\documentclass[leqno,14pt]{amsart}

\usepackage[left=2.3cm, right=2.3cm,top=2.6cm,bottom=2.6cm]{geometry}

\usepackage{times}
\usepackage[all]{xy}
\usepackage{tikz}
\usepackage{tikz-cd}
\usepackage[scr=boondox, cal=esstix]{mathalpha}
\usepackage{pgfplots}
\usepgfplotslibrary{fillbetween}

\usepackage{subcaption}

\usetikzlibrary{arrows}

\usepackage{amsmath, amssymb, amsfonts, amsthm}
\usepackage{mathrsfs}

\usepackage[colorlinks=true,citecolor=blue, urlcolor=blue, linkcolor=blue, pagebackref]{hyperref}

\usepackage[capitalise]{cleveref} 
\usepackage{comment} 

\usepackage{graphicx}
\usepackage{wrapfig}
\usepackage{longtable}
\usepackage{enumerate}
\usepackage{nicematrix}
\usepackage{bbm}
\usepackage{subcaption}

\usepackage{tkz-euclide} 
\tikzset{
  dotted/.style={pattern=dots,pattern color=#1},
  dotted/.default=black
}
\tikzset{
  fdotted/.style={pattern=crosshatch dots,pattern color=#1},
  fdotted/.default=black
}
\tikzset{
  scopedlines/.style={pattern=north east lines,pattern color=#1},
  scopedlines/.default=black
}
\tikzset{
  hrlines/.style={pattern=horizontal lines,pattern color=#1},
  hrlines/.default=black
}

\makeindex

\def\A{\ensuremath{\mathbb{A}}}
\def\C{\ensuremath{\mathbb{C}}}

\def\N{\ensuremath{\mathbb{N}}}
\def\P{\ensuremath{\mathbb{P}}}

\def\R{\ensuremath{\mathbb{R}}}
\def\V{\ensuremath{\mathbb{V}}}
\def\Z{\ensuremath{\mathbb{Z}}}

\def\cH{\ensuremath{\mathcal H}}

\def\aa{\ensuremath{\mathbf a}}

\def\uu{\ensuremath{\mathbf u}}
\def\vv{\ensuremath{\mathbf v}}

\def\ee{\ensuremath{\mathbf e}}

\def\deg{\mathop{\mathrm{deg}}\nolimits}
\def\dim{\mathop{\mathrm{dim}}\nolimits}

\def\gr{\mathop{\mathrm{Gr}}\nolimits}

\def\Hilb{\mathop{\mathrm{Hilb}}\nolimits}

\def\sHilb{\mathop{\mathcal{Hilb}}\nolimits}
\def\Sym{\mathop{\mathrm{Sym}}\nolimits}
\def\sp{\mathrm{Spec}}

\newtheorem{Thm}{Theorem}[section]
\newtheorem{Prop}[Thm]{Proposition}
\newtheorem{PropDef}[Thm]{Proposition and Definition}
\newtheorem{Lem}[Thm]{Lemma}
\newtheorem{Cor}[Thm]{Corollary}

\newtheorem{Ques}{Question}

\newtheorem*{Ques*}{Question}

\newtheorem{ThmLetter}{Theorem}

\theoremstyle{definition}
\newtheorem{Def}[Thm]{Definition}
\newtheorem{Rem}[Thm]{Remark}

\newtheorem{Ex}[Thm]{Example}

\makeatletter
\newtheoremstyle{italicsname}
 {3pt}
 {3pt}
 {\itshape}
 {}
 {\itshape}
 {.}
 {.5em}
 {\thmname{#1}\thmnumber{\@ifnotempty{#1}{ }#2}%
 \thmnote{ {\the\thm@notefont(#3)}}}
\makeatother
\theoremstyle{italicsname}

\title[Hilbert schemes of points on fold-like curves and their combinatorics]{Hilbert schemes of points on fold-like curves and their combinatorics}

\author[Ángel David Ríos Ortiz, Javier Sendra Arranz]{Ángel David Ríos Ortiz, Javier Sendra Arranz}

\address{Université Paris Cité and Sorbonne Université, CNRS, IMJ-PRG, F-75013 Paris, France}
\email{riosortiz@imj-prg.fr}
\address{Department of Mathematics, CUNEF Universidad, Madrid, Spain}
\email{ javier.sendraarranz@cunef.edu}

\usepackage{pgfplots}
\pgfplotsset{
compat=newest,
every axis plot/.append style={no marks,thick},
every axis/.style={
  axis lines=middle,
  width=7cm,
  height=3cm,
  }
}

\begin{document}

\begin{abstract} \noindent 
We investigate the Hilbert scheme of points on curves with 
n-fold singularities, that is curves that look locally around their singular points as the axis in an affine space. We describe the structure and number of its irreducible components, and provide a detailed analysis of their singularities, revealing rich combinatorial patterns governing its geometry.
\end{abstract}
\maketitle
\tableofcontents

\section*{Introduction}

The Hilbert scheme of $m$ points on a variety $X$ over $\C$, which in this paper will be denoted by $\sHilb^m(X)$, is one of the most studied objects in algebraic geometry because it not only parametrizes how points in $X$ behave when they start to collide, but because of its beauty and complexity. Starting from curves, the geometry of Hilbert schemes of points tends to be very rich. In the case of smooth curves, these Hilbert schemes are isomorphic to the  symmetric product of the curve, which is smooth. However, when we allow singularities on the curves, their geometry becomes much more complicated, cf. for example \cite{KivinenReducibleLocallyPlanar,luan2023irreduciblecomponentshilbertscheme} and references therein. If an integral curve has \textit{locally planar singularities}, meaning that it is contained in a smooth surface, then the Hilbert scheme of points is irreducible \cite{AltmanIarrobinoKleinman,BGS1981}, and its singularities are in a deep relation with the compactified Jacobian of the curve, see e.g. \cite{Esteves2001,MRV2017} and the survey \cite{Migliorini2020} for recent applications to knot theory. 

When the curve ceases to be locally planar, there has been growing interest in obtaining new \emph{invariants} \cite{AgostonNemethiLatticeCohomology, NemethiScheflerLatticeCohomology} that might serve as substitutes for the topological invariants which are effective in the case of plane curve singularities. In this vein, the study of the geometry of the Hilbert scheme of points associated with these more general singularities becomes particularly meaningful.

One of the main features of Hilbert schemes of points is the so called \emph{Murphy's law} \cite{VakilMurphy,Jelisiejew2020} that essentially says that, in general, it is out of reach to understand them. Therefore, finding explicit descriptions for the Hilbert schemes of points for specific varieties usually yields to a good amount of geometry. Ran in \cite{Ran2005} studied the case of nodal singularities on curves, he gave a very precise description of the Hilbert scheme of points of an irreducible curve with nodal singularities, describing completely their structure. The case of curves which are not contained in smooth surfaces, to the knowledge of the authors, has not been explored yet. 

In this work, we address the Hilbert scheme of points for a class of curves with \textit{rational $n$--fold singularities}, for which we found fascinating geometry and combinatorics. Given a reduced curve $C$, a point $p\in C$ is a rational $n$--fold singularity of $C$ if locally around $p$, the curve $C$ is analytically isomorphic to the union of the axis in $\C^n$. Nodal singularities are the case $n=2$ and when $n\geq 3$ they are no longer locally planar. Rational $n$--fold singularities have been studied because, as nodal singularities, they are \emph{semi-normal} \cite{BombieriSeminormality,DavisSeminormality}. In a very recent work \cite{HanKassSatrianoExtendingTorelli} the authors construct an alternative compactification of the moduli space of curves by adding \emph{stable and separating} fold-like curves, cf. \cite[Theorem 1.2]{HanKassSatrianoExtendingTorelli}, see also \cite{SmythModularCompactifications}.

With this motivation, and with the aim of describing explicitly their compactified Jacobians, we need to first study their Hilbert scheme of points. Suppose now that $C$ is an irreducible curve whose unique singularity is a rational $n$--fold singularity. One of our main results is a precise characterization of the irreducible components of its Hilbert scheme of points.

\begin{ThmLetter}\label{thm:irred comp}
    Let $C$ be an irreducible curve with a unique rational $n$--fold singularity and denote by $C_{\mathrm{sm}}$ its smooth locus. The irreducible components of $\sHilb^m(C)$ are birational to
    \[
        \begin{array}{ccc}
            \sHilb^m(C_\mathrm{sm})&  \text{ and } & 
            \sHilb^{m-m'}(C_{\mathrm{sm}}) \times \mathrm{Gr}(n+1-m',n)\,\,\, \text{ for } 2\leq m'\leq\min\{m,n-1\}.
        \end{array}
    \]    
In particular, the number of irreducible components of $\sHilb^m(C)$ is $\mathrm{min}\{n-1,m\}$. 
\end{ThmLetter}
A direct consequence of \cref{thm:irred comp} is that the number of irreducible components of $\sHilb^m(C)$ is $n-1$ as long as $m\geq n-1$ (see  \cref{fig:graph irred comp}).
As far as the authors are aware, this phenomenon is rather unexpected, since it is typically observed that as the number of points increases, a non-irreducible Hilbert scheme of points tends to have an increasing  number of components.  Notice also that there exists components of different dimensions whenever $n\geq 4$ and $m\geq 2$. So in these cases, the Hilbert scheme of points is not  Cohen-Macaulay by \cite[Corollary 18.11]{Eisenbud1995}.

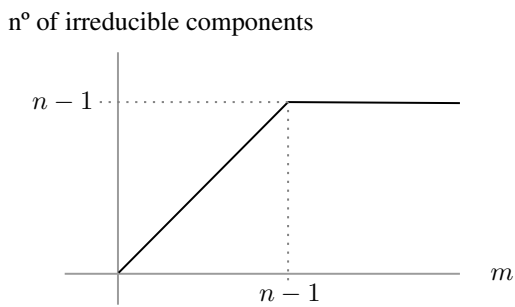
\begin{figure}[h]
    \centering
\tikzset{every picture/.style={line width=0.75pt}} 

\begin{tikzpicture}[x=0.65pt,y=0.65pt,yscale=-1,xscale=1]

\draw [color={rgb, 255:red, 155; green, 155; blue, 155 }  ,draw opacity=1 ]   (80,200) -- (310,200) ;
\draw [color={rgb, 255:red, 155; green, 155; blue, 155 }  ,draw opacity=1 ]   (111,71) -- (111,218) ;
\draw [color={rgb, 255:red, 0; green, 0; blue, 0 }  ,draw opacity=1 ]   (210,100) -- (111.4,199.4) ;
\draw [color={rgb, 255:red, 0; green, 0; blue, 0 }  ,draw opacity=1 ]   (210,100) -- (310,100.6) ;
\draw [color={rgb, 255:red, 130; green, 130; blue, 130 }  ,draw opacity=1 ] [dash pattern={on 0.84pt off 2.51pt}]  (100.4,100) -- (210,100) ;
\draw [color={rgb, 255:red, 130; green, 130; blue, 130 }  ,draw opacity=1 ] [dash pattern={on 0.84pt off 2.51pt}]  (210,203.8) -- (210,100) ;

\draw (191.5,204.4) node [anchor=north west][inner sep=0.75pt]    {$n-1$};
\draw (59.3,92.4) node [anchor=north west][inner sep=0.75pt]    {$n-1$};
\draw (326,195.4) node [anchor=north west][inner sep=0.75pt]    {$m$};
\draw (45.5,45.6) node [anchor=north west][inner sep=0.75pt]    {nº of irreducible components};

\end{tikzpicture}

    \caption{Graph of the number of irreducible components of the Hilbert scheme of $m$ points on a curve with an $n$-fold singularity.}
    \label{fig:graph irred comp}
\end{figure}

The strategy to prove \cref{thm:irred comp} is to calculate the \emph{elementary components} of $\sHilb^m(C)$. To do so, we first study the locus in $\sHilb^m(C)$ of those subschemes supported at the singularity, the so-called \emph{punctual Hilbert scheme} $\sHilb_p^m(C)$. Since this is a local problem we can assume that we are studying the axis in $\C^n$, which we denote by $X_n$, and the singular point to be the origin $\mathbf{0}$. We classify the possible ideals that appear in $\sHilb_\mathbf{0}^m(X_n)$. From there, we obtain the irreducible components of $\sHilb^m_\mathbf{0}(X_n)$ and also the identification with the corresponding Grassmannians.

With this observation in place, we proceed to examine the structure of the Hilbert scheme of points. For this purpose, we make use of the combinatorial structures associated with these curves; in particular, we derive the following result concerning $\sHilb_\mathbf{0}^m(X_n)$. 

\begin{ThmLetter}\label{thm:moment map and Grassmannians}
The punctual Hilbert scheme $\sHilb_\mathbf{0}^m(X_n)$ is a union of Grassmanians of the form $\gr(l,n)$ for $\max\{1,n+1-m\}\leq l\leq n-1$, where $\gr(l,n)$ appears $\binom{l+m-2}{n-1}$--times. Moreover,
there is a well-defined moment map     
        $$\mu_m:\sHilb^m_{\mathbf{0}}(X_n)\longrightarrow (m-1)\cdot\Delta_{n-1}. $$ 
\end{ThmLetter}

The moment map allows us to study the geometry of $\sHilb_\mathbf{0}^m(X_n)$ through the combinatorics of a \emph{hypersimplicial complex} in $(m-1)\cdot\Delta_{n-1}$.
The strategy for the proof of \cref{thm:moment map and Grassmannians} is motivated by the natural toric action on  $X_n$. We identify which irreducible components of $\sHilb_\mathbf{0}^m(X_n)$ are lifted to \emph{elementary components} of $\sHilb^m(X_n)$ and relate them to the combinatorics of the moment map. To this end, we carry out a detailed analysis of the singularities and scheme structure of the local Hilbert scheme, employing deformation theory and the natural torus action to extend our results to the global setting. This leads to our other main theorem, which provides a detailed description of the singularities occurring in its components. Here, the combinatorics developed previously play a prominent role.

\begin{ThmLetter}\label{thm: singularities HilbX_n}
The Hilbert scheme $\sHilb^m(X_n)$ is reduced. Moreover, each smoothable component of $\sHilb^m(X_n)$ is normal and has toric singularities. Each non-smoothable component is smooth.  
\end{ThmLetter}

When we replace $X_n$ by an \emph{irreducible} curve $C$ with a unique rational $n$--fold singularity, the smoothable component is no longer normal and its singularities are not toric. Moreover, the non-smoothable components are singular. However, their structure can be made explicit. This yields our final main result.

\begin{ThmLetter}\label{thm: singularities HilbC}
Let $C$ be an irreducible curve with a unique rational $n$--fold singularity. Then, $\sHilb^m(C)$ is reduced.  The singularities of the smoothable components are locally unions of normal toric varieties, while those of the non-smoothable components are locally unions of affine spaces. 
In addition, there is an explicit description of the normalization of the non-smoothable components.
\end{ThmLetter}

\subsection*{Structure of the paper}  
In \cref{sec:irred comp punctual} the main result is \cref{punctual irred comp}, where we classify the irreducible components of $\sHilb^m_{\mathbf{0}}(X_n)$. In \cref{sec:moment map} we construct a moment map for $\sHilb^m_{\mathbf{0}}(X_n)$ leading to \cref{thm: moment map}, obtaining \cref{thm:moment map and Grassmannians} for the reduced structure of $\sHilb_\mathbf{0}^m(X_n)$. In this section, we also explore the relation between the geometry of $\sHilb^m_{\mathbf{0}}(X_n)$ and a hypersimplicial complex. Some of the combinatoric lemmas needed for this purpose are given in \cref{app: The hypersimplicial complex}. In \cref{sec:localtoglobal}, we prove \cref{co: irred comp reduced structure} that establishes \cref{thm:irred comp} for the reduced structure of $\sHilb^m(C)$. In \cref{sec:local hilbert scheme} we deduce \cref{theo: reduced} that shows that the punctual Hilbert scheme is reduced. This completes the proof of \cref{thm:moment map and Grassmannians}. Then by \cref{sec:localtoglobal} we obtain that the same happens for the whole Hilbert scheme of points, completing the proof of \cref{thm:irred comp}. This makes use of some amount of commutative algebra computations which are given in \cref{app: commutative algebra}. From there, in \cref{sec:singularites and local hilbert scheme} we start studying the singular locus of $\sHilb^m(X_n)$ where, by using combinatoric methods explained in \cref{app: The hypersimplicial complex}, we characterize it completely. The main result of this section is \cref{theo:sing locus}. 
Afterwards in \cref{sec:sm and non sm} we focus on the description of the smoothable and non-smoothable components.  Propositions \ref{prop: nonsmooth comp Xn} and \ref{prop:sing smooth comp} complete the proof of Theorem \ref{thm: singularities HilbX_n}. Theorem \ref{thm:normalization} and Corollaries \ref{cor: sing nonsmooth} and \ref{cor: sing smooth irred} lead to Theorem \ref{thm: singularities HilbC}.
Finally in \cref{sec:ongoing work and future questions} we report some ongoing work and state some open questions and future research directions.

The theory presented in this paper is complemented by a variety of examples, intended to offer deeper insight into the problems under consideration and to illustrate the geometric structures that emerge from them.

\subsection*{Acknowledgments} The authors are grateful to Daniele Agostini, Marie Brandenburg, Michele Graffeo, Joachim Jelisiejew, Christian Lehn, Bernd Sturmfels, and to the people at the MPI-MiS Leipzig for their interest, useful conversations, remarks and for pointing out relevant literature on the topics treated in this paper.
\subsection*{Funding} 
Rios Ortiz was supported by the European Research Council (ERC) under the European Union’s Horizon 2020 research and innovation programme (ERC-2020-SyG-854361-HyperK).

\subsection*{Notations} We will work over $\C$. All arguments remain valid over any algebraically closed field of characteristic zero. The case of positive characteristic is unknown to the authors. Vectors in $\C^n$ are written in boldface. The origin in $\C^n$ is denoted by $\mathbf{0}$ and the vector $(1,\dots,1)$ as $\mathbf{1}$. The standard basis for $\C^n$ is $\mathbf{e}_1,\dots,\mathbf{e}_n$. For a subset $S\subset [n]$, we set
\[
\mathbf{e}_S:=\sum_{i\in S}\mathbf{e}_i.
\]
As is customary, $\binom{[n]}{l}$ denotes the set of subsets of $\{1,\ldots,n\}$ with $l$ elements.

\begin{itemize}
    \item $X_n$: The union of the axis $L_1,\ldots,L_n$, where $L_i:=\mathbb{V}(x_j:j\neq i)\subseteq \C^n$ (\cref{def: X_n}).
    \item $R_n$: The coordinate ring of $X_n$ (\cref{def:R_n}).
    \item $\Sym^m(C)$: The $m$--th symmetric product of $C$ (\cref{rem: number com X_n non smooth}).
    \item $\Hilb^m(X_n)$: The reduced Hilbert scheme of $m$ points of $X_n$ (\cref{def:punctualHilbertscheme}).
    \item $\sHilb^m_{\mathbf{0}}(X_n)$: The  punctual Hilbert scheme at the origin (\cref{def:punctualHilbertscheme}).
    \item $\Hilb^m_{\mathbf{0}}(X_n)$: The reduced  punctual Hilbert scheme at the origin (\cref{def:punctualHilbertscheme}).
    \item $\Sigma(m,l,\uu)$: A subvariety of $\Hilb^m_{\mathbf{0}}(X_n)$ constructed by an integer $l\leq n$ and a partition $\uu$ (\cref{def:irred comp punct}).
    \item $\Lambda_{\uu}$: The vector subspace generated by monomials indexed by $\uu$ (\cref{def:irred comp punct}).
    \item $\langle\Gamma\rangle$: The ideal generated by the vector subspace $ \Gamma\subseteq\Lambda_{\uu}$ (\cref{eq:grass map}). 
    \item $\kappa(\mathbf{w},k)$: The function that gives the indexes $i$ for which $\mathbf{w}_i=k$ (\cref{def:kappa indexes}).
    \item $\mu_{l,n}$: The moment map (\cref{def:momentmap}).
    \item $\mu_{\uu,l}$: The moment map defined on $\Sigma(m,l,\uu)$  (\cref{eq:piece of moment map}).
    \item $\mu_m$: The moment map defined on $\Hilb^m_{\mathbf{0}}(X_n)$ (\cref{thm: moment map}).
    \item $\mathcal{K}_n^{[m]}$: The $(n,m)$--hypersimplicial complex (\cref{prop: hypersimplicial complex}).
    \item $\mathcal{K}_n^{[m]}(S_1,S_2,l,\uu)$: The faces of the complex $\mathcal{K}_n^{[m]}$ (\cref{eq:hyper faces}).
    \item $\mathcal{K}_{l,n}^{[m]}$: subcomplex of $\mathcal{K}_n^{[m]}$ formed by hypersimplices of the form $\Delta_{l,n}$.
    \item $\mathcal{G}_n^m$: The variety obtained by gluing Grassmannians following $\mathcal{K}_n^{[m]}$ (\cref{eq:disjoint union}).
\item $\mathcal{G}_{l,n}^m$: The subvariety of  $\mathcal{G}_n^m$ formed by the Grassmannians of the form $\gr(l,n)$.
\item $\Hilb^{m,m'}(X_n)$: The non-smoothable components of $\Hilb^m(X_n)$ (\cref{eq:induc map gen}).
\item $\mathbb{S}_k$ and $\widehat{\mathbb{S}}_k$: The simplicial complexes describing the singularities of $\sHilb^m(X_n)$ (\cref{prop:complex singularity}).
    \item $\sHilb^{m,m',u}(C)$: The strata of the non-smoothable components of $\sHilb^m(C)$ (\cref{eq:strata component}).

\end{itemize}

\section{Irreducible components of the punctual Hilbert scheme}\label{sec:irred comp punctual}
A classical strategy to analyze the irreducible components of Hilbert schemes of points is to focus on elementary components. Following \cite{IarrobinoElementaryComponents}, an elementary component is an irreducible component of the Hilbert scheme of $m$ points that parameterizes subschemes supported at a single point.  In the case of a curve $C$ with a rational $n$--fold singularity $p\in C$, an elementary component must parametrize length $m$ subschemes supported at the singularity of $p$. We start this section defining this type of singularities.

\begin{Def}\label{def: X_n}
Let $X_n$ be the union of the axis $L_1,\ldots,L_n$ of $\mathbb{C}^n$, where $L_i:=\mathbb{V}(x_j:j\neq i)$ and let $\mathbf{0} = (0,\ldots,0)\in\C^n$ be the singular point of $X_n$. A curve $C$ has a rational $n$--fold singularity at $p\in C$ if, locally around $p$, $C$ is analytically isomorphic to $X_n$ around $\mathbf{0}$.
\end{Def}
Algebraically, $p\in C$ is a rational $n$--fold singularity if there exists an isomorphism between the completed stalks
$
\widehat{\mathcal{O}}_{C,p}\simeq \widehat{\mathcal{O}}_{X_n,\mathbf{0}}.
$
We denote the coordinate ring of $X_n$ by $R_n$, which is defined by
\begin{equation}\label{def:R_n}
R_n = \C[x_1,\ldots,x_n]/\langle x_ix_j:1\leq i<j\leq n\rangle.  
\end{equation}

A rational $2$--fold singularity is a nodal singularity. However, for $n\geq 3$ such singularities have embedding dimension $n$; in particular, they can no longer be embedded in a smooth surface.
By the very definition of elementary components, we can replace $C$ and $p$ by $X_n$ and $\mathbf{0}$, respectively. In this section, we analyze ideals in $\sHilb^m(X_n)$ supported at $\mathbf{0}$, and then extend the analysis to the full Hilbert scheme. We first perform this analysis on the reduced Hilbert scheme of points $\Hilb^m(X_n)$, which is $\sHilb^m(X_n)$ endowed with its reduced structure.

\begin{Def}\label{def:punctualHilbertscheme}
    The punctual Hilbert scheme at $\mathbf{0}$, denoted by $\sHilb^m_{\mathbf{0}}(X_n)$, is the locus in $\sHilb^m(X_n)$ of length $m$ ideals in $R_n$ supported at $\mathbf{0}$. The variety $\Hilb^m_{\mathbf{0}}(X_n)$ is the punctual Hilbert scheme endowed with its \emph{reduced structure}.
\end{Def}

Notice that an elementary component of  $\sHilb^m(X_n)$ is an irreducible component of 
$\sHilb^m_{\mathbf{0}}(X_n)$. We first study the irreducible components of $\sHilb^m_{\mathbf{0}}(X_n)$, which can be done within $\Hilb^m_{\mathbf{0}}(X_n)$. In \cref{sec:localtoglobal}, we determine which of these lift to elementary components of $\sHilb^m(X_n)$.

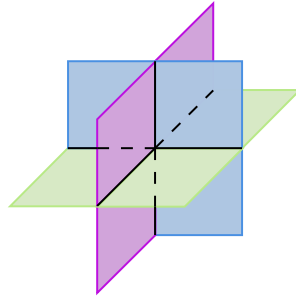
\begin{figure}
    \centering
    
\tikzset{every picture/.style={line width=0.75pt}} 

\begin{tikzpicture}[x=0.55pt,y=0.55pt,yscale=-1,xscale=1]

\draw  [color={rgb, 255:red, 184; green, 233; blue, 134 }  ,draw opacity=1 ][fill={rgb, 255:red, 215; green, 233; blue, 197 }  ,fill opacity=1 ][line width=0.75]  (209.61,100) -- (269.53,100) -- (229.8,140.02) -- (169.88,140.02) -- cycle ;
\draw  [color={rgb, 255:red, 74; green, 144; blue, 226 }  ,draw opacity=1 ][fill={rgb, 255:red, 174; green, 198; blue, 226 }  ,fill opacity=0.99 ] (209.76,40.16) -- (269.6,40.16) -- (269.6,100) -- (209.76,100) -- cycle ;
\draw  [color={rgb, 255:red, 74; green, 144; blue, 226 }  ,draw opacity=1 ][fill={rgb, 255:red, 174; green, 198; blue, 226 }  ,fill opacity=0.99 ] (269.6,100) -- (329.44,100) -- (329.44,159.84) -- (269.6,159.84) -- cycle ;
\draw  [color={rgb, 255:red, 184; green, 233; blue, 134 }  ,draw opacity=1 ][fill={rgb, 255:red, 215; green, 233; blue, 197 }  ,fill opacity=1 ][line width=0.75]  (309.33,59.98) -- (369.24,59.98) -- (329.52,100) -- (269.6,100) -- cycle ;
\draw  [color={rgb, 255:red, 189; green, 16; blue, 224 }  ,draw opacity=1 ][fill={rgb, 255:red, 209; green, 164; blue, 218 }  ,fill opacity=1 ] (309.6,0.36) -- (309.6,60.24) -- (269.6,100) -- (269.6,40.13) -- cycle ;
\draw  [color={rgb, 255:red, 189; green, 16; blue, 224 }  ,draw opacity=1 ][fill={rgb, 255:red, 209; green, 164; blue, 218 }  ,fill opacity=1 ] (269.8,100) -- (269.8,159.87) -- (229.8,199.64) -- (229.8,139.76) -- cycle ;
\draw  [color={rgb, 255:red, 189; green, 16; blue, 224 }  ,draw opacity=1 ][fill={rgb, 255:red, 209; green, 164; blue, 218 }  ,fill opacity=1 ] (269.8,40.36) -- (269.8,100.24) -- (229.8,140) -- (229.8,80.13) -- cycle ;
\draw  [color={rgb, 255:red, 74; green, 144; blue, 226 }  ,draw opacity=1 ][fill={rgb, 255:red, 174; green, 198; blue, 226 }  ,fill opacity=0.99 ] (269.6,40.16) -- (329.44,40.16) -- (329.44,100) -- (269.6,100) -- cycle ;
\draw  [color={rgb, 255:red, 184; green, 233; blue, 134 }  ,draw opacity=1 ][fill={rgb, 255:red, 215; green, 233; blue, 197 }  ,fill opacity=1 ][line width=0.75]  (269.8,100.24) -- (329.72,100.24) -- (289.99,140.25) -- (230.07,140.25) -- cycle ;
\draw    (229.8,139.76) -- (269.6,100) ;
\draw    (269.6,100) -- (269.6,40.16) ;
\draw    (269.6,100) -- (329.44,100) ;
\draw    (269.6,140) -- (269.6,159.81) ;
\draw    (209.61,100) -- (229.8,100) ;
\draw  [dash pattern={on 4.5pt off 4.5pt}]  (229.8,100) -- (269.6,100) ;
\draw  [dash pattern={on 4.5pt off 4.5pt}]  (269.6,100) -- (269.6,140) ;
\draw  [dash pattern={on 4.5pt off 4.5pt}]  (269.6,100) -- (309.6,60.24) ;

\end{tikzpicture}

    \caption{Planes spanned by each pairs of lines in $X_3$.}
    \label{fig:tangent star}
\end{figure}

\begin{Ex}\label{ex:m=2}
    For $m=2$, $\Hilb^2_\mathbf{0}(X_n)=\sHilb^2_\mathbf{0}(X_n)$ is the projectivization of the tangent space of $X_{n}$ at $\mathbf{0}$. In particular, $\Hilb^2_\mathbf{0}(X_n)$  is  $\P(T_{\mathbf{0}}X_n)\simeq\P^{n-1}$ and we can identify ideals in $\Hilb^2_\mathbf{0}(X_n)$ with the tangent directions at $\mathbf{0}$. The intersection of $\Hilb^2_\mathbf{0}(X_n)$  and the smoothable component consists of all tangent directions at $\mathbf{0}$ that lie in the planes spanned by two of the lines $L_1,\ldots,L_n$. This is also called the \emph{tangent star} of $X_{n}$ (see \cref{fig:tangent star}).  We refer to \cite[Chapter 1]{RussoOnTheGeometryOfSomeSpecialVarieties} for the general definition of the tangent star.
    In particular,  $\Hilb^2_\mathbf{0}(X_n)\simeq\P^{n-1}$ must be an irreducible component of $\Hilb^2(X_n)$. This fact can also be derived by a dimension argument.
\end{Ex}

We now give a description of the generators of the ideals in $\Hilb^m_{\mathbf{0}}(X_n)$.

\begin{Prop}\label{ideals punctual}
    Let $l\in[n]$ and $\uu\in\Z_{\geq 1}^n$ be a strictly positive partition of $m+l-1$. Consider a full rank matrix $A$ of size $l\times n$ with no vanishing column. Then, the ideal generated by the polynomials
    \[
    \begin{pmatrix}
    f_1\\ \vdots \\ f_l
    \end{pmatrix} = A\,\,\begin{pmatrix}
    x_1^{u_1}\\ \vdots\\ x_n^{u_n}
    \end{pmatrix}.
    \]
    lies in $\Hilb^m_{\mathbf{0}}(X_n)$. Moreover, all ideals in $\Hilb^m_{\mathbf{0}}(X_n)$ are of this form.
\end{Prop}
\begin{proof}
 Let $J$ be an ideal as above. First we check that $J$ is supported at $\mathbf{0}$. Since $A$ has not vanishing columns, for every $1\leq i\leq n$  there exists  $1\leq j\leq l$ such that the entry $A_{j,i}$ of $A$ is nonzero. Therefore it holds  $x_if_j=A_{j,i}x_i^{u_i+1}$. We deduce that $J$ contains $x_1^{u_1+1},\ldots,x_n^{u_n+1}$ and therefore, $J$ is supported at $\mathbf{0}$. We are left to show that $J$ has length $m$. Since $x_1^{u_1+1},\ldots,x_n^{u_n+1}\in J$, the quotient $R_n/J$ is generated by $1,x_1,\ldots,x_1^{u_1},\ldots,x_n,\ldots,x_n^{u_n}$. The generators of $J$ induce $l$ linearly independent linear relations among  $x_1^{u_1},\ldots,x_n^{u_n}$. We deduce that 
\[
\dim R_n/J = 1+u_1+\cdots+u_n-l=|\uu|+1-l=m
\]
as claimed. 

We will now prove the second part of the proposition. Let $J$ be an ideal of $R_n$ supported at $\mathbf{0}$ of length $m$, and let $f_1,\ldots,f_l$ be minimal generators of $J$. Since $J$ is supported at $\mathbf{0}$, then $f_1,\ldots,f_l$ do not have independent term. We write each $f_i$ as 
    \[
    f_i= f_{i,1}(x_1) +\cdots +f_{i,n}(x_n),
    \]
    where $f_{i,j}(x_i)$ is a polynomial in $x_i$ such that $f_{i,j}(0)=0$. In the case of $x_1$ we get
    \[
    f_{i,1}=\sum_{k=1}^{d_{i}} 
    a_{i,j}x_1^j,
    \]
     where $d_{i}=\deg f_{i,1}$ and $a_{i,j}\in\C$. Notice that if $f_{i,1}$ vanishes for all $i$, then the dimension of $R_n/J$ as a $\C$--vector space is infinite, and this implies that $J$ does not have finite length. Therefore we can assume that $a_{1,d_{1}}$ is nonzero.  Now, in $R_n$ we have that
    \[
    x_1f_1=x_1f_{i,1}=\sum_{j=1}^{d_1} a_{1,j}x_1^{j+1}
    \]
    is in $J$. Since $J$ is supported at $\mathbf{0}$, the only root of $x_1f_{i,1}$ is $0$. We conclude that $f_{i,1}$ is the monomial $a_{i,d_i}x_1^{d_i+1}$. We can assume  $d_1$ is the minimum between those $d_1,\ldots,d_n$ that are nonzero. By replacing $f_i$ by $a_{1,d_1}f_i-a_{i,d_i}x_1^{d_i-d_1}f_1$, we may assume that $d_1=d_2=\cdots=d_n$.

    Let $d_i$ be the degree of $f_{i,1}$ and suppose that there exists $2\leq i\leq l$ such that $d_{i}>d_{1}$. In $R_n$ we have that the monomial
    \[
    x_1^{d_{i}-d_{1}}f_1=x_1^{d_{i}-d_{1}}f_{1,1}=\sum_{j=1}^{d_1} a_{1,j}x_1^{j+d_{i}-d_{1}}
    \]
    is in $J$. Since $J$ is supported at $\mathbf{0}$, the only root of $x_1^{d_{i}-d_{1}}f_{1,1}$ is $0$. We conclude that $f_{i,1}$ is the monomial $a_{1,d_1}x_1^{d_1}$. Repeating this process with the variables $x_2,\cdots,x_n$ we get that there exist $d_1,\ldots,d_n$ such that $f_1,\ldots,f_n$ are a linear combination of $x_1^{d_1},\ldots,x_n^{d_n}$.
    
\end{proof}

Notice that in the extremal case $l=n$ in \cref{ideals punctual} leads to the ideal $\langle x_1^{u_1}, \ldots, x_n^{u_n}\rangle$.

\begin{Def}\label{def:irred comp punct}
    Given $\uu\in\Z_{\geq 1}^n$, let $\Lambda_{\uu}:= \langle x_1^{u_1},\ldots,x_n^{u_n}\rangle_{\C}$ be the $\C$--vector space generated by $ x_1^{u_1},\ldots,x_n^{u_n}$. 
    Fix $m\geq 1$. For $l\in [n]$ such that $|\uu|=m+l-1$, the subvariety $\Sigma(m,l,\uu)\subseteq \Hilb^m_{\mathbf{0}}(X_n)$ is the closure of the ideals of the form $\langle f_1,\ldots,f_l\rangle$ where $f_1,\ldots,f_l$ are linearly independent elements of $\Lambda_{\uu}$.
\end{Def}

With the notation introduced above we immediately obtain the following.

\begin{Cor}\label{cor:big union}
    Let $m,n$ be positive integers, then there is a decomposition 
\[
\displaystyle \Hilb^m_{\mathbf{0}}(X_n)=\bigcup_{l=1}^{n}\bigcup_{{\tiny\begin{array}{c}\uu\in\Z_{\geq 1}^n\\ |\uu|=m+l-1\end{array}}} \Sigma(m,l,\uu).
\]
\end{Cor}
Hence the varieties $ \Sigma(m,l,\uu)$ are the candidates to be irreducible components of the punctual Hilbert scheme. Next, we will describe the geometry of these varieties. For any $\Gamma\in\gr(l,\Lambda_{\uu})$, let $\langle\Gamma\rangle$ be the ideal generated by $\Gamma$ in $R_n$. Define the rational map 
\begin{equation}\label{eq:grass map}
    \begin{array}{cccc}
     \varphi_{l,\uu}:&\gr(l,\Lambda_{\uu})&\dashrightarrow& \Sigma(m,l,\uu)\\
      & \Gamma&\mapsto&\langle\Gamma\rangle.
    \end{array}
\end{equation}  
Since $m+l-1=|\uu|$, by \cref{ideals punctual} the map $\varphi_{l,\uu}$ is well--defined in an open subset of the Grassmannian.

\begin{Lem}\label{lem:base locus}
    The base locus of $\varphi_{l,\uu}$ is contained in the union
    \begin{equation}\label{eq:base locus}
    \bigcup_{i=1}^n\mathcal{H}_i,    
    \end{equation}
    where $\mathcal{H}_i := \{\Gamma\in\gr(l,\Lambda_{\uu}):\Gamma\subseteq \langle x_j^{u_j}:j\neq i\rangle_{\C}\}$.
\end{Lem}
\begin{proof}
    If $\Gamma\in \gr(l,\Lambda_{\uu})$, then $\langle\Gamma\rangle$ is an ideal of $R_n$ supported at $\mathbf{0}$. Hence, the base locus of $\varphi_{l,\uu}$ coincide with the locus of $\Gamma\in \gr(l,\Lambda_{\uu})$ such that $\langle\Gamma\rangle$ is not a length $m$ ideal. Now, if $\Gamma\in \cH_i$ for $1\leq i\leq n$, then  $\langle\Gamma\rangle$ has no finite length. Assume now that $\Gamma\not \in\cH_i$ for all $1\leq i\leq n$. Then, $\langle\Gamma\rangle$ is generated by $l$ polynomials $f_1,\ldots,f_l$ of the form 
    \[
    f_j=\sum_{i=1}^na_{i,j}x_j^{u_i}.
    \]
    Since $\Gamma\not \in\cH_i$ for all $i$, we deduce that for any $1\leq i\leq n$, there exists $1\leq j\leq l$ such that $a_{i,j}\neq 0$. In particular,  $x_1^{u_1+1},\ldots,x_n^{u_n+1}$ are contained in $\langle\Gamma\rangle$, and hence, $\langle\Gamma\rangle$ has finite length. Moreover, this length is given by $u_1+\cdots+u_n+1-l=|\uu|+1-l=m$. We conclude that $\varphi_{l,\uu}$ is well--defined away from \eqref{eq:base locus}.
\end{proof}

In the next lemma we will show that the map $\varphi_{l,\uu}$ can be extended along \eqref{eq:base locus}.

\begin{Lem}\label{lem:extension}
    With the notation of \cref{lem:base locus}, let $1\leq k\leq n$ and assume $\Gamma\in\cH_{i_1}\cap\cdots\cap\cH_{i_k}$ where $1\leq i_1<\cdots<i_k\leq n$ and $\Gamma\not\in\cH_j$ for $j\neq i_1,\ldots,i_k$. Then 
    \[
    \varphi_{l,\uu}(\Gamma) := \langle\Gamma\rangle+\langle x_{i_1}^{u_{i_1}+1},\ldots,x_{i_k}^{u_{i_k}+1}\rangle
    \]
    extends $\varphi_{l,\uu}$ to all $\gr(l,\Lambda_{\uu})$.
\end{Lem}
\begin{proof}
    Let $J$ be an ideal in the image of $\varphi_{l,\uu}$. By construction, $x_1^{u_1+1},\ldots,x_n^{u_n+1}$ are contained in $J$. In other words, $\langle x_1^{u_1+1},\ldots,x_n^{u_n+1}\rangle$ is contained in $J$. Since this containment is a closed condition in $\Hilb^m(X_n)$, we deduce that for any $J$ in the closure of the image of $\varphi_{l,\uu}$, we have that $\langle x_1^{u_1+1},\ldots,x_n^{u_n+1}\rangle\subseteq J$.  
    
    Let $\Gamma$ be an element of the base locus of $\varphi_{l,\uu}$. By  \cref{lem:base locus}, there exist $1\leq i_1<\cdots<i_k\leq n$ such that $\Gamma\in\cH_{i_1}\cap\cdots\cap\cH_{i_k}$ for $1\leq i_1<\cdots<i_k\leq n$ such that $\Gamma\not\in\cH_j$ for $j\neq i_1,\ldots,i_k$. In particular, $x_j^{u_j+1}$ is contained in $\langle\Gamma\rangle$ for $j\neq i_1,\ldots,i_k$. 
    Now, let $C$ be a smooth curve in $\gr(l,\Lambda_{\uu})$ passing through $\Gamma$ and not contained in the base locus of $\varphi_{l,\uu}$. Then, the restriction of  $\varphi_{l,\uu}$ to $C$ extents to all $C$. Let $I$ be the image of $\Gamma$ via this extension. Note that by construction, $\langle\Gamma\rangle$ is contained in $I$. Therefore, we deduce that 
    \begin{equation}\label{eq:inclusion}
    \langle\Gamma\rangle+\langle x_{i_1}^{u_{i_1}+1},\ldots,x_{i_k}^{u_{i_k}+1}\rangle\subseteq I.
    \end{equation}
    Both ideals in \eqref{eq:inclusion} are the same since both have length $m$. Then the proof follows from the fact that the ideal $I$ does not depend on the curve $C$.
\end{proof}

\begin{Rem}\label{rem:boundary}
    The varieties $\Sigma(m,l,\uu)$ in  \cref{def:irred comp punct} are given as the closure of the ideals minimally generated by $f_1,\ldots,f_l\in \Lambda_{\uu}$. Using \cref{lem:extension}, there is a complete description of the elements in the boundary of this closure. This boundary is exactly the image of \eqref{eq:base locus} through $\varphi_{l,\uu}$. 
\end{Rem}

\begin{Prop}\label{irr comp Grass}
    The map $\varphi_{l,\uu}$ extends uniquely to an isomorphism $\mathrm{Gr}(l,\Lambda_{\uu})\cong \Sigma(m,l,\uu)$.
\end{Prop}
\begin{proof}
    The inverse of $\varphi_{l,\uu}$ is the map 
    \[
    \psi_{l,\uu}\colon \Sigma(m,l,\uu)\dashrightarrow \mathrm{Gr}(l,\Lambda_{\uu})
    \]
    that associates to an ideal $J$ in  $\Sigma(m,l,\uu)$ the linear subspace in $\Lambda_{\uu}$ generated by its minimal set of generators, whenever this has dimension $l$. The only case where this does not happen is in the boundary of $\Sigma(m,l,\uu)$. By \cref{rem:boundary}, such an ideal is, up to labeling, of the form
    \begin{equation}\label{eq:ideals boundary}
    \langle f_1,\ldots,f_l\rangle+\langle x_k^{u_k+1},\ldots,x_n^{u_n+1}\rangle,  
    \end{equation}
    where $f_1,\ldots,f_l\in\langle x_1^{u_1},\ldots,x_{k-1}^{u_{k-1}}\rangle_{\C}$ are linearly independent and $1\leq k\leq n$. Then, the extension of $\psi_{l,\uu}$ sends $J$ to  $\langle f_1,\ldots,f_l\rangle_\C$.
\end{proof}

Using the map $\varphi_{l,\uu}$, we can understand the intersection of two varieties of the form $\Sigma(m,l,\uu)$ and $\Sigma(m,l',\vv)$. Let $k\in\Z$ and $\mathbf{w}\in\Z^n$. Define 
\begin{equation}\label{def:kappa indexes}
    \kappa(\mathbf{w},k) := \{i\in [n] : \mathbf{w}_i=k\}\subseteq [n].    
\end{equation}

\begin{Prop}\label{prop: intersections}
  Let $l,l'\in [n-1]$ and $\uu,\vv\in\Z_{\geq 1}^n$ such that $|\uu|=m+l-1$ and $|\vv|=m+l'-1$. Then, $\Sigma(m,l,\uu)\cap \Sigma(m,l',\vv)$ is nonempty if and only if $\uu-\vv\in\{0,1,-1\}^n\setminus\{\mathbf{1},-\mathbf{1}\}$. In this case, the intersection $\Sigma(m,l,\uu)\cap \Sigma(m,l',\vv)$ consists on the ideals of the form:
  \begin{equation}\label{eq:ideal intersect}
  \langle f_1,\ldots, f_r\rangle+\langle x_i^{u_i}:i\in \kappa(\uu-\vv,1)\rangle+\langle  x_i^{u_i+1}:i\in\kappa(\uu-\vv,-1)\rangle,
   \end{equation}
  with $f_1,\ldots, f_r\in\langle x_i^{u_i}:i\in\kappa(\uu-\vv,0)\rangle_\C$ linearly independent and  $r=l-|\kappa(\uu-\vv,1)|$.
\end{Prop}
\begin{proof}
    Let $[I]\in \Sigma(m,l,\uu)$. Then by \eqref{eq:ideals boundary} there exist $S\subseteq [n]$ such that
    \[
    I = \langle f_1,\ldots,f_l\rangle+\langle x_i^{u_i+1}:i\in S\rangle,
    \]
    with $f_1,\ldots,f_l\in\langle x_i^{u_i}:i\not\in S\rangle_\C$ linearly independent. With this presentation some of the generators $f_1,\ldots,f_l$ might be equal to $x_i^{u_i}$ for $i\not\in S$. Rewrite $I$ as 
    \begin{equation}\label{eq: gen intersect}
    I=\langle x_i^{u_i+1}:i\in S\rangle + \langle x_i^{u_i}:i\in T\rangle  +\langle f_1,\ldots,f_r\rangle,
    \end{equation}
    where $T\subset[n]$ is such that $S\cap T=\emptyset$, $r = l-|T|$ and $f_1,\ldots,f_r\in\langle x_i^{u_i}:i\not \in S\cup T\rangle.$
    Analogously, if $[I]\in \Sigma(m,l',\vv)$, then there exists $S',T'\subsetneq[n]$ disjoint such that
    \begin{equation}\label{eq: gen intersect 2}
    I = \langle x_i^{v_i+1}:i\in S'\rangle + \langle x_i^{v_i}:i\in T'\rangle  +\langle f_1',\ldots,f_{r'}'\rangle,
    \end{equation}
    with $r' = l'-|T'|$ and $f_1,\ldots,f_{r'}\in\langle x_i^{v_i}:i\not \in S'\cup T'\rangle.$ From \eqref{eq: gen intersect} and \eqref{eq: gen intersect 2} we deduce that $r=r'$, $\langle f_1',\ldots,f_{r'}'\rangle=\langle f_1,\ldots,f_{r}\rangle$ and $S\cup T = S'\cup T'$. Therefore, $u_i=v_i$ for $i\not \in S\cup T$. For $i\in S$ we get that $u_i+1=v_i$ or $u_i+1=v_i+1$. Similarly, $u_i=v_i$ or $u_i=v_i+1$ for $i\in T$. Hence, $u_i-v_i\in \{0,1,-1\}$ for any $i\in [n]$. Finally, notice that $S\neq [n]$, since the ideal $\langle x_1^{u_1+1},\ldots,x_n^{u_n+1}\rangle$ has length $|\uu|+n+1-n=m+l$ which is greater than $m$ and the same holds for $S'$. This shows the first inclusion. 
    
    For the other inclusion, assume that $\uu-\vv\in\{0,1,-1\}^n\setminus\{\mathbf{1},-\mathbf{1}\}$ and let $J$ be an ideal as in \eqref{eq:ideal intersect}. Then $[J]\in\Sigma(m,l,\uu)$ since it is the image of 
    \[ 
    \langle f_1,\ldots, f_r\rangle_\C+\langle x_i^{u_i}:i\in \kappa(\uu-\vv,1)\rangle_\C\subseteq \Lambda_{\uu}
    \]
    via $\varphi_{l,\uu}$. Similarly, we can rewrite $J$ as
    \[
     \langle f_1,\ldots, f_r\rangle+\langle x_i^{v_i+1}:i\in \kappa(\uu-\vv,1)\rangle+\langle  x_i^{v_i}:i\in\kappa(\uu-\vv,-1)\rangle,
    \]
    which is the image of 
    \[ 
    \langle f_1,\ldots, f_r\rangle_\C+\langle x_i^{v_i}:i\in \kappa(\uu-\vv,-1)\rangle_\C\subseteq \Lambda_{\vv}
    \]
    via $\varphi_ {\vv,l'}$. We conclude that $[J]\in \Sigma(m,l,\uu)\cap \Sigma(m,l',\vv)$.
\end{proof}

\begin{Cor}\label{Cor:intersect}
      Let $l,l'\in [n-1]$ and $\uu,\vv\in\Z_{\geq 1}^n$ such that $|\uu|=m+l-1$ and $|\vv|=m+l'-1$ and $\uu-\vv\in\{0,1,-1\}^n\setminus\{\mathbf{1},-\mathbf{1}\}$. Then
      \[
      \Sigma(m,l,\uu)\cap \Sigma(m,l',\vv)\cong \gr(l-|\kappa(\uu-\vv,1)|,|\kappa(\uu-\vv,0)|),
      \]
      where the intersection is taken with the reduced structure.
\end{Cor}
\begin{proof}
    Let $U=\langle x_i^{u_i}:i\in |\kappa(\uu-\vv,0)|\rangle_\C$. We get a closed embedding of $\mathrm{Gr}(l-|\kappa(\uu-\vv,1)|,U)$ into $\mathrm{Gr}(l,\Lambda_{\uu})$ by sending $\Gamma$ to $\Gamma+\langle x_i^{u_i}:i\in \kappa(\uu-\vv,1)\rangle_\C$. The composition of this closed embedding with $\varphi_{\uu,l}$ gives the isomorphism.
\end{proof}

Having described the varieties $\Sigma(m,l,\uu)$, we can finally identify which of these varieties are the irreducible components of $\Hilb^m_{\mathbf{0}}(X_n)$. 

\begin{Thm}\label{punctual irred comp}
The irreducible components of $\Hilb^m_{\mathbf{0}}(X_n)$ are such $\Sigma(m,l,\uu)$ 
for which $\max\{1,n+1-m\}\leq l\leq n-1$ and $\uu\in\Z_{\geq1}^{n}$ with $|\uu|=m+l-1$.
\end{Thm}
\begin{proof}
\cref{cor:big union} decomposes $\Hilb^m_{\mathbf{0}}(X_n)$ as the union of closed subvarieties, each of them irreducible, hence it suffices to show which $\Sigma(m,l,\uu)$ are irreducible components of $\Hilb^m_{\mathbf{0}}(X_n)$. 

We will first show that $\Sigma(m,n,\uu)$ is not an irreducible component. If $[J]\in\Sigma(m,n,\uu)$ is generic, then it is minimally generated by $f_1,\ldots,f_n$ where $f_1,\ldots,f_n$ are linearly independent elements of $\langle x_1^{u_1},\ldots,x_n^{u_n}\rangle_\C$. Since the latter is an $n$--dimensional vector space we deduce that $J=\langle x_1^{u_1},\ldots,x_n^{u_n}\rangle$. Since $|\uu|= m+n-1$ and $m\geq 2$, there exists $1\leq i\leq n$ such that $u_i\geq 2$. We claim that $[J]\in \Sigma(m,n-1,\uu-\mathbf{e}_i)$. Indeed, the family of length $m$ ideals $$J_\lambda := \langle x_k^{u_k}+\lambda x_i^{u_i-1}:k\in[n]\setminus\{i\}\rangle$$ with $\lambda\in\C^{*}$ satisfies $[J_\lambda]\in \Sigma(m,n-1,\uu-\mathbf{e}_i)$ and for $\lambda= 0$ this family extends uniquely to $J_0=J$. We get that $\Sigma(m,n,\uu)$ is contained in $\Sigma(m,n-1,\uu-\mathbf{e}_i)$ and hence, it is not an irreducible component of $\Hilb^m_{\mathbf{0}}(X_n)$.

Now assume that  $\Sigma(m,l,\uu)$ is not an irreducible component of $\Hilb^m_{\mathbf{0}}(X_n)$ for $1\leq l \leq n-1$. Then, there exists $\vv$ and $l'$ such that $\Sigma(m,l',\vv)$ is an irreducible component and $\Sigma(m,l,\uu)$ is contained in $\Sigma(m,l',\vv)$. By  \cref{prop: intersections}, $\uu-\vv\in\{0,1,-1\}^n\setminus\{\mathbf{1},-\mathbf{1}\}$. Moreover, by \cref{irr comp Grass} and  \cref{Cor:intersect} we get that 
\[
\mathrm{Gr}(l,n)\simeq \Sigma(m,l,\uu)=\Sigma(m,l,\uu)\cap \Sigma(m,l',\vv) \simeq \mathrm{Gr}(l-|\kappa(\uu-\vv,1)|,|\kappa(\uu-\vv,0)|).
\]
We deduce that $|\kappa(\uu-\vv,0)|=n$. Hence, $\uu=\vv$ and $l=l'$ which is a contradiction.  Therefore, $\Sigma(m,l,\uu)$ is an irreducible component of $\Hilb^m_{\mathbf{0}}(X_n)$.
\end{proof} 
\begin{Cor}
    The number of irreducible components of $\sHilb_{\mathbf{0}}^{m}(X_n)$ is 
    \[
    \begin{cases}
    \displaystyle \sum_{l=n-m+1}^{n-1}\binom{l+m-2}{n-1} = \frac{(m-1)}{n}\binom{m+n-2}{n-1} & \text{if $n\geq m$},\\
    \displaystyle \sum_{l=1}^{n-1}\binom{m+l-2}{n-1} = \frac{(m-1)}{n}\binom{m+n-2}{n-1} + \frac{(n-m)}{n}\binom{m + n - 2}{-1 + n} & \text{if $n\leq m$.}
    \end{cases}
    \]
    
\end{Cor}
\begin{proof}
    Since the number of irreducible components is topological we may work in $\Hilb^m_{\mathbf{0}}(X_n)$. By \cref{punctual irred comp}, the quantity on the left-hand side is determined, while the expression on the right-hand side follows from classical combinatorial identities.
\end{proof}

From the above reasoning we deduce that $\Hilb^m_{\mathbf{0}}(X_n)$ is the union of some Grassmannians glued together through closed subvarieties. \cref{ex:n=2 part 1} shows how this gluing is done for $n=2$.

\begin{Ex}\label{ex:n=2 part 1}
    For $n=2$, the only possible value of $l$ is $1$, and $\uu=(u_1,u_2)$ is a strictly positive partition of $m$. Hence
    $$\Hilb^m_{\mathbf{0}}(X_2)=\bigcup_{i=1}^{m-1} \Sigma(m,1,(i,m-i))$$
    has $m-1$ irreducible components.
    A generic point in $\Sigma(m,1,(i,m-i))$ corresponds to an ideal of the form $\langle \lambda_1x_1^{i}+\lambda_2x_2^{m-i}\rangle$ for $\lambda_1,\lambda_2\neq 0$. One can check that $\Sigma(m,1,(i,m-i))$ is isomorphic to $\P^1$, whose torus is identified with the ideals of the above form, and the torus invariant points are the ideals $\langle x_1^{i+1},x_2^{m-i}\rangle$ and $\langle x_1^{i},x_2^{m-i+1}\rangle$. Moreover, the ideal $\langle x_1^{i+1},x_2^{m-i}\rangle$ is the intersection of $\Sigma(m,1,(i,m-i))$ and $\Sigma(m,1,(i+1,m-i_1))$. With further work, it can be shown that $\Hilb^m_{\mathbf{0}}(X_2)$ is a chain of rational curves with nodal singularities obtained by gluing consecutively  $\Sigma(m,1,(i,m-i))$ and $\Sigma(m,1,(i+1,m-i_1))$ through the point associated to the ideal $\langle x_1^{i+1},x_2^{m-i}\rangle$ (see \cref{fig:n2 rational nodal}). This is precisely \cite[Theorem 1]{Ran2005}. 
\end{Ex}

\begin{figure}[h]
    \centering

\tikzset{every picture/.style={line width=0.75pt}} 

\begin{tikzpicture}[x=0.85pt,y=0.85pt,yscale=-1,xscale=1]

\draw    (101.5,150.25) -- (170,80.75) ;
\draw    (210.5,149.75) -- (140.5,80.25) ;
\draw    (180.5,149.75) -- (249,80.25) ;
\draw  [dash pattern={on 0.84pt off 2.51pt}]  (260.5,120) -- (391,120.25) ;
\draw    (450,149.75) -- (380.5,79.75) ;
\draw    (421.5,150.75) -- (489.5,80.25) ;
\draw  [color={rgb, 255:red, 208; green, 2; blue, 27 }  ,draw opacity=1 ][fill={rgb, 255:red, 208; green, 2; blue, 27 }  ,fill opacity=1 ] (153.5,95.88) .. controls (153.5,94.7) and (154.45,93.75) .. (155.63,93.75) .. controls (156.8,93.75) and (157.75,94.7) .. (157.75,95.88) .. controls (157.75,97.05) and (156.8,98) .. (155.63,98) .. controls (154.45,98) and (153.5,97.05) .. (153.5,95.88) -- cycle ;
\draw  [color={rgb, 255:red, 208; green, 2; blue, 27 }  ,draw opacity=1 ][fill={rgb, 255:red, 208; green, 2; blue, 27 }  ,fill opacity=1 ] (193.5,134.54) .. controls (193.5,133.37) and (194.45,132.42) .. (195.63,132.42) .. controls (196.8,132.42) and (197.75,133.37) .. (197.75,134.54) .. controls (197.75,135.72) and (196.8,136.67) .. (195.63,136.67) .. controls (194.45,136.67) and (193.5,135.72) .. (193.5,134.54) -- cycle ;
\draw  [color={rgb, 255:red, 208; green, 2; blue, 27 }  ,draw opacity=1 ][fill={rgb, 255:red, 208; green, 2; blue, 27 }  ,fill opacity=1 ] (433.83,136.21) .. controls (433.83,135.03) and (434.78,134.08) .. (435.96,134.08) .. controls (437.13,134.08) and (438.08,135.03) .. (438.08,136.21) .. controls (438.08,137.38) and (437.13,138.33) .. (435.96,138.33) .. controls (434.78,138.33) and (433.83,137.38) .. (433.83,136.21) -- cycle ;

\draw (55.5,102.4) node [anchor=north west][inner sep=0.75pt]    {{\footnotesize  $\Sigma(m,1,(1,m-1))$}};
\draw (159,91) node [anchor=north west][inner sep=0.75pt]    {{\tiny$\Sigma(m,1,(2,m-2))$}};
\draw (460,110.4) node [anchor=north west][inner sep=0.75pt]    {{\footnotesize $\Sigma(m,1,( m-1,1))$}};

\end{tikzpicture}

    \caption{The $m$ rational irreducible components of $\Hilb^m_{\mathbf{0}}(X_2)$.}
    \label{fig:n2 rational nodal}
\end{figure}
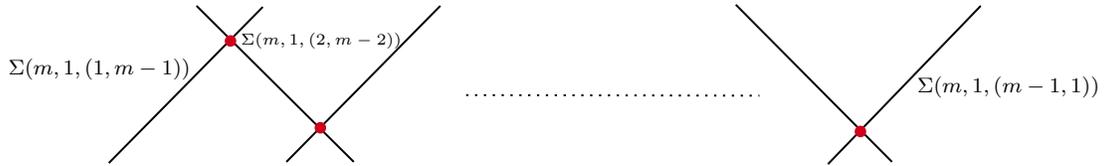

\section{The moment map and its combinatorics}\label{sec:moment map}

 We keep the same notation as in \cref{sec:irred comp punctual}. Our next aim is to describe how $\Hilb^m_{\mathbf{0}}(X_n)$ is obtained by gluing the Grassmannians $\Sigma(m,l,\uu)$ using the combinatorial framework developed in \cref{app: The hypersimplicial complex}. To encode the combinatorics of these Grassmannians, we make use of the moment map associated with the natural action of the algebraic torus on $X_n$, cf. e.g. \cite{AudinTorisActions,KirwanCohomologyofQuotients} for more details about moment maps in symplectic and algebraic geometry. We consider the Plücker coordinates $q_A$ for $A$ in $\binom{[n]}{l}$. For $l<n$ the moment map $\mu_{l,n}$ is: 
\begin{equation}\label{def:momentmap}
    \begin{array}{cccc}
    \mu_{l,n}:&\mathrm{Gr}(l,n)&\longrightarrow &\R^n\\ \displaystyle 
    & (q_A)_{A\in  \binom{[n]}{l}}&\longmapsto& \displaystyle \frac{1}{\displaystyle\sum_{A}|q_A|^2}\left(\sum_{A} |q_A|^2 \ee_A\right).
    \end{array}    
\end{equation}

Note that the moment map is not algebraic and it is defined over the $\C$--points of $\mathrm{Gr}(l,n)$.
 The image of the moment map is the hypersimplex $\Delta_{l,n}$, which
 lies in $l\cdot\Delta_{n-1}$. Here, $l\cdot\Delta_{n-1}\subset\R^n$ is the dilation by $l$ of the $(n-1)$--dimensional simplex. The definition of $\Delta_{l,n}$ and the description of its faces can be found in  \cref{app: The hypersimplicial complex}.
The vertices of $\Delta_{l,n}$ are exactly the vectors $\mathbf{e}_{i_1}+\cdots+\mathbf{e}_{i_l}$ in $\R^n$ for $1\leq i_1<\cdots <i_l\leq n$, which correspond via $\mu_{l,n}$ to the point $[\langle \mathbf{e}_{i_1},\cdots,\mathbf{e}_{i_l}\rangle_\C]$ in $\mathrm{Gr}(l,n)$. More generally, a face 
$\Delta_{l,n}(S_1,S_2)$ for $S_1\sqcup S_2\subseteq [n]$ (see \cref{eq:faces hypersimplex}) 
is isomorphic to the hypersimplex $\Delta_{l-|S_2|,n-r+1}$, and via the moment map, it corresponds to the variety
\[
\left\{[E]\in\mathrm{Gr}(l,n):\langle \mathbf{e}_i:i\in S_2\rangle_\C\subseteq E\subseteq \langle \mathbf{e}_i:i\not\in S_1\rangle_\C\right\},
\]
which is isomorphic to $\mathrm{Gr}(l-|S_2|,n-r+1)$.
The goal of this section is to provide a complete combinatorial description of $\Hilb^m_{\mathbf{0}}(X_n)$ by constructing a moment map on it. To do so, we first construct a moment map on each irreducible component. On $\Sigma(m,l,\uu)\simeq\mathrm{Gr}(l,\Lambda_{\uu})\simeq\mathrm{Gr}(n-l,\Lambda_{\uu}^{*})$, define the following moment map:
\begin{equation}
    \label{eq:piece of moment map}
    \mu_{\uu,l}:\Sigma(m,l,\uu)\simeq\mathrm{Gr}(n-l,\Lambda_{\uu}^{*})\longrightarrow \Delta_{n-l,n}+\uu-\mathbf{1}
\end{equation}
sending a point $\Gamma$ to $\mu_{n-l,n}(\Gamma^{*})+\uu-\mathbf{1}$. In other words, the map $\mu_{\uu,l}$ is the composition of the isomorphism $\Sigma(m,l,\uu)\simeq\mathrm{Gr}(n-l,n)$ with the moment map $\mu_{n-l,n}$ and the translation by the vector $\uu-\mathbf{1}$. In the Plücker coordinates of $\mathrm{Gr}(l,\Lambda_{\uu})$, the moment map $\mu_{\uu,l}$ is given by
\begin{equation}
    \label{eq:moment map Plucker}
    \mu_{\uu,l}((q_A)_{A\in\binom{[n]}{l} })=\displaystyle  \frac{1}{\displaystyle\sum_{I}|q_A|^2}\left(\sum_{A} |q_A|^2 \ee_{[n]\setminus I}\right)+\uu -\mathbf{1}.
\end{equation}

\begin{Thm}\label{thm: moment map}
    For distinct pairs $(\uu,l)$ and $(\vv,l')$, the moment maps $\mu_{\uu,l}$ and $\mu_{\vv,l'}$ coincide in the intersection of $\Sigma(m,l,\uu)$ and $\Sigma(m,l',\vv)$. Therefore, there is a well-defined moment map  
    \begin{equation}
        \mu_m:\Hilb^m_{\mathbf{0}}(X_n)\longrightarrow (m-1)\Delta_{n-1}    
    \end{equation}
    whose restriction to each irreducible component $\Sigma(m,l,\uu)$ is $\mu_{\uu,l}$.
\end{Thm}
\begin{proof}
Let $[I]\in\Sigma(m,l,\uu)\cap\Sigma(m,l',\vv)$.
By \cref{prop: intersections}, $\uu-\vv\in\{0,1,-1\}^n$ and 
\[
 I =  \langle f_1,\ldots, f_r\rangle+\langle x_i^{u_i}:i\in \kappa(\uu-\vv,1)\rangle+\langle  x_i^{u_i+1}:i\in\kappa(\uu-\vv,-1)\rangle,
   \]
  with $f_1,\ldots, f_r\in\langle x_i^{u_i}:i\in\kappa(\uu-\vv,0)\rangle_\C$ linearly independent and  $r=l-|\kappa(\uu-\vv,1)|$. As an element of $\mathrm{Gr}(l,\Lambda_{\uu})$, $[I]$ corresponds to the linear subspace 
  \[
   \langle f_1,\ldots, f_r\rangle_\C+\langle x_i^{u_i}:i\in \kappa(\uu-\vv,1)\rangle_\C.
  \]
We will compute explicitly the moment map in coordinates. Let $M$ be the $l\times n$ matrix whose rows consists on the coefficients of $x_1^{u_1},\ldots,x_n^{u_n}$ in $f_1,\ldots,f_r$ and $x_i^{u_i}$ for $i\in \kappa(\uu-\vv,1)$. In other words, $M$ is a matrix of the form
   \[
M = \begin{pNiceArray}{w{c}{1cm}w{c}{1cm}|w{c}{1cm}w{c}{1cm}|w{c}{1cm}w{c}{1cm}}
\Block{2-2}<\Large>{\mathrm{Id}_{|\kappa(\uu-\vv,1)|}} & &\Block{2-2}<\Large>{\mathbf{0}} & &\Block{2-2}<\Large>{\mathbf{0}} & \\
 & & & & & \\

  \hline
\Block{2-2}<\Large>{\mathbf{0}} & &\Block{2-2}<\Large>{N} & &\Block{2-2}<\Large>{\mathbf{0}} & \\
 & & & & & \\
  \CodeAfter
 \UnderBrace[shorten,yshift=3pt]{4-1}{4-2}{\kappa(\uu-\vv,1)}
 \UnderBrace[shorten,yshift=3pt]{4-3}{4-4}{\kappa(\uu-\vv,0)}
 \UnderBrace[shorten,yshift=3pt]{4-5}{4-6}{\kappa(\uu-\vv,-1)}
\end{pNiceArray},
   \]
   
   \vspace{6mm}
   
   where $N$ is the $r\times |\kappa(\uu-\vv,0)|$ matrix 
   \[
   N= \begin{pmatrix}
       \cdots & f_1 &\cdots\\ & \vdots & \\  \cdots & f_r &\cdots
   \end{pmatrix}
   \]
   whose rows are the coefficients of $x_i^{u_i}$ for $i\in\kappa(\uu-\vv,0)$ in $f_1,\ldots,f_r$. Thus, the Plücker coordinate $q_A$ of $[I]$, where $A\in\binom{[n]}{l}$, corresponds to the $l\times A$ minor of $M$. We denote this minor by $\det (M,A)$. Note that if $A\cap \kappa(\uu-\vv,-1)\neq \emptyset$, the submatrix of $M$ given by the columns in $A$ has a vanishing column. Thus, we have that $q_A= 0$ if $A\cap \kappa(\uu-\vv,-1)\neq \emptyset$. Similarly, if $\kappa(\uu-\vv,1)$ is not contained in $A$, the submatrix of $M$ given by the columns in $A$ has a vanishing row. Hence we get $q_A= 0$ if $\kappa(\uu-\vv,1) \not\subset A$. Thus, the only nonvanishing Plücker coordinates of $[I]$ in $\mathrm{Gr}(l,\Lambda_{\uu})$ are such $q_A$ with $A=\kappa(\uu-\vv,1)\cup B$ for some $B\in\binom{\kappa(\uu-\vv,0)}{r}$. In this case, $q_A$ is the minor $r\times B$ of $N$, i.e. $q_A=\det(N,B)$. Using \eqref{eq:piece of moment map}, we get that
   \[
   \begin{split}
    \mu_{\uu,l}([I]) &= \displaystyle  \frac{1}{\displaystyle\sum_{B\in\binom{\kappa(\uu-\vv,0)}{r}}|\det(N,B)|^2}\left(\sum_{B\in\binom{\kappa(\uu-\vv,0)}{r}} |\det(N,B)|^2 \ee_{[n]\setminus (B\cup \kappa(\uu-\vv,1))}\right)+\uu -\mathbf{1}\\[-4pt]
    &= \displaystyle  \frac{1}{\displaystyle\sum_{B}|\det(N,B)|^2}\left(\sum_{B} |\det(N,B)|^2 \ee_{(\kappa(\uu-\vv,0)\setminus B)\cup \kappa(\uu-\vv,-1)}\right)+\uu -\mathbf{1}\\[-4pt]
    &= \displaystyle  \frac{1}{\displaystyle\sum_{B}|\det(N,B)|^2}\left(\sum_{B} |\det(N,B)|^2 (\ee_{\kappa(\uu-\vv,0)\cup \kappa(\uu-\vv,-1)}-\ee_B)\right)+\uu -\mathbf{1}\\[-6pt]
    &= \ee_{\kappa(\uu-\vv,0)\cup \kappa(\uu-\vv,-1)}-
  \frac{1}{\displaystyle\sum_{B}|\det(N,B)|^2}\left(\sum_{B} |\det(N,B)|^2 \ee_B)\right)
  +\uu -\mathbf{1}\\[-6pt]
    &= \ee_{\kappa(\uu-\vv,0)}+\ee_{\kappa(\uu-\vv,-1)}-
  \frac{1}{\displaystyle\sum_{B}|\det(N,B)|^2}\left(\sum_{B} |\det(N,B)|^2 \ee_B)\right)
  +\uu -\mathbf{1}.  
   \end{split}
   \]
   A similar computation shows that 
   \[
   \mu_{\vv,l'}([I])=  \ee_{\kappa(\uu-\vv,0)}+\ee_{\kappa(\uu-\vv,1)}-
  \frac{1}{\displaystyle\sum_{B}|\det(N,B)|^2}\left(\sum_{B} |\det(N,B)|^2 \ee_B)\right)
  +\vv -\mathbf{1}.
   \]
   The proof follows from the fact that $\uu-\vv = \ee_{\kappa(\uu-\vv,1)}-\ee_{\kappa(\uu-\vv,-1)}$.
\end{proof}

\begin{Ex}
   Continuing \cref{ex:n=2 part 1}, for $n=2$, the $m-1$ irreducible components of $\Hilb^m_{\mathbf{0}}(X_2)$ are $\Sigma(m,1,(1,m-1)),\ldots,\Sigma(m,1,(m-1))$. Each of them is isomorphic to $\P^1$. The moment map $\mu_{(i,m-i),1}$ is defined as
   \[
   \begin{array}{cccccl}
   \mu_{(i,m-i),1}:&\Sigma(m,1,(i,m-i))\simeq \P^1&\longrightarrow&(\P^1)^{*}&\longrightarrow& (m-1)\cdot\Delta_1\\
   &[a_0,a_1]&\longmapsto&[a_1,a_0]&\longmapsto&(\frac{|a_1|^2}{|a_0|^2+|a_1|^2}+i-1,\frac{|a_0|^2}{|a_0|^2+|a_1|^2}+m-i-1)
   \end{array}.
   \]
Using the above formula a direct computation yields
   \[\begin{array}{c}
   \mu_{(i,m-i),1}(\langle x_1^{i+1},x_2^{m-i}\rangle)=\mu_{(i,m-i),1}([0,1])=(i,m-i-1)=\mu_{(i+1,m-i-1),1}([1,0])\\=\mu_{(i+1,m-i-1),1}(\langle x_1^{i+1},x_2^{m-i}\rangle).
   \end{array}
   \]
In particular, $ \mu_{(i,m-i),1}$ and $ \mu_{(i+1,m-i-1),1}$ coincide in the intersection of $\Sigma(m,1,(i,m-i))$ and $\Sigma(m,1,(i+1,m-i))$. The image of $\mu_{(i,m-i),1}$ is the segment between $(i,m-i-1)$ and $(i-1,m-i)$. These segments form a subdivision of $(m-1)\cdot\Delta_1$, which is the image of the moment map $\mu_m$. In \cref{fig:moment n 2}, the image of $\mu_2$, $\mu_3$ and $\mu_4$ is depicted.
\end{Ex}

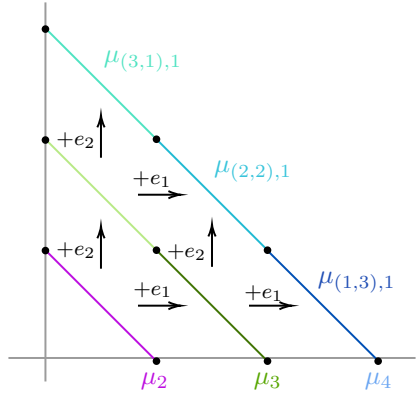
\begin{figure}[h]
    \centering

\tikzset{every picture/.style={line width=0.75pt}} 

\begin{tikzpicture}[x=0.7pt,y=0.7pt,yscale=-1,xscale=1]

\draw [color={rgb, 255:red, 155; green, 155; blue, 155 }  ,draw opacity=1 ]   (100,260) -- (320,260) ;
\draw [color={rgb, 255:red, 155; green, 155; blue, 155 }  ,draw opacity=1 ]   (120,67.71) -- (120,272.71) ;
\draw [color={rgb, 255:red, 189; green, 16; blue, 224 }  ,draw opacity=1 ]   (120,201.71) -- (180,261.71) ;
\draw [color={rgb, 255:red, 184; green, 233; blue, 134 }  ,draw opacity=1 ]   (120,141.71) -- (180,201.71) ;
\draw [color={rgb, 255:red, 65; green, 117; blue, 5 }  ,draw opacity=1 ]   (180,201.71) -- (240,261.71) ;
\draw  [fill={rgb, 255:red, 0; green, 0; blue, 0 }  ,fill opacity=1 ] (178.49,201.71) .. controls (178.49,200.88) and (179.16,200.2) .. (180,200.2) .. controls (180.84,200.2) and (181.51,200.88) .. (181.51,201.71) .. controls (181.51,202.55) and (180.84,203.23) .. (180,203.23) .. controls (179.16,203.23) and (178.49,202.55) .. (178.49,201.71) -- cycle ;
\draw [color={rgb, 255:red, 80; green, 227; blue, 194 }  ,draw opacity=1 ]   (120,81.71) -- (180,141.71) ;
\draw [color={rgb, 255:red, 33; green, 186; blue, 211 }  ,draw opacity=1 ]   (180,141.71) -- (240,201.71) ;
\draw [color={rgb, 255:red, 0; green, 80; blue, 181 }  ,draw opacity=1 ]   (239.78,201.71) -- (300,261.71) ;
\draw  [fill={rgb, 255:red, 0; green, 0; blue, 0 }  ,fill opacity=1 ] (238.49,201.71) .. controls (238.49,200.88) and (239.16,200.2) .. (240,200.2) .. controls (240.84,200.2) and (241.51,200.88) .. (241.51,201.71) .. controls (241.51,202.55) and (240.84,203.23) .. (240,203.23) .. controls (239.16,203.23) and (238.49,202.55) .. (238.49,201.71) -- cycle ;
\draw  [fill={rgb, 255:red, 0; green, 0; blue, 0 }  ,fill opacity=1 ] (178.49,141.71) .. controls (178.49,140.88) and (179.16,140.2) .. (180,140.2) .. controls (180.84,140.2) and (181.51,140.88) .. (181.51,141.71) .. controls (181.51,142.55) and (180.84,143.23) .. (180,143.23) .. controls (179.16,143.23) and (178.49,142.55) .. (178.49,141.71) -- cycle ;
\draw [color={rgb, 255:red, 0; green, 0; blue, 0 }  ,draw opacity=1 ][line width=0.75]    (170,231.71) -- (191,231.71) ;
\draw [shift={(193,231.71)}, rotate = 180] [color={rgb, 255:red, 0; green, 0; blue, 0 }  ,draw opacity=1 ][line width=0.75]    (6.56,-1.97) .. controls (4.17,-0.84) and (1.99,-0.18) .. (0,0) .. controls (1.99,0.18) and (4.17,0.84) .. (6.56,1.97)   ;
\draw [color={rgb, 255:red, 0; green, 0; blue, 0 }  ,draw opacity=1 ][line width=0.75]    (230,231.71) -- (251,231.71) ;
\draw [shift={(253,231.71)}, rotate = 180] [color={rgb, 255:red, 0; green, 0; blue, 0 }  ,draw opacity=1 ][line width=0.75]    (6.56,-1.97) .. controls (4.17,-0.84) and (1.99,-0.18) .. (0,0) .. controls (1.99,0.18) and (4.17,0.84) .. (6.56,1.97)   ;
\draw [color={rgb, 255:red, 0; green, 0; blue, 0 }  ,draw opacity=1 ][line width=0.75]    (170,171.71) -- (191,171.71) ;
\draw [shift={(193,171.71)}, rotate = 180] [color={rgb, 255:red, 0; green, 0; blue, 0 }  ,draw opacity=1 ][line width=0.75]    (6.56,-1.97) .. controls (4.17,-0.84) and (1.99,-0.18) .. (0,0) .. controls (1.99,0.18) and (4.17,0.84) .. (6.56,1.97)   ;
\draw [color={rgb, 255:red, 0; green, 0; blue, 0 }  ,draw opacity=1 ][line width=0.75]    (150,211.71) -- (150,190.71) ;
\draw [shift={(150,188.71)}, rotate = 90] [color={rgb, 255:red, 0; green, 0; blue, 0 }  ,draw opacity=1 ][line width=0.75]    (6.56,-1.97) .. controls (4.17,-0.84) and (1.99,-0.18) .. (0,0) .. controls (1.99,0.18) and (4.17,0.84) .. (6.56,1.97)   ;
\draw [color={rgb, 255:red, 0; green, 0; blue, 0 }  ,draw opacity=1 ][line width=0.75]    (150,151.71) -- (150,130.71) ;
\draw [shift={(150,128.71)}, rotate = 90] [color={rgb, 255:red, 0; green, 0; blue, 0 }  ,draw opacity=1 ][line width=0.75]    (6.56,-1.97) .. controls (4.17,-0.84) and (1.99,-0.18) .. (0,0) .. controls (1.99,0.18) and (4.17,0.84) .. (6.56,1.97)   ;
\draw [color={rgb, 255:red, 0; green, 0; blue, 0 }  ,draw opacity=1 ][line width=0.75]    (210,210.71) -- (210,189.71) ;
\draw [shift={(210,187.71)}, rotate = 90] [color={rgb, 255:red, 0; green, 0; blue, 0 }  ,draw opacity=1 ][line width=0.75]    (6.56,-1.97) .. controls (4.17,-0.84) and (1.99,-0.18) .. (0,0) .. controls (1.99,0.18) and (4.17,0.84) .. (6.56,1.97)   ;
\draw  [fill={rgb, 255:red, 0; green, 0; blue, 0 }  ,fill opacity=1 ] (178.49,261.71) .. controls (178.49,260.88) and (179.16,260.2) .. (180,260.2) .. controls (180.84,260.2) and (181.51,260.88) .. (181.51,261.71) .. controls (181.51,262.55) and (180.84,263.23) .. (180,263.23) .. controls (179.16,263.23) and (178.49,262.55) .. (178.49,261.71) -- cycle ;
\draw  [fill={rgb, 255:red, 0; green, 0; blue, 0 }  ,fill opacity=1 ] (238.49,261.71) .. controls (238.49,260.88) and (239.16,260.2) .. (240,260.2) .. controls (240.84,260.2) and (241.51,260.88) .. (241.51,261.71) .. controls (241.51,262.55) and (240.84,263.23) .. (240,263.23) .. controls (239.16,263.23) and (238.49,262.55) .. (238.49,261.71) -- cycle ;
\draw  [fill={rgb, 255:red, 0; green, 0; blue, 0 }  ,fill opacity=1 ] (298.27,261.71) .. controls (298.27,260.88) and (298.94,260.2) .. (299.78,260.2) .. controls (300.61,260.2) and (301.29,260.88) .. (301.29,261.71) .. controls (301.29,262.55) and (300.61,263.23) .. (299.78,263.23) .. controls (298.94,263.23) and (298.27,262.55) .. (298.27,261.71) -- cycle ;
\draw  [fill={rgb, 255:red, 0; green, 0; blue, 0 }  ,fill opacity=1 ] (118.49,201.91) .. controls (118.49,201.08) and (119.16,200.4) .. (120,200.4) .. controls (120.84,200.4) and (121.51,201.08) .. (121.51,201.91) .. controls (121.51,202.75) and (120.84,203.43) .. (120,203.43) .. controls (119.16,203.43) and (118.49,202.75) .. (118.49,201.91) -- cycle ;
\draw  [fill={rgb, 255:red, 0; green, 0; blue, 0 }  ,fill opacity=1 ] (118.49,142.14) .. controls (118.49,141.3) and (119.16,140.62) .. (120,140.62) .. controls (120.84,140.62) and (121.51,141.3) .. (121.51,142.14) .. controls (121.51,142.97) and (120.84,143.65) .. (120,143.65) .. controls (119.16,143.65) and (118.49,142.97) .. (118.49,142.14) -- cycle ;
\draw  [fill={rgb, 255:red, 0; green, 0; blue, 0 }  ,fill opacity=1 ] (118.49,82.14) .. controls (118.49,81.3) and (119.16,80.62) .. (120,80.62) .. controls (120.84,80.62) and (121.51,81.3) .. (121.51,82.14) .. controls (121.51,82.97) and (120.84,83.65) .. (120,83.65) .. controls (119.16,83.65) and (118.49,82.97) .. (118.49,82.14) -- cycle ;

\draw (170,266.11) node [anchor=north west][inner sep=0.75pt]    {$\textcolor[rgb]{0.74,0.06,0.88}{\mu }\textcolor[rgb]{0.74,0.06,0.88}{_{2}}$};
\draw (231,266.11) node [anchor=north west][inner sep=0.75pt]    {$\textcolor[rgb]{0.34,0.65,0}{\mu }\textcolor[rgb]{0.34,0.65,0}{_{3}}$};
\draw (291,266.11) node [anchor=north west][inner sep=0.75pt]    {$\textcolor[rgb]{0.42,0.65,0.95}{\mu }\textcolor[rgb]{0.42,0.65,0.95}{_{4}}$};
\draw (124.43,195.86) node [anchor=north west][inner sep=0.75pt]   [align=left] {{\footnotesize $+e_2$}};
\draw (165.51,157.51) node [anchor=north west][inner sep=0.75pt]   [align=left] {{\footnotesize  $+e_1$}};
\draw (165.91,217.17) node [anchor=north west][inner sep=0.75pt]   [align=left] {{\footnotesize  $+e_1$}};
\draw (225.91,218.94) node [anchor=north west][inner sep=0.75pt]   [align=left] {{\footnotesize  $+e_1$}};
\draw (184.14,196.71) node [anchor=north west][inner sep=0.75pt]   [align=left] {{\footnotesize  $+e_2$}};
\draw (124.43,136.43) node [anchor=north west][inner sep=0.75pt]   [align=left] {{\footnotesize  $+e_2$}};
\draw (150.03,92.09) node [anchor=north west][inner sep=0.75pt]    {$\textcolor[rgb]{0.31,0.89,0.76}{\mu }\textcolor[rgb]{0.31,0.89,0.76}{_{( 3,1) ,1}}$};
\draw (209.63,151.69) node [anchor=north west][inner sep=0.75pt]    {$\textcolor[rgb]{0.22,0.77,0.86}{\mu _{( 2,2) ,1}}$};
\draw (265.63,211.69) node [anchor=north west][inner sep=0.75pt]    {$\textcolor[rgb]{0.13,0.31,0.71}{\mu _{( 1,3) ,1}}$};

\end{tikzpicture}

    \caption{Decomposition of $(m-1)\cdot\Delta_1$ according to the moment map $\mu_m$ for $m=2,3,4$ and $n=2$. It corresponds to the hypersimplicial complexes $\mathcal{K}^{[2]}_2$, $\mathcal{K}^{[3]}_2$, and $\mathcal{K}^{[4]}_2$.}
    \label{fig:moment n 2}
\end{figure}

The image of the moment map $\mu$ is the 
union of all hypersimplices $\Delta_{n-l,n}+\uu-\mathbf{1}$ for $\max\{1,n+1-m\}\leq l\leq n-1 $ and $\uu\in\Z_{\geq 1}^n$ with $\uu=m+l-1$.

\begin{Def}\label{def:hypersimplicial complex K_n^m}
    The hypersimplices $\Delta_{n-l,n}+\uu-\mathbf{1}$ for $\max\{1,n+1-m\}\leq l\leq n-1 $ and $\uu\in\Z_{\geq 1}^n$ with $\uu=m+l-1$ form a hypersimplicial complex called the \emph{$(n,m)$--hypersimplicial complex}. We denote this hypersimplicial complex by $\mathcal{K}_n^{[m]}$.
\end{Def}

By \cref{prop: hypersimplicial complex}, the $(n,m)$--hypersimplicial complex $\mathcal{K}_n^{[m]}$ is indeed a hypersimplicial complex, and it subdivides $(m-1)\cdot\Delta_{n-1}$. 
Further properties of this complex are discussed in \cref{app: The hypersimplicial complex}.
By the proof of \cref{thm: moment map}, we have that 
\[
\mu\left(
\Sigma(m,l,\uu)\cap\Sigma(m,l',\mathbf{v}')
\right)= \left(\Delta_{n-l,n}+\uu-\mathbf{1}\right)\cap\left(\Delta_{n-l',n}+\mathbf{v}-\mathbf{1}\right).
\]
Therefore, the faces of $\mathcal{K}_n^{[m]}$ encode the intersection of the distinct Grassmannian components of $\Hilb^m_{\mathbf{0}}(X_n)$. \cref{fig:moment n 2} depicts this hypersimplicial complex for the case $n=2$.

\begin{Ex}\label{ex: n 3 part 1}
For $n=3$, the parameter $l$ can be either be $1$ or $2$. In the case $l=2$, we get that $\uu=(u_1,u_2,u_3)$ is a partition of $m+1$. This leads to components of the form $\Sigma(m,2,\uu)\simeq \mathrm{Gr}(2,3)=(\P^2)^{*}$. In  Plücker coordinates the moment map $\mu_{\uu,2}$ of these components is
\[
\begin{array}{cccl}
\mu_{\uu,2}:&\Sigma(m,2,\uu)\simeq (\P^2)^{*}&\longrightarrow&(m-1)\Delta_{1,3}+\uu-\mathbf{1}\\
 & [a_{23},a_{13},a_{12}]&\longrightarrow& \left(
\frac{|a_{23}|^2}{|a_{23}|^2+|a_{13}|^2+|a_{12}|^2},\frac{|a_{13}|^2}{|a_{23}|^2+|a_{13}|^2+|a_{12}|^2},\frac{|a_{12}|^2}{|a_{23}|^2+|a_{13}|^2+|a_{12}|^2}
 \right)+\uu-\mathbf{1}
    
\end{array}
\]
The image of $\mu_{\uu,2}$ is the triangle defined by the vertices $(u_1,u_2-1,u_3-1)$, $(u_1-1,u_2,u_3-1)$ and $(u_1-1,u_2-1,u_3)$. These triangles are illustrated in blue in \cref{fig: n=3 simplex}.

For $l=1$, $\uu$ is a partition of $m$ and we get components of the form $\Sigma(m,1,\uu)\simeq \mathrm{Gr}(1,3)=\P^2$. Since $\uu$ can not have zero entries, these type of components only appear for $m\geq 3$. In Plücker coordinates, the moment map $\mu_{\uu,1}$ of these components is
\[
\begin{array}{cccl}
\mu_{\uu,1}:&\Sigma(m,1,\uu)\simeq \P^2&\longrightarrow&(m-1)\Delta_{2,3}+\uu-\mathbf{1}\\
 & [a_{1},a_{2},a_{3}]&\longrightarrow& \left(
\frac{|a_{2}|^2+|a_{3}|^2}{|a_{1}|^2+|a_{2}|^2+|a_{3}|^2},\frac{|a_{1}|^2+|a_{3}|^2}{|a_{1}|^2+|a_{2}|^2+|a_{3}|^2},\frac{|a_{1}|^2+|a_{2}|^2}{|a_{1}|^2+|a_{2}|^2+|a_{3}|^2}
 \right)+\uu-\mathbf{1}.
    
\end{array}
\]

\begin{figure}[h]
    \centering

\tikzset{every picture/.style={line width=0.75pt}} 

\begin{tikzpicture}[x=0.65pt,y=0.65pt,yscale=-1,xscale=1]

\draw [color={rgb, 255:red, 155; green, 155; blue, 155 }  ,draw opacity=1 ]   (118.67,120) -- (118.67,79.96) ;
\draw [color={rgb, 255:red, 155; green, 155; blue, 155 }  ,draw opacity=1 ][line width=0.75]    (198.67,160) -- (158.67,160) ;
\draw [color={rgb, 255:red, 155; green, 155; blue, 155 }  ,draw opacity=1 ]   (78.67,200) -- (98.67,180) ;
\draw [color={rgb, 255:red, 155; green, 155; blue, 155 }  ,draw opacity=1 ]   (98.67,180) -- (118.67,160) ;
\draw [color={rgb, 255:red, 155; green, 155; blue, 155 }  ,draw opacity=1 ]   (158.67,160) -- (118.67,160) ;
\draw [color={rgb, 255:red, 155; green, 155; blue, 155 }  ,draw opacity=1 ]   (118.67,160) -- (118.67,119.96) ;
\draw  [color={rgb, 255:red, 74; green, 144; blue, 226 }  ,draw opacity=1 ][fill={rgb, 255:red, 74; green, 144; blue, 226 }  ,fill opacity=0.18 ] (118.67,120) -- (158.67,160) -- (98.67,180) -- cycle ;
\draw [color={rgb, 255:red, 155; green, 155; blue, 155 }  ,draw opacity=1 ]   (280.67,120) -- (280.67,79.96) ;
\draw [color={rgb, 255:red, 155; green, 155; blue, 155 }  ,draw opacity=1 ]   (220.67,220) -- (240.67,200) ;
\draw [color={rgb, 255:red, 155; green, 155; blue, 155 }  ,draw opacity=1 ]   (400.67,160) -- (360.67,160) ;
\draw [color={rgb, 255:red, 155; green, 155; blue, 155 }  ,draw opacity=1 ][line width=0.75]    (360.67,160) -- (320.67,160) ;
\draw [color={rgb, 255:red, 155; green, 155; blue, 155 }  ,draw opacity=1 ]   (240.67,200) -- (260.67,180) ;
\draw [color={rgb, 255:red, 155; green, 155; blue, 155 }  ,draw opacity=1 ]   (260.67,180) -- (280.67,160) ;
\draw [color={rgb, 255:red, 155; green, 155; blue, 155 }  ,draw opacity=1 ]   (320.67,160) -- (280.67,160) ;
\draw [color={rgb, 255:red, 155; green, 155; blue, 155 }  ,draw opacity=1 ]   (280.67,160) -- (280.67,119.96) ;
\draw  [color={rgb, 255:red, 74; green, 144; blue, 226 }  ,draw opacity=1 ][fill={rgb, 255:red, 74; green, 144; blue, 226 }  ,fill opacity=0.18 ] (260.67,140) -- (300.67,180) -- (240.67,200) -- cycle ;
\draw  [color={rgb, 255:red, 74; green, 144; blue, 226 }  ,draw opacity=1 ][fill={rgb, 255:red, 74; green, 144; blue, 226 }  ,fill opacity=0.18 ] (280.67,79.96) -- (320.67,119.96) -- (260.67,139.96) -- cycle ;
\draw  [color={rgb, 255:red, 74; green, 144; blue, 226 }  ,draw opacity=1 ][fill={rgb, 255:red, 74; green, 144; blue, 226 }  ,fill opacity=0.18 ] (320.67,120) -- (360.67,160) -- (300.67,180) -- cycle ;
\draw  [color={rgb, 255:red, 189; green, 16; blue, 224 }  ,draw opacity=1 ][fill={rgb, 255:red, 189; green, 16; blue, 224 }  ,fill opacity=0.25 ] (260.67,140) -- (320.67,120) -- (300.67,180) -- cycle ;
\draw [color={rgb, 255:red, 155; green, 155; blue, 155 }  ,draw opacity=1 ]   (482,120) -- (482,79.96) ;
\draw [color={rgb, 255:red, 155; green, 155; blue, 155 }  ,draw opacity=1 ]   (422,220) -- (442,200) ;
\draw [color={rgb, 255:red, 155; green, 155; blue, 155 }  ,draw opacity=1 ]   (602,160) -- (562,160) ;
\draw [color={rgb, 255:red, 155; green, 155; blue, 155 }  ,draw opacity=1 ][line width=0.75]    (562,160) -- (522,160) ;
\draw [color={rgb, 255:red, 155; green, 155; blue, 155 }  ,draw opacity=1 ]   (442,200) -- (462,180) ;
\draw [color={rgb, 255:red, 155; green, 155; blue, 155 }  ,draw opacity=1 ]   (462,180) -- (482,160) ;
\draw [color={rgb, 255:red, 155; green, 155; blue, 155 }  ,draw opacity=1 ]   (522,160) -- (482,160) ;
\draw [color={rgb, 255:red, 155; green, 155; blue, 155 }  ,draw opacity=1 ]   (482,160) -- (482,119.96) ;
\draw  [color={rgb, 255:red, 74; green, 144; blue, 226 }  ,draw opacity=1 ][fill={rgb, 255:red, 74; green, 144; blue, 226 }  ,fill opacity=0.18 ] (462,100) -- (502,140) -- (442,160) -- cycle ;
\draw  [color={rgb, 255:red, 74; green, 144; blue, 226 }  ,draw opacity=1 ][fill={rgb, 255:red, 74; green, 144; blue, 226 }  ,fill opacity=0.18 ] (482,39.92) -- (522,79.92) -- (462,99.92) -- cycle ;
\draw  [color={rgb, 255:red, 74; green, 144; blue, 226 }  ,draw opacity=1 ][fill={rgb, 255:red, 74; green, 144; blue, 226 }  ,fill opacity=0.18 ] (522,80) -- (562,120) -- (502,140) -- cycle ;
\draw  [color={rgb, 255:red, 189; green, 16; blue, 224 }  ,draw opacity=1 ][fill={rgb, 255:red, 189; green, 16; blue, 224 }  ,fill opacity=0.25 ] (462,100) -- (522,80) -- (502,140) -- cycle ;
\draw [color={rgb, 255:red, 155; green, 155; blue, 155 }  ,draw opacity=1 ]   (482,79.96) -- (482,39.92) ;
\draw  [color={rgb, 255:red, 74; green, 144; blue, 226 }  ,draw opacity=1 ][fill={rgb, 255:red, 74; green, 144; blue, 226 }  ,fill opacity=0.18 ] (562,120) -- (602,160) -- (542,180) -- cycle ;
\draw  [color={rgb, 255:red, 74; green, 144; blue, 226 }  ,draw opacity=1 ][fill={rgb, 255:red, 74; green, 144; blue, 226 }  ,fill opacity=0.18 ] (502,140) -- (542,180) -- (482,200) -- cycle ;
\draw  [color={rgb, 255:red, 74; green, 144; blue, 226 }  ,draw opacity=1 ][fill={rgb, 255:red, 74; green, 144; blue, 226 }  ,fill opacity=0.18 ] (442,160) -- (482,200) -- (422,220) -- cycle ;
\draw  [color={rgb, 255:red, 189; green, 16; blue, 224 }  ,draw opacity=1 ][fill={rgb, 255:red, 189; green, 16; blue, 224 }  ,fill opacity=0.25 ] (442,160) -- (502,140) -- (482,200) -- cycle ;
\draw  [color={rgb, 255:red, 189; green, 16; blue, 224 }  ,draw opacity=1 ][fill={rgb, 255:red, 189; green, 16; blue, 224 }  ,fill opacity=0.25 ] (502,140) -- (562,120) -- (542,180) -- cycle ;
\draw [color={rgb, 255:red, 155; green, 155; blue, 155 }  ,draw opacity=1 ]   (280.67,79.96) -- (280.67,39.92) ;

\end{tikzpicture}

    \caption{Hypersimplicial complex $\mathcal{K}^{[m]}_3$ subdividing $(m-1)\Delta_2$ for $m=1,2$ and $3$. The blue triangles correspond to the image of the maps $\mu_{\uu,2}$ whereas the purple triangles correspond to the image of the maps $\mu_{\uu,1}$.}
    \label{fig: n=3 simplex}
\end{figure}
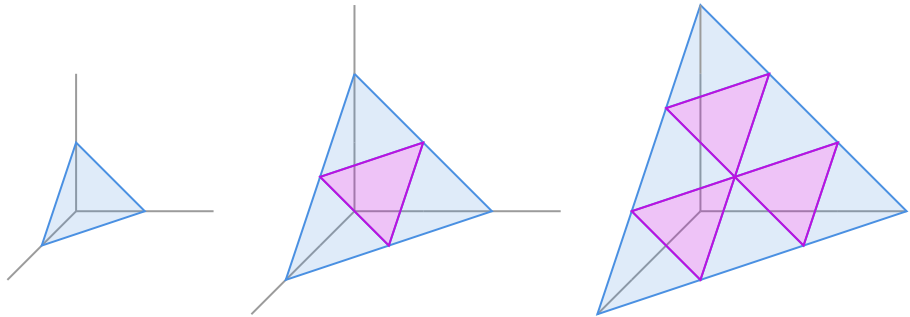

The image of $\mu_{\uu,1}$ is the triangle defined by the vertices $(u_1,u_2,u_3-1)$, $(u_1-1,u_2,u_3)$ and $(u_1,u_2-1,u_3)$. These triangles are illustrated in purple in \cref{fig: n=3 simplex}. 
These $2$--dimensional hypersimplices form the complex $\mathcal{K}_3^{[m]}$ that subdivides $(m-1)\cdot\Delta_2$. \cref{fig: n=3 simplex} depicts this complex for $m=2,3,4$ and \cref{fig:n 3 and m 5} the case $m=5$. We deduce that $\Hilb^m_\mathbf{0}(X_3)$ has $\binom{m}{2}$
components of the form $\Sigma(m,2,\uu)$ and $\binom{m-1}{2}$ components of the form $\Sigma(m,1,\uu)$. All these components are toric and they intersect each other in the closure of toric orbits. The complex $\mathcal{K}^{[m]}_3$ encodes the toric representation of $\Hilb^m_\mathbf{0}(X_3)$ and how its components intersect.
\end{Ex}

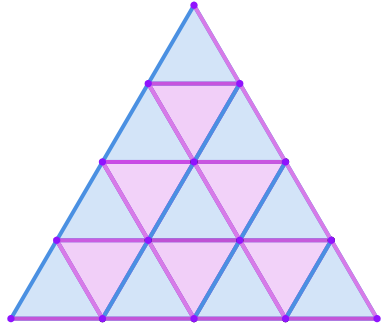
\begin{figure}[h]
    \centering

\tikzset{every picture/.style={line width=0.75pt}} 

\begin{tikzpicture}[x=0.5pt,y=0.5pt,yscale=-1,xscale=1]

\draw  [color={rgb, 255:red, 74; green, 144; blue, 226 }  ,draw opacity=1 ][fill={rgb, 255:red, 74; green, 144; blue, 226 }  ,fill opacity=0.24 ][line width=1.5]  (389.04,210.68) -- (423.67,270) -- (354.42,270) -- cycle ;
\draw  [color={rgb, 255:red, 74; green, 144; blue, 226 }  ,draw opacity=1 ][fill={rgb, 255:red, 74; green, 144; blue, 226 }  ,fill opacity=0.24 ][line width=1.5]  (319.8,210.68) -- (354.42,270) -- (285.17,270) -- cycle ;
\draw  [color={rgb, 255:red, 74; green, 144; blue, 226 }  ,draw opacity=1 ][fill={rgb, 255:red, 74; green, 144; blue, 226 }  ,fill opacity=0.24 ][line width=1.5]  (250.55,210.68) -- (285.17,270) -- (215.93,270) -- cycle ;
\draw  [color={rgb, 255:red, 74; green, 144; blue, 226 }  ,draw opacity=1 ][fill={rgb, 255:red, 74; green, 144; blue, 226 }  ,fill opacity=0.24 ][line width=1.5]  (285.17,151.37) -- (319.8,210.68) -- (250.55,210.68) -- cycle ;
\draw  [color={rgb, 255:red, 74; green, 144; blue, 226 }  ,draw opacity=1 ][fill={rgb, 255:red, 74; green, 144; blue, 226 }  ,fill opacity=0.24 ][line width=1.5]  (354.42,151.37) -- (389.04,210.68) -- (319.8,210.68) -- cycle ;
\draw  [color={rgb, 255:red, 74; green, 144; blue, 226 }  ,draw opacity=1 ][fill={rgb, 255:red, 74; green, 144; blue, 226 }  ,fill opacity=0.24 ][line width=1.5]  (181.3,210.68) -- (215.93,270) -- (146.68,270) -- cycle ;
\draw  [color={rgb, 255:red, 74; green, 144; blue, 226 }  ,draw opacity=1 ][fill={rgb, 255:red, 74; green, 144; blue, 226 }  ,fill opacity=0.24 ][line width=1.5]  (215.93,151.37) -- (250.55,210.68) -- (181.3,210.68) -- cycle ;
\draw  [color={rgb, 255:red, 74; green, 144; blue, 226 }  ,draw opacity=1 ][fill={rgb, 255:red, 74; green, 144; blue, 226 }  ,fill opacity=0.24 ][line width=1.5]  (250.55,92.05) -- (285.17,151.37) -- (215.93,151.37) -- cycle ;
\draw  [color={rgb, 255:red, 74; green, 144; blue, 226 }  ,draw opacity=1 ][fill={rgb, 255:red, 74; green, 144; blue, 226 }  ,fill opacity=0.24 ][line width=1.5]  (319.8,92.05) -- (354.42,151.37) -- (285.17,151.37) -- cycle ;
\draw  [color={rgb, 255:red, 74; green, 144; blue, 226 }  ,draw opacity=1 ][fill={rgb, 255:red, 74; green, 144; blue, 226 }  ,fill opacity=0.24 ][line width=1.5]  (285.17,32.74) -- (319.8,92.05) -- (250.55,92.05) -- cycle ;
\draw  [color={rgb, 255:red, 198; green, 84; blue, 221 }  ,draw opacity=1 ][fill={rgb, 255:red, 189; green, 16; blue, 224 }  ,fill opacity=0.2 ][line width=1.5]  (285.17,151.37) -- (250.55,92.05) -- (319.8,92.05) -- cycle ;
\draw  [color={rgb, 255:red, 201; green, 86; blue, 224 }  ,draw opacity=1 ][fill={rgb, 255:red, 189; green, 16; blue, 224 }  ,fill opacity=0.19 ][line width=1.5]  (319.8,210.68) -- (285.17,151.37) -- (354.42,151.37) -- cycle ;
\draw  [color={rgb, 255:red, 203; green, 74; blue, 229 }  ,draw opacity=1 ][fill={rgb, 255:red, 189; green, 16; blue, 224 }  ,fill opacity=0.19 ][line width=1.5]  (250.55,210.68) -- (215.93,151.37) -- (285.17,151.37) -- cycle ;
\draw  [color={rgb, 255:red, 202; green, 93; blue, 224 }  ,draw opacity=1 ][fill={rgb, 255:red, 189; green, 16; blue, 224 }  ,fill opacity=0.2 ][line width=1.5]  (215.93,270) -- (181.3,210.68) -- (250.55,210.68) -- cycle ;
\draw  [color={rgb, 255:red, 190; green, 82; blue, 212 }  ,draw opacity=1 ][fill={rgb, 255:red, 189; green, 16; blue, 224 }  ,fill opacity=0.2 ][line width=1.5]  (285.17,270) -- (250.55,210.68) -- (319.8,210.68) -- cycle ;
\draw  [color={rgb, 255:red, 197; green, 91; blue, 219 }  ,draw opacity=1 ][fill={rgb, 255:red, 189; green, 16; blue, 224 }  ,fill opacity=0.19 ][line width=1.5]  (354.42,270) -- (319.8,210.68) -- (389.04,210.68) -- cycle ;
\draw  [fill={rgb, 255:red, 0; green, 0; blue, 0 }  ,fill opacity=1 ] (352.43,151.37) .. controls (352.43,150.27) and (353.32,149.37) .. (354.42,149.37) .. controls (355.52,149.37) and (356.41,150.27) .. (356.41,151.37) .. controls (356.41,152.47) and (355.52,153.36) .. (354.42,153.36) .. controls (353.32,153.36) and (352.43,152.47) .. (352.43,151.37) -- cycle ;
\draw  [fill={rgb, 255:red, 0; green, 0; blue, 0 }  ,fill opacity=1 ] (387.05,210.68) .. controls (387.05,209.58) and (387.94,208.69) .. (389.04,208.69) .. controls (390.14,208.69) and (391.04,209.58) .. (391.04,210.68) .. controls (391.04,211.78) and (390.14,212.68) .. (389.04,212.68) .. controls (387.94,212.68) and (387.05,211.78) .. (387.05,210.68) -- cycle ;
\draw  [fill={rgb, 255:red, 0; green, 0; blue, 0 }  ,fill opacity=1 ] (283.18,151.37) .. controls (283.18,150.27) and (284.07,149.37) .. (285.17,149.37) .. controls (286.27,149.37) and (287.17,150.27) .. (287.17,151.37) .. controls (287.17,152.47) and (286.27,153.36) .. (285.17,153.36) .. controls (284.07,153.36) and (283.18,152.47) .. (283.18,151.37) -- cycle ;
\draw  [fill={rgb, 255:red, 0; green, 0; blue, 0 }  ,fill opacity=1 ] (248.56,210.68) .. controls (248.56,209.58) and (249.45,208.69) .. (250.55,208.69) .. controls (251.65,208.69) and (252.54,209.58) .. (252.54,210.68) .. controls (252.54,211.78) and (251.65,212.68) .. (250.55,212.68) .. controls (249.45,212.68) and (248.56,211.78) .. (248.56,210.68) -- cycle ;
\draw  [fill={rgb, 255:red, 0; green, 0; blue, 0 }  ,fill opacity=1 ] (317.8,210.68) .. controls (317.8,209.58) and (318.7,208.69) .. (319.8,208.69) .. controls (320.9,208.69) and (321.79,209.58) .. (321.79,210.68) .. controls (321.79,211.78) and (320.9,212.68) .. (319.8,212.68) .. controls (318.7,212.68) and (317.8,211.78) .. (317.8,210.68) -- cycle ;
\draw  [fill={rgb, 255:red, 0; green, 0; blue, 0 }  ,fill opacity=1 ] (283.18,270) .. controls (283.18,268.9) and (284.07,268.01) .. (285.17,268.01) .. controls (286.27,268.01) and (287.17,268.9) .. (287.17,270) .. controls (287.17,271.1) and (286.27,271.99) .. (285.17,271.99) .. controls (284.07,271.99) and (283.18,271.1) .. (283.18,270) -- cycle ;
\draw  [fill={rgb, 255:red, 0; green, 0; blue, 0 }  ,fill opacity=1 ] (352.43,270) .. controls (352.43,268.9) and (353.32,268.01) .. (354.42,268.01) .. controls (355.52,268.01) and (356.41,268.9) .. (356.41,270) .. controls (356.41,271.1) and (355.52,271.99) .. (354.42,271.99) .. controls (353.32,271.99) and (352.43,271.1) .. (352.43,270) -- cycle ;
\draw  [fill={rgb, 255:red, 0; green, 0; blue, 0 }  ,fill opacity=1 ] (213.93,270) .. controls (213.93,268.9) and (214.83,268.01) .. (215.93,268.01) .. controls (217.03,268.01) and (217.92,268.9) .. (217.92,270) .. controls (217.92,271.1) and (217.03,271.99) .. (215.93,271.99) .. controls (214.83,271.99) and (213.93,271.1) .. (213.93,270) -- cycle ;
\draw [color={rgb, 255:red, 74; green, 144; blue, 226 }  ,draw opacity=1 ][line width=1.5]    (319.8,92.05) -- (215.93,270) ;
\draw [color={rgb, 255:red, 74; green, 144; blue, 226 }  ,draw opacity=1 ][line width=1.5]    (354.42,151.37) -- (285.17,270) ;
\draw [color={rgb, 255:red, 74; green, 144; blue, 226 }  ,draw opacity=1 ][line width=1.5]    (389.04,210.68) -- (354.42,270) ;
\draw [color={rgb, 255:red, 195; green, 85; blue, 217 }  ,draw opacity=1 ][line width=1.5]    (423.67,270) -- (146.68,270) ;
\draw [color={rgb, 255:red, 215; green, 123; blue, 232 }  ,draw opacity=1 ][line width=1.5]    (285.17,32.74) -- (423.67,270) ;
\draw [color={rgb, 255:red, 215; green, 123; blue, 232 }  ,draw opacity=1 ][line width=1.5]    (250.55,92.05) -- (354.42,270) ;
\draw [color={rgb, 255:red, 215; green, 123; blue, 232 }  ,draw opacity=1 ][line width=1.5]    (215.93,151.37) -- (285.17,270) ;
\draw [color={rgb, 255:red, 215; green, 123; blue, 232 }  ,draw opacity=1 ][line width=1.5]    (181.3,210.68) -- (215.93,270) ;
\draw  [color={rgb, 255:red, 144; green, 19; blue, 254 }  ,draw opacity=1 ][fill={rgb, 255:red, 144; green, 19; blue, 254 }  ,fill opacity=1 ] (283.18,32.74) .. controls (283.18,31.63) and (284.07,30.74) .. (285.17,30.74) .. controls (286.27,30.74) and (287.17,31.63) .. (287.17,32.74) .. controls (287.17,33.84) and (286.27,34.73) .. (285.17,34.73) .. controls (284.07,34.73) and (283.18,33.84) .. (283.18,32.74) -- cycle ;
\draw  [color={rgb, 255:red, 144; green, 19; blue, 254 }  ,draw opacity=1 ][fill={rgb, 255:red, 144; green, 19; blue, 254 }  ,fill opacity=1 ] (317.8,92.05) .. controls (317.8,90.95) and (318.7,90.06) .. (319.8,90.06) .. controls (320.9,90.06) and (321.79,90.95) .. (321.79,92.05) .. controls (321.79,93.15) and (320.9,94.05) .. (319.8,94.05) .. controls (318.7,94.05) and (317.8,93.15) .. (317.8,92.05) -- cycle ;
\draw  [color={rgb, 255:red, 144; green, 19; blue, 254 }  ,draw opacity=1 ][fill={rgb, 255:red, 144; green, 19; blue, 254 }  ,fill opacity=1 ] (179.31,210.68) .. controls (179.31,209.58) and (180.2,208.69) .. (181.3,208.69) .. controls (182.41,208.69) and (183.3,209.58) .. (183.3,210.68) .. controls (183.3,211.78) and (182.41,212.68) .. (181.3,212.68) .. controls (180.2,212.68) and (179.31,211.78) .. (179.31,210.68) -- cycle ;
\draw  [color={rgb, 255:red, 144; green, 19; blue, 254 }  ,draw opacity=1 ][fill={rgb, 255:red, 144; green, 19; blue, 254 }  ,fill opacity=1 ] (352.43,151.37) .. controls (352.43,150.27) and (353.32,149.37) .. (354.42,149.37) .. controls (355.52,149.37) and (356.41,150.27) .. (356.41,151.37) .. controls (356.41,152.47) and (355.52,153.36) .. (354.42,153.36) .. controls (353.32,153.36) and (352.43,152.47) .. (352.43,151.37) -- cycle ;
\draw  [color={rgb, 255:red, 144; green, 19; blue, 254 }  ,draw opacity=1 ][fill={rgb, 255:red, 144; green, 19; blue, 254 }  ,fill opacity=1 ] (213.93,151.37) .. controls (213.93,150.27) and (214.83,149.37) .. (215.93,149.37) .. controls (217.03,149.37) and (217.92,150.27) .. (217.92,151.37) .. controls (217.92,152.47) and (217.03,153.36) .. (215.93,153.36) .. controls (214.83,153.36) and (213.93,152.47) .. (213.93,151.37) -- cycle ;
\draw  [color={rgb, 255:red, 144; green, 19; blue, 254 }  ,draw opacity=1 ][fill={rgb, 255:red, 144; green, 19; blue, 254 }  ,fill opacity=1 ] (283.18,151.37) .. controls (283.18,150.27) and (284.07,149.37) .. (285.17,149.37) .. controls (286.27,149.37) and (287.17,150.27) .. (287.17,151.37) .. controls (287.17,152.47) and (286.27,153.36) .. (285.17,153.36) .. controls (284.07,153.36) and (283.18,152.47) .. (283.18,151.37) -- cycle ;
\draw  [color={rgb, 255:red, 144; green, 19; blue, 254 }  ,draw opacity=1 ][fill={rgb, 255:red, 144; green, 19; blue, 254 }  ,fill opacity=1 ] (248.56,92.05) .. controls (248.56,90.95) and (249.45,90.06) .. (250.55,90.06) .. controls (251.65,90.06) and (252.54,90.95) .. (252.54,92.05) .. controls (252.54,93.15) and (251.65,94.05) .. (250.55,94.05) .. controls (249.45,94.05) and (248.56,93.15) .. (248.56,92.05) -- cycle ;
\draw  [color={rgb, 255:red, 144; green, 19; blue, 254 }  ,draw opacity=1 ][fill={rgb, 255:red, 144; green, 19; blue, 254 }  ,fill opacity=1 ] (283.18,270) .. controls (283.18,268.9) and (284.07,268.01) .. (285.17,268.01) .. controls (286.27,268.01) and (287.17,268.9) .. (287.17,270) .. controls (287.17,271.1) and (286.27,271.99) .. (285.17,271.99) .. controls (284.07,271.99) and (283.18,271.1) .. (283.18,270) -- cycle ;
\draw  [color={rgb, 255:red, 144; green, 19; blue, 254 }  ,draw opacity=1 ][fill={rgb, 255:red, 144; green, 19; blue, 254 }  ,fill opacity=1 ] (248.56,210.68) .. controls (248.56,209.58) and (249.45,208.69) .. (250.55,208.69) .. controls (251.65,208.69) and (252.54,209.58) .. (252.54,210.68) .. controls (252.54,211.78) and (251.65,212.68) .. (250.55,212.68) .. controls (249.45,212.68) and (248.56,211.78) .. (248.56,210.68) -- cycle ;
\draw  [color={rgb, 255:red, 144; green, 19; blue, 254 }  ,draw opacity=1 ][fill={rgb, 255:red, 144; green, 19; blue, 254 }  ,fill opacity=1 ] (213.93,270) .. controls (213.93,268.9) and (214.83,268.01) .. (215.93,268.01) .. controls (217.03,268.01) and (217.92,268.9) .. (217.92,270) .. controls (217.92,271.1) and (217.03,271.99) .. (215.93,271.99) .. controls (214.83,271.99) and (213.93,271.1) .. (213.93,270) -- cycle ;
\draw  [color={rgb, 255:red, 144; green, 19; blue, 254 }  ,draw opacity=1 ][fill={rgb, 255:red, 144; green, 19; blue, 254 }  ,fill opacity=1 ] (144.69,270) .. controls (144.69,268.9) and (145.58,268.01) .. (146.68,268.01) .. controls (147.78,268.01) and (148.67,268.9) .. (148.67,270) .. controls (148.67,271.1) and (147.78,271.99) .. (146.68,271.99) .. controls (145.58,271.99) and (144.69,271.1) .. (144.69,270) -- cycle ;
\draw  [color={rgb, 255:red, 144; green, 19; blue, 254 }  ,draw opacity=1 ][fill={rgb, 255:red, 144; green, 19; blue, 254 }  ,fill opacity=1 ] (421.67,270) .. controls (421.67,268.9) and (422.56,268.01) .. (423.67,268.01) .. controls (424.77,268.01) and (425.66,268.9) .. (425.66,270) .. controls (425.66,271.1) and (424.77,271.99) .. (423.67,271.99) .. controls (422.56,271.99) and (421.67,271.1) .. (421.67,270) -- cycle ;
\draw  [color={rgb, 255:red, 144; green, 19; blue, 254 }  ,draw opacity=1 ][fill={rgb, 255:red, 144; green, 19; blue, 254 }  ,fill opacity=1 ] (387.05,210.68) .. controls (387.05,209.58) and (387.94,208.69) .. (389.04,208.69) .. controls (390.14,208.69) and (391.04,209.58) .. (391.04,210.68) .. controls (391.04,211.78) and (390.14,212.68) .. (389.04,212.68) .. controls (387.94,212.68) and (387.05,211.78) .. (387.05,210.68) -- cycle ;
\draw  [color={rgb, 255:red, 144; green, 19; blue, 254 }  ,draw opacity=1 ][fill={rgb, 255:red, 144; green, 19; blue, 254 }  ,fill opacity=1 ] (352.43,270) .. controls (352.43,268.9) and (353.32,268.01) .. (354.42,268.01) .. controls (355.52,268.01) and (356.41,268.9) .. (356.41,270) .. controls (356.41,271.1) and (355.52,271.99) .. (354.42,271.99) .. controls (353.32,271.99) and (352.43,271.1) .. (352.43,270) -- cycle ;
\draw  [color={rgb, 255:red, 144; green, 19; blue, 254 }  ,draw opacity=1 ][fill={rgb, 255:red, 144; green, 19; blue, 254 }  ,fill opacity=1 ] (317.8,210.68) .. controls (317.8,209.58) and (318.7,208.69) .. (319.8,208.69) .. controls (320.9,208.69) and (321.79,209.58) .. (321.79,210.68) .. controls (321.79,211.78) and (320.9,212.68) .. (319.8,212.68) .. controls (318.7,212.68) and (317.8,211.78) .. (317.8,210.68) -- cycle ;

\end{tikzpicture}

    \caption{Hypersimplicial complex $\mathcal{K}^{[5]}_3$ enconding the toric representation of $\Hilb^5_0(X_3)$ obtained by gluing together $16$ copies of $\P^2$.}
    \label{fig:n 3 and m 5}
\end{figure}

\begin{Ex}
 For $n=4$, the possible values of $l$ are $l=1,2,3$. 
 For $l=1$,  $\uu$ is a partition $m$ and we get components of the form $\Sigma(m,1,\uu)\simeq\mathrm{Gr}(1,4)\simeq\P^3$. The image of $\mu_{\uu,1}$ is the translated hypersimplex $\Delta_{3,4}+\uu-\mathbf{1}$. 
 Note that since $\uu$ has no vanishing entries, these components only appear for $m\geq 4$. For $m=4$, the only choice of $\uu$ is $\mathbf{1}$. In this case, the hypersimplex $\Delta_{3,4}$ is illustrated in purple in \cref{fig:n 4 m 4}.
 
 For $l=2$,  $\uu$ is a partition of $m+1$ and we get components of the form $\Sigma(m,2,\uu)\simeq\mathrm{Gr}(2,4)$. The image of $\mu_{\uu,2}$ is the translated hypersimplex $\Delta_{2,4}+\uu-\mathbf{1}$. These components appear for $m\geq 3$. For $m=3$ we have that $\uu=\mathbf{1}$. The hypersimplex $\Delta_{2,4}$ is depicted in purple in \cref{fig:n 4 m 3}. For $m=4$, the possible choices of $\uu$ are $\mathbf{1}+\mathbf{e}_i$ for $i\in[4]$. The hypersimplices $\Delta_{2,4}+\mathbf{e}_i$ are illustrated in green in \cref{fig:n 4 m 4}.

 Finally, for $l=3$, $\uu$ is a partition of $m+2$ and we get components of the form $\Sigma(m,3,\uu)\simeq\mathrm{Gr}(3,4)\simeq\P^3$. The image of $\mu_{\uu,3}$ is the translated hypersimplex $\Delta_{1,4}+\uu-\mathbf{1}$. For $m=2$, we have that $\mathbf{u}=\mathbf{1}$ and $\Delta_{1,4}$ is the usual simplex $\Delta_{3}$. For $m=3$, $\uu=\mathbf{1}+\mathbf{e}_i$ for $i\in[4]$. The hypersimplices $\Delta_{1,4}+\mathbf{e}_i$ are illustrated in blue in \cref{fig:n 4 m 3}. For $m=4$ the possible choices of $\uu$ are $\uu=\mathbf{1}+\mathbf{e}_i+\mathbf{e}_j$ for $i,j\in[4]$. The $10$ hypersimplices $\Delta_{1,4}+\mathbf{e}_i+\mathbf{e}_j$ for $i,j\in[4]$ are depicted in blue in \cref{fig:n 4 m 4}. 

 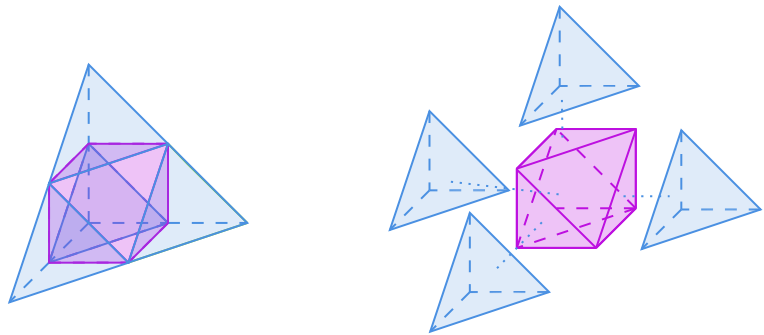
\begin{figure}[h]
     \centering

  
\tikzset {_10qj243rh/.code = {\pgfsetadditionalshadetransform{ \pgftransformshift{\pgfpoint{0 bp } { 0 bp }  }  \pgftransformscale{1 }  }}}
\pgfdeclareradialshading{_h21rpgd2u}{\pgfpoint{0bp}{0bp}}{rgb(0bp)=(0.49,0.83,0.13);
rgb(0bp)=(0.49,0.83,0.13);
rgb(25bp)=(0,0,0);
rgb(400bp)=(0,0,0)}
\tikzset{_u4vbe2pa5/.code = {\pgfsetadditionalshadetransform{\pgftransformshift{\pgfpoint{0 bp } { 0 bp }  }  \pgftransformscale{1 } }}}
\pgfdeclareradialshading{_2lmcxllsa} { \pgfpoint{0bp} {0bp}} {color(0bp)=(transparent!40);
color(0bp)=(transparent!40);
color(25bp)=(transparent!0);
color(400bp)=(transparent!0)} 
\pgfdeclarefading{_rgq9dvoam}{\tikz \fill[shading=_2lmcxllsa,_u4vbe2pa5] (0,0) rectangle (50bp,50bp); } 

  
\tikzset {_niy9s8pxh/.code = {\pgfsetadditionalshadetransform{ \pgftransformshift{\pgfpoint{0 bp } { 0 bp }  }  \pgftransformscale{1 }  }}}
\pgfdeclareradialshading{_2chw11kv6}{\pgfpoint{0bp}{0bp}}{rgb(0bp)=(0.49,0.83,0.13);
rgb(0bp)=(0.49,0.83,0.13);
rgb(25bp)=(0,0,0);
rgb(400bp)=(0,0,0)}
\tikzset{_6qv16zwts/.code = {\pgfsetadditionalshadetransform{\pgftransformshift{\pgfpoint{0 bp } { 0 bp }  }  \pgftransformscale{1 } }}}
\pgfdeclareradialshading{_cxiv1xz59} { \pgfpoint{0bp} {0bp}} {color(0bp)=(transparent!40);
color(0bp)=(transparent!40);
color(25bp)=(transparent!0);
color(400bp)=(transparent!0)} 
\pgfdeclarefading{_k3vsvihp3}{\tikz \fill[shading=_cxiv1xz59,_6qv16zwts] (0,0) rectangle (50bp,50bp); } 

  
\tikzset {_jq0sp0kra/.code = {\pgfsetadditionalshadetransform{ \pgftransformshift{\pgfpoint{0 bp } { 0 bp }  }  \pgftransformscale{1 }  }}}
\pgfdeclareradialshading{_112r4hh53}{\pgfpoint{0bp}{0bp}}{rgb(0bp)=(0.49,0.83,0.13);
rgb(0bp)=(0.49,0.83,0.13);
rgb(25bp)=(0,0,0);
rgb(400bp)=(0,0,0)}
\tikzset{_task4ukxs/.code = {\pgfsetadditionalshadetransform{\pgftransformshift{\pgfpoint{0 bp } { 0 bp }  }  \pgftransformscale{1 } }}}
\pgfdeclareradialshading{_gna4zgivx} { \pgfpoint{0bp} {0bp}} {color(0bp)=(transparent!40);
color(0bp)=(transparent!40);
color(25bp)=(transparent!0);
color(400bp)=(transparent!0)} 
\pgfdeclarefading{_7r82lrqto}{\tikz \fill[shading=_gna4zgivx,_task4ukxs] (0,0) rectangle (50bp,50bp); } 

  
\tikzset {_pljgi783y/.code = {\pgfsetadditionalshadetransform{ \pgftransformshift{\pgfpoint{0 bp } { 0 bp }  }  \pgftransformscale{1 }  }}}
\pgfdeclareradialshading{_gn6bfzwfa}{\pgfpoint{0bp}{0bp}}{rgb(0bp)=(0.49,0.83,0.13);
rgb(0bp)=(0.49,0.83,0.13);
rgb(25bp)=(0,0,0);
rgb(400bp)=(0,0,0)}
\tikzset{_r9sxjnl12/.code = {\pgfsetadditionalshadetransform{\pgftransformshift{\pgfpoint{0 bp } { 0 bp }  }  \pgftransformscale{1 } }}}
\pgfdeclareradialshading{_l7nfpuies} { \pgfpoint{0bp} {0bp}} {color(0bp)=(transparent!40);
color(0bp)=(transparent!40);
color(25bp)=(transparent!0);
color(400bp)=(transparent!0)} 
\pgfdeclarefading{_9fx2zw7qt}{\tikz \fill[shading=_l7nfpuies,_r9sxjnl12] (0,0) rectangle (50bp,50bp); } 
\tikzset{every picture/.style={line width=0.75pt}} 

\begin{tikzpicture}[x=0.75pt,y=0.75pt,yscale=-1,xscale=1]

\draw [color={rgb, 255:red, 74; green, 144; blue, 226 }  ,draw opacity=1 ] [dash pattern={on 4.5pt off 4.5pt}]  (205.07,101.4) -- (205.07,141.4) ;
\draw [color={rgb, 255:red, 189; green, 16; blue, 224 }  ,draw opacity=1 ] [dash pattern={on 4.5pt off 4.5pt}]  (185.07,161.4) -- (205.07,141.4) ;
\draw [color={rgb, 255:red, 189; green, 16; blue, 224 }  ,draw opacity=1 ] [dash pattern={on 4.5pt off 4.5pt}]  (245.07,141.4) -- (205.07,141.4) ;
\draw [color={rgb, 255:red, 126; green, 211; blue, 33 }  ,draw opacity=1 ]   (285.07,181.4) -- (254.86,151.57) -- (245.07,141.4) ;
\draw [color={rgb, 255:red, 189; green, 16; blue, 224 }  ,draw opacity=1 ]   (225.07,201.4) -- (185.07,161.4) ;
\draw [color={rgb, 255:red, 189; green, 16; blue, 224 }  ,draw opacity=1 ]   (225.07,201.4) -- (245.07,141.4) ;
\draw [color={rgb, 255:red, 189; green, 16; blue, 224 }  ,draw opacity=1 ]   (185.07,161.4) -- (245.07,141.4) ;
\draw [color={rgb, 255:red, 189; green, 16; blue, 224 }  ,draw opacity=1 ] [dash pattern={on 4.5pt off 4.5pt}]  (185.07,161.4) -- (185.07,201.4) ;
\draw [color={rgb, 255:red, 189; green, 16; blue, 224 }  ,draw opacity=1 ] [dash pattern={on 4.5pt off 4.5pt}]  (245.07,141.4) -- (245.07,181.4) ;
\draw [color={rgb, 255:red, 189; green, 16; blue, 224 }  ,draw opacity=1 ] [dash pattern={on 4.5pt off 4.5pt}]  (185.07,201.4) -- (205.07,141.4) ;
\draw [color={rgb, 255:red, 189; green, 16; blue, 224 }  ,draw opacity=1 ] [dash pattern={on 4.5pt off 4.5pt}]  (245.07,181.4) -- (205.07,141.4) ;
\draw [color={rgb, 255:red, 74; green, 144; blue, 226 }  ,draw opacity=1 ] [dash pattern={on 4.5pt off 4.5pt}]  (205.07,181.4) -- (205.07,141.36) ;
\draw [color={rgb, 255:red, 126; green, 211; blue, 33 }  ,draw opacity=1 ]   (285.07,181.4) -- (225.07,201.4) ;
\draw [color={rgb, 255:red, 74; green, 144; blue, 226 }  ,draw opacity=1 ] [dash pattern={on 4.5pt off 4.5pt}]  (205.07,181.4) -- (245.07,181.4) ;
\draw [color={rgb, 255:red, 74; green, 144; blue, 226 }  ,draw opacity=1 ] [dash pattern={on 4.5pt off 4.5pt}]  (245.07,181.4) -- (285.07,181.4) ;
\draw [color={rgb, 255:red, 189; green, 16; blue, 224 }  ,draw opacity=1 ] [dash pattern={on 4.5pt off 4.5pt}]  (185.07,201.4) -- (225.07,201.4) ;
\draw [color={rgb, 255:red, 74; green, 144; blue, 226 }  ,draw opacity=1 ] [dash pattern={on 4.5pt off 4.5pt}]  (185.07,201.4) -- (165.07,221.36) ;
\draw [color={rgb, 255:red, 74; green, 144; blue, 226 }  ,draw opacity=1 ] [dash pattern={on 4.5pt off 4.5pt}]  (205.07,181.4) -- (185.07,201.36) ;
\draw [color={rgb, 255:red, 189; green, 16; blue, 224 }  ,draw opacity=1 ] [dash pattern={on 4.5pt off 4.5pt}]  (245.07,181.4) -- (225.07,201.36) ;
\draw [color={rgb, 255:red, 189; green, 16; blue, 224 }  ,draw opacity=1 ] [dash pattern={on 4.5pt off 4.5pt}]  (245.07,181.4) -- (185.07,201.4) ;
\draw  [color={rgb, 255:red, 74; green, 144; blue, 226 }  ,draw opacity=1 ][fill={rgb, 255:red, 74; green, 144; blue, 226 }  ,fill opacity=0.18 ] (205.07,141.44) -- (245.07,181.44) -- (185.07,201.4) -- cycle ;
\draw  [color={rgb, 255:red, 189; green, 16; blue, 224 }  ,draw opacity=1 ][fill={rgb, 255:red, 189; green, 16; blue, 224 }  ,fill opacity=0.25 ] (185.07,161.48) -- (205.07,141.44) -- (245.07,141.48) -- (245.07,181.44) -- (225.07,201.44) -- (185.07,201.48) -- cycle ;
\draw  [color={rgb, 255:red, 74; green, 144; blue, 226 }  ,draw opacity=1 ][fill={rgb, 255:red, 74; green, 144; blue, 226 }  ,fill opacity=0.18 ] (205.07,101.48) -- (245.07,141.48) -- (185.07,161.48) -- cycle ;
\draw  [color={rgb, 255:red, 74; green, 144; blue, 226 }  ,draw opacity=1 ][fill={rgb, 255:red, 74; green, 144; blue, 226 }  ,fill opacity=0.18 ] (245.07,141.4) -- (285.07,181.4) -- (225.07,201.4) -- cycle ;
\draw  [color={rgb, 255:red, 74; green, 144; blue, 226 }  ,draw opacity=1 ][fill={rgb, 255:red, 74; green, 144; blue, 226 }  ,fill opacity=0.18 ] (185.07,161.48) -- (225.07,201.48) -- (165.07,221.48) -- cycle ;
\draw [color={rgb, 255:red, 189; green, 16; blue, 224 }  ,draw opacity=1 ]   (481,134) -- (421,154) ;
\draw [color={rgb, 255:red, 189; green, 16; blue, 224 }  ,draw opacity=1 ]   (461,194) -- (481,134.5) ;
\draw [color={rgb, 255:red, 189; green, 16; blue, 224 }  ,draw opacity=1 ]   (461,194) -- (421,154.5) ;
\draw [color={rgb, 255:red, 189; green, 16; blue, 224 }  ,draw opacity=1 ]   (461,194) -- (481,174) ;
\draw [color={rgb, 255:red, 189; green, 16; blue, 224 }  ,draw opacity=1 ]   (481,134) -- (481,174) ;
\draw [color={rgb, 255:red, 189; green, 16; blue, 224 }  ,draw opacity=1 ] [dash pattern={on 4.5pt off 4.5pt}]  (481,174) -- (441,134) ;
\draw [color={rgb, 255:red, 189; green, 16; blue, 224 }  ,draw opacity=1 ] [dash pattern={on 4.5pt off 4.5pt}]  (481,174) -- (421,194) ;
\draw [color={rgb, 255:red, 189; green, 16; blue, 224 }  ,draw opacity=1 ] [dash pattern={on 4.5pt off 4.5pt}]  (481,174) -- (441,174) ;
\draw [color={rgb, 255:red, 189; green, 16; blue, 224 }  ,draw opacity=1 ]   (461,194) -- (421,194) ;
\draw [color={rgb, 255:red, 189; green, 16; blue, 224 }  ,draw opacity=1 ]   (481,134) -- (441,134) ;
\draw [color={rgb, 255:red, 189; green, 16; blue, 224 }  ,draw opacity=1 ] [dash pattern={on 4.5pt off 4.5pt}]  (421,194) -- (441,134) ;
\draw [color={rgb, 255:red, 189; green, 16; blue, 224 }  ,draw opacity=1 ]   (421,154) -- (441,134) ;
\draw [color={rgb, 255:red, 189; green, 16; blue, 224 }  ,draw opacity=1 ]   (421,154) -- (421,194) ;
\draw  [color={rgb, 255:red, 189; green, 16; blue, 224 }  ,draw opacity=1 ][fill={rgb, 255:red, 189; green, 16; blue, 224 }  ,fill opacity=0.25 ] (441,134) -- (481,134) -- (481,174) -- (461,194) -- (421,194) -- (421,154) -- cycle ;
\draw [color={rgb, 255:red, 74; green, 144; blue, 226 }  ,draw opacity=1 ][shading=_h21rpgd2u,_10qj243rh,path fading= _rgq9dvoam ,fading transform={xshift=2}] [dash pattern={on 4.5pt off 4.5pt}]  (482.62,112.29) -- (442.62,112.29) ;
\draw [color={rgb, 255:red, 74; green, 144; blue, 226 }  ,draw opacity=1 ] [dash pattern={on 0.84pt off 2.51pt}]  (474.79,167.94) -- (498.06,167.93) ;
\draw [color={rgb, 255:red, 74; green, 144; blue, 226 }  ,draw opacity=1 ] [dash pattern={on 0.84pt off 2.51pt}]  (411,204.27) -- (433.73,180.99) ;
\draw [color={rgb, 255:red, 74; green, 144; blue, 226 }  ,draw opacity=1 ] [dash pattern={on 0.84pt off 2.51pt}]  (443.73,119.32) -- (444.06,141.32) ;
\draw [color={rgb, 255:red, 74; green, 144; blue, 226 }  ,draw opacity=1 ] [dash pattern={on 0.84pt off 2.51pt}]  (387.73,160.61) -- (445.06,167.32) ;
\draw  [color={rgb, 255:red, 74; green, 144; blue, 226 }  ,draw opacity=1 ][fill={rgb, 255:red, 74; green, 144; blue, 226 }  ,fill opacity=0.18 ] (442.62,72.29) -- (482.62,112.29) -- (422.62,132.29) -- cycle ;
\draw [color={rgb, 255:red, 74; green, 144; blue, 226 }  ,draw opacity=1 ] [dash pattern={on 4.5pt off 4.5pt}]  (442.62,72.29) -- (442.62,112.29) ;
\draw [color={rgb, 255:red, 74; green, 144; blue, 226 }  ,draw opacity=1 ] [dash pattern={on 4.5pt off 4.5pt}]  (442.62,112.29) -- (422.62,132.25) ;
\draw [color={rgb, 255:red, 74; green, 144; blue, 226 }  ,draw opacity=1 ][shading=_2chw11kv6,_niy9s8pxh,path fading= _k3vsvihp3 ,fading transform={xshift=2}] [dash pattern={on 4.5pt off 4.5pt}]  (543.95,174.62) -- (503.95,174.62) ;
\draw  [color={rgb, 255:red, 74; green, 144; blue, 226 }  ,draw opacity=1 ][fill={rgb, 255:red, 74; green, 144; blue, 226 }  ,fill opacity=0.18 ] (503.95,134.62) -- (543.95,174.62) -- (483.95,194.62) -- cycle ;
\draw [color={rgb, 255:red, 74; green, 144; blue, 226 }  ,draw opacity=1 ] [dash pattern={on 4.5pt off 4.5pt}]  (503.95,134.62) -- (503.95,174.62) ;
\draw [color={rgb, 255:red, 74; green, 144; blue, 226 }  ,draw opacity=1 ] [dash pattern={on 4.5pt off 4.5pt}]  (503.95,174.62) -- (483.95,194.58) ;
\draw [color={rgb, 255:red, 74; green, 144; blue, 226 }  ,draw opacity=1 ][shading=_112r4hh53,_jq0sp0kra,path fading= _7r82lrqto ,fading transform={xshift=2}] [dash pattern={on 4.5pt off 4.5pt}]  (437.29,216.29) -- (397.29,216.29) ;
\draw  [color={rgb, 255:red, 74; green, 144; blue, 226 }  ,draw opacity=1 ][fill={rgb, 255:red, 74; green, 144; blue, 226 }  ,fill opacity=0.18 ] (397.29,176.29) -- (437.29,216.29) -- (377.29,236.29) -- cycle ;
\draw [color={rgb, 255:red, 74; green, 144; blue, 226 }  ,draw opacity=1 ] [dash pattern={on 4.5pt off 4.5pt}]  (397.29,176.29) -- (397.29,216.29) ;
\draw [color={rgb, 255:red, 74; green, 144; blue, 226 }  ,draw opacity=1 ] [dash pattern={on 4.5pt off 4.5pt}]  (397.29,216.29) -- (377.29,236.25) ;
\draw [color={rgb, 255:red, 74; green, 144; blue, 226 }  ,draw opacity=1 ][shading=_gn6bfzwfa,_pljgi783y,path fading= _9fx2zw7qt ,fading transform={xshift=2}] [dash pattern={on 4.5pt off 4.5pt}]  (416.95,164.95) -- (376.95,164.95) ;
\draw  [color={rgb, 255:red, 74; green, 144; blue, 226 }  ,draw opacity=1 ][fill={rgb, 255:red, 74; green, 144; blue, 226 }  ,fill opacity=0.18 ] (376.95,124.95) -- (416.95,164.95) -- (356.95,184.95) -- cycle ;
\draw [color={rgb, 255:red, 74; green, 144; blue, 226 }  ,draw opacity=1 ] [dash pattern={on 4.5pt off 4.5pt}]  (376.95,124.95) -- (376.95,164.95) ;
\draw [color={rgb, 255:red, 74; green, 144; blue, 226 }  ,draw opacity=1 ] [dash pattern={on 4.5pt off 4.5pt}]  (376.95,164.95) -- (356.95,184.91) ;

\end{tikzpicture}

     \caption{Hypersimplicial complex $\mathcal{K}^{[3]}_4$ encoding the intersection of the irreducible components of $\Hilb^3_0(X_4)$.}
     \label{fig:n 4 m 3}
 \end{figure}

 For instance, for $m=3$ there is no component of the form $\Sigma(3,1,\uu)$, there is a component of the form $\Sigma(3,2,\mathbf{1})\simeq\mathrm{Gr}(2,4)$ and there are four components of the form $\Sigma(3,3,\uu)$ for $\uu =(2,1,1,1),\ldots,(1,1,1,2)$. \cref{fig:n 4 m 3} illustrates the hypersimplicial complex $\mathcal{K}_4^{[3]}$ subdividing $2\Delta_3$. The maximal faces of the complex are the octahedron $\Delta_{2,4}$ in purple and the simplices $\Delta_3+\mathbf{e}_1,\ldots,\Delta_3+\mathbf{e}_4$ in blue. As illustrated in \cref{fig:n 4 m 3}, the gluing of the distinct components of $\Hilb^3_0(X_4)$ is done through projective planes. In $\mathrm{Gr}(2,4)$, these projective planes correspond to the $2$--dimensional linear subspaces in  Equation \eqref{eq:base locus}. In the components isomorphic to $\P^3$, the gluing is done through the torus-invariant projective planes.

  For $m=4$, $\Hilb^4_0(X_4)$ has $15$ irreducible components. One of them is $\Sigma(4,1,\mathbf{1})\simeq \P^3$. Four of them are of the form $\Sigma(4,2,\uu)\simeq \mathrm{Gr}(2,4)$ for $\uu =(2,1,1,1),\ldots,(1,1,1,2)$. Finally, there are $10$ irreducible components of the form $\Sigma(4,3,\uu)\simeq \P^3$ where $\uu=\mathbf{1}+\ee_i+\ee_j$ for $1\leq i\leq j\leq 4$. The hypersimplicial complex $\mathcal{K}_4^{[4]}$ encoding the intersection of these components is illustrated in \cref{fig:n 4 m 4}. The blue tetrahedron corresponds to the images of through the moment map of the components of the form $\Sigma(4,3,\mathbf{1}+\ee_i+\ee_j)$. The four green octahedrons correspond to the images of the components $\Sigma(4,2,\mathbf{1}+\ee_i$. Finally, the purple tetrahedron corresponds to the image of the component $\Sigma(4,1,\mathbf{1})$.

  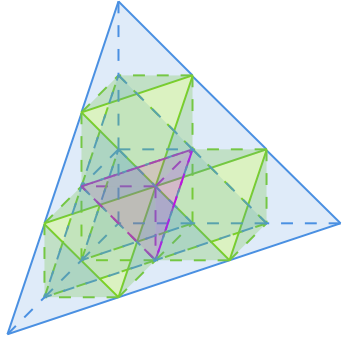
\begin{figure}
      \centering

\tikzset{every picture/.style={line width=0.75pt}} 

\begin{tikzpicture}[x=0.7pt,y=0.7pt,yscale=-1,xscale=1]

\draw [color={rgb, 255:red, 74; green, 144; blue, 226 }  ,draw opacity=1 ]   (290,120) -- (250,80) ;
\draw [color={rgb, 255:red, 74; green, 144; blue, 226 }  ,draw opacity=1 ]   (230,140) -- (250,80) ;
\draw [color={rgb, 255:red, 74; green, 144; blue, 226 }  ,draw opacity=1 ] [dash pattern={on 4.5pt off 4.5pt}]  (250,80) -- (250,120) ;
\draw [color={rgb, 255:red, 126; green, 211; blue, 33 }  ,draw opacity=1 ] [dash pattern={on 4.5pt off 4.5pt}]  (230,140) -- (250,120) ;
\draw [color={rgb, 255:red, 126; green, 211; blue, 33 }  ,draw opacity=1 ] [dash pattern={on 4.5pt off 4.5pt}]  (290,120) -- (250,120) ;
\draw [color={rgb, 255:red, 74; green, 144; blue, 226 }  ,draw opacity=1 ]   (330,160) -- (290,120.5) ;
\draw [color={rgb, 255:red, 126; green, 211; blue, 33 }  ,draw opacity=1 ]   (270,180) -- (230,140) ;
\draw [color={rgb, 255:red, 126; green, 211; blue, 33 }  ,draw opacity=1 ]   (270,180) -- (290,120.5) ;
\draw [color={rgb, 255:red, 74; green, 144; blue, 226 }  ,draw opacity=1 ]   (210,200) -- (230,140.5) ;
\draw [color={rgb, 255:red, 126; green, 211; blue, 33 }  ,draw opacity=1 ]   (230,140) -- (290,120) ;
\draw [color={rgb, 255:red, 126; green, 211; blue, 33 }  ,draw opacity=1 ]   (330,160) -- (270,180) ;
\draw [color={rgb, 255:red, 126; green, 211; blue, 33 }  ,draw opacity=1 ]   (270,180) -- (210,200) ;
\draw [color={rgb, 255:red, 126; green, 211; blue, 33 }  ,draw opacity=1 ] [dash pattern={on 4.5pt off 4.5pt}]  (330,160) -- (290,160) ;
\draw [color={rgb, 255:red, 189; green, 16; blue, 224 }  ,draw opacity=1 ] [dash pattern={on 4.5pt off 4.5pt}]  (270,180) -- (230,180) ;
\draw [color={rgb, 255:red, 126; green, 211; blue, 33 }  ,draw opacity=1 ] [dash pattern={on 4.5pt off 4.5pt}]  (210,200) -- (230,180) ;
\draw [color={rgb, 255:red, 126; green, 211; blue, 33 }  ,draw opacity=1 ] [dash pattern={on 4.5pt off 4.5pt}]  (230,140) -- (230,180) ;
\draw [color={rgb, 255:red, 126; green, 211; blue, 33 }  ,draw opacity=1 ] [dash pattern={on 4.5pt off 4.5pt}]  (290,120) -- (290,160) ;
\draw [color={rgb, 255:red, 189; green, 16; blue, 224 }  ,draw opacity=1 ] [dash pattern={on 4.5pt off 4.5pt}]  (270,180) -- (290,160) ;
\draw [color={rgb, 255:red, 126; green, 211; blue, 33 }  ,draw opacity=1 ] [dash pattern={on 4.5pt off 4.5pt}]  (230,180) -- (250,120) ;
\draw [color={rgb, 255:red, 126; green, 211; blue, 33 }  ,draw opacity=1 ] [dash pattern={on 4.5pt off 4.5pt}]  (290,159.96) -- (250,119.96) ;
\draw [color={rgb, 255:red, 189; green, 16; blue, 224 }  ,draw opacity=1 ] [dash pattern={on 4.5pt off 4.5pt}]  (290,160) -- (230,180) ;
\draw [color={rgb, 255:red, 74; green, 144; blue, 226 }  ,draw opacity=1 ] [dash pattern={on 4.5pt off 4.5pt}]  (250,160) -- (290,160) ;
\draw [color={rgb, 255:red, 74; green, 144; blue, 226 }  ,draw opacity=1 ] [dash pattern={on 4.5pt off 4.5pt}]  (250,160) -- (250,119.96) ;
\draw [color={rgb, 255:red, 74; green, 144; blue, 226 }  ,draw opacity=1 ] [dash pattern={on 4.5pt off 4.5pt}]  (250,160) -- (230,179.96) ;
\draw [color={rgb, 255:red, 74; green, 144; blue, 226 }  ,draw opacity=1 ]   (370,200) -- (330,160.5) ;
\draw [color={rgb, 255:red, 126; green, 211; blue, 33 }  ,draw opacity=1 ]   (250,240) -- (270,180.5) ;
\draw [color={rgb, 255:red, 74; green, 144; blue, 226 }  ,draw opacity=1 ]   (190,260) -- (210,200.5) ;
\draw [color={rgb, 255:red, 74; green, 144; blue, 226 }  ,draw opacity=1 ]   (250,240) -- (190,260) ;
\draw [color={rgb, 255:red, 74; green, 144; blue, 226 }  ,draw opacity=1 ]   (310,220) -- (250,240) ;
\draw [color={rgb, 255:red, 74; green, 144; blue, 226 }  ,draw opacity=1 ]   (370,200) -- (310,220) ;
\draw [color={rgb, 255:red, 126; green, 211; blue, 33 }  ,draw opacity=1 ]   (310,220) -- (330,160.5) ;
\draw [color={rgb, 255:red, 126; green, 211; blue, 33 }  ,draw opacity=1 ]   (310,220) -- (270,180) ;
\draw [color={rgb, 255:red, 126; green, 211; blue, 33 }  ,draw opacity=1 ]   (250,240) -- (210,200.5) ;
\draw [color={rgb, 255:red, 74; green, 144; blue, 226 }  ,draw opacity=1 ] [dash pattern={on 4.5pt off 4.5pt}]  (190,260) -- (210,240) ;
\draw [color={rgb, 255:red, 126; green, 211; blue, 33 }  ,draw opacity=1 ] [dash pattern={on 4.5pt off 4.5pt}]  (250,240) -- (270,220) ;
\draw [color={rgb, 255:red, 126; green, 211; blue, 33 }  ,draw opacity=1 ] [dash pattern={on 4.5pt off 4.5pt}]  (310,220) -- (330,200) ;
\draw [color={rgb, 255:red, 74; green, 144; blue, 226 }  ,draw opacity=1 ] [dash pattern={on 4.5pt off 4.5pt}]  (370,200) -- (330,200) ;
\draw [color={rgb, 255:red, 126; green, 211; blue, 33 }  ,draw opacity=1 ] [dash pattern={on 4.5pt off 4.5pt}]  (310,220) -- (270,220) ;
\draw [color={rgb, 255:red, 126; green, 211; blue, 33 }  ,draw opacity=1 ] [dash pattern={on 4.5pt off 4.5pt}]  (250,240) -- (210,240) ;
\draw [color={rgb, 255:red, 126; green, 211; blue, 33 }  ,draw opacity=1 ] [dash pattern={on 4.5pt off 4.5pt}]  (210,200) -- (210,240) ;
\draw [color={rgb, 255:red, 189; green, 16; blue, 224 }  ,draw opacity=1 ] [dash pattern={on 4.5pt off 4.5pt}]  (270,180) -- (270,220) ;
\draw [color={rgb, 255:red, 126; green, 211; blue, 33 }  ,draw opacity=1 ] [dash pattern={on 4.5pt off 4.5pt}]  (330,160) -- (330,200) ;
\draw [color={rgb, 255:red, 126; green, 211; blue, 33 }  ,draw opacity=1 ] [dash pattern={on 4.5pt off 4.5pt}]  (330,200) -- (290,160) ;
\draw [color={rgb, 255:red, 189; green, 16; blue, 224 }  ,draw opacity=1 ] [dash pattern={on 4.5pt off 4.5pt}]  (270,220) -- (230,180) ;
\draw [color={rgb, 255:red, 126; green, 211; blue, 33 }  ,draw opacity=1 ] [dash pattern={on 4.5pt off 4.5pt}]  (210,240) -- (230,180) ;
\draw [color={rgb, 255:red, 126; green, 211; blue, 33 }  ,draw opacity=1 ] [dash pattern={on 4.5pt off 4.5pt}]  (330,200) -- (270,220) ;
\draw [color={rgb, 255:red, 126; green, 211; blue, 33 }  ,draw opacity=1 ] [dash pattern={on 4.5pt off 4.5pt}]  (270,219.96) -- (210,239.96) ;
\draw [color={rgb, 255:red, 189; green, 16; blue, 224 }  ,draw opacity=1 ] [dash pattern={on 4.5pt off 4.5pt}]  (270,220) -- (290,160) ;
\draw [color={rgb, 255:red, 126; green, 211; blue, 33 }  ,draw opacity=1 ] [dash pattern={on 4.5pt off 4.5pt}]  (230,219.96) -- (230,179.92) ;
\draw [color={rgb, 255:red, 126; green, 211; blue, 33 }  ,draw opacity=1 ] [dash pattern={on 4.5pt off 4.5pt}]  (290,200) -- (290,159.96) ;
\draw [color={rgb, 255:red, 74; green, 144; blue, 226 }  ,draw opacity=1 ] [dash pattern={on 4.5pt off 4.5pt}]  (330,200) -- (290,200) ;
\draw [color={rgb, 255:red, 126; green, 211; blue, 33 }  ,draw opacity=1 ] [dash pattern={on 4.5pt off 4.5pt}]  (270,219.96) -- (230,219.96) ;
\draw [color={rgb, 255:red, 126; green, 211; blue, 33 }  ,draw opacity=1 ] [dash pattern={on 4.5pt off 4.5pt}]  (270,220) -- (290,200) ;
\draw [color={rgb, 255:red, 74; green, 144; blue, 226 }  ,draw opacity=1 ] [dash pattern={on 4.5pt off 4.5pt}]  (210,240) -- (230,220) ;
\draw [color={rgb, 255:red, 126; green, 211; blue, 33 }  ,draw opacity=1 ] [dash pattern={on 4.5pt off 4.5pt}]  (230,219.96) -- (250,159.96) ;
\draw [color={rgb, 255:red, 126; green, 211; blue, 33 }  ,draw opacity=1 ] [dash pattern={on 4.5pt off 4.5pt}]  (290,200) -- (230,220) ;
\draw [color={rgb, 255:red, 74; green, 144; blue, 226 }  ,draw opacity=1 ] [dash pattern={on 4.5pt off 4.5pt}]  (230,220) -- (250,200) ;
\draw [color={rgb, 255:red, 74; green, 144; blue, 226 }  ,draw opacity=1 ] [dash pattern={on 4.5pt off 4.5pt}]  (290,200) -- (250,200) ;
\draw [color={rgb, 255:red, 74; green, 144; blue, 226 }  ,draw opacity=1 ] [dash pattern={on 4.5pt off 4.5pt}]  (250,200) -- (250,159.96) ;
\draw  [color={rgb, 255:red, 74; green, 144; blue, 226 }  ,draw opacity=1 ][fill={rgb, 255:red, 74; green, 144; blue, 226 }  ,fill opacity=0.18 ][dash pattern={on 4.5pt off 4.5pt}] (250,159.96) -- (290,199.96) -- (230,219.96) -- cycle ;
\draw  [color={rgb, 255:red, 74; green, 144; blue, 226 }  ,draw opacity=1 ][fill={rgb, 255:red, 74; green, 144; blue, 226 }  ,fill opacity=0.18 ][dash pattern={on 4.5pt off 4.5pt}] (290,159.92) -- (330,199.92) -- (270,219.92) -- cycle ;
\draw  [draw opacity=0][fill={rgb, 255:red, 126; green, 211; blue, 33 }  ,fill opacity=0.28 ] (270,180) -- (290,159.96) -- (330,160) -- (330,199.96) -- (310,219.96) -- (270,220) -- cycle ;
\draw  [color={rgb, 255:red, 74; green, 144; blue, 226 }  ,draw opacity=1 ][fill={rgb, 255:red, 74; green, 144; blue, 226 }  ,fill opacity=0.18 ][dash pattern={on 4.5pt off 4.5pt}] (230,180.04) -- (270,220.04) -- (210,240.04) -- cycle ;
\draw  [color={rgb, 255:red, 74; green, 144; blue, 226 }  ,draw opacity=1 ][fill={rgb, 255:red, 74; green, 144; blue, 226 }  ,fill opacity=0.18 ][dash pattern={on 4.5pt off 4.5pt}] (250,119.92) -- (290,159.92) -- (230,179.92) -- cycle ;
\draw  [color={rgb, 255:red, 189; green, 16; blue, 224 }  ,draw opacity=1 ][fill={rgb, 255:red, 189; green, 16; blue, 224 }  ,fill opacity=0.25 ][dash pattern={on 4.5pt off 4.5pt}] (230,180) -- (290,160) -- (270,220) -- cycle ;
\draw  [draw opacity=0][fill={rgb, 255:red, 126; green, 211; blue, 33 }  ,fill opacity=0.28 ][dash pattern={on 4.5pt off 4.5pt}] (210,199.96) -- (230,179.92) -- (270,179.96) -- (270,219.92) -- (250,239.92) -- (210,239.96) -- cycle ;
\draw  [draw opacity=0][fill={rgb, 255:red, 126; green, 211; blue, 33 }  ,fill opacity=0.28 ] (230,140.04) -- (250,120) -- (290,120.04) -- (290,160) -- (270,180) -- (230,180.04) -- cycle ;
\draw  [draw opacity=0][fill={rgb, 255:red, 74; green, 144; blue, 226 }  ,fill opacity=0.18 ] (250,80) -- (267.03,97.03) -- (290,120) -- (230,140) -- cycle ;
\draw  [draw opacity=0][fill={rgb, 255:red, 74; green, 144; blue, 226 }  ,fill opacity=0.18 ] (290,120.5) -- (330,160.5) -- (270,180.5) -- cycle ;
\draw  [draw opacity=0][fill={rgb, 255:red, 74; green, 144; blue, 226 }  ,fill opacity=0.18 ] (230,140.04) -- (270,180.04) -- (210,200.04) -- cycle ;
\draw  [draw opacity=0][fill={rgb, 255:red, 74; green, 144; blue, 226 }  ,fill opacity=0.18 ] (210,200.5) -- (250,240.5) -- (190,260.5) -- cycle ;
\draw  [draw opacity=0][fill={rgb, 255:red, 74; green, 144; blue, 226 }  ,fill opacity=0.18 ] (270,180) -- (310,220) -- (250,240) -- cycle ;
\draw  [draw opacity=0][fill={rgb, 255:red, 74; green, 144; blue, 226 }  ,fill opacity=0.18 ] (330,160) -- (370,200) -- (310,219.96) -- cycle ;

\end{tikzpicture}

      \caption{Hypersimplicial complex $\mathcal{K}^{[4]}_4$ encoding the intersection of the irreducible components of $\Hilb^4_0(X_4)$.}
      \label{fig:n 4 m 4}
  \end{figure}

\end{Ex}

The faces of the hypersimplicial complex $\mathcal{K}_n^{[m]}$ are characterized in \cref{app: The hypersimplicial complex}. The $(n-r)$--faces may be described as follows. Let $ \max\{ n-m+1,1\}\leq l\leq n-1$, $\mathbf{u}\in \Z_{\geq 1}^n$ with $|\uu|=m+l-1$, and let
$S_1$ and $S_2$ be two disjoint subsets of $[n]$ with $|S_1|+|S_2|= r-1$ and $l\leq |S_1|$ and $|S_2|\leq n-l$. Then, the codimension $r$ faces of $\mathcal{K}_n^{[m]}$ are of the form
\begin{equation}\label{eq:hyper faces}
\begin{array}{c}
\displaystyle\mathcal{K}_n^{[m]}(S_1,S_2,l,\uu)\!:=\!\mathrm{Conv}(\ee_{i_1}\!+\cdots+\ee_{i_{n-l-|S_2|}}\!: i_1,\ldots,i_{n-l-|S_2|}\in [n]\setminus(S_1\sqcup S_2) \text{ distinct})\!+\!\!\sum_{i\in S_2}\!\ee_i\!+\!\uu\!-\!\mathbf{1}\\ = 
\displaystyle
\mathcal{K}_n^{[m]}(S_1,S_2,l,\mathbf{u})=
\left\{
\sum_{i\not\in S_1\sqcup S_2}\lambda_i\ee_i+\sum_{i\in S_2}\ee_i:0\leq \lambda_i\leq 1\text{ and }\sum_{i\not\in S_1\sqcup S_2}\lambda_i=n-l-|S_2|
\right\}+\uu-\mathbf{1}.
\end{array}
\end{equation}
Note that that the face $\mathcal{K}_n^{[m]}(S_1,S_2,l,\mathbf{u})$ is obtained by setting $\lambda_i=0$ for $i\in S_1$ and $\lambda_i=1$ for $i\in S_2$.
The ideals $[J]$ in $\mathrm{Hilb}_0^m(X_n)$ lying on the face $\mathcal{K}_n^{[m]}(S_1,S_2,l,\uu)$ are of the form
\begin{equation}\label{eq: ideal faces}
J=\langle x_i^{u_i}:i\in S_1\rangle +\langle x_i^{u_i+1}:i\in S_2\rangle +\langle g_1,\ldots,g_{l-|S_1|}\rangle,
\end{equation}
where $g_1,\ldots,g_{l-|S_1|}$ are linearly independent polynomials in $\langle x_i^{u_i}:i\in[n]\setminus (S_1\cup S_2)\rangle_{\mathbb{C}}$. 
Geometrically, these ideals form a Grassmannian $\mathrm{Gr}(l-|S_1|,n-r+1)$. 
Note that the ideals in \eqref{eq: ideal faces} are of the same form as those described in \cref{prop: intersections}. In particular, we have that 
\[
\begin{array}{c}
\displaystyle
\mu\left(\Sigma(m,l,\mathbf{u})\cap \Sigma(m,l',\mathbf{v})\right)=\left(\Delta_{n-l,n}+\mathbf{u}-\mathbf{1} \right)\cap \left(\Delta_{n-l',n}+\mathbf{v}-\mathbf{1} \right)=\mathcal{K
}_n^{[m]}(\kappa(\mathbf{u}-\mathbf{v},1),\kappa(\mathbf{u}-\mathbf{v},-1),l,\mathbf{u})\\ \\
\displaystyle = \mathcal{K
}_n^{[m]}(\kappa(\mathbf{u}-\mathbf{v},-1),\kappa(\mathbf{u}-\mathbf{v},-1),l',\mathbf{v}).
\end{array}
\]
We conclude that the hypersimplicial complex $\mathcal{K}_n^{[m]}$ encodes the geometry of irreducible components of $\mathrm{Hilb}_0^m(X_n)$: these components correspond to the Grassmannians associated with hypersimplices, and their intersections are likewise recorded. However, $\mathcal{K}_n^{[m]}$ cannot describe whether the intersection of these components is transversal or not. In \cref{sec:local hilbert scheme}, this question will be addressed. To do so, we associate to
 the hypersimplicial complex $\mathcal{K}_n^{[m]}$ the variety $\mathcal{G}_n^m$ 
 defined as follows. For $\max\{1,n-m+1\}\leq l\leq n-1$ and $\mathbf{u}\in\Z_{\geq1}$ with $|\mathbf{u}|=m+l-1$, we consider the Grassmannian $\mathrm{Gr}(l,\Lambda_\mathbf{u})$ (see \cref{eq:grass map}). Recall that $\mathrm{Gr}(l,\Lambda_\mathbf{u})$ is isomorphic to $\Sigma(m,l,\mathbf{u})$ via the map $\varphi_{l,\mathbf{u}}$ (see \cref{eq:grass map} and \cref{irr comp Grass} ). Using this map, we can consider the equivalence relation in the disjoint union
\begin{equation}\label{eq:disjoint union}
\bigsqcup_{l=\max\{1,n-m+1\}}^{n-1}\bigsqcup_{{\tiny\begin{array}{c}\uu\in\Z_{\geq 1}^n\\ |\uu|=m+l-1\end{array}}} \mathrm{Gr}(l,\Lambda_\mathbf{u})
\end{equation}
given by $[E]\sim [E']$ for $[E]\in \mathrm{Gr}(l,\Lambda_\mathbf{u})$
and $[E']\in\mathrm{Gr}(l',\Lambda_\mathbf{v})$ if and only if $\varphi_{l,\uu}([E])=\varphi_{l',\mathbf{v}}([E'])$. 
This condition occurs only in the intersection of $\Sigma(m,l,\uu)$ and $\Sigma(m,l',\mathbf{v})$, which by \cref{Cor:intersect} is a Grassmannian. The variety $\mathcal{G}_n^m$ is the variety obtained by quotienting \eqref{eq:disjoint union} by this equivalence relation. In other words, $\mathcal{G}_n^m$ is obtained by gluing the Grassmannians $\mathrm{Gr}(l,\Lambda_\mathbf{u})$ via smaller Grassmannians. Since the intersections of the Grassmannians $\Sigma(m,l,\uu)$ are described by $\mathcal{K}_n^{[m]}$, we obtain that $\mathcal{G}_n^m$ is obtained by gluing the Grassmannians $\mathrm{Gr}(l,\Lambda_\mathbf{u})$ via the smaller Grassmannians corresponding to the faces of $\mathcal{K}_n^{[m]}$. Formally, this gluing may be done iteratively. In \eqref{eq:disjoint union} we first glue together the points that corresponds to the same vertex in $\mathcal{K}^{[m]}_n$. Then, we glue the lines that correspond to the same edges of $\mathcal{K}^{[m]}_n$. Inductively on the dimension of the faces, we glue together the subgrassmannians that corresponds to the same face of the complex. 

\begin{Ex}\label{ex: gr n 3}
    Fix $n=3$. For $m=2$, the hypersimplicial complex $\mathcal{K}_3^{[2]}$ coincides with $\Delta_{1,3}$, and hence, $\mathcal{G}_3^{2}\simeq \P^2$. For $m=3$, $\mathcal{K}_3^{[3]}$ has four $2$--dimensional hypersimplices as shown in \cref{ex: n 3 part 1}. These hypersimplices are $\Delta_{2,3}$ and $\Delta_{1,3}+\mathbf{e}_i$ for $i\in[3]$, which are depicted in \cref{fig: n=3 simplex}. The variety $\mathcal{G}_3^{3}$ is obtained by gluing $4$ copies of $\P^2$ following the intersections of the corresponding hypersimplices in $\mathcal{K}_3^{[3]}$: at each of the three $(\C^{*})^3$--invariants lines of $\P^2$ we glue a copy of $\P^2$ through one of its invariant lines. Similarly, for $m=4$, $\mathcal{G}_3^{4}$ is obtained by gluing $9$ copies of $\P^2$ through torus invariant lines following the hypersimplicial complex $\mathcal{K}_3^{[4]}$ depicted in \cref{fig: n=3 simplex}. 
\end{Ex}

\begin{Ex}\label{ex: gr n 4}
    Fix $n=4$. For $m=2$, the hypersimplicial complex $\mathcal{K}_4^{[2]}$ coincides with $\Delta_{1,4}$, and hence, $\mathcal{G}_4^{2}\simeq \P^3$. For $m=3$, $\mathcal{K}_4^{[3]}$ has five $3$ dimensional hypersimplices as shown in \cref{ex: gr n 4}. These hypersimplices are $\Delta_{2,4}$ and $\Delta_{1,4}+\mathbf{e}_i$ for $i\in[4]$, which are depicted in \cref{fig:n 4 m 3}. The variety $\mathcal{G}_4^{3}$ is obtained by gluing $4$ copies of $\P^2$ to $\gr(2,\langle x_1,x_2,x_3,x_4\rangle_\C)\simeq\mathrm{Gr}(2,4)$
    following the intersections of the corresponding hypersimplices in $\mathcal{K}_4^{[3]}$ (see  \cref{fig:n 4 m 3}). The hypersimplex $\Delta_{2,4}$ has eight $2$--dimensional faces isomorphic to the simplex $\Delta_2$. Four of these faces coincide with a faces of each of the hypersimplices $\Delta_{1,4}+\mathbf{e}_i$ for $i\in[4]$. Geometrically, the associated Grassmannian  $\mathrm{Gr}(2,4)$ has eight $\P^2$ embedded which are of the form:
     \[
     \begin{array}{ccc}
    Y_i:=\displaystyle \left\{[E]\in\gr(2,4): E\subseteq \langle x_j:j\neq i\rangle_\C
    \right\} & \text{ and } & \displaystyle Y_i^{*}:= \left\{[E]\in\gr(2,4): \langle x_i\rangle_\C\subseteq E
    \right\},
     \end{array}
     \]
     for $i\in[4]$. The ideals associated to $Y_i$ are of the form $\langle x_i^2, f_1,f_2\rangle$ for $f_1,f_2\in \langle x_j:j\neq i\rangle_\C$, and the ideals associated to $Y_i^*$ are of the form $\langle x_i,  f\rangle$ for $f\in \langle x_j:j\neq i\rangle_\C$. The varieties $Y_i$ correspond to the four faces $\Delta_{2,4}(\emptyset,\{i\},2,\mathbf{1})$ of $\Delta_{2,4}$ that intersect with the hypersimplices $\Delta_{1,4}+\mathbf{e}_i$. For each $i\in [4]$ we glue $\gr(2,4)$ and $\P^3$ through $Y_i$ and a toric invariant plane of $\P^3$. The variety obtained by this glue is $\mathcal{G}_4^{3}$.
\end{Ex}

\begin{figure}[h]
    \centering

\tikzset{every picture/.style={line width=0.75pt}} 

\begin{tikzpicture}[x=0.47pt,y=0.47pt,yscale=-1,xscale=1]

\draw  [color={rgb, 255:red, 245; green, 166; blue, 35 }  ,draw opacity=1 ][fill={rgb, 255:red, 247; green, 202; blue, 130 }  ,fill opacity=0.95 ][line width=1.5]  (488.85,153.87) -- (454.22,94.56) -- (523.47,94.56) -- cycle ;
\draw  [color={rgb, 255:red, 245; green, 166; blue, 35 }  ,draw opacity=1 ][fill={rgb, 255:red, 247; green, 202; blue, 130 }  ,fill opacity=0.95 ][line width=1.5]  (454.22,213.19) -- (419.6,153.87) -- (488.85,153.87) -- cycle ;
\draw  [color={rgb, 255:red, 245; green, 166; blue, 35 }  ,draw opacity=1 ][fill={rgb, 255:red, 247; green, 202; blue, 130 }  ,fill opacity=0.95 ][line width=1.5]  (523.47,213.19) -- (488.85,153.87) -- (558.09,153.87) -- cycle ;
\draw  [color={rgb, 255:red, 206; green, 6; blue, 245 }  ,draw opacity=0.55 ][fill={rgb, 255:red, 189; green, 16; blue, 224 }  ,fill opacity=0.29 ][line width=1.5]  (236.85,34.24) -- (271.47,93.56) -- (202.22,93.56) -- cycle ;
\draw  [color={rgb, 255:red, 206; green, 6; blue, 245 }  ,draw opacity=0.55 ][line width=1.5]  (236.85,152.87) -- (202.22,93.56) -- (271.47,93.56) -- cycle ;
\draw  [color={rgb, 255:red, 206; green, 6; blue, 245 }  ,draw opacity=0.55 ][line width=1.5]  (202.22,212.19) -- (167.6,152.87) -- (236.85,152.87) -- cycle ;
\draw  [color={rgb, 255:red, 206; green, 6; blue, 245 }  ,draw opacity=0.55 ][line width=1.5]  (271.47,212.19) -- (236.85,152.87) -- (306.09,152.87) -- cycle ;
\draw  [color={rgb, 255:red, 206; green, 6; blue, 245 }  ,draw opacity=0.55 ][fill={rgb, 255:red, 189; green, 16; blue, 224 }  ,fill opacity=0.29 ][line width=1.5]  (271.47,93.56) -- (306.09,152.87) -- (236.85,152.87) -- cycle ;
\draw  [color={rgb, 255:red, 206; green, 6; blue, 245 }  ,draw opacity=0.55 ][fill={rgb, 255:red, 189; green, 16; blue, 224 }  ,fill opacity=0.29 ][line width=1.5]  (306.09,152.87) -- (340.71,212.19) -- (271.47,212.19) -- cycle ;
\draw  [color={rgb, 255:red, 206; green, 6; blue, 245 }  ,draw opacity=0.55 ][fill={rgb, 255:red, 189; green, 16; blue, 224 }  ,fill opacity=0.29 ][line width=1.5]  (167.6,152.87) -- (202.22,212.19) -- (132.98,212.19) -- cycle ;
\draw  [color={rgb, 255:red, 206; green, 6; blue, 245 }  ,draw opacity=0.55 ][fill={rgb, 255:red, 189; green, 16; blue, 224 }  ,fill opacity=0.29 ][line width=1.5]  (202.22,93.56) -- (236.85,152.87) -- (167.6,152.87) -- cycle ;
\draw  [color={rgb, 255:red, 206; green, 6; blue, 245 }  ,draw opacity=0.55 ][fill={rgb, 255:red, 189; green, 16; blue, 224 }  ,fill opacity=0.29 ][line width=1.5]  (236.85,152.87) -- (271.47,212.19) -- (202.22,212.19) -- cycle ;

\end{tikzpicture}

    \caption{Hypersimplicial complex $\mathcal{K}_{3,1}^{[4]}$ and $\mathcal{K}_{3,2}^{[4]}.$}
    \label{fig:hypercomplex 3 1 4}
\end{figure}
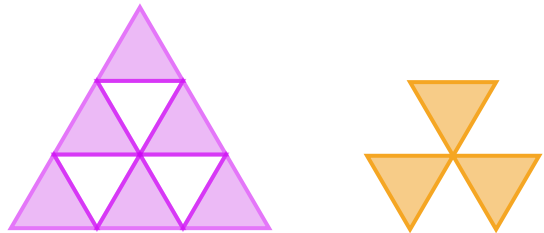

For $\max\{1,n-m+1\}\leq l\leq n-1$, we consider the subvariety $\mathcal{G}_{l,n}^{m}$  of $\mathcal{G}_n^{m}$ given only by the components of $\mathcal{G}_n^{m}$ that are Grassmannians of the form $\gr(l,n)$. Analogously, we consider the hypersimplicial complex $\mathcal{K}_{n-l,n}^{[m]}$, which is the hypersimplicial subcomplex of $\mathcal{K}_{n}^{[m]}$ given by the hypersimplices of the form $\Delta_{n-l,n}+\uu$. We refer to \cref{app: The hypersimplicial complex} for further details on $\mathcal{K}_{n-l,n}^{[m]}$. The variety $\mathcal{G}_{l,n}^{m}$ is the variety obtained by gluing the Grassmannian $\Sigma(m,l,\uu)$ via the faces of $\mathcal{K}_{n-l,n}^{[m]}$.

\begin{Ex}
    The variety $\mathcal{G}_{2,3}^{4}$ consists of $6$ copies of $\P^2$ that are glued together via torus invariant points as in the vertices of the hypersimplicial complex $\mathcal{K}_{1,3}^{[4]}$ illustated in \cref{fig:hypercomplex 3 1 4}. Similarly, $\mathcal{G}_{1,3}^{4}$ is obtained by taking $3$ copies of $\P^2$ and gluing together a torus invariant point on each of them. The corresponding hypersimplicial complex is $\mathcal{K}_{2,3}^{[4]}$ illustrated in \cref{fig:hypercomplex 3 1 4}.

\end{Ex}

\begin{Ex}
 \cref{fig:hypercomplex 4 1 3} depicts the hypersimplicial complexes $\mathcal{K}_{1,4}^{[3]}$ and $\mathcal{K}_{2,4}^{[4]}$. These two complexes represent the varieties $\mathcal{G}_{1,3}^{3}$ and $\mathcal{G}_{2,4}^{4}$. The variety $\mathcal{G}_{1,3}^{3}$ consists of $4$ copies of $\P^3$ glued together via invariant torus points. Similarly, $\mathcal{G}_{2,4}^{4}$  consists of $4$ copies of $\gr(2,4)$ that are glued together via torus invariant lines.

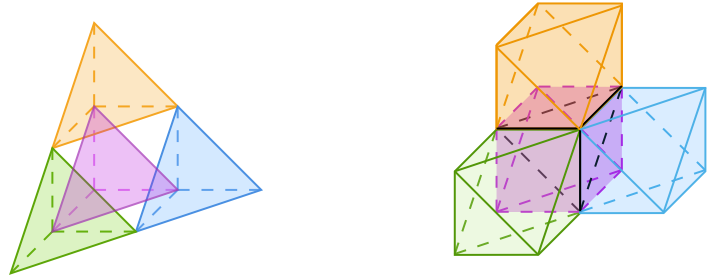
\begin{figure}[h]
    \centering

\tikzset{every picture/.style={line width=0.75pt}} 

\begin{tikzpicture}[x=0.79pt,y=0.79pt,yscale=-1,xscale=1]

\draw [color={rgb, 255:red, 0; green, 0; blue, 0 }  ,draw opacity=1 ]   (445.67,137.67) -- (405.67,137.67) ;
\draw [color={rgb, 255:red, 82; green, 151; blue, 8 }  ,draw opacity=1 ] [dash pattern={on 4.5pt off 4.5pt}]  (385.67,197.86) -- (405.67,137.86) ;
\draw [color={rgb, 255:red, 82; green, 151; blue, 8 }  ,draw opacity=1 ] [dash pattern={on 4.5pt off 4.5pt}]  (385.67,197.86) -- (445.67,177.86) ;
\draw [color={rgb, 255:red, 189; green, 16; blue, 224 }  ,draw opacity=1 ] [dash pattern={on 4.5pt off 4.5pt}]  (405.67,137.67) -- (425.67,117.67) ;
\draw [color={rgb, 255:red, 189; green, 16; blue, 224 }  ,draw opacity=1 ] [dash pattern={on 4.5pt off 4.5pt}]  (465.67,117.67) -- (425.67,117.67) ;
\draw [color={rgb, 255:red, 0; green, 0; blue, 0 }  ,draw opacity=1 ] [dash pattern={on 4.5pt off 4.5pt}]  (445.67,177.67) -- (405.67,137.67) ;
\draw [color={rgb, 255:red, 0; green, 0; blue, 0 }  ,draw opacity=1 ] [dash pattern={on 4.5pt off 4.5pt}]  (405.67,137.67) -- (465.67,117.67) ;
\draw [color={rgb, 255:red, 189; green, 16; blue, 224 }  ,draw opacity=1 ] [dash pattern={on 4.5pt off 4.5pt}]  (405.67,137.67) -- (405.67,177.67) ;
\draw [color={rgb, 255:red, 189; green, 16; blue, 224 }  ,draw opacity=1 ] [dash pattern={on 4.5pt off 4.5pt}]  (465.67,117.67) -- (465.67,157.67) ;
\draw [color={rgb, 255:red, 189; green, 16; blue, 224 }  ,draw opacity=1 ] [dash pattern={on 4.5pt off 4.5pt}]  (405.67,177.67) -- (425.67,117.67) ;
\draw [color={rgb, 255:red, 189; green, 16; blue, 224 }  ,draw opacity=1 ] [dash pattern={on 4.5pt off 4.5pt}]  (465.67,157.63) -- (425.67,117.63) ;
\draw  [draw opacity=0][fill={rgb, 255:red, 189; green, 16; blue, 224 }  ,fill opacity=0.33 ] (405.67,137.78) -- (425.67,117.74) -- (465.67,117.78) -- (465.67,157.74) -- (445.67,177.74) -- (405.67,177.78) -- cycle ;
\draw [color={rgb, 255:red, 189; green, 16; blue, 224 }  ,draw opacity=1 ] [dash pattern={on 4.5pt off 4.5pt}]  (445.67,177.63) -- (465.67,157.63) ;
\draw [color={rgb, 255:red, 189; green, 16; blue, 224 }  ,draw opacity=1 ] [dash pattern={on 4.5pt off 4.5pt}]  (445.67,177.67) -- (405.67,177.67) ;
\draw [color={rgb, 255:red, 189; green, 16; blue, 224 }  ,draw opacity=1 ] [dash pattern={on 4.5pt off 4.5pt}]  (405.67,177.63) -- (465.67,157.63) ;
\draw [color={rgb, 255:red, 0; green, 0; blue, 0 }  ,draw opacity=1 ]   (445.67,137.71) -- (445.67,177.71) ;
\draw [color={rgb, 255:red, 0; green, 0; blue, 0 }  ,draw opacity=1 ] [dash pattern={on 4.5pt off 4.5pt}]  (445.67,177.63) -- (465.67,117.63) ;
\draw [color={rgb, 255:red, 73; green, 176; blue, 231 }  ,draw opacity=0.97 ] [dash pattern={on 4.5pt off 4.5pt}]  (505.67,157.63) -- (445.67,177.63) ;
\draw [color={rgb, 255:red, 73; green, 176; blue, 231 }  ,draw opacity=0.97 ] [dash pattern={on 4.5pt off 4.5pt}]  (505.67,157.78) -- (465.67,118.28) ;
\draw [color={rgb, 255:red, 238; green, 151; blue, 4 }  ,draw opacity=1 ]   (405.67,97.67) -- (425.67,77.67) ;
\draw [color={rgb, 255:red, 0; green, 0; blue, 0 }  ,draw opacity=1 ]   (445.67,137.71) -- (465.67,117.71) ;
\draw [color={rgb, 255:red, 238; green, 151; blue, 4 }  ,draw opacity=1 ]   (405.67,97.67) -- (405.67,137.67) ;
\draw [color={rgb, 255:red, 238; green, 151; blue, 4 }  ,draw opacity=1 ] [dash pattern={on 4.5pt off 4.5pt}]  (405.67,137.67) -- (425.67,77.67) ;
\draw [color={rgb, 255:red, 238; green, 151; blue, 4 }  ,draw opacity=1 ]   (465.67,77.63) -- (465.67,117.63) ;
\draw [color={rgb, 255:red, 238; green, 151; blue, 4 }  ,draw opacity=1 ] [dash pattern={on 4.5pt off 4.5pt}]  (465.67,117.17) -- (425.67,77.67) ;
\draw  [color={rgb, 255:red, 82; green, 151; blue, 8 }  ,draw opacity=1 ][fill={rgb, 255:red, 184; green, 233; blue, 134 }  ,fill opacity=0.25 ] (385.67,158.32) -- (405.67,138.28) -- (445.67,138.32) -- (445.67,178.28) -- (425.67,198.28) -- (385.67,198.32) -- cycle ;
\draw  [color={rgb, 255:red, 73; green, 176; blue, 231 }  ,draw opacity=0.97 ][fill={rgb, 255:red, 59; green, 164; blue, 255 }  ,fill opacity=0.19 ] (445.67,138.32) -- (465.67,118.28) -- (505.67,118.32) -- (505.67,158.28) -- (485.67,178.28) -- (445.67,178.32) -- cycle ;
\draw  [color={rgb, 255:red, 238; green, 151; blue, 4 }  ,draw opacity=1 ][fill={rgb, 255:red, 245; green, 166; blue, 35 }  ,fill opacity=0.34 ] (405.67,97.86) -- (425.67,77.82) -- (465.67,77.86) -- (465.67,107.35) -- (465.67,117.82) -- (445.67,137.82) -- (405.67,137.86) -- cycle ;
\draw [color={rgb, 255:red, 0; green, 0; blue, 0 }  ,draw opacity=1 ][line width=0.75]    (445.67,137.78) -- (405.67,137.78) ;
\draw [color={rgb, 255:red, 0; green, 0; blue, 0 }  ,draw opacity=1 ][line width=0.75]    (445.67,137.38) -- (465.67,117.38) ;
\draw [color={rgb, 255:red, 82; green, 151; blue, 8 }  ,draw opacity=1 ]   (425.67,197.78) -- (445.67,138.28) ;
\draw [color={rgb, 255:red, 82; green, 151; blue, 8 }  ,draw opacity=1 ]   (385.67,158.28) -- (445.67,138.28) ;
\draw [color={rgb, 255:red, 82; green, 151; blue, 8 }  ,draw opacity=1 ]   (425.67,198.28) -- (385.67,158.28) ;
\draw [color={rgb, 255:red, 73; green, 176; blue, 231 }  ,draw opacity=0.97 ]   (485.67,177.78) -- (445.67,138.28) ;
\draw [color={rgb, 255:red, 73; green, 176; blue, 231 }  ,draw opacity=0.97 ]   (485.67,177.78) -- (505.67,118.28) ;
\draw [color={rgb, 255:red, 0; green, 0; blue, 0 }  ,draw opacity=1 ][line width=0.75]    (445.67,137.38) -- (445.67,177.38) ;
\draw [color={rgb, 255:red, 238; green, 151; blue, 4 }  ,draw opacity=1 ]   (465.67,78.82) -- (405.67,98.82) ;
\draw [color={rgb, 255:red, 238; green, 151; blue, 4 }  ,draw opacity=1 ]   (445.67,138.32) -- (465.67,78.82) ;
\draw [color={rgb, 255:red, 238; green, 151; blue, 4 }  ,draw opacity=1 ]   (445.67,138.32) -- (405.67,98.82) ;
\draw [color={rgb, 255:red, 73; green, 176; blue, 231 }  ,draw opacity=0.97 ]   (505.67,118.28) -- (445.67,138.28) ;
\draw [color={rgb, 255:red, 245; green, 166; blue, 35 }  ,draw opacity=1 ] [dash pattern={on 4.5pt off 4.5pt}]  (213,87.14) -- (213,127.14) ;
\draw [color={rgb, 255:red, 74; green, 144; blue, 226 }  ,draw opacity=1 ][fill={rgb, 255:red, 133; green, 192; blue, 255 }  ,fill opacity=0.35 ] [dash pattern={on 4.5pt off 4.5pt}]  (253,167.18) -- (293,167.18) ;
\draw [color={rgb, 255:red, 92; green, 165; blue, 8 }  ,draw opacity=1 ][fill={rgb, 255:red, 126; green, 211; blue, 33 }  ,fill opacity=0.27 ] [dash pattern={on 4.5pt off 4.5pt}]  (193,187.14) -- (173,207.11) ;
\draw [color={rgb, 255:red, 215; green, 78; blue, 237 }  ,draw opacity=0.67 ] [dash pattern={on 4.5pt off 4.5pt}]  (213,167.18) -- (253,167.18) ;
\draw [color={rgb, 255:red, 215; green, 78; blue, 237 }  ,draw opacity=0.67 ] [dash pattern={on 4.5pt off 4.5pt}]  (213,167.18) -- (193,187.14) ;
\draw [color={rgb, 255:red, 215; green, 78; blue, 237 }  ,draw opacity=0.67 ] [dash pattern={on 4.5pt off 4.5pt}]  (213,127.18) -- (213,167.18) ;
\draw [color={rgb, 255:red, 245; green, 166; blue, 35 }  ,draw opacity=1 ] [dash pattern={on 4.5pt off 4.5pt}]  (213,127.18) -- (253,127.18) ;
\draw [color={rgb, 255:red, 92; green, 165; blue, 8 }  ,draw opacity=1 ][fill={rgb, 255:red, 126; green, 211; blue, 33 }  ,fill opacity=0.27 ] [dash pattern={on 4.5pt off 4.5pt}]  (193,187.14) -- (233,187.14) ;
\draw [color={rgb, 255:red, 245; green, 166; blue, 35 }  ,draw opacity=1 ] [dash pattern={on 4.5pt off 4.5pt}]  (213,127.14) -- (193,147.11) ;
\draw [color={rgb, 255:red, 74; green, 144; blue, 226 }  ,draw opacity=1 ] [dash pattern={on 4.5pt off 4.5pt}]  (253,167.18) -- (233,187.14) ;
\draw [color={rgb, 255:red, 92; green, 165; blue, 8 }  ,draw opacity=1 ][fill={rgb, 255:red, 126; green, 211; blue, 33 }  ,fill opacity=0.27 ] [dash pattern={on 4.5pt off 4.5pt}]  (193,147.14) -- (193,187.14) ;
\draw [color={rgb, 255:red, 74; green, 144; blue, 226 }  ,draw opacity=1 ] [dash pattern={on 4.5pt off 4.5pt}]  (253,127.18) -- (253,167.18) ;
\draw  [color={rgb, 255:red, 245; green, 166; blue, 35 }  ,draw opacity=1 ][fill={rgb, 255:red, 245; green, 166; blue, 35 }  ,fill opacity=0.27 ] (213,87.18) -- (253,127.18) -- (193,147.18) -- cycle ;
\draw  [color={rgb, 255:red, 74; green, 144; blue, 226 }  ,draw opacity=1 ][fill={rgb, 255:red, 133; green, 192; blue, 255 }  ,fill opacity=0.35 ] (253,127.18) -- (293,167.18) -- (233,187.18) -- cycle ;
\draw  [color={rgb, 255:red, 93; green, 166; blue, 8 }  ,draw opacity=1 ][fill={rgb, 255:red, 126; green, 211; blue, 33 }  ,fill opacity=0.27 ] (193,147.18) -- (233,187.18) -- (173,207.18) -- cycle ;
\draw  [color={rgb, 255:red, 161; green, 44; blue, 180 }  ,draw opacity=0.7 ][fill={rgb, 255:red, 189; green, 16; blue, 224 }  ,fill opacity=0.26 ] (213,127.14) -- (253,167.14) -- (193,187.14) -- cycle ;

\end{tikzpicture}

    \caption{Hypersimplicial complexes $\mathcal{K}_{1,4}^{[3]}$ and $\mathcal{K}_{2,4}^{[4]}$.}
    \label{fig:hypercomplex 4 1 3}
\end{figure}

\end{Ex}

 \cref{prop: linear and hypersimplices} described the intersection of $\mathcal{K}_{n}^{[m]}$ with a linear subspace of the form 
\begin{equation}\label{eq:linear subspace cut 2}
H(S,\mathbf{a}) :=\{\lambda_i=a_i: i\in S\}\ \text{ for } S\subseteq [n] \text{ and } \aa\in\Z_{\geq 0}^n\text{ with }|\aa|\leq m-1.
\end{equation}
Such intersection is isomorphic to the hypersimplicial complex $\mathcal{K}_{n-|S|}^{[m-|a|]}$. On the levels of ideals, 
$\mu([J])$ for $[J]\in\Hilb_{\mathbf{0}}^m(X_n)$ is contained in 
$H(S,\aa)$  if and only if $x_i^{a_i+1}$ for $i\in S$ among its minimal generators. Therefore, the intersection of $\Hilb_{\mathbf{0}}^m(X_n)$ with this condition is isomorphic to $\Hilb_{\mathbf{0}}^{m-\sum a_i}(X_{n-|S|})$. This intersection coincides with the image of the map
\begin{equation}\label{eq:inductive punctual map}
\begin{array}{cccc}
\iota_{S,\aa}&\Hilb_{\mathbf{0}}^{m-|\aa|}(X_{n-|S|})&\longrightarrow& \Hilb_{\mathbf{0}}^m(X_n)
\\ 

&[J] & \longmapsto & \left[J+\langle x_i^{a_i+1}:i\in S\rangle \right]
\end{array}
\end{equation}
which is an isomorphism onto its image. Moreover, such a map fits in the following commutative diagram
\begin{equation}\label{eq:diagram inductive}
\begin{tikzcd}
\Hilb_{\mathbf{0}}^{m-|\aa|}(X_{n-|S|}) \arrow[r] \arrow[d] & \Hilb_{\mathbf{0}}^{m}(X_{n}) \arrow[d] \\
(m-|\aa|-1)\Delta_{n-|S|-1} \arrow[r]              & (m-1)\Delta_{n-1}              
\end{tikzcd},
\end{equation}
where the vertical maps are the moment map and
the map below is the translation by $\sum_{i\in S}a_i\mathbf{e}_i$.

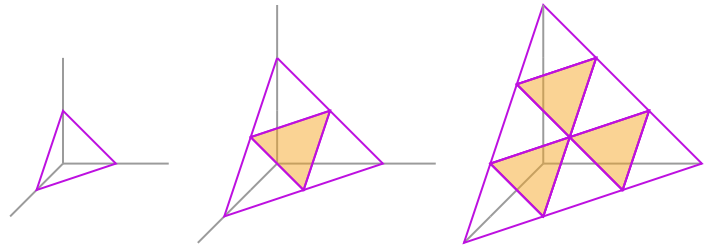
\begin{figure}[h]
    \centering
    \

\tikzset{every picture/.style={line width=0.75pt}} 

\begin{tikzpicture}[x=0.5pt,y=0.5pt,yscale=-1,xscale=1]

\draw [color={rgb, 255:red, 155; green, 155; blue, 155 }  ,draw opacity=1 ]   (138.67,140) -- (138.67,99.96) ;
\draw [color={rgb, 255:red, 155; green, 155; blue, 155 }  ,draw opacity=1 ][line width=0.75]    (218.67,180) -- (178.67,180) ;
\draw [color={rgb, 255:red, 155; green, 155; blue, 155 }  ,draw opacity=1 ]   (98.67,220) -- (118.67,200) ;
\draw [color={rgb, 255:red, 155; green, 155; blue, 155 }  ,draw opacity=1 ]   (118.67,200) -- (138.67,180) ;
\draw [color={rgb, 255:red, 155; green, 155; blue, 155 }  ,draw opacity=1 ]   (178.67,180) -- (138.67,180) ;
\draw [color={rgb, 255:red, 155; green, 155; blue, 155 }  ,draw opacity=1 ]   (138.67,180) -- (138.67,139.96) ;
\draw  [color={rgb, 255:red, 189; green, 16; blue, 224 }  ,draw opacity=1 ] (138.67,140) -- (178.67,180) -- (118.67,200) -- cycle ;
\draw [color={rgb, 255:red, 155; green, 155; blue, 155 }  ,draw opacity=1 ]   (300.67,140) -- (300.67,99.96) ;
\draw [color={rgb, 255:red, 155; green, 155; blue, 155 }  ,draw opacity=1 ]   (240.67,240) -- (260.67,220) ;
\draw [color={rgb, 255:red, 155; green, 155; blue, 155 }  ,draw opacity=1 ]   (420.67,180) -- (380.67,180) ;
\draw [color={rgb, 255:red, 155; green, 155; blue, 155 }  ,draw opacity=1 ][line width=0.75]    (380.67,180) -- (340.67,180) ;
\draw [color={rgb, 255:red, 155; green, 155; blue, 155 }  ,draw opacity=1 ]   (260.67,220) -- (280.67,200) ;
\draw [color={rgb, 255:red, 155; green, 155; blue, 155 }  ,draw opacity=1 ]   (280.67,200) -- (300.67,180) ;
\draw [color={rgb, 255:red, 155; green, 155; blue, 155 }  ,draw opacity=1 ]   (340.67,180) -- (300.67,180) ;
\draw [color={rgb, 255:red, 155; green, 155; blue, 155 }  ,draw opacity=1 ]   (300.67,180) -- (300.67,139.96) ;
\draw  [color={rgb, 255:red, 189; green, 16; blue, 224 }  ,draw opacity=1 ] (280.67,160) -- (320.67,200) -- (260.67,220) -- cycle ;
\draw  [color={rgb, 255:red, 189; green, 16; blue, 224 }  ,draw opacity=1 ] (300.67,99.96) -- (340.67,139.96) -- (280.67,159.96) -- cycle ;
\draw  [color={rgb, 255:red, 189; green, 16; blue, 224 }  ,draw opacity=1 ] (340.67,140) -- (380.67,180) -- (320.67,200) -- cycle ;
\draw  [color={rgb, 255:red, 189; green, 16; blue, 224 }  ,draw opacity=1 ][fill={rgb, 255:red, 245; green, 166; blue, 35 }  ,fill opacity=0.49 ] (280.67,160) -- (340.67,140) -- (320.67,200) -- cycle ;
\draw [color={rgb, 255:red, 155; green, 155; blue, 155 }  ,draw opacity=1 ]   (502,140) -- (502,99.96) ;
\draw [color={rgb, 255:red, 155; green, 155; blue, 155 }  ,draw opacity=1 ]   (442,240) -- (462,220) ;
\draw [color={rgb, 255:red, 155; green, 155; blue, 155 }  ,draw opacity=1 ]   (622,180) -- (582,180) ;
\draw [color={rgb, 255:red, 155; green, 155; blue, 155 }  ,draw opacity=1 ][line width=0.75]    (582,180) -- (542,180) ;
\draw [color={rgb, 255:red, 155; green, 155; blue, 155 }  ,draw opacity=1 ]   (462,220) -- (482,200) ;
\draw [color={rgb, 255:red, 155; green, 155; blue, 155 }  ,draw opacity=1 ]   (482,200) -- (502,180) ;
\draw [color={rgb, 255:red, 155; green, 155; blue, 155 }  ,draw opacity=1 ]   (542,180) -- (502,180) ;
\draw [color={rgb, 255:red, 155; green, 155; blue, 155 }  ,draw opacity=1 ]   (502,180) -- (502,139.96) ;
\draw  [color={rgb, 255:red, 189; green, 16; blue, 224 }  ,draw opacity=1 ] (482,120) -- (522,160) -- (462,180) -- cycle ;
\draw  [color={rgb, 255:red, 189; green, 16; blue, 224 }  ,draw opacity=1 ] (502,59.92) -- (542,99.92) -- (482,119.92) -- cycle ;
\draw  [color={rgb, 255:red, 189; green, 16; blue, 224 }  ,draw opacity=1 ] (542,100) -- (582,140) -- (522,160) -- cycle ;
\draw  [color={rgb, 255:red, 189; green, 16; blue, 224 }  ,draw opacity=1 ] (482,120) -- (542,100) -- (522,160) -- cycle ;
\draw [color={rgb, 255:red, 155; green, 155; blue, 155 }  ,draw opacity=1 ]   (502,99.96) -- (502,59.92) ;
\draw  [color={rgb, 255:red, 189; green, 16; blue, 224 }  ,draw opacity=1 ] (582,140) -- (622,180) -- (562,200) -- cycle ;
\draw  [color={rgb, 255:red, 189; green, 16; blue, 224 }  ,draw opacity=1 ] (522,160) -- (562,200) -- (502,220) -- cycle ;
\draw  [color={rgb, 255:red, 189; green, 16; blue, 224 }  ,draw opacity=1 ] (462,180) -- (502,220) -- (442,240) -- cycle ;
\draw  [color={rgb, 255:red, 189; green, 16; blue, 224 }  ,draw opacity=1 ] (462,180) -- (522,160) -- (502,220) -- cycle ;
\draw  [color={rgb, 255:red, 189; green, 16; blue, 224 }  ,draw opacity=1 ] (522,160) -- (582,140) -- (562,200) -- cycle ;
\draw [color={rgb, 255:red, 155; green, 155; blue, 155 }  ,draw opacity=1 ]   (300.67,99.96) -- (300.67,59.92) ;
\draw  [color={rgb, 255:red, 189; green, 16; blue, 224 }  ,draw opacity=1 ][fill={rgb, 255:red, 245; green, 166; blue, 35 }  ,fill opacity=0.49 ] (482,120) -- (542,100) -- (522,160) -- cycle ;
\draw  [color={rgb, 255:red, 189; green, 16; blue, 224 }  ,draw opacity=1 ][fill={rgb, 255:red, 245; green, 166; blue, 35 }  ,fill opacity=0.49 ] (462,180) -- (522,160) -- (502,220) -- cycle ;
\draw  [color={rgb, 255:red, 189; green, 16; blue, 224 }  ,draw opacity=1 ][fill={rgb, 255:red, 245; green, 166; blue, 35 }  ,fill opacity=0.49 ] (522,160) -- (582,140) -- (562,200) -- cycle ;

\end{tikzpicture}

    \caption{Smoothable faces of $\mathcal{K}_3^{[2]}$, $\mathcal{K}_3^{[3]}$ and $\mathcal{K}_3^{[4]}$}
    \label{fig:Smoothable faces n 3}
\end{figure}

Now, we describe the ideals in $\Hilb_{\mathbf{0}}^m(X_n)$ that are mapped to a smoothable face of $\mathcal{K}_n^{[m]}$ via the moment map. We recall that a face $\Gamma$ of $\mathcal{K}_n^{[m]}$ is smoothable face (see  \cref{def: smoothable face}) if one of the following conditions is satisfied
\begin{itemize}
    \item $n=0$ or $n=1$.
    \item The face $\Gamma$ is contained in $\Delta_{n-1,n}+\mathbf{v}-\mathbf{1}$ for certain $\mathbf{v}$.
    \item The face $\Gamma$ is contained in a linear subspace $H(S,\mathbf{a})$ as in \eqref{eq:linear subspace cut 2}, and in the intersection of $H(S,\mathbf{a})$ and $\mathcal{K}_n^{[m]}$, the face $\Gamma$ is smoothable.
\end{itemize} 

In \cref{app: The hypersimplicial complex} equivalent definitions of smoothable faces are given. \cref{fig:Smoothable faces n 3} illustrates the smoothable faces for $\mathcal{K}_3^{[m]}$ for $m=2,3,4$. The smoothable faces of $\mathcal{K}_4^{[m]}$ for $m=3,4$ are illustrated in  \cref{fig:smoothable faces}. 

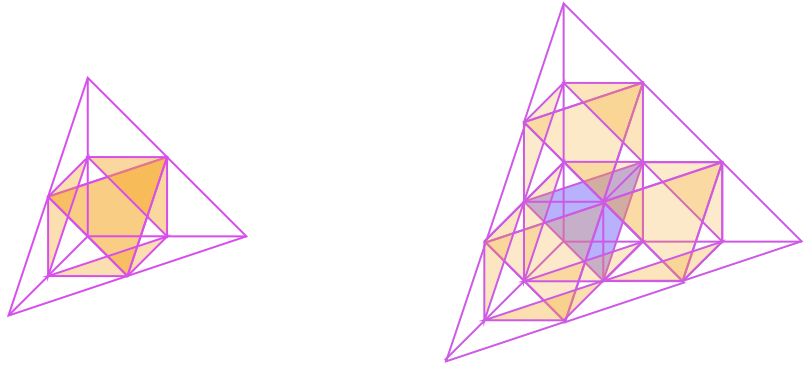
\begin{figure}
    \centering

\tikzset{every picture/.style={line width=0.75pt}} 

\begin{tikzpicture}[x=0.75pt,y=0.75pt,yscale=-1,xscale=1]

\draw [color={rgb, 255:red, 215; green, 78; blue, 237 }  ,draw opacity=1 ]   (187.57,82.85) -- (187.57,122.85) ;
\draw [color={rgb, 255:red, 215; green, 78; blue, 237 }  ,draw opacity=1 ]   (227.57,162.89) -- (267.57,162.89) ;
\draw [color={rgb, 255:red, 215; green, 78; blue, 237 }  ,draw opacity=1 ]   (167.57,182.93) -- (147.57,202.89) ;
\draw  [color={rgb, 255:red, 215; green, 78; blue, 237 }  ,draw opacity=1 ][fill={rgb, 255:red, 74; green, 144; blue, 226 }  ,fill opacity=0 ] (187.57,82.85) -- (227.57,122.85) -- (167.57,142.85) -- cycle ;
\draw  [color={rgb, 255:red, 215; green, 78; blue, 237 }  ,draw opacity=1 ][fill={rgb, 255:red, 245; green, 166; blue, 35 }  ,fill opacity=0.56 ] (227.57,122.89) -- (207.57,182.89) -- (167.57,142.89) -- cycle ;
\draw  [color={rgb, 255:red, 215; green, 78; blue, 237 }  ,draw opacity=1 ][fill={rgb, 255:red, 245; green, 166; blue, 35 }  ,fill opacity=0.44 ] (227.57,122.89) -- (227.57,162.89) -- (187.57,122.85) -- cycle ;
\draw  [color={rgb, 255:red, 215; green, 78; blue, 237 }  ,draw opacity=1 ][fill={rgb, 255:red, 245; green, 166; blue, 35 }  ,fill opacity=0.36 ] (187.57,122.93) -- (167.57,142.93) -- (167.57,182.93) -- cycle ;
\draw  [color={rgb, 255:red, 215; green, 78; blue, 237 }  ,draw opacity=1 ][fill={rgb, 255:red, 245; green, 166; blue, 35 }  ,fill opacity=0.36 ] (207.57,182.97) -- (227.57,162.97) -- (167.57,182.93) -- cycle ;
\draw  [color={rgb, 255:red, 215; green, 78; blue, 237 }  ,draw opacity=1 ][fill={rgb, 255:red, 74; green, 144; blue, 226 }  ,fill opacity=0 ] (167.57,142.89) -- (207.57,182.89) -- (147.57,202.89) -- cycle ;
\draw  [color={rgb, 255:red, 215; green, 78; blue, 237 }  ,draw opacity=1 ][fill={rgb, 255:red, 74; green, 144; blue, 226 }  ,fill opacity=0 ] (227.57,122.89) -- (267.57,162.89) -- (249,169.08) -- (207.57,182.89) -- cycle ;
\draw [color={rgb, 255:red, 215; green, 78; blue, 237 }  ,draw opacity=1 ]   (187.57,162.85) -- (227.57,162.85) ;
\draw [color={rgb, 255:red, 215; green, 78; blue, 237 }  ,draw opacity=1 ]   (187.57,162.85) -- (167.57,182.81) ;
\draw [color={rgb, 255:red, 215; green, 78; blue, 237 }  ,draw opacity=1 ]   (187.57,122.85) -- (187.57,162.85) ;
\draw  [color={rgb, 255:red, 205; green, 89; blue, 229 }  ,draw opacity=1 ][fill={rgb, 255:red, 183; green, 177; blue, 252 }  ,fill opacity=0.98 ] (407.57,145.31) -- (467.57,125.31) -- (447.57,185.31) -- cycle ;
\draw [color={rgb, 255:red, 205; green, 89; blue, 229 }  ,draw opacity=1 ]   (427.57,125.55) -- (427.57,165.55) ;
\draw [color={rgb, 255:red, 205; green, 89; blue, 229 }  ,draw opacity=1 ]   (467.57,125.39) -- (507.57,125.39) ;
\draw [color={rgb, 255:red, 205; green, 89; blue, 229 }  ,draw opacity=1 ]   (407.57,145.43) -- (387.57,165.39) ;
\draw  [color={rgb, 255:red, 205; green, 89; blue, 229 }  ,draw opacity=1 ][fill={rgb, 255:red, 245; green, 166; blue, 35 }  ,fill opacity=0.25 ] (467.57,85.27) -- (447.57,145.27) -- (407.57,105.27) -- cycle ;
\draw  [color={rgb, 255:red, 205; green, 89; blue, 229 }  ,draw opacity=1 ][fill={rgb, 255:red, 245; green, 166; blue, 35 }  ,fill opacity=0.22 ] (467.57,125.31) -- (467.57,165.31) -- (427.57,125.27) -- cycle ;
\draw  [color={rgb, 255:red, 205; green, 89; blue, 229 }  ,draw opacity=1 ][fill={rgb, 255:red, 245; green, 166; blue, 35 }  ,fill opacity=0.23 ] (427.57,85.43) -- (407.57,105.43) -- (407.57,145.43) -- cycle ;
\draw  [color={rgb, 255:red, 205; green, 89; blue, 229 }  ,draw opacity=1 ][fill={rgb, 255:red, 74; green, 144; blue, 226 }  ,fill opacity=0 ] (408.29,105.98) -- (448.29,145.98) -- (388.29,165.98) -- cycle ;
\draw  [color={rgb, 255:red, 205; green, 89; blue, 229 }  ,draw opacity=1 ][fill={rgb, 255:red, 74; green, 144; blue, 226 }  ,fill opacity=0 ] (467.57,85.39) -- (507.57,125.39) -- (489,131.58) -- (447.57,145.39) -- cycle ;
\draw [color={rgb, 255:red, 205; green, 89; blue, 229 }  ,draw opacity=1 ]   (427.57,165.55) -- (467.57,165.55) ;
\draw [color={rgb, 255:red, 205; green, 89; blue, 229 }  ,draw opacity=1 ]   (427.57,165.31) -- (407.57,185.27) ;
\draw [color={rgb, 255:red, 205; green, 89; blue, 229 }  ,draw opacity=1 ]   (427.57,85.35) -- (427.57,125.35) ;
\draw [color={rgb, 255:red, 205; green, 89; blue, 229 }  ,draw opacity=1 ]   (447.57,145.31) -- (447.57,185.31) ;
\draw [color={rgb, 255:red, 205; green, 89; blue, 229 }  ,draw opacity=1 ]   (507.57,125.39) -- (507.57,165.39) ;
\draw [color={rgb, 255:red, 205; green, 89; blue, 229 }  ,draw opacity=1 ]   (467.57,165.55) -- (507.57,165.55) ;
\draw [color={rgb, 255:red, 205; green, 89; blue, 229 }  ,draw opacity=1 ]   (507.57,165.55) -- (487.57,185.51) ;
\draw  [color={rgb, 255:red, 205; green, 89; blue, 229 }  ,draw opacity=1 ][fill={rgb, 255:red, 74; green, 144; blue, 226 }  ,fill opacity=0 ] (507.57,125.39) -- (547.57,165.39) -- (487.57,185.39) -- cycle ;
\draw  [color={rgb, 255:red, 205; green, 89; blue, 229 }  ,draw opacity=1 ][fill={rgb, 255:red, 245; green, 166; blue, 35 }  ,fill opacity=0.24 ] (448.29,145.98) -- (428.29,205.98) -- (388.29,165.98) -- cycle ;
\draw  [color={rgb, 255:red, 205; green, 89; blue, 229 }  ,draw opacity=1 ][fill={rgb, 255:red, 245; green, 166; blue, 35 }  ,fill opacity=0.23 ] (507.57,125.51) -- (507.57,165.51) -- (467.57,125.47) -- cycle ;
\draw  [color={rgb, 255:red, 205; green, 89; blue, 229 }  ,draw opacity=1 ][fill={rgb, 255:red, 245; green, 166; blue, 35 }  ,fill opacity=0.24 ] (427.57,125.55) -- (407.57,145.55) -- (407.57,185.55) -- cycle ;
\draw  [color={rgb, 255:red, 205; green, 89; blue, 229 }  ,draw opacity=1 ][fill={rgb, 255:red, 245; green, 166; blue, 35 }  ,fill opacity=0.25 ] (447.57,185.59) -- (467.57,165.59) -- (407.57,185.55) -- cycle ;
\draw  [color={rgb, 255:red, 205; green, 89; blue, 229 }  ,draw opacity=1 ][fill={rgb, 255:red, 74; green, 144; blue, 226 }  ,fill opacity=0 ] (388.29,165.98) -- (417.57,195.25) -- (428.29,205.98) -- (368.29,225.98) -- cycle ;
\draw [color={rgb, 255:red, 205; green, 89; blue, 229 }  ,draw opacity=1 ]   (447.57,185.31) -- (487.57,185.31) ;
\draw [color={rgb, 255:red, 205; green, 89; blue, 229 }  ,draw opacity=1 ]   (467.57,165.31) -- (447.57,185.27) ;
\draw [color={rgb, 255:red, 205; green, 89; blue, 229 }  ,draw opacity=1 ]   (467.57,125.31) -- (467.57,165.31) ;
\draw [color={rgb, 255:red, 205; green, 89; blue, 229 }  ,draw opacity=1 ]   (407.57,145.43) -- (447.57,145.43) ;
\draw [color={rgb, 255:red, 205; green, 89; blue, 229 }  ,draw opacity=1 ]   (467.57,125.47) -- (447.57,145.43) ;
\draw  [color={rgb, 255:red, 205; green, 89; blue, 229 }  ,draw opacity=1 ][fill={rgb, 255:red, 245; green, 166; blue, 35 }  ,fill opacity=0.24 ] (507.57,125.51) -- (487.57,185.51) -- (447.57,145.51) -- cycle ;
\draw [color={rgb, 255:red, 205; green, 89; blue, 229 }  ,draw opacity=1 ]   (507.57,165.55) -- (547.57,165.55) ;
\draw  [color={rgb, 255:red, 205; green, 89; blue, 229 }  ,draw opacity=1 ][fill={rgb, 255:red, 74; green, 144; blue, 226 }  ,fill opacity=0 ] (448.29,146.21) -- (477.57,175.49) -- (488.29,186.21) -- (428.29,206.21) -- cycle ;
\draw [color={rgb, 255:red, 205; green, 89; blue, 229 }  ,draw opacity=1 ]   (387.57,205.51) -- (367.57,225.47) ;
\draw [color={rgb, 255:red, 205; green, 89; blue, 229 }  ,draw opacity=1 ]   (427.57,45.43) -- (427.57,85.43) ;
\draw  [color={rgb, 255:red, 205; green, 89; blue, 229 }  ,draw opacity=1 ][fill={rgb, 255:red, 74; green, 144; blue, 226 }  ,fill opacity=0 ] (427.57,45.31) -- (467.57,85.31) -- (407.57,105.31) -- cycle ;
\draw [color={rgb, 255:red, 205; green, 89; blue, 229 }  ,draw opacity=1 ]   (407.57,185.55) -- (387.57,205.51) ;
\draw [color={rgb, 255:red, 205; green, 89; blue, 229 }  ,draw opacity=1 ]   (427.57,125.55) -- (407.57,145.51) ;
\draw  [color={rgb, 255:red, 205; green, 89; blue, 229 }  ,draw opacity=1 ][fill={rgb, 255:red, 245; green, 166; blue, 35 }  ,fill opacity=0.28 ] (407.57,145.55) -- (387.57,165.55) -- (387.57,205.55) -- cycle ;
\draw  [color={rgb, 255:red, 205; green, 89; blue, 229 }  ,draw opacity=1 ][fill={rgb, 255:red, 245; green, 166; blue, 35 }  ,fill opacity=0.3 ] (467.57,85.47) -- (467.57,125.47) -- (427.57,85.43) -- cycle ;
\draw  [color={rgb, 255:red, 205; green, 89; blue, 229 }  ,draw opacity=1 ][fill={rgb, 255:red, 245; green, 166; blue, 35 }  ,fill opacity=0.3 ] (487.57,185.51) -- (507.57,165.51) -- (447.57,185.47) -- cycle ;
\draw  [color={rgb, 255:red, 205; green, 89; blue, 229 }  ,draw opacity=1 ][fill={rgb, 255:red, 245; green, 166; blue, 35 }  ,fill opacity=0.36 ] (427.57,205.27) -- (447.57,185.27) -- (387.57,205.23) -- cycle ;

\end{tikzpicture}

    \caption{Smoothable faces of the complexes $\mathcal{K}_4^{[3]}$ and $\mathcal{K}_4^{[4]}$.}
    \label{fig:smoothable faces}
\end{figure}

\begin{Prop}
    \label{prop: ideals to smoothable faces}
    Let $[J]\in \Hilb_{\mathbf{0}}^m(X_n)$, then, $\mu([J])$ lies in a smoothable face of $\mathcal{K}_n^{[m]}$ if and only if $J=\langle x_i^{u_i}:i\in S\rangle +\langle f\rangle$ for $S\subset [n]$, $\uu\in\Z_{\geq 1}^n$ with $|\uu|=m+|S|$, and $f\in\langle x_i^{u_i}:i\not\in S\rangle_\C$.
\end{Prop}
\begin{proof}
    Assume first that $J=\langle x_i^{u_i}:i\in S\rangle +\langle f\rangle$ and consider the integer vectors $\uu_S=\sum_{i\in S} u_i\mathbf{e}_i$ and  $\uu_{[n]\setminus S}=\sum_{i\not\in S} u_i\mathbf{e}_i$. Consider the ideal $J'=\langle f\rangle \subset\C[x_i:i\not\in S]$. Then, $J'$ is a length $|\uu_{[n]\setminus S}|$ and via the map \eqref{eq:map induct}, we get that $[J]=\iota_{S,\uu_S}([J'])$. Now, $[J']\in \Hilb_{\mathbf{0}}^{|\uu_{[n]\setminus S}|}(X_{n-|S|})$ is contained in the Grassmannian $\Sigma(|\uu_{[n]\setminus S}|,1,\uu_{[n]\setminus S})$. Therefore, $\mu([J'])$ lies in $\Delta_{n-1,n}+\uu_{[n]\setminus S}-\mathbf{1}$, and hence,  $\mu([J'])$ is contained in a smoothable face. Using the commutative diagram \eqref{eq:diagram inductive}, we deduce that $\mu([J])$ is contained in a smoothable face.

    Assume now that $\mu([J])$ lies in a smoothable face $\Gamma$. Then, $\Gamma$ is a hypersimplex of the form $\Delta_{n'-1,n'}$  for $n'\leq n$ (see \cref{app: The hypersimplicial complex}). By the commutative diagram \eqref{eq:diagram inductive}, it is enough to check the case when $n'=n$. In other words, $\Gamma =\Delta_{n-1,n}+\uu-\mathbf{1}$.
    In this case,  $\mu([J])$ lies in $\Gamma$ if and only if $[J]\in \Sigma(m,1,\uu)$. The proof follows from the fact that any ideal in $\Sigma(m,1,\uu)$ is of the form  $\langle x_i^{u_i+1}:i\in S\rangle +\langle f\rangle$ for $S\subsetneq [n]$ and $f\in\langle x_i^{u_i}:i\not\in S\rangle_\C$.
\end{proof}

Analogously to the notion of smoothable face, the notion of singular face is introduced in \cref{app: The hypersimplicial complex}, which we recall here for convenience: A face $\Gamma$ of $\mathcal{K}_n^{[m]}$ is \emph{singular} if one of the following conditions is satisfied:
    \begin{itemize}
        \item The face $\Gamma$ is in the intersection of two distinct maximal faces.
        \item The face $\Gamma$ is smoothable of dimension at most $n-2$, i.e. at most codimension $1$.
    \end{itemize}

In the following Proposition we describe the ideals that are contained in a singular face via the moment map.

\begin{Prop}
    \label{prop: ideal singular face}
    Let $[J]\in \Hilb_{\mathbf{0}}^{m}(X_n)$. Then $\mu([J])$ is contained in a singular face if and only if one of the following conditions is satisfied
    \begin{itemize}
    \item $J$ admits a minimal generator of the form $x_i^{u_i}$ for $u_i\geq 2$. 
    \item  $J$ admits a minimal representation of the form $\langle f,x_i:i\in S\rangle$ for $f\in \langle x_i^{u_i};i\not\in S\rangle_\C$ and $\emptyset\subsetneq S\subset[n]$.
\end{itemize}
\end{Prop}
\begin{proof}
Let $[J]\in \Hilb^m_\mathbf{0}(X_n)$ lying in the component $\Sigma(m,l,\uu)$.
    By \cref{prop: ideals to smoothable faces}, $\mu([J])$ lies in a smoothable face if and only if $J=\langle f,x_i:i\in S\rangle$ for $f\in \langle x_i^{u_i};i\not\in S\rangle_\C$ and $ S\subset[n]$. Moreover, such a face has dimension $n-1$ if and only if $S=\emptyset$. Therefore, the second condition of the definition of singular face corresponds to the second condition in  \cref{prop: ideal singular face}. 
    
    Assume now that $J$ contains a generator of the form $x_i^{u_i}$ with $u_i\geq 2$. Then $[J]$ also lies in the Grassmannian $\Sigma(m,r+|S|-1,\uu-\mathbf{e}_i)$. Therefore, $\mu([J])$ is singular, since it lies in the intersection of the hypersimplices corresponding to these two Grassmannians. Now assume that $\mu([J])$ is contained in the intersection of two distinct hypersimplices $\Delta_{l,n}+\uu-\mathbf{1}$ and $\Delta_{l',n}+\mathbf{v}-\mathbf{1}$. Then, $[J]$ is contained in the intersection of the Grassmannians $\Sigma(m,l,\uu)$ and $\Sigma(m,l',\mathbf{v})$. 
    By \cref{prop: intersections}, $J$ is as in \eqref{eq:ideal intersect}. Since $\uu\neq \mathbf{v}$, either $\kappa(\uu-\mathbf{v},1)$ or $\kappa(\uu-\mathbf{v},-1)$ is nonempty. Without loss of generality assume that $\kappa(\uu-\mathbf{v},1)\neq \emptyset$. Then, there exists $i\in [n]$ such that $x_i^{u_i}=x_i^{v_i+1}$ is a minimal generator of $J$. Since $v_i\geq1$, we deduce that $J$ admits a minimal generators of the form $x_i^{u_i}$ with $u_i\geq 2$.
\end{proof}

In \cref{sec:singularites and local hilbert scheme} and \cref{sec:sm and non sm}, we relate the notions of smoothable and singular faces with the smoothable ideals and the singular locus of $\sHilb^m(X_n)$.

\section{From the punctual to the global Hilbert scheme}\label{sec:localtoglobal}

With the notation and definitions of the previous sections, we will study the relation behind the combinatorics and the geometry of the Hilbert scheme $\Hilb^m(X_n)$. From this interplay we will get all the irreducible components of $\Hilb^m(X_n)$. 

For  $2\leq m'\leq \min\{m,n-1\}$, $l=n+1-m'$, and for $\uu\in\Z_{\geq 0}^n$ with $|\uu|=m-m'$, consider the isomorphism
\begin{equation}\label{eq: new map grass}
\phi_{m',\uu}:\Sigma(m',l,\mathbf{1})\rightarrow \Sigma(m,l,\uu+\mathbf{1})
\end{equation}
given by the composition
\[
\begin{array}{ccccccc}
\Sigma(m',l,\mathbf{1})&\overset{\varphi_{l,\mathbf{1}}^{-1}}{\longrightarrow }&\gr(l,\Lambda_{\mathbf{1}})&\longrightarrow&\gr(l,\Lambda_{\uu+\mathbf{1}})&\overset{\varphi_{l,\uu}}{\longrightarrow} &\Sigma(m,l,\uu+\mathbf{1}),
\end{array}
\]
where the middle map is induced by the isomorphism of vector spaces $\Lambda_{\mathbf{1}}\rightarrow \Lambda_{\uu+\mathbf{1}}$ that sends $x_i$ to $x_i^{u_i+1}$.
In other words, for $[J]\in\Sigma(m',l,\mathbf{1})$, its image $\phi_{m',\uu}([J])$ is the ideal whose generators are obtained by replacing $x_i$ by $x_i^{u_i+1}$. Moreover, 
the following diagram commutes
\[
\begin{tikzcd}
{\Sigma(m',l,\mathbf{1})} \arrow[d, "\mu"'] \arrow[rr, "{\phi_{m',\mathbf{u}}}"] &  & {\Sigma(m',l,\mathbf{u}+\mathbf{1})} \arrow[d, "\mu"] \\
{\Delta_{n-l,n}} \arrow[rr, "+\mathbf{u}"']                                      &  & {\Delta_{n-l,n}+\mathbf{u}}                          
\end{tikzcd}.
\]
In particular, the map $\phi_{m',\uu-\mathbf{1}}$ is the geometric analogous to the translation $+\uu-\mathbf{1}$ in the definition of the moment map.

\begin{Prop}\label{prop:limits}
    For $2\leq m'\leq \min\{m,n-1\}$ and for $\uu\in\Z_{\geq 0}^n$ with $|\uu|=m-m'$, consider the rational map
    \begin{equation}
        \label{eq:induc map lines}
        \begin{array}{ccc}
 \displaystyle\Sigma(m',n+1-m',\mathbf{1})\times\prod_{i\in[n]}\Sym^{u_i}L_i & \dashrightarrow&\Hilb^m(X_n)\\
         ([J],q_1,\ldots,q_n)&\longmapsto&\mathbb{V}(J)\cup q_1\cup\cdots \cup q_n
        \end{array}.
    \end{equation}
   Let $[J]\in \Sigma(m',n+1-m',\mathbf{1})$ and let $Z$ be a one parameter family in the domain of \eqref{eq:induc map lines} that contains the point $z_0=([J],u_1\cdot\mathbf{0},\ldots,u_n\cdot\mathbf{0})$ and the domain of definition of \eqref{eq:induc map lines} intersect $Z$ in a dense open subset. Then,
    the restriction of the map \eqref{eq:induc map lines} to $Z$ can be uniquely extended to $z_0$ and the image of $z_0$ is $\phi_{m',\uu}([J])$, which does not depend on the family $Z$.
\end{Prop}
\begin{proof}
    Let $z=([J],q_1,\ldots,q_n)\in Z$ lying in the domain of definition of \eqref{eq:induc map lines}. In particular, $q_i$ represents $u_i$ points in $L_i\setminus\{\mathbf{0}\}$. Let $I_i=\langle x_j:j\neq i\rangle +\langle f_i\rangle$ be the ideal of $q_i$. In other words, we can write $f_i$ as $f_i=1+f_i'$ where $f_i'$ is a degree $u_i$ polynomial in $x_i$ with $f_i'(0)=0$. Then, the ideal of $q_1\cup\cdots \cup q_n$ in $R_n$ is $\langle f_1\cdots f_n\rangle =\langle 1+\sum f_j'\rangle$. Let $g_1,\ldots,g_{n+1-m'}$ be the generators of $J$, where $g_i=a_{i,1}x_1+\ldots,a_{i,n}x_n$. Therefore, the ideal of $\mathbb{V}(J)\cup q_1\cup\cdots \cup q_n$ is generated by 
    \[
    \langle g_1 \left(  1+\sum f_j\right) ,\ldots, g_{n+1-m'} \left(  1+\sum f_j\right)\rangle.
    \]
    Now, we have that 
    \[
    g_i \left(  1+\sum_{j=1}^n f_j\right)=g_i+\sum_{j=1}^n a_{i,j}x_jf_j'=\sum_{j=1}^n a_{i,j}x_j + a_{i,j}x_jf_j'=\sum_{j=1}^n a_{i,j}x_j(1+f_j')=\sum_{j=1}^n a_{i,j}x_jf_j.
    \]
    Now, consider any limit of the form $\underset{z\rightarrow z_0}{\lim} z$. In other words, we are taking the limit when $f_j$ goes to $x_j^{u_j}$ for all $j\in[n]$.
Then, the image of such a limit is a length $m$ ideal that must contained the ideal generated by
\[
\sum_{j=1}^n a_{i,j}x_j^{u_j+1} \text{ for all } i\in[n].
\]
Note that this ideal is exactly $\phi_{m',\uu}([J])$. In particular, 
the length of this ideal is $|\uu|+n+1-(n+1-m')=m$, and we conclude that $\phi_{m',\uu}([J])$ is the image of the limit.
\end{proof}

\begin{Rem}\label{remark: comb vs geometry}
     \cref{prop:limits} allows us to relate the combinatorics studied in \cref{sec:moment map} and the geometry of the Hilbert scheme. Mainly, it shows the relation between the map \eqref{eq:induc map lines} and the translations made in the definition of the moment map and in the hypersimplices of $\mathcal{K}_n^{[m]}$. Consider the translation by $\uu$ between the hypersimplices $\Delta_{m'-1,n}$ and  $\Delta_{m'-1,n}+\uu$ for $\uu\in\Z_{\geq 0}^{n}$ and $|\uu|=m-m'$. Let $[J]\in\Sigma(m',n+1-m',\mathbf{1})$, then $\mu([J])$ is in $\Delta_{m'-1,n}$. The translation $\mu([J])+\uu$ can be interpreted geometrically as follows. Let $Z$ be the length $m$ ideal obtained by adding $u_i$ nonzero points in the line $L_i$  to the scheme $\mathbb{V}(J)$ for all $i$. By \cref{prop:limits}, collapsing all these nonzero points to the singularity leads to the length $m$ ideal $\phi_{m',\uu}([J])$ in $\Hilb_{\mathbf{0}}^{m}(X_n)$. The image via the moment map of this ideal is exactly $\mu([J])+\uu$. Thus, the translation in the definition of the moment map is interpreted geometrically as adding to $\mathbb{V}(J)$ extra nonzero points in the lines of $X_n$ and collapsing them to the singularity. 
     This relation is explored in more detail in \cref{sec: nonsm}.
\end{Rem}


We now calculate the irreducible components of $\Hilb^m(X_n)$. The strategy is to use \cref{prop:limits} to distinguish which irreducible components of the punctual Hilbert scheme $\Hilb_{\mathbf{0}}^m(X_n)$ lift to an elementary component of $\Hilb^m(X_n)$. To do so, we first introduce the candidates to irreducible components of $\Hilb^{m}(X_n)$.
 For $\uu\in \Z_{\geq 1}^{n}$ with $|\uu|=m+l-1$,  we define the rational map 
\begin{equation}
    \label{eq:induc map}
    X_n\times \Sigma(m,l,\uu)\dashrightarrow \Hilb^{m+1}
(X_n)
\end{equation}
that sends a point $q\in X_n$ and an ideal $[J]\in \Sigma(m,l,\uu)$
to the length-$m$ subscheme $\{q\}\cup \V(J)$. The map \eqref{eq:induc map} is not defined in $\mathbf{0}\times \Sigma(m,l,\uu)$. As a consequence, the image of \eqref{eq:induc map} is not a closed subvariety. \cref{prop:limits} allows us to extend \eqref{eq:induc map} to a well-defined map. The image of the base locus is the union of the Grassmannians $\Sigma(m+1,l,\uu+\mathbf{e}_i)$ for $i\in[n]$. In general, for $ 2\leq m'\leq \min\{n-1,m\}$, we define the  map 
\begin{equation}\label{eq:induc map gen}
    \begin{array}{ccc}
    \Hilb_{\mathrm{sm}}^{m-m'}(X_n\setminus\{\mathbf{0}\})\times \Sigma(m',n+1-m',\mathbf{1}) & \longrightarrow &\Hilb^m(X_{n})\\
     (Y,[J])&\longmapsto& Y\cup \V(J) .
    \end{array}
\end{equation}
The extension of \eqref{eq:induc map gen} to the closure of its domain will be studied in  \cref{sec: nonsm}.
We denote the closure of the image of \eqref{eq:induc map gen} by $\Hilb^{m,m'}(X_n)$.  In other words, $\Hilb^{m,m'}(X_n)$ is the reduced version of $\sHilb^{m,m'}(X_n)$, defined as the closure of the locus of points $[J]\in\Hilb^m(X_n)$ such that there exists $J_0$ in the primary decomposition of $J$ supported at $\mathbf{0}$ with $[J_0]\in \Sigma(m',n+1-m',\mathbf{1})$. 
The following result describes which irreducible components of $\Hilb^m_\mathbf{0}(X_n)$ lift to elementary components of $\Hilb^m(X_n)$ and which not. This allows us to compute the irreducible components of $\Hilb^m(X_n)$.

\begin{Thm}\label{theo: irred comp reduced structure}
    Fix $n\geq 1$ and $m\geq 1$. Let $l\in[n-1]$ and let $\uu\in\Z_{\geq 1}^n$ be partition of $m+l-1$.
    \begin{enumerate}
        \item If $\uu = \mathbf{1} = (1,\ldots,1)$ and $m=n$, then $l=1$ and $\Sigma(n,1,\mathbf{1})$ is contained in $\Hilb^m_{\mathrm{sm}}(X_n)$.
        
        \item If $\uu = \mathbf{1} = (1,\ldots,1)$ and $2\leq m\leq n-1$, then $2\leq l=n+1-m\leq n-1$. In this case, $\Sigma(m,n+1-m,\mathbf{1})$ is an irreducible component of $\Hilb^m(X_n)$. In particular, $\Sigma(m,n+1-m,\mathbf{1})$ is an elementary component of $\Hilb^m(X_n)$.
        
        \item If there exists $1\leq i\leq n$ such that $u_i\geq 2$ then $\Sigma(m,l,\uu)$ is contained in the closure of the image of the map 
        \begin{equation}\label{eq:map point in a line}
        (L_i\setminus\{\mathbf{0}\})\times \Sigma(m-1,l,\uu-\ee_I)\longrightarrow \Hilb^m(X_n)
        \end{equation}
        sending a pair $(q,[I])\in(L_i\setminus\{\mathbf{0}\})\times \Sigma(n,l,\uu-\ee_I)$ to the length $m$ scheme $\{q\}\cup\V(I)$. In particular, $\Sigma(m,l,\uu)$ is not an elementary component of $\Hilb^m(X_n)$.
    \end{enumerate}
    Therefore, the irreducible components of $\Hilb^m(X_n)$ are either an irreducible component of the smoothable component or an irreducible component of $\Hilb^{m,m'}(X_n)$ for $2\leq m'\leq \min\{m,n-1\}$.
\end{Thm}
\begin{proof}

    Assume first that $u_i\geq 2$ for some $1\leq i\leq n$. Without loss of generality, we can assume that $u_1\geq 2$. 
        Let $J$ be a generic ideal in $\Sigma(m,l,\uu)$. Then $J$ is generated by $l$ polynomials $f_1,\ldots,f_l$ of the form 
    \[  f_i=a_{i,1}x_1^{u_1}+\cdots+a_{i,n}x_n^{u_n}.
    \]
    We can assume that $a_{1,1}= 1$. Replacing $f_i$ by $a_{1,1}f_i-a_{i,1}f_1$ for $2\leq i\leq l$, we can also assume that $a_{i,1}=0$ for $2\leq i\leq l$. We consider the ideal 
    \[
    J_\lambda=\langle f_1+\lambda x_1^{u_1-1},f_2,\ldots,f_l\rangle
    \]
    for $\lambda\in\C$. Note that $J_0=J$ and $J_\lambda$ is a length $m$ ideal in $R_n$  for all $\lambda\in \C$.
    As in Proposition \ref{prop:limits}, for $\lambda\neq 0$, we may write $J_\lambda$ as
    \[
    J_\lambda = \langle
    (\lambda+x_1)(x_1^{u_1-1}+\lambda^{-1}a_{1,2}x_2^{u_2}+\cdots+\lambda^{-1}a_{1,n}x_2^{u_n}),(\lambda+x_1)f_2,\ldots,(\lambda+x_1)f_n
    \rangle.
    \]
    We deduce that for $\lambda\neq 0$, $\V(J_\lambda)$ is the union of the point $(-\lambda,0,\ldots,0)$ and a length $m-1$ subscheme in $\Sigma(m,l,\uu-\mathbf{e}_1)$ given by the ideal $$J'_\lambda=\langle x_1^{u_1-1}+\lambda^{-1}a_{1,2}x_2^{u_2}+\cdots+\lambda^{-1}a_{1,n}x_2^{u_n},f_2,\ldots,f_n \rangle .$$
    Therefore, for $\lambda\neq 0$, $[J_\lambda]$ is the image via \eqref{eq:map point in a line} of the tuple $((-\lambda,0\ldots,0),[J'_\lambda])$. By \cref{prop:limits}, we deduce that $[J]$ lies in the closure of the image of the map \eqref{eq:map point in a line}. Since $[J]$ is a generic element of $\Sigma(m,l,\uu)$, we conclude that $\Sigma(m,l,\uu)$ is also contained in this closure.

        Now, assume that $l=1$, $m=n$ and $\uu= \mathbf{1}$. 
        In this case, ideals in $\Sigma(n,1,\mathbf{1})$ are generated by a linear form, so they correspond to hyperplanes passing through $\mathbf{0}$. Let $J$ be a generic element in $\Sigma(n,1,\mathbf{1})$  and let $H$ be the corresponding hyperplane. Moreover, let $v$ be a generic vector supported at $\mathbf{0}$ and consider the family of affine hyperplanes $H_t=H+tv$. Since $H$  and $v$ are generic, for $t\neq 0$, $H_t$ intersects each of the lines $L_1,\ldots,L_n$ in a point distinct than $\mathbf{0}$. In particular, $H_t\cap X_n$ consists of $n$ distinct points in $X_n$ and we obtain a family of length $n$ schemes $H_t\cap X_n$ that are smoothable for $t\neq 0$. We conclude that the scheme $H_0\cap X_n=\V(J)$ is smoothable.

        Assume that $2\leq m\leq n-1$, $\uu=(1,\ldots,1)$ and $l=n+1-m$. This implies that $n\geq 3$. We show that $\Sigma(m,n+1-m,\mathbf{1})$ is an irreducible component of $\Hilb^m(X_n)$. For $m=2$, we saw that $\Sigma(2,n-1,\mathbf{1})$ is an irreducible component of $\Hilb^2(X_n)$ in \cref{ex:m=2} since the tangent star is properly contained in the tangent space $\mathrm{T}_\mathbf{0}X_n$. Now assume that $m\geq3$. For $0\leq k\leq m$, we consider the map
        \begin{equation}\label{eq:map induct}
        \Hilb^k(X_n\setminus\{\mathbf{0}\})\times \Hilb^{m-k}_\mathbf{0}(X_n)\dashrightarrow \Hilb^m(X_n)
        \end{equation}
        sending $k$ distinct points $\{q_1,\ldots,q_k\}$ in $X_n\setminus\{\mathbf{0}\}$ and a length $m-k$ scheme $Z$ supported at $\mathbf{0}$ to $\{q_1,\ldots,q_k\}\cup Z$.  We denote the closure of the image of this map by $Y^m(n,k)$. Note that $Y^m(n,m)$ is the smoothable component of $\Hilb^m(X_n)$. 

        To check that $\Sigma(m,n+1-m,\mathbf{1})$ is an elementary component of $\Hilb^m(X_n)$ it is enough to show that it is not contained in  $Y^m(n,k)$ for all $1\leq k\leq m$. We argue by contradiction as follows. Let $[J]$ be a generic element in $\Sigma(m,n+1-m,\mathbf{1})$ and assume that $[J]$ lies in $Y^m(n,k)$ for some $1\leq k\leq m$.
        Note that generic ideals in $\Sigma(m,n+1-m,\mathbf{1})$ are generated by $n+1-m$ linear forms and they correspond to generic $(m-1)$--dimensional linear subspaces in $\mathbb{A}^n$ passing through $\mathbf{0}$. Let $\Gamma$ be the $(m-1)$--dimensional subspace associated to $J$. Then there exists an irreducible reduced curve $C$, a one--dimensional family of length $m$ schemes $\mathcal{Z}\rightarrow C$ and $t_0\in C$ such that the fiber $\mathcal{Z}_{t_0}$ is $\V(J)$ and the fibers $Z_t$ for $t\neq t_0$ lie in the image of \eqref{eq:map induct}.  Since $1\leq k$, there exists $1\leq i\leq n$ such that $\mathcal{Z}_t$ is contained in the image of the map
        \[
        (L_i\setminus\{\mathbf{0}\})\times \Hilb^{k-1}(X_n\setminus\{\mathbf{0}\})\times  \Hilb^{m-k}_\mathbf{0}(X_n)\dashrightarrow \Hilb^m(X_n).
        \]
        for $t\neq t_0$. Let $\Gamma_t$ be the smallest linear subspace containing $Z_t$. Assume first that $1\leq k\leq m-1$. Since $k\leq m-1$, $Z_t$ contains $\mathbf{0}$ and a point in $L_i\setminus\{\mathbf{0}\}$ for $t\neq t_0$. Therefore, $\Gamma_t$ is contained in $L_i$ for $t\neq 0$. We deduce that $\Gamma_0=\Gamma$ contains $L_i$. This is a contradiction since $\Gamma$ is a generic linear subspace of dimension $m-1$ containing $\mathbf{0}$.  Next, assume that $[J]$ lies in $Y^m(n,m)$, i.e., assume that $J$ is smoothable. Then, $\mathcal{Z}_t$ consists of $m$ distinct points for $t\neq t_0$. Moreover, there exists $i_1,\ldots,i_m$ such that $\mathcal{Z}_t$ is contained in the image of 
           \[
        (L_{i_1}\setminus\{\mathbf{0}\})\times\cdots\times (L_{i_m}\setminus\{\mathbf{0}\})\dashrightarrow \Hilb^m(X_n)
        \]
        for $t\neq 0$.
        Let $M_{i_1,\ldots,i_k}$ be the affine subspace expand by $L_{i_1},\ldots,L_{i_k}$. Then $\mathcal{Z}_t$ is contained in $M_{i_1,\ldots,i_k}$ for any $t\in C$. Therefore, $\Gamma$ is contained in $M_{i_1,\ldots,i_k}$. Since $m\leq n-1$, we have that $M_{i_1,\ldots,i_k}\subsetneq \mathbb{A}^n$. We deduce that if $J$ is smoothable, $\Gamma$ would be contained in 
        \[
        \bigcup_{1\leq i_1,\ldots,i_k\leq n} M_{i_1,\ldots,i_k} \subsetneq \mathbb{A}^n.
        \]
        This is a contradiction since $J$ and $\Gamma$ are generic.

        By the first statement of the theorem, we deduce that for $m\geq n$, $\Hilb^m(X_n)$ does not have any elementary component. Similarly, the second statament of the theorem implies that for $2\leq m\leq n-1$, the only elementary component of $\Hilb^m(X_n)$ is $\Sigma(m,n+1-m,\mathbf{1})$. Therefore, we conclude that the irreducible components of $\Hilb^m(X_n)$  are the smoothable components of the irreducible components of $\Hilb^{m,m'}(X_n)$ for $2\leq m'\leq \min\{m,n-1\}$.
\end{proof}

The main consequences of \cref{theo: irred comp reduced structure} are the following results.

\begin{Cor}\label{co: irred comp Xn}
    The number of irreducible components of $\Hilb^m(X_n)$ is
    \[
    \binom{m+n-1}{m}+\sum_{m'=2}^{\min\{m,n-1\}}\binom{m-m'+n-1}{m-m'}.
    \]
\end{Cor}
\begin{proof}
By \cref{theo: irred comp reduced structure}, the irreducible components of $\Hilb^m(X_n)$ are the irreducible components of $\Hilb^m_\mathrm{sm}(X_n)$ and the irreducible components of $\Hilb^{m,m'}(X_n)$ for $2\leq m'\leq \min\{n-1,m\}$. The irreducible components of $\Hilb^m_\mathrm{sm}(X_n)$ are given by the possible distribution of $m$ distinct points among the $n$ of $X_n$. Therefore, $\Hilb^m_\mathrm{sm}(X_n)$ has $\binom{m+n-1}{m}$ components. Similarly, the components of $\Hilb^{m,m'}(X_n)$ are in correspondence with the components of $\Hilb^{m-m'}_\mathrm{sm}(X_n)$ via the map \eqref{eq:induc map gen}. We deduce that the number of irreducible components of $\Hilb^{m,m'}(X_n)$ is $\binom{m-m'+n-1}{m-m'}$.
\end{proof}

\begin{Rem}\label{rem: number com X_n non smooth}
    The number of irreducible components of $\Hilb^{m,m'}(X_n)$ is $\binom{m-m'+n-1}{m-m'}$. These components are birational to 
    \[
   \Sym^{u_1-1}( L_1)\times \cdots\times  \Sym^{u_n-1}(L_n)\times \Sigma(m',n+1-m',\mathbf{1}),
    \]
    where $\uu\in \Z^n_{\geq 1}$ with $|\uu|=m-m'+n$.
    These components are in bijection with the hypersimplices of the form $\Delta_{m'-1,n}+\uu-\mathbf{1}$ in $\mathcal{K}^{[m]}_n$. 
\end{Rem}

\begin{Cor}
    \label{co: irred comp reduced structure}
    Let $C$ be an irreducible curve whose only singularity is a rational $n$-fold singularity. Then, the number of irreducible components of $\sHilb^m(C)$ is $\min\{n-1,m\}$. Moreover, these irreducible components are $\Hilb^{m}_{\mathrm{sm}}(C)$ and $\Hilb^{m,m'}(C)$ for $2\leq m'\leq \min\{n-1,m\}$, which are birational to
    \[
\begin{array}{ccc}
\Sym^m(C)&  \text{or} & \Sym^{m-m'}(C) \times \mathrm{Gr}(n+1-m',n)\,\,\, \text{ for } 2\leq m'\leq \min\{n-1,m\}.
\end{array}
\]   

\end{Cor}
\begin{proof}
    Follows from \cref{theo: irred comp reduced structure}, since in this case the smoothable component is irreducible.
\end{proof}

\begin{Rem}
    To the best of our knowledge, this is the first example of Hilbert schemes of $m$ points where the number of irreducible components initially increases and then remains constant as $m$ varies. The graph of the number of irreducible components for $n$ fixed as $m$ varies is illustrated in \cref{fig:graph irred comp}.
\end{Rem}

\begin{Ex}
    Let $C$ be an irreducible curve whose unique singularity is a rational $3$--fold singularity. The number of irreducible components of $\Hilb^m(C)$ is $\mathrm{min}\{2,m\}.$ Assuming that $m\geq 2$, the two irreducible components of $\Hilb^m(C)$ are the smoothable component and $\Hilb^{m,2}(C)$, which is birational to $\P^2\times\Sym^{m-2}(C)$.
\end{Ex}

\begin{Ex}
    Let $C$ be an irreducible curve whose unique singularity is a rational $4$--fold singularity. The number of irreducible components of $\Hilb^m(C)$ is $\mathrm{min}\{3,m\}.$ For $m=2$, the irreducible components are the smoothable component $\Hilb^2(C_{\mathrm{sm}})$ and a $\Hilb^{2,2}(C)\simeq\P^3$. For $m\geq 3$, the three irreducible components of $\Hilb^m(C)$ are $\Hilb^3(C_{\mathrm{sm}})\simeq_{\mathrm{bir}}\Sym^m(C)$, $\Hilb^{m,2}(C)\simeq_{\mathrm{bir}}\Sym^{m-2}C\times\P^3$ and $\Hilb^{m,3}(C)\simeq_{\mathrm{bir}}\Sym^{m-3}C\times\mathrm{Gr}(2,4)$. Note that two of the first two irreducible components have dimension $m$, while the third irreducible component has dimension $m+1$.
\end{Ex}

\section{The local Hilbert scheme and its schematic structure}\label{sec:local hilbert scheme}

In the previous sections we have carried out a study of $\sHilb^m(X_n)$ from the perspective of algebraic varieties, and not considering the possibility of a non--reduced structure. This is reflected in \cref{theo: irred comp reduced structure} where the reduced structure of the irreducible components of $\sHilb^m(X_n)$ is given. The next goal is to analyse the reducedness of $\sHilb^m(X_n)$. We use the notation and classical results from deformation theory as presented in \cite{Sernesi2006}, to which we refer the reader for further details. Consider the local Hilbert scheme $H_J^n$ of $\mathbb{V}(J)$ in $X_n$. This is defined as the functor 
\[\begin{array}{cccc}
  H_J^n:   & \mathcal{Art}&\longrightarrow&\text{Sets} \\
     & A&\longmapsto &\left\{ \begin{array}{c}\text{ Deformations of } \mathbb{V}(J) \text{ in } X_n \\ \text{ over } A\end{array}\right\}.
\end{array}
\]
Then $H_J^n$ is prorepresented by $
\widehat{\mathcal{O}}_{\sHilb^m(X_n),[J]}
$. We compute $
\widehat{\mathcal{O}}_{\sHilb^m(X_n),[J]}
$ by calculating the complete local ring prorepresenting the local Hilbert scheme using the same strategy as \cite{Ran2005}. 
We carry out this computation for the most singular points of $\sHilb^m(X_n)$, which are the ideals $J$ such that $\mu_m(J)$ is a vertex of $\mathcal{K}_n^{[m]}$. In other words, $J=\langle x_1^{u_1},\ldots,x_n^{u_n}\rangle$. Up to labeling of the variables, we may assume $J=\langle x_1^{u_1},\ldots,x_k^{u_k},x_{k+1},\ldots,x_n\rangle$, where $u_i\geq 2$ for $i\leq k$. Note that $k\geq1$, since for $k=0$, the ideal $J$ has length $1$. The value of $k$ has the following combinatorial meaning. The hypersimplicial complex $\mathcal{K}_n^{[m]}$ fills the simplex $(m-1)\Delta_{n-1}$, which it can also be seen as a simplicial complex. Then, $k-1$ is the dimension of the face of $(m-1)\Delta_{n-1}$ where $\mu([J])$ lies. For instance, if $k=1$, then $\mu([J])$ is a vertex of $(m-1)\Delta_{n-1}$. If  $k=n$, then $\mu([J])$ is an interior point of $(m-1)\Delta_{n-1}$.

Consider the ring 
\[
S_{k} = \C\left[A_1,\ldots,A_n,\alpha_{i,j,l}:i\in[n],j\in[k] \text{ and } l\in[u_j-1]\,\right]
\]
and the ideal of $S_k$ given by 
\begin{equation}\label{eq:gens ideal def 1}
\begin{array}{ll}
J_k=&\langle A_i A_j: \text{ for } k+1\leq j\leq n \text{ and } i\in [n]\setminus\{j\}\rangle\\
&+\langle A_i\alpha_{j,r,s}: \text{ for } k+1\leq j\leq n,i\in [n]\setminus\{j\}, r\in [k] \text{ and } s\in[u_r-1]\rangle\\
&+\langle \alpha_{i,j,u_j-1}\alpha_{j,j,1}-A_i:\text{ for }j\in[k]\text{ and }i\in[n]\setminus\{j\}\rangle\\
&+\langle \alpha_{i,j,u_j-1}A_j:\text{ for }j\in[k]\text{ and }i\in[n]\setminus\{j\}\rangle\\
&+\langle \alpha_{i,j,u_j-1}\alpha_{j,j,l+1}-\alpha_{i,j,l}:\text{ for }j\in[k],i\in[n]\setminus\{j\}\text{ and }l\in [u_j-2]\rangle\\
&+\langle \alpha_{i,j,u_j-1}\alpha_{j,r,l}:\text{ for }j\in[k],i\in[n]\setminus\{j\},r\in[k]\setminus\{j\}\text{ and }l\in [u_j-1]\rangle.\\
\end{array}
\end{equation}
Using this ideal we compute $
\widehat{\mathcal{O}}_{\sHilb^m(X_n),[J]}
$ in the following:
\begin{Thm}
    \label{theo:def theory}
    Let $0\leq k\leq n$ and let $\uu=\in\Z_{\geq 1}^n$ such that $|\uu|=m+n-1$, $u_i\geq 2 $ for $1\leq i\leq k$ and $u_i=1$ for $k+1\leq i\leq n$. Then, the local Hilbert scheme of 
    $J=\langle x_1^{u_1},\ldots,x_k^{u_k},x_{k+1},\ldots,x_n\rangle$ in $X_n$ is prorepresented by the completion of the quotient $S_k/J_k$ localized at the origin.
\end{Thm}
\begin{proof}
Let $S\in\mathcal{A}$ be a local Artinian $\C$--algebra with residue field $\C$, and let $\mathfrak{m}_S$ its maximal ideal.
Let $R_S:=R\otimes_{\C} S$.
A deformation of $\mathbb{V}(J)$ in $\sp(R_n)$ over $\sp(S) $ is an ideal $J_S$ of $R_S$ such that $R_S/J_S$ is flat over $S$ and $(R_S/J_S)\otimes_S\C=R/J $. In other words, $J_S$ is generated by 
\[
\begin{array}{ccl}
\displaystyle    f_1& = &x_1^{u_1}+A_1+f_{1,1}(x_1)+\cdots+f_{1,n}(x_n),\\
&\vdots &\\
\displaystyle    f_k &=& x_k^{u_k}+A_k+f_{k,1}(x_1)+\cdots+f_{1,n}(x_n),\\
\displaystyle    f_{k+1}& =& x_{k+1}+A_{k+1}+f_{k+1,1}(x_1)+\cdots+f_{k+1,n}(x_n),\\
&\vdots&\\
\displaystyle    f_n&=& x_n+A_n+f_{n,1}(x_1)+\cdots+f_{n,n}(x_n),\\
\end{array}
\]
where $A_i$ lies in $\mathfrak{m}_S$ and $f_{i,j}(x_j)$ is a polynomial in $x_j$ with coefficients in $\mathfrak{m}_S$ and with no independent coefficient.
By \cite[\href{https://stacks.math.columbia.edu/tag/051G}{Tag 051G}]{stacks-project}, $R_S/J_S$ is a free $S$-module of rank $m$. By Nakayama's Lemma, $R_S/J_S$ is freely generated by
\[
1,x_1,\ldots,x_1^{u_1-1},\cdots,x_k,\ldots,x_k^{u_k-1}.
\]
We can write $x_i^d$ with $d_i\geq u_i$ as a linear combination of $1,x_1,\ldots,x_1^{u_1-1},\cdots,x_k,\ldots,x_k^{u_k-1}$ with coefficients in $S$. Therefore, we may assume that the degree of $f_{i,j}$ is at most $u_j-1$. In particular, $f_{i,j}=0$ for $k+1\leq j\leq n$ and we get
\[
f_{i,j}=\sum_{l=1}^{u_j-1} \alpha_{i,j,l}x_j^l
\]
for $1\leq j\leq k$. We deduce that $J_S$ is generated by 
\begin{equation}\label{eq:gens def}
\begin{array}{ccl}
\displaystyle    f_1& =&\displaystyle   x_1^{u_1}+A_1+ \sum_{l=1}^{u_1-1} \alpha_{1,1,l}x_1^l+\cdots+\sum_{l=1}^{u_k-1} \alpha_{1,k,l}x_k^l,\\
&\vdots& \\
\displaystyle    f_k& = &\displaystyle  x_k^{u_k}+A_k+ \sum_{l=1}^{u_1-1} \alpha_{k,1,l}x_1^l+\cdots+\sum_{l=1}^{u_k-1} \alpha_{k,k,l}x_k^l,\\
\displaystyle    f_{k+1} &=&\displaystyle   x_{k+1}+A_{k+1}+ \sum_{l=1}^{u_1-1} \alpha_{k+1,1,l}x_1^l+\cdots+\sum_{l=1}^{u_k-1} \alpha_{k+1,k,l}x_k^l,\\
&\vdots&\\
\displaystyle    f_n& =& \displaystyle  x_n+A_n+ \sum_{l=1}^{u_1-1} \alpha_{n,1,l}x_1^l+\cdots+\sum_{l=1}^{u_k-1} \alpha_{n,k,l}x_k^l.\\
\end{array}
\end{equation}
For  $k+1\leq j\leq n$ and $i\in[n]\setminus\{ j\}$,we consider in $R_S/J_S$ the relation  
\begin{equation}\label{eq:local 1}
\begin{array}{c}
\displaystyle0=A_if_j-x_jf_i=A_i\left(x_j+A_j+ \sum_{l=1}^{u_1-1} \alpha_{j,1,l}x_1^l+\cdots+\sum_{l=1}^{u_k-1} \alpha_{j,k,l}x_k^l\right)-A_ix_j=\\
\displaystyle
A_iA_j+ \sum_{l=1}^{u_1-1}A_i \alpha_{j,1,l}x_1^l+\cdots+\sum_{l=1}^{u_k-1} A_i\alpha_{j,k,l}x_k^l.
\end{array}
\end{equation}
Note that \eqref{eq:local 1} is a relation among the free generators of $R_S/J_S$. We deduce that the coefficients of \eqref{eq:local 1} vanish in $S$, and we get that for $1\leq i\leq n$ and $k+1\leq j\neq i\leq n$ 
\begin{equation}
    \label{eq:local 3}
    \begin{array}{ll}
      A_iA_j=0   &  \\
     A_i\alpha_{j,a,b}=0 & \text{ for } 1\leq a\leq k \text{ and } 1\leq b\leq u_a-1.
    \end{array}
\end{equation}

Similarly, for $1\leq j\leq k$ and $i\in[n]\setminus\{j\}$, we consider in $R_S/J_S$ the relation 
\begin{equation}
    \label{eq:local 2}
    \begin{array}{c}
    0\!= \!\displaystyle \alpha_{i,j,u_j-1}f_j-x_jf_i=\alpha_{i,j,u_j-1} \left(\!x_j^{u_j}+A_j+ \!\sum_{l=1}^{u_1-1} \alpha_{j,1,l}x_1^l+\cdots+\sum_{l=1}^{u_k-1} \alpha_{j,k,l}x_k^l\!\right)\!-\!A_ix_j\!-\!\sum_{l=1}^{u_j-1}\alpha_{i,j,l}x_j^{l+1}\\
      = \displaystyle(\alpha_{i,j,u_j-1}\alpha_{j,j,1}-A_i)x_j +\!\sum_{l=1}^{u_j-2}(\alpha_{i,j,u_j-1}\alpha_{j,j,l+1}-\alpha_{i,j,l})x_j^{l+1}+\alpha_{i,j,u_j-1}A_j+\!\sum_{a\neq j}\sum_{l=1}^{u_a-1}\alpha_{i,j,u_j-1}\alpha_{j,a,l}x_a^l.
      
    \end{array}
\end{equation}
Again, \eqref{eq:local 2} is a relation among the free generators of $R_S/J_S$. We deduce that for $1\leq i\leq n$ and $1\leq j\leq k$ with $i\neq j$ we have
\begin{equation}\label{eq:local 4}
    \begin{array}{ll}
    \alpha_{i,j,u_j-1}\alpha_{j,j,1}-A_i=0& \\
    \alpha_{i,j,u_j-1}A_j = 0 & \\
    \alpha_{i,j,u_j-1}\alpha_{j,j,l+1}-\alpha_{i,j,l}=0& \text{ for } 1\leq l\leq u_j-2\\
    \alpha_{i,j,u_j-1}\alpha_{j,a,l}= 0& \text{ for } 1\leq a\neq j\leq k, \text{ and } 1\leq l\leq u_a-1.
    \end{array}
\end{equation}
Therefore, if the $R_S/J_S$ is a deformation over $S$, then the coefficients of $f_1,\ldots,f_n$ satisfy equations \eqref{eq:local 3} and \eqref{eq:local 4}. Conversely, assume that the coefficients of $f_1,\ldots,f_n$ satisfy equations \eqref{eq:local 3} and \eqref{eq:local 4}. Then, $f_1,\ldots,f_n$ satisfy the relations \eqref{eq:local 1} and \eqref{eq:local 2}. Modulo $S/\mathfrak{m}_s\simeq \C$, these relations are exactly
\[
x_ix_j^{u_j} = 0
\]
for $i\neq j$, which are exactly the syzygies of $J$. By \cite[Corollary A.11]{Sernesi2006}, $R_S/J_S$ is flat over $S$ and we conclude that $R_S/J_S$ is a deformation over $S$. Therefore, the local Hilbert scheme of $[J]$ is prorepresented by the completion of the stalk at the origin of the quotient of
\[
S_k=\C[A_1,\ldots,A_n,\alpha_{i,j,l}: 1\leq i\leq n,\,\, 1\leq j\leq k,\,\, 1\leq l\leq u_j-1]
\]
by the ideal generated by the relations in \eqref{eq:local 3} and \eqref{eq:local 4}, which coincide with $J_k$. 
\end{proof}

 \cref{theo:def theory} computes the stalk $
\widehat{\mathcal{O}}_{\sHilb^m(X_n),[J]}
$ as the completion of the quotient $S_k/J_k$. However, the representation of this quotient is quite complicated. In \cref{app: commutative algebra} we give a better representation of this ring. In particular, \cref{lemma: primary k1} and  \cref{lemma:primary k2} together with \cref{lemma: new rings} compute the irreducible components of the stalk of $\sHilb^m(X_n)$ at the point $[J] = [\langle x_1^{u_1},\ldots,x_k^{u_k},x_{k+1},\ldots,x_n\rangle]$. Now, we identify each of these components with the corresponding components of $\sHilb^m(X_n)$. 

\begin{Prop}\label{prop: ideal translation}
 For $k=1$, $
\widehat{\mathcal{O}}_{\sHilb^m(X_n),[J]}
$ has $n+1$ irreducible components and they correspond to the ideals in \cref{lemma: primary k1}. A generic point in the component corresponding to  $\langle A_1,\alpha_{1,1}\rangle$ represents a length $m$ scheme that is a point in 
\begin{equation}\label{eq:irred comp 1}
\sHilb^{m-2}(L_1)\times\Sigma (2,n-1,\mathbf{1}).
\end{equation}
A generic point in the component associated to  $\langle \alpha_2,\ldots,\alpha_n\rangle$ represents a smoothable length $m$ scheme which is a point in $\sHilb^{m}(L_1)$. A generic point in $\langle A_1,\alpha_j: 2\leq j\leq n \text{ and } j\neq i\rangle$ for $2\leq i\leq n$ represents a smoothable length $m$ scheme that 
is a point in $\sHilb^{m-1}(L_1)\times L_i$.
\end{Prop}
\begin{proof}
    Let $K$ be an ideal in the irreducible component corresponding to the ideal  $\langle \alpha_2,\ldots,\alpha_n\rangle$. We write the generators of $K$ as in \eqref{eq:gens def}. Via the isomorphism in \cref{lemma: new rings}, we deduce that $A_2=\cdots=A_n=0$ and $\alpha_{2,1,u_1-1}=\cdots =\alpha_{n,1,u_1-1} = 0$. Using the ideal \eqref{eq:gens ideal def 1} we deduce that the generators of $K$ are 
    \[\begin{array}{cll}
    f_1 &=& \displaystyle A_1+\sum_{l=1}^{u_1-1}\alpha_{1,1,l}x_1^l+x_1^{u_1},\\
    f_2&=&x_2,\\
    \vdots& &\vdots \\
    f_n&=&x_n.
    \end{array}
    \]
    Hence, $K$ represents a length $m$ ideal in $\sHilb^{m}(L_1)$. Assume now that $K$ corresponds to a point in the ideal  $\langle A_1,\alpha_j: 2\leq j\leq n \text{ and } j\neq i\rangle$ for $2\leq i\leq n$. As before, pulling back these conditions via the isomorphism build in \cref{lemma: new rings}, we deduce that the generators of $K$ are of the form 
    \[\begin{array}{cll}
    f_1 &=& \displaystyle  \sum_{l=1}^{u_1-1}\alpha_{1,1,l}x_1^l+x_1^{u_1}=x_1\left(\sum_{l=0}^{u_1-2}\alpha_{1,1,l+1}x_1^l+x_1^{u_1-1} \right),\\
    f_2&=&x_2,\\
    \vdots& &\vdots \\
    f_i&=&\displaystyle x_i+ \alpha_{i,1,u_1-1}\left(\sum_{l=0}^{u_1-2}\alpha_{1,1,l+1}x_1^l+x_1^{u_1-1} \right) ,  \\
    \vdots& &\vdots \\
    f_n&=&x_n.
    \end{array}
    \]
    We deduce that the primary decomposition of $K$ consists on the ideals $\langle \sum_{l=0}^{u_1-2}\alpha_{1,1,l+1}x_1^l+x_1^{u_1-1},x_2,\ldots,x_n \rangle$, which represents a point in $\sHilb^{m-1}(L_1)$ and $\langle x_1,x_2,\ldots,x_i+ \alpha_{i,1,u_1-1} \alpha_{1,1,1},\ldots,x_n\rangle$, which represents a point in $L_i$.  Similarly, if $K$ is in the component corresponding to  $\langle A_1,\alpha_{1,1}\rangle$, the generators of $K$ are
    \[\begin{array}{cll}
    f_1 &=& \displaystyle  \sum_{l=2}^{u_1-1}\alpha_{1,1,l}x_1^l+x_1^{u_1}=x_1^2\left(\sum_{l=0}^{u_1-3}\alpha_{1,1,l+2}x_1^l+x_1^{u_1-2} \right),\\
    f_2&=&\displaystyle x_2+\sum_{l=1}^{u_1-1}\alpha_{2,1,l}x_1^l= x_2+\alpha_{2,1,u_1-1}x_1\left(\sum_{l=0}^{u_1-3}\alpha_{1,1,l+2}x_1^l\right),\\
    \vdots& &\vdots \\
    f_n&=&\displaystyle x_n+\sum_{l=1}^{u_1-1}\alpha_{n,1,l}x_1^l=x_n+\alpha_{n,1,u_1-1}x_1\left(\sum_{l=0}^{u_1-3}\alpha_{1,1,l+2}x_1^l\right).
    \end{array}
    \]
    We deduce that the primary decomposition of $K$ consists of two ideals. First, the ideal $\langle \sum_{l=0}^{u_1-3}\alpha_{1,1,l+2}x_1^l,x_2,\ldots,x_n\rangle$ which leads to length $m-2$ scheme in $L_1$. The second ideal is $$\langle x_1^2,x_2-\alpha_{2,1,u_1-1}\alpha_{1,1,2}x_1,\ldots,x_n-\alpha_{n,1,u_1-1}\alpha_{1,1,2}x_1\rangle,$$ which represent a length $2$ scheme in $\Sigma(2,n-1,\mathbf{1})$.
\end{proof}

\begin{Prop}\label{prop: ideal translation 2}
 For $k\geq 2$, the number of irreducible components of $
\widehat{\mathcal{O}}_{\sHilb^m(X_n),[J]}
$ is 
\begin{equation}\label{eq:numb irred compo  local}
n+\sum_{i=1}^{\min\{k,n-2\}}\binom{k}{i} = \left\{
\begin{array}{ll}
n+2^k-1 & \text{ if }k\leq n-2\\
n+2^{n-1}-2& \text{ if }k= n-1\\
2^k-2 & \text{ if }k = n
\end{array}
\right.,
\end{equation}
and each irreducible component  corresponds to the ideals in \cref{lemma:primary k2}. A generic point in the components corresponding to $\mathcal{K}_a$ represents length $m$ schemes that are points in 
\[
\sHilb^{u_a}(L_a)\times\prod_{i\in[n]\setminus\{a\}}
\sHilb^{u_i-1}(L_i).
\]
In particular, points in the components corresponding to $\mathcal{K}_a$ are smoothable.
A generic point in the components corresponding to $\mathcal{J}_S$ represents length $m$ schemes that are points in 
\begin{equation}\label{eq: local components}
\prod_{i\in S}\sHilb^{u_i-2}(L_i)\times \prod_{i\in[k]\setminus S}\sHilb^{u_i-1}(L_i)\times \Sigma(|S|+1,n-|S|,\mathbf{1}).
\end{equation}

\end{Prop}
\begin{proof}
We proceed as in the proof of \cref{prop: ideal translation}. We use the isomorphism of \cref{lemma: new rings} to translate the conditions imposed by the ideals in \cref{lemma:primary k2} to the ring $S_k/J_k$ and the generators \eqref{eq:gens def}. Let $K$ be an ideal in the irreducible component corresponding to the ideal $\mathcal{J}_S$ for $S$ as in \cref{lemma:primary k2}. We may write the generators of $K$ as in  \eqref{eq:gens def}. To simplify these generators, we consider the polynomials 
\[
f_i':= \left\{
\begin{array}{ll}\displaystyle
x_i^{u_i-2}+\sum_{l=2}^{u_i-1}\alpha_{i,i,l}x_i^{l-2}\,\,\,  \text{ for } i\in S,\\  \displaystyle
x_i^{u_i-1}+\sum_{l=1}^{u_i-1}\alpha_{i,i,l}x_i^{l-1}\,\,\,  \text{ for } i\in [k]\setminus S.\\
\end{array}
\right.
\]
Modulo the pullback of $\mathcal{J}_S$ via the isomorphism in \cref{lemma: new rings}, we may write the generators of $K$ as 
\[
\begin{array}{ll}
\displaystyle f_i = x_i^{u_i}+\sum_{l=2}^{u_i-1}\alpha_{i,i,l}x_i^{l}=x_i^2f_i'&\text{ for }i\in S,\\

\displaystyle f_i = x_i^{u_i}+\sum_{l=1}^{u_i-1}\alpha_{i,i,l}x_i^{l} +\sum_{j\in S}\alpha_{i,j,u_j-1}x_jf_j' 

&\text{ for }i\in [k]\setminus S,\\
\displaystyle f_i = x_i +\sum_{j\in S}\alpha_{i,j,u_j-1}x_jf_j'

&\text{ for }i\not \in [k].
\end{array}
\]

Now, we analyze the scheme defined by these equations. Assume first that $f_i'=0$ for some $i\in S$. This implies that $0=x_jf_i' = \alpha_{i,i,2}x_j$  for $j\neq i$. We get a component of $K$ the form $\langle f_i',x_j:j\neq i\rangle$ for $i\in S$ which represent a point in $\sHilb^{u_i-2}(L_i)$. Assume on the contrary that $f_i'\neq 0$ for every $i\in S$. Then, $x_i^2 = 0 $ for every $i\in S$. The generators of $K$ modulo this condition are
\[
\begin{array}{ll}

\displaystyle f_i = x_i^{u_i}+\sum_{l=1}^{u_i-1}\alpha_{i,i,l}x_i^{l} +\sum_{j\in S}\alpha_{i,j,u_j-1}\alpha_{i,i,2}x_j

&\text{ for }i\in [k]\setminus S,\\
\displaystyle f_i = x_i  +\sum_{j\in S}\alpha_{i,j,u_j-1}\alpha_{i,i,2}x_j

&\text{ for }i\not \in [k].
\end{array}
\]
For $i\in [k]\setminus S$ we get that 
$
0 = x_if_i =x_i^2f_i'.
$
If $f_i'=0$, then, multiplying by $x_j$ for $j\neq i$ we get that $x_j=0 $ for $j\neq i$. Therefore, we obtained the ideal $\langle f_i',x_j:j\neq i\rangle $ for $i\in [k]\setminus S$, which represent a point in $\sHilb^{u_i-1}(L_i)$. If on the contrary $f_i'\neq 0$ for every $i\in[k]$. Since $x_i^2f_i'=0$, we deduce that $x_i^2 = 0$ for every $i\in [k]$. Therefore, we obtain the ideal 
\[\begin{array}{c}
\displaystyle
\langle x_i^2:i\in[k]\rangle +\langle \alpha_{i,i,1}x_i+\sum_{j\in S}\alpha_{i,j,u_j-1}\alpha_{i,i,2}x_j: i\in[k]\setminus S \rangle +\langle x_i  +\sum_{j\in S}\alpha_{i,j,u_j-1}\alpha_{i,i,2}x_j:i\not\in [k]\rangle\\
\displaystyle=
\langle \alpha_{i,i,1}x_i+\sum_{j\in S}\alpha_{i,j,u_j-1}\alpha_{i,i,2}x_j: i\in[k]\setminus S \rangle +\langle x_i  +\sum_{j\in S}\alpha_{i,j,u_j-1}\alpha_{i,i,2}x_j:i\not\in [k]\rangle,
\end{array}
\]
which represents a point in $\Sigma( |S|+1,n-|S| ,\mathbf{1} )$. We conclude that $K$ defines a length $m$ scheme which corresponds to a point in 
\[
\prod_{i\in S}\sHilb^{u_i-2}(L_i)\times \prod_{i\in[k]\setminus S}\sHilb^{u_i-1}(L_i)\times \Sigma(|S|+1,n-|S|,\mathbf{1}).
\]
\end{proof}

We can relate \cref{prop: ideal translation} and \cref{prop: ideal translation 2} with the combinatorics in \cref{sec:moment map}. For $k=1$, the ideal $J=\langle x_1^m,x_2,\ldots,x_n\rangle$ is the intersection of the $n+1$ irreducible components of the Hilbert scheme described in \cref{prop: ideal translation}.
The ideal $J$ corresponds to the vertex $(m-1) e_1$ of the simplex $(m-1)\Delta_{n-1}$. The only hypersimplex of the hypersimplicial complex $\mathcal{K}^{[m]}_n$ containing this vertex is $\Delta_{1,n}+(m-2) e_1$. 
The component of the form \eqref{eq:irred comp 1} corresponds to this translated hypersimplex. The translation by $(m-2) e_1$  geometrically corresponds to the factor $\sHilb^{m-2}(L_1)$ of the irreducible components.  The rest of the irreducible components containig the point $[J]$ are smoothable and cannot be seen from the complex $\mathcal{K}_n^{[m]}$. Therefore, these components do not come from the punctual Hilbert scheme. In particular, the number of irreducible components of $\sHilb^m(X_n)$ and $\sHilb_{\mathbf{0}}^m(X_n)$ that contain $[J]$ is different. Similarly, for $k\geq 2$, we can associate to the components of the form \eqref{eq: local components} corresponding to $\mathcal{J}_S$ the hypersimplex $\Delta_{n,|S|}+\uu-\mathbf{e}_S-\mathbf{1}$. Such a hypersimplex corresponds to the component $\Sigma(m,n-|S|,\uu-\mathbf{e}_S)$ of the punctual Hilbert scheme. This component is exactly the intersection of the punctual Hilbert scheme and the component \eqref{eq: local components}. For $k=n-1$ and  $a\not\in [k]$, we can associate to the component corresponding to $\mathcal{K}_a$ in \cref{prop: ideal translation 2} the hypersimplex $\Delta_{n-1,n}+\uu-\mathbf{e}_{[n]\setminus\{a\}}-\mathbf{1}$. Similarly, for $n=k$ and $a\in[k]$, we can associate to the component corresponding to $\mathcal{K}_a$ in  \cref{prop: ideal translation 2} the hypersimplex $\Delta_{n-1,n}+\uu-\mathbf{e}_{[n]\setminus\{a\}}-\mathbf{1}$.
\begin{Rem}
    Fix the ideal $J=\langle x_1^{u_1},\ldots, x_k^{u_k},x_{k+1},\ldots,x_n\rangle$ with $u_i\geq 2$ for $i\in[k]$ and $|\uu|=m+n-1$. Then,  $\mu([J])=\uu$ is a vertex of $\mathcal{K}_n^{[m]}$ that is contained in the relative interior of a $(k-1)$--dimensional face of $(m-1)\cdot\Delta_{n-1}$. The number of hypersimplices containing the vertex $\mu$ is exactly
   $ 2^k-1$ if $k<n$  and  $2^k-2$  if  $k=n$. Therefore, this number is the number of irreducible components of $\sHilb_{\mathbf{0}}^m(X_n)$ containing $[J]$. On the other hand,
    by \cref{prop: ideal translation} and \cref{prop: ideal translation 2}, the number of irreducible components of $\sHilb^m(X_n)$ containing $[J]$ is given by \eqref{eq:numb irred compo  local}. For $k\leq n-1$, this number of irreducible components of $\sHilb^m(X_n)$ and $\sHilb_{\mathbf{0}}^m(X_n)$ differs by $n$ if $k\leq n-2$ and by $n-1$ if $k=n-1$. These extra components that do not appear in $\sHilb_{\mathbf{0}}^m(X_n)$ are described in \cref{prop: ideal translation 2} and they are smoothable irreducible components of $\sHilb^m(X_n)$ that contain $[J]$.  For $k=n$, the number of irreducible components of the local and punctual Hilbert scheme coincides.

\end{Rem}

Now, we show the local structure of $\sHilb_{\mathbf{0}}^m(X_n)$ inside the stalk
 $
\widehat{\mathcal{O}}_{\sHilb^m(X_n),[J]}
$.

\begin{Prop}\label{prop:local punctual}
Let $1\leq k\leq n$ and let $\uu=\in\Z_{\geq 1}^n$ such that $|\uu|=m+n-1$, $u_i\geq 2 $ for $1\leq i\leq k$ and $u_i=1$ for $k+1\leq i\leq n$. Then,  the completion of the stalk of the punctual Hilbert scheme  
at $J=\langle x_1^{u_1},\ldots,x_k^{u_k},x_{k+1},\ldots,x_n\rangle$, denoted by $ 
\widehat{\mathcal{O}}_{\sHilb_{\mathbf{0}}^m(X_n),[J]}$
is isomorphic to the completion of the quotient
\begin{equation}\label{eq:ring punctual}
\frac{\Z[\alpha_{i,j}:i\in[n],j\in[k]\setminus\{i\}]}{\langle 
\alpha_{i,j}\alpha_{j,r}:j\in[k],\,i\in[n]\setminus\{i\},\,r\in[k]\setminus\{j\}
\rangle}
\end{equation}
localized at the origin. Through this isomorphism, the irreducible components of $ 
\widehat{\mathcal{O}}_{\sHilb_{\mathbf{0}}^m(X_n),[J]}$ corresponds to the ideals of \eqref{eq:ring punctual}
\begin{equation}\label{eq:punctual primary decomposition}
\mathcal{J}_{[k],T}:=\langle \alpha_{i,j}:i,j\in[k]\setminus T,\,i\neq j\rangle +
\langle\alpha_{i,j}: j\in T \text{ and } i \in [n]\setminus\{j\} \rangle
\end{equation}
for $T\subsetneq [k]$ if $k\neq n$, or  $\emptyset\neq T\subsetneq [k]$ if $k= n$. Moreover,the ideal $\mathcal{J}_{[k],T}$ represents the irreducible component $$\Sigma(m, n-k+|T|, \mathbf{u}-\mathbf{e}_{[k]\setminus T})$$ of the punctual Hilbert scheme.
\end{Prop}
\begin{proof}
By \cref{theo:def theory},  $ 
\widehat{\mathcal{O}}_{\sHilb_{\mathbf{0}}^m(X_n),[J]}$ corresponds to the ideal of $S_k/I_k$ describing the deformations of $\mathbb{V}(J)$ supported at the origin. 
Let $K$ be an ideal corresponding to a point in  $S_k/J_k$. In other words, the generators of $K$ are given by \eqref{eq:gens def}. We need to check when the ideal $K$ is an ideal supported at the origin. The polynomial $x_if_i = x_i(x_i^{u_i}+A_i+\sum_{l=1}^{u_i-1}\alpha_{i,i,l}x_i^l)$ is a polynomial in $K$. If $K$ is supported only at the origin, then $x_i=0$ must be the only solution of $x_if_i=0$. We deduce that $A_i=\alpha_{i,i,1}=\cdots = \alpha_{i,i,u_i-1}=0$.
Modulo this relation, we get that the generators of $K$ are 
\[
f_i= x_i^{u_i}+\sum_{i\in [k]\setminus\{i\}}\alpha_{i,j,u_j-1}x_j^{u_j-1},
\]
where $\alpha_{i,j,u_j-1}\alpha_{j,r,u_r-1}=0$ for every $i\in[n]$,  $j,r\in[k]\setminus\{i\}$ and $j\neq r$.  Together with the generators of $J_k$ and identifying  $\alpha_{i,j,u_j-1}$ with $\alpha_{i,j}$ we get the ring \eqref{eq:ring punctual}. The primary decomposition of ideal in \eqref{eq:ring punctual} is given by \cref{lemma: primary technical} for $S=[k]$. Now fix an ideal $\mathcal{J}_{[k],T}$ in the primary decomposition of \eqref{eq:ring punctual}. Modulo $\mathcal{J}_{[k],T}$ , the generatos of the length $m$ ideal $K$ are
\[
f_i=\left\{
\begin{array}{ll}
\displaystyle x_i^{u_i} & \text{ for } i\in [k]\setminus T,\\
\displaystyle x_i^{u_i}+\sum_{j\in[k]\setminus{T}}\alpha_{i,j}x_j^{u_j-1} & \text{ for } i\in T,\\
\displaystyle x_i^{u_i}+\sum_{j\in[k]\setminus{T}}\alpha_{i,j}x_j^{u_j-1} & \text{ for } i\not\in[k].\\
\end{array}
\right.
\]
Note that the generators $f_i$ for $i\in [k]\setminus T$ are not required since the are recovered from the multiplication $x_if_j$ for $j\not\in [k]\setminus T$.
We deduce that the ideals in the component corresponding to $\mathcal{J}_{[k],T}$ are given by $n-k+|T|$ linearly independent polynomials in the vector space
\[
\langle x_i^{u_i}:i\in T\rangle_\C+
\langle x_i^{u_i-1}:i\in [k]\setminus T\rangle_\C+
\langle x_i:i\not\in [k]\rangle_\C.
\]
Therefore, the component $\mathcal{J}_{[k],T}$ correspond with the irreducible component 
$$\Sigma(m, n-k+|T|, \mathbf{u}-\mathbf{e}_{[k]\setminus T})$$
of the punctual Hilbert scheme.
\end{proof}

 \cref{prop:local punctual} allows us to carry out the local study of $\sHilb_{\mathbf{0}}^m(X_n)$, for the reducedness of this scheme and the transversality of the intersection of the irreducible components. To derive these results,  we need first the following lemmas.

\begin{Lem}\label{lem: reduced vertex}
    Let $[J]\in\sHilb_{\mathbf{0}}^m(X_n)$ such that $\mu([J])$ is a vertex of $\mathcal{K}_n^{[m]}$. Then, $[J]$ is a reduced point of $\sHilb^m(X_n)$ and $\sHilb_{\mathbf{0}}^m(X_n)$. 

\end{Lem}
\begin{proof}
    To show that $[J]$ is a reduced point of $\Hilb^m(X_n)$, it is enough to show that the stalk $\mathcal{O}_{\Hilb^m(X_n),[J]}$ is reduced. By Lemma \cite[Lemma 07NZ]{stacks-project}, it is enough to check that the completion of this stalk is reduced. By \cref{theo:def theory} and \cref{lemma: new rings}, this completion is the quotient $\mathcal{S}_k/\mathcal{J}_k$ (see \eqref{eq:ideal 1},\eqref{eq:ring Sk} and \eqref{eq:ideal Ik} for the relevant definitions). Since the primary decomposition of the ideal $\mathcal{J}_k$ calculated in \cref{lemma: primary k1} and \cref{lemma:primary k2} is given by prime ideals, the ideal $\mathcal{J}_k$ is radical. Hence, $\mathcal{S}_k/\mathcal{J}_k$ is a reduced ring. 

    Similarly,  to show that $[J]$ is a reduced point of $\Hilb_{\mathbf{0}}^m(X_n)$, it is enough to check that the ring \eqref{eq:ring punctual} is reduced. The primary decomposition of the ideal in \eqref{eq:ring punctual} is given in Equation \eqref{eq:punctual primary decomposition} in \cref{prop:local punctual}. The proof follows from the fact that the ideals \eqref{eq:punctual primary decomposition}  are prime.
\end{proof}

\begin{Lem}\label{lemma:one parameter family}
   Let $[J]\in \sHilb^m(X_n)$, then there exists $[J_0]\in\sHilb_\mathbf{0}^m(X_n)$ in the closure of the $(\C^{*})^n$--orbit of $[J]$ such that $\mu([J_0])$ is a vertex of $\mathcal{K}_n^{[m]}$. 
\end{Lem}
\begin{proof}
  The statement of the Lemma is independent of the schematic structure of the Hilbert scheme. Therefore, we can replace $\sHilb^m(X_n)$ by $\Hilb^m(X_n)$. Assume first that $J$ is supported at $\mathbf{0}$, i.e. $[J]\in \Hilb^m_\mathbf{0}(X_n)$.
  By \cref{punctual irred comp}, it is enough to check the analogous statement for Grassmannians. Let $[E]\in\gr(l,n)$ be generated by the image of an $n\times l$ matrix $A=(a_{i,j})$. Without loss of generality, we may assume that $A$ is a block matrix of the form 
  \[
 A =  \begin{pmatrix}
      \mathrm{Id}_l\\
      \hline
      A'
  \end{pmatrix}.
  \]
  For $t=(t_1,\ldots,t_n)\in(\C^{*})^n$, $t\cdot [E]$ is the vector subspace generated by image of the matrix $t\cdot A=\mathrm{Diag}(t_1,\ldots,t_n)\cdot A$. Therefore, taking the limit when $t_{l+1},\ldots,t_n$ goes to zero we get the linear subspace $\langle \mathbf{e}_1,\ldots,\mathbf{e}_l\rangle$, which is a torus invariant point. Therefore, it corresponds to a vertex of $\Delta_{l,n}$ via the moment map, and it lies in the closure of the orbit of $[E]$.

  Assume now that $[J]\not\in\Hilb_\mathbf{0}^m(X_n)$. 
  Since the statement of the lemma holds for ideals in the punctual Hilbert scheme, it is enough to check that in the closure of the $(\C^{*})^n$--orbit of $[J]$ there is a point in the punctual Hilbert scheme.
  Consider a one parameter family $\lambda:\C^{*}\rightarrow(\C^{*})^n$ such that 
  the limit of $\lambda(t)$ when $t$ goes to $0$ is $\mathbf{0}$. Then, for any point $q$ in $X_n$, the limit of $\lambda(t)\cdot q$ when $t$ goes to $0$ is the singularity $\mathbf{0}$. Therefore, the limit of $\lambda(t)\cdot [J]$ when $t$ goes to $0$ is a length $m$ ideal supported at $\mathbf{0}$. We conclude that in the closure of the $(\C^{*})^n$--orbit of $[J]$ there is a point in $\Hilb_\mathbf{0}^m(X_n)$.
\end{proof}

\begin{Rem}
      \cref{prop: ideal translation} and \cref{prop: ideal translation 2} and  \cref{lemma:one parameter family} provide an alternative proof to \cref{theo: irred comp reduced structure}: Let $Z$ be an elementary component of $\sHilb^m(X_n)$. By   \cref{lemma:one parameter family}, $Z$ contains a point $[J]\in\Hilb_\mathbf{0}^m(X_n)$ corresponding to a vertex of $\mathcal{K}^{[m]}_n$. Since $Z$ is an irreducible component of $\sHilb^m(X_n)$, it should correspond to an irreducible component of the completion of the stalk of $[J]$ calculated in \cref{theo:def theory}. \cref{prop: ideal translation} and \cref{prop: ideal translation 2} give a geometrical interpretation to each of the irreducible components of this stalk. The only cases where these irreducible components are entirely contained in $\sHilb_\mathbf{0}^m(X_n)$ are exactly the ones described in \cref{theo: irred comp reduced structure}.
\end{Rem}

\begin{Prop}
    \label{prop: cohen Mac}
    Let $C$ be an irreducible curve whose singularities are all rational $n$--fold singularities. Then, 
    for $m\geq 2$, the Hilbert scheme $\sHilb^m(C)$ is Cohen-Macaulay if and only if $n\leq 3$.
\end{Prop}
\begin{proof}
    First, note that for $n\geq 4$, the Hilbert scheme $\sHilb^m(C)$ is not equidimensional, and hence, it is not Cohen-Macaulay. Assume that $n\leq 3$. Without loss of generality, we may assume that $C=X_n$. Let $[J]\in\Hilb^m(X_n)$, we need to check that the completion of the stalk of $[J]$ is Cohen-Macaulay.
    By \cref{lemma:one parameter family}, we may assume that $[J]\in\Hilb^m_\mathbf{0}(X_n)$ and $\mu([J])$ is a vertex of $\mathcal{K}^{[m]}_n$. By \cref{theo:def theory} and \cref{lemma: new rings}, it is enough to check that the ring $\mathcal{S}_k/\mathcal{J}_k$ is Cohen-Macaulay for $n\leq 3$ and $k\in[n]$. A computation in \texttt{Macaulay2} \cite{macaulay2} shows that $\mathcal{S}_k/\mathcal{J}_k$  is Cohen-Macaulay for $n\leq 3$  and $k\in[n]$.
\end{proof}

\begin{Thm}\label{theo:reduced punctual}
    The punctual Hilbert scheme $\sHilb_{\mathbf{0}}^m(X_n)$ is reduced and isomorphic to $\mathcal{G}_n^{m}$. 
\end{Thm}
\begin{proof}
    First, we prove that $\sHilb_{\mathbf{0}}^m(X_n)$ is reduced. Assume on the contrary that 
    $[J]$ is a nonreduced point in  $\sHilb_{\mathbf{0}}^m(X_n)$. This implies that $(\C^{*})^n$--orbit of $[J]$ is nonreduced, and hence, the closure of the $(\C^{*})^n$--orbit of $[J]$ is nonreduced. Bby \cref{lemma:one parameter family}, there exists a ideal $[J']$ in the closure of the orbit such that $\mu([J'])$ is a vertex of $\mathcal{K}_n^{[m]}$. This is a contradiction since by \cref{lem: reduced vertex}, $[J']$ is a reduced point.

    We know check that $\sHilb_{\mathbf{0}}^m(X_n)$ is isomorphic to $\mathcal{G}^m_n$. Note that by the universal property of pushouts, we have a map $\varphi:\mathcal{G}_n^m\rightarrow\sHilb_\mathbf{0}^{m}(X_n)$ that on each component $\gr(l,\Lambda_\uu)$ of $\mathcal{G}_n^m$ it is the map $\varphi_{l,\uu}$. Therefore, $\varphi$ is injective and its restriction to each irreducible component of $\mathcal{G}_n^m$ and $\sHilb_\mathbf{0}^{m}(X_n)$ leads to an isomorphism. Therefore, to check that $\varphi$ is an isomorphism, it is enough to check what happens at the intersections. In other words, we need to check that in $\sHilb_{\mathbf{0}}^m(X_n)$ the intersection of the Grassmannians $\Sigma(m,l,\uu)$ is locally the intersection of affine spaces. By  \cref{lemma:one parameter family} and using the torus action, it is enough to check this condition around an ideal $[J]$ corresponding to a vertex of $\mathcal{K}_n^{[m]}$. Then the proof follows from \cref{prop:local punctual} since all the components of the completion of the stalk at $[J]$ are affine spaces.
\end{proof}

Using the same technique as in  \cref{theo:reduced punctual} we derive that $\sHilb^m(X_n)$ is reduced.

\begin{Thm}\label{theo: reduced}
   The Hilbert scheme of points $\sHilb^m(X_n)$ is reduced.
\end{Thm}
\begin{proof}

Let $[I]\in\sHilb^m(X_n)$ be a nonreduced point. Then the $(\C^{*})^n$--orbit of $[I]$ is nonreduced, and hence, the closure of the $(\C^{*})^n$--orbit is also nonreduced. By  \cref{lemma:one parameter family}, the closure of this orbit contains an ideal $[J]$ in $\sHilb_{\mathbf{0}}^{m}(X_n)$ associated to a vertex of $\mathcal{K}_n^{m}$ via the moment map.  
 This is a contradiction since by \cref{lem: reduced vertex}, $[J]$ is a reduced point of $\sHilb^m(X_n)$. Thus, we conclude that $[I]$ is a reduced point of $\sHilb^m(X_n)$.
\end{proof}

As a consequence of Theorems \ref{theo:reduced punctual} and \ref{theo: reduced}, we have that $\Hilb^m_\mathbf{0}(X_n)=\sHilb^m_\mathbf{0}(X_n)$ and $\Hilb^m(X_n)=\sHilb^m(X_n)$. Therefore, for the rest of the paper, we will use the notation $\sHilb^m_\mathbf{0}(X_n)$ and $\sHilb^m(X_n)$.
The following improvement of \cref{co: irred comp reduced structure} where we no longer require to take the reduced structure follows from \cref{theo: reduced} together with \cref{theo: irred comp reduced structure} . 

\begin{Cor}
    \label{co: irred comp structure}
    Let $C$ be an irreducible curve with a unique rational $n$-fold singularity. Then, the irreducible components of  $\sHilb^m(C)$ are 
    \[
    \sHilb_{\mathrm{sm}}^{m}(C) \text{ and }\sHilb^{m,m'}(C) \text{ for } 2\leq m' \leq \min\{m,n-1\}.
    \]
The number of irreducible components is $\min\{n-1,m\}$. Moreover, these irreducible components are birational to
    \[
\begin{array}{ccc}
\Sym^m(C)&  \text{or} & \Sym^{m-m'}(C) \times \mathrm{Gr}(n+1-m',n)\,\,\, \text{ for }  2\leq m' \leq \min\{m,n-1\}.
\end{array}
\]   
\end{Cor}

We can generalize \cref{co: irred comp structure} to irreducible curves with several rational fold like singularities. 
Given integers $k\in\N$, and $m,n_1,\ldots,n_k\in Z_{\geq 2}$, we define the number $\rho(k,m,n_1,\ldots,n_k)$ as the cardinality of the set 
\[
S(k,m,n_1,\ldots,n_k):=\left\{\mathbf{m}=(m_1,\ldots,m_k)\in\Z_{\geq0}^k: |\mathbf{m}|\leq m, \,m_i\neq 1 \text{ and } 0\leq m_i\leq \min\{m,n_i-1\}\right\}.
\]

\begin{Cor}\label{cor: component more singularities}
    Let $C$ be an irreducible curve whose singularities are $p_1,\ldots,p_k\in C$ where $p_i$ is a rational $n_i$-fold singularity. Then, the number of irreducible components of $\sHilb^m(C)$ is $\rho(k,m,n_1,\ldots,n_k)$. Moreover, $\sHilb^m(C)$ is reduced and its irreducible components of $\sHilb^m(C)$ are birational to 
    \[\begin{array}{ccc}
    \Sym^m(C) &\text{ and }&\displaystyle\Sym^{m-|\mathbf{m}|}(C)\times\prod_{i=1}^k\Sigma(m_i,n_i+1-m_i,\mathbf{1})
    \end{array}
    \]
    for $\mathbf{m}\in S(k,m,n_1,\ldots,n_k)$.
\end{Cor}
\begin{proof}
We first show that $\sHilb^m(C)$ is reduced. Given a point $[J]\in\sHilb^m(C)$, we can decompose $\mathbb{V}(J)$ as 
$\mathbb{V}(J)=\mathbb{V}(J_0)\cup\mathbb{V}(J_1)\cup\cdots \cup\mathbb{V}(J_k)$ where $J_0$ is supported at the smooth locus of $C$ and $J_i$ is supported at $p_i$ for every $i\in[k]$. Let $m_i$ be the length of $J_i$.
Around $[J]$, étale locally $\sHilb^{m}(C)$ is  isomorphic to the product $\sHilb^{m_0}_\mathrm{sm}(C)\times\sHilb^{m_1}(C)\times\cdots\times\sHilb^{m_k}(C)$ for $([J_0],[J_1],\ldots,[J_k])$ (cf. \cite{JelisiejewElementaryComponents}). Note that $[J_0]$ is reduced in $\sHilb^{m_0}_\mathrm{sm}(C)$. The punctual Hilbert scheme $\sHilb^{m_i}_{p_i}(C)$ is isomorphic to $\sHilb^{m_i}_\mathbf{0}(X_{n_i})$. Therefore, we may see $[J_i]$ also as a point in  $\sHilb^{m_i}_\mathbf{0}(X_{n_i})\subset \sHilb^{m_i}(X_{n_i})$ through this isomorphism. Then,
the completion of the stalk of $[J_i]$ in $\sHilb^{m_i}(C)$ is isomorphic to the completion of the stalk of $[J_i]$ in  $\sHilb^{m_i}(X_{n_i})$ which is reduced by \cref{theo: reduced}.

Next, we calculate the number of irreducible components. An elementary component of  $\sHilb^m(C)$ parametrizes length $m$ subschemes supported at one fixed singular point $p_i$. Therefore, elementary components of $\sHilb^m(C)$ are in bijection with elementary components of $\sHilb^m(X_{n_i})$ for $i\in[k]$. These components correspond to the vectors in $S(k,m,n_1,\ldots,n_k)$ of the form $m\mathbf{e}_i$. By \eqref{eq:induc map gen} we obtain that the non-elementary components of $\sHilb^m(C)$ correspond to the closure of the image of the map
\begin{equation}\label{eq:recursion more sing}
\begin{array}{ccc}
\displaystyle\sHilb^{m-|\mathbf{m}|}_{\mathrm{sm}}(C\setminus\{p_1,\ldots,p_k\})\times\prod_{i=1}^k\Sigma(m_i,n_i+1-m_i,\mathbf{1})&\longrightarrow&\sHilb^m(C)\\
([J_0],[J_1],\dots,[J_k])&\longmapsto&\left[\mathbb{V}(J_0)\cup\cdots\cup\mathbb{V}(J_k)\right]
\end{array}
\end{equation}
for $\mathbf{m}\in S(k,m,n_1,\ldots,n_k)$ with $\mathbf{m}\neq m\mathbf{e}_i$ and $i\in[k]$. Here, $\Sigma(m_i,n_i+1-m_i,\mathbf{1})$ is seen as the corresponding elementary component in the punctual Hilbert scheme $\sHilb^{m_i}_{p_i}(C)$. Moreover, the map \eqref{eq:recursion more sing} is birational onto its image.
Note that for $\mathbf{m}=\mathbf{0}$, the corresponding component is the smoothable component. In particular, we conclude that the irreducible components of $\sHilb^m(C)$ are in bijection with vectors in $S(k,m,n_1,\ldots,n_k)$.    
\end{proof}

\begin{Rem}
    One can derive a formula for the cardinality of $\rho(k,m,n_1,\ldots,n_k)$. To do so, we introduce the number $\chi(k,m,n_1,\ldots,n_k)$ as the cardinality of the set
    \[
    \{\mathbf{m}\in\Z_{\geq 0}^k: 0\leq |\mathbf{m}|\leq m \text{ and } 0\leq m_i\leq n_i\}.
    \]
    We set $\chi(0,m)=1$ for $k=0$.
    Using the Exclusion-Inclusion formula, one may check that 
    \[
    \chi(k,m,n_1,\ldots,n_k)= \displaystyle\sum_{J\subseteq[k]}(-1)^{|J|}\binom{m-\sum_{j\in J}(n_j+1)+k}{k}.
    \]
    Decompose the set $S(k,m,n_1,\ldots,n_k)$ as 
    \[
    S=\bigsqcup_{J\subseteq[k]}S(k,m,n_1,\ldots,n_k)\cap\{m_i=0:i\not \in J\}\cap\{m_i\geq2:i\in J\}.
    \]
    For $J\subseteq[k]$, the  cardinality of the corresponding set in the above disjoint union is 
    $\chi(|J|,m-2,n_i-2:i\in J)$. We conclude that the number of irreducible components in \cref{cor: component more singularities} is 
    \[
    \rho(k,m,n_1,\ldots,n_k)=\displaystyle\sum_{J\subseteq[k]}
    \chi(|J|,m-2,n_i-2:i\in J)=\displaystyle\sum_{J\subseteq[k]}
     \displaystyle\sum_{I\subseteq J}(-1)^{|I|}\binom{ m-2-\sum_{i\in I}(n_i-1)+|J|}{|J|}.
    \]
\end{Rem}

\section{Singularities}\label{sec:singularites and local hilbert scheme}
 
The goal of this section is to describe the singular locus of $\sHilb^m(X_n)$. For doing this, we will rely heavily in Combinatorics, in particular the hypersimplicial complex, as given in \cref{sec:moment map,app: The hypersimplicial complex}.

\subsection{Singular Locus}

We will first compute the singular locus of $\sHilb^m(X_n)$. By a classical result in deformation theory (cf. \cite{Sernesi2006}), the dimension of the tangent space of $\sHilb^m(X_n)$ at the point $[J]$ is given by 
\[
\dim_\C \mathrm{T}_{[J]}\sHilb^m(X_n) = \dim_\C \mathrm{Hom}_R( J, R/J).
\]
To apply this formula, we need the syzygies of $J$, which are computed in \cref{lemma:syz}.
Using this lemma, we will describe the singular locus of $\sHilb^m(X_n)$ by the combinatorics of $\mathcal{K}_n^{[m]}$ using the notion of singular face introduced in \cref{sec:moment map}.

\begin{Prop}\label{prop:singlocus}
    Let $[J]\in \sHilb^m_{\mathbf{0}}(X_n)$. Then the following are equivalent:
    \begin{enumerate}
        \item $[J]$ is a singular point of $\sHilb^m(X_n)$.
        \item $\mu_m([J])$ lies in a singular face of $\mathcal{K}_n^{[m]}$.
        \item $J$ admits a minimal generator of the form $x_i^{u_i}$ with $u_i\geq 2$ or $J$ is minimally generated by $\langle f,x_{i}:i\in S\rangle $ for $f\in\langle x_j^{u_j}: j\not\in S\rangle_\C$ and $\emptyset\subsetneq S\subset[n]$.
    \end{enumerate}
\end{Prop}
 \begin{proof}
By Proposition \cref{prop: ideal singular face}, (2) and (3) are equivalent.
So it is enough to show that (1) and (3) are equivalent. By \cref{ideals punctual} there exists $1\leq l\leq n$ and $\uu\in\Z_{\geq 1}$ such that  $J$ is minimally generated by $f_1,\ldots,f_l$ where 
       \[
    \begin{pmatrix}
    f_1\\ \vdots \\ f_l
    \end{pmatrix} = A\,\,\begin{pmatrix}
    x_1^{u_1}\\ \vdots\\ x_n^{u_n}
    \end{pmatrix}.
    \]
    and $A$ is a size $l\times n$ matrix as in  \cref{ideals punctual}. Every $\varphi\in\mathrm{Hom}_R( J, R/J)$ is uniquely determined by $l$ elements $\alpha_1,\ldots,\alpha_l\in R/J$ satisfying the syzygies \eqref{eq:syz} of $J$, where $\alpha_j:=\varphi(f_j)$ for $1\leq j\leq l$. Write $\alpha_j$ as follows 
    \[
    \alpha_j = \alpha_{j,0}+\displaystyle\sum_{r=1}^n\sum_{s=1}^{u_r}\alpha_{j,r,s}x_r^s.
    \]
    For $1\leq i\leq n$ and $1\leq j<k\leq l$ we have the following equalities:
    \begin{equation}
        \label{eq:rel sing locus}
        \begin{array}{c}
    0= A_{k,i}x_i\alpha_j-A_{j,i}x_i\alpha_k = A_{k,i}\alpha_{j,0}x_i+\displaystyle\sum_{s=1}^{u_i-1}A_{k,i}\alpha_{j,i,s}x_i^{s+1}-
    \sum_{s=1}^{u_i-1}A_{j,i}\alpha_{k,i,s}x_i^{s+1}= \\
    (A_{k,i}\alpha_{j,0}-A_{j,i}\alpha_{k,0})x_i+\displaystyle\sum_{s=1}^{u_i-1}(A_{k,i}\alpha_{j,i,s}-A_{j,i}\alpha_{k,i,s})x_i^{s+1}.
    \end{array}
    \end{equation}
   We first assume that $J$ does not admit a minimal generator of the form $x_i^{u_i}$. Then, none of the terms in \eqref{eq:rel sing locus} vanishes in $R/J$, and we
    obtain that $\alpha_1,\ldots,\alpha_l$ satisfy the relations
    \begin{equation}
        \label{eq:relations tangent space}
            \begin{array}{ll}
       A_{k,i}\alpha_{j,0}-A_{j,i}\alpha_{k,0} = 0  & \text{ for } 1\leq j<k\leq l \text{ and } 1\leq i \leq n, \\
       A_{k,i}\alpha_{j,i,s}-A_{j,i}\alpha_{k,i,s} = 0  &  \text{ for } 1\leq j<k\leq l, \text{ } 1\leq i\leq n \text{ and } 2\leq s \leq u_i.
    \end{array}
    \end{equation}
    Rewrite the first type of relations in  \eqref{eq:relations tangent space} as 
    \[
    (\alpha_{k,0},-\alpha_{j,0})
    \begin{pmatrix}
        A_{j,1}&\cdots&A_{j,n}\\ A_{k,1}&\cdots&A_{k,n}
    \end{pmatrix} =\begin{pmatrix}
        0\\ 0
    \end{pmatrix}.
    \]
    Since the matrix $A$ has maximum rank we get that the matrix 
    \[
     \begin{pmatrix}
        A_{j,1}&\cdots&A_{j,n}\\ A_{k,1}&\cdots&A_{k,n}
    \end{pmatrix}
    \]
    has rank $2$. This implies that $\alpha_{j,0} = 0$ for all $1\leq j\leq l$. The second type of relations in \eqref{eq:relations tangent space} can be written as the $2\times 2$ minors of the matrix
    \begin{equation}
        \label{eq:mat rel sing}
           \begin{pmatrix}
        A_{1,i} & \alpha_{1,i,s} \\ \vdots&\vdots \\ A_{l,i} &\alpha_{l,i,s}
    \end{pmatrix}
    \end{equation}

    for $1\leq i\leq n$ and $2\leq s \leq u_i$.
    Since $A$ has no vanishing columns, there exists $1\leq j\leq l$ such that $A_{j,i}$ is nonzero. Thus, the $2\times 2$ minors of the matrix \eqref{eq:mat rel sing} give $l-1$ relations among $\alpha_{1,i,s},\ldots,\alpha_{l,i,s}$ for $1\leq i\leq n$ and $2\leq s\leq u_i$. Hence $\alpha_1,\ldots,\alpha_l$ satisfy $l+\sum_{i=1}^n(u_i-1)(l-1)=l+(|\uu|-n)(l-1)$ relations. We conclude that the dimension of the tangent space at $[J]$ is
    \[
    l\, m-l-(|\uu|-n)(l-1) = l\, m-l-(m+l-1-n)(l-1)=l(n-l)+(m+l-1-n).
    \]
    On the other hand, $[J]$ is contained in 
    $\sHilb^{m,n+1-l}(X_n)$, which has dimension $l(n-l)+(m+l-1-n)$. Thus $[J]$ is a smooth point of $\sHilb^m(X_n)$.

    Now, assume that $J$ admits a generator of the form $x_r^{u_r}$.  If $u_r \geq 2$, then $J$ is in the intersection of two irreducible components of $\sHilb^m(X_n)$ and therefore $J$ is singular a singular point of $\sHilb^m(X_n)$. Hence we can assume  that $u_r =1$ for every such generator.
    Let $f_1,\ldots,f_a$ be the minimal generators of $J$ that are not of the form $x_r$. In other words, the minimal generators of $J$ are $f_1,\ldots, f_a,x_{n-l+a+1},\ldots,x_n$, and therefore we can write these generators as
    \[
    \begin{pmatrix}
        f_1\\ \vdots \\ f_a\\ x_{n-l+a+1}\\ \vdots\\x_n
    \end{pmatrix}
    =
\begin{pNiceArray}{ccc|ccc}
  \Block{3-3}<\Large>{A} &  &  & \Block{3-3}<\Large>{\mathbf{0}} & & \\
   &  \\
 &  \\
  \hline
  \Block{3-3}<\Large>{\mathbf{0}} && &\Block{3-3}<\Large>{\mathrm{Id}_{l-a}} & &   \\
  &
  \\
  &
\end{pNiceArray}
    \begin{pmatrix}
        x_1^{u_1}\\ \vdots\\ x_{n-l+a}^{u_{n-l+a}} \\x_{n-l+a+1}\\ \vdots \\ x_{n}
    \end{pmatrix},
    \]
    where $A$ has maximal rank and it has no vanishing row or column. Hence the relations among $\alpha_1,\ldots,\alpha_l$ in $R/J$ are 
    \begin{equation}
        \label{eq:const case 2}
         \begin{array}{ll}
     x_i \alpha_j =0  &  \text{ for } a+1 \leq j\leq l \text{ and } i \neq n-l+j, \\
     x_i\alpha_j = 0 &   \text{ for } 1 \leq j\leq a \text{ and } n-l+a+1\leq i \leq n, \\
     A_{k,i}x_i\alpha_j-A_{j,i}x_i\alpha_k = 0 &    \text{ for } 1 \leq j<k\leq a \text{ and } 1\leq i \leq n-l+a. \\
    \end{array}
    \end{equation}
    Since $x_i\in J$ for $n-l+a+1\leq i \leq n$, the left-hand side of the second equation in \eqref{eq:const case 2} is zero in $R/J$ and is therefore redundant. Similarly, the first relation is identically zero in $R/J$ for $n-l+a+1\leq i\leq n$. We obtain
     \begin{equation}
        \label{eq:const case 2.2}
         \begin{array}{ll}
     x_i \alpha_j =0  &  \text{ for } a+1 \leq j\leq l \text{ and } 1\leq i\leq n-l+a, \\
     A_{k,i}x_i\alpha_j-A_{j,i}x_i\alpha_k = 0 &    \text{ for } 1 \leq j<k\leq a \text{ and } 1\leq i \leq n-l+a. \\
    \end{array}
    \end{equation}
    From the first equation in \eqref{eq:const case 2.2}, we get that 
    \[
    \alpha_{j,0}= 0 \text{ and } \alpha_{j,i,s} = 0 \text{ for } a+1\leq j\leq l,\,\, 1\leq i\leq n-l+a \text{ and } 1\leq s\leq u_i-1.
    \]
    In particular, there are $(l-a)(m+l-n)$ linear equations on $\alpha_{a+1},\ldots,\alpha_l$. As in the first part of the proof, the number of linearly independent equations obtained from the second equations in \eqref{eq:const case 2.2} is
    \[
    a\displaystyle+\sum_{i=1}^{n-l+a}(a-1)(u_i-1)=a+(a-1)(m+l-n-1).
    \]
    Now, assume first that $a=1$. In other words, $J$ is satisfy condition (3) of  \cref{prop:singlocus}. Then, we have no equation of the second type in \eqref{eq:const case 2.2}. In particular, the dimension of the tangent space is
    \[
    ml-(l-1)(m+l-n)=l(n-l)+m+l-n.
    \]
    On the other hand, $J$ lies in 
    $\sHilb^{m,n+1-l}(X_n)$, which has dimension $l(n-l)+m+l-n-1$. We conclude that $[J]$ is a singular point of $\sHilb^m(X_n)$.

    Assume now that $a\geq 2$. Then, the dimension of the tangent space at $[J]$ is
    \[
    ml-(l-a)(m+l-n) -a-(a-1)(m+l-n-1)=l(n-l)+m+l-n-1.
    \]
    Moreover, $[J]$ is contained in $\sHilb^{m,n+1-l}(X_n)$ which has also dimension $l(n-l)+m+l-n-1$. We conclude that $[J]$ is a smooth point of $\sHilb^m(X_n)$.
 \end{proof}

 \cref{prop:singlocus} characterizes the points in $\sHilb^m_{\mathbf{0}}(X_n)$ that are singular in $\sHilb^m(X_n)$. The study done in \cref{sec:localtoglobal}, allow us to move from the punctual Hilbert scheme to the global Hilbert scheme, giving a  characterization of the singular locus of $\sHilb^m(X_n)$.

\begin{Thm}\label{theo:sing locus}
   Let $J$ be a length $m$ ideal of $R$ and let $J_1,\ldots,J_k$ be its primary decomposition, where $J_a$ has length $m_a$. Then, $[J]$ is singular in $\sHilb^m(X_n)$
       if and only if  there exists $1\leq a\leq k$ such that $[J_a]$ is contained in $\sHilb^{u_a}_0(X_n)$ and $J_a$ satisfies one of the conditions in  \cref{prop:singlocus}.
\end{Thm}
\begin{proof}
    Let $[J]\in\sHilb^m(X_n)$ and let $J_1,\ldots,J_k$ be its primary decomposition. Let $Z$ be the subscheme of $X_n$ defined by $J$ and let $Z_i$ be the subscheme defined by $J_i$ for $1\leq i\leq k$. 
   Using \cite[Section 4.6.5]{Sernesi2006}, we get that 
    \[
    \mathrm{T}_{[J]} \sHilb^m(X_n) = H^0(Z,N_{Z/X})=\bigoplus_{1\leq i\leq k}H^0(Z_i,N_{Z_i/X}) = \bigoplus_{1\leq i\leq k}\mathrm{T}_{[J_i]} \sHilb^m(X_n).
    \]
    The proof then follows from \cref{prop:singlocus}.
\end{proof}

\begin{Rem}
    Using the notation in \cref{theo:sing locus}, if $m_a=1$, then $[J_a]\in \sHilb_{\mathbf{0}}^{1}(X_n)=\{p\}$, which is singular. In particular, $\Sym^{m-1}\left(X_n\setminus\{p\}\right)\times \{p\}$ is a subset of the singular locus of $\sHilb^m(X_n)$.
\end{Rem}

\subsection{Local picture of the singularities}

Next, we give a description of the singularities of $\sHilb^m(X_n)$. By \cref{lemma:one parameter family}, the closure of any torus orbit in $\sHilb^m(X_n)$ contains a point $[\mathbb{V}(J)]\in\sHilb^m_{\mathbf{0}}(X_n)$ that corresponds to a vertex of $\mathcal{K}_n^{[m]}$ via the moment map. Therefore, it is enough to describe the singularities of $\sHilb^m(X_n)$ at points of the form $[J]=[\langle x_1^{u_1},\ldots,x_k^{u_k},x_{k+1},\ldots,x_n\rangle]$ for $1\leq k \leq n$ and $u_i\geq 2$ for $i\in[k]$.
Using \cref{theo:def theory} and \cref{lemma: new rings}, it is enough to analyze the singularty at the origin of the variety defined by the ideal $\mathcal{J}_k$. By \cref{lemma: primary k1}, \cref{lemma:primary k2} and \cref{lem:normal polytope}, locally at the origin $\mathbb{V}(\mathcal{J}_k)$ is the union of normal toric varieties. We describe the singularity through the gluing of the polytopes associated to these toric varieties. To do so, we first do some simplification on the coordinate ring $\mathcal{S}_k/\mathcal{J}_k$. For $k=1$, the variables $\alpha_{1,2},\ldots,\alpha_{1,m-1}$ do not appear in the ideal $\mathcal{J}_1$. Therefore, we have that 
$$\mathcal{S}_1/\mathcal{J}_1\simeq \C[\alpha_{1,2},\ldots,\alpha_{1,m-1}]\otimes (\mathcal{S}_1'/\mathcal{J}_1),$$ where $\mathcal{S}_1'=\C[\alpha_2,\ldots,\alpha_n,A_1,\alpha_{1,1}]$  is the polynomial ring in the rest of the variables.
Similarly, for $k\geq 2$, the variables $\beta_{1,s}$ for $i\in[k]$ and $2\leq s\leq u_i-1$ do not appear in the ideal $\mathcal{J}_K$. Therefore, we have that 
$$\mathcal{S}_k/\mathcal{J}_k\simeq \C[\beta_{i,j}:1\leq i\leq k \text{ and }2\leq j\leq u_i-1]\otimes (\mathcal{S}_k'/\mathcal{J}_k),$$ 
where $\mathcal{S}_k'=\C[\beta_1,\ldots,\beta_k,\alpha_{i,j}:i\in[n],j\in[k],i\neq j]$  is the polynomial ring in the rest of the variables. In particular, the singularity type of $\sHilb^m(X_n)$ at $[J]$ is the same as the affine variety $\mathrm{spec}(\mathcal{S}'_k/\mathcal{J}_k)$ at the origin.

We start by the case $k=1$. In other words, assume that $J=\langle x_1^m,x_2,\ldots,x_n\rangle $ and $\mu([J])$ is an edge of $(m-1)\cdot\Delta_{n-1}$.
    By \cref{lemma: primary k1}, the primary decomposition of $\mathcal{J}_1$ is given by the ideals $J_0=\langle A_1,\alpha_{1,1}\rangle$, $J_1=\langle \alpha_2,\ldots,\alpha_n\rangle$ and $J_i=\langle A_,\alpha_j:2\leq j\leq n, \,j\neq i\rangle$. Note that $\mathbb{V}(J_0)$ is a linear subspace of dimension $n-1$, and $\mathbb{V}(J_i)$ for $1\leq i\leq n$ is a linear subspace of dimension $2$. By \cref{prop: ideal translation}, $J_0$ corresponds to the component $\sHilb^{m,2}(X_n)$ and $J_i$ for $i\geq 1$ correspond to the smoothable component.  
    Now, $\mathbb{V}(J_0)$ and $\mathbb{V}(J_1)$ intersect in the origin. Similarly,  $\mathbb{V}(J_0)$ and $\mathbb{V}(J_i)$ for $2\leq i\leq n$  intersect in the line corresponding to $\C[\alpha_{i}]$. Finally,  $\mathbb{V}(J_i)$ and $\mathbb{V}(J_j)$ for $1\leq i<j\leq n$  intersect in the line $\C[\alpha_{1,1}]$. 
    We can associate to $J_0$ the simplex $\Delta_{n}$ whose vertices are labeled by $v_0,v_2\ldots,v_n$. Moreover, we associate to $J_i$ the simplex $M_i:=\Delta_{2}$ whose vertices are labeled by $w_{i,0}$, $w_{i,1}$ and $w_{i,2}$. We construct the simplicial complex $ \mathbb{S}_1$ obtained by the following gluing:
    \begin{itemize}
        \item For $2\leq i\leq n$, glue the edge $\overline{v_0,v_i}$ of $\Delta_{n}$ with the edge $\overline{w_{i,0},w_{i,2}} $ of $M_i$.
        \item For $1\leq i<j\leq n$, glue the edge $\overline{w_{i,0},w_{i,1}} $ of $M_i$ with the edge $\overline{w_{j,0},w_{j,1}} $ of $M_j$.
    \end{itemize}
    In \cref{fig:simplex 3} and \cref{fig:simplex 4}, the simplicial complex $\mathbb{S}_1$ is depicted for $n=3$ and $n=4$ respectively. In these figures, the simplex in blue corresponds to $\Delta_{n}$ and $J_0$, the simplex in orange represents $M_1$ and $J_1$,
 the simplices in purple are $M_i$ and $J_i$ for $2\leq i\leq n$.
    The simplicial complex $\mathbb{S}_1$ around the origin $\mathbf{0}$ describes how the components of $\mathcal{J}_1$ intersect at the origin. Here the origin $\mathbf{0}$ is the vertex obtained from the gluing of $v_0,w_{1,0},\ldots, w_{n,0}$.  In particular, the singularity type is described by the complex $\widehat{\mathbb{S}}_1$.

   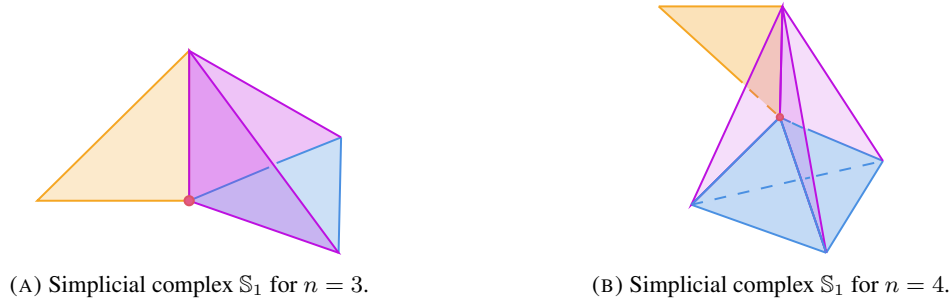
\begin{figure}
\begin{subfigure}{.45\textwidth}
  \centering

\tikzset{every picture/.style={line width=0.75pt}} 
\begin{tikzpicture}[x=0.95pt,y=0.95pt,yscale=-1,xscale=1]

\draw  [color={rgb, 255:red, 245; green, 166; blue, 35 }  ,draw opacity=1 ][fill={rgb, 255:red, 245; green, 166; blue, 35 }  ,fill opacity=0.23 ] (223.67,173.5) -- (284.17,113.75) -- (284.17,173.5) -- cycle ;
\draw  [draw opacity=0][fill={rgb, 255:red, 189; green, 16; blue, 224 }  ,fill opacity=0.25 ] (284.17,113.75) -- (344.67,148.25) -- (284.17,173.5) -- cycle ;
\draw  [draw opacity=0][fill={rgb, 255:red, 74; green, 144; blue, 226 }  ,fill opacity=0.28 ] (284.17,173.5) -- (344.67,148.25) -- (343.67,194.25) -- cycle ;
\draw [color={rgb, 255:red, 74; green, 144; blue, 226 }  ,draw opacity=1 ]   (284.17,173.5) -- (316,160.17) ;
\draw [color={rgb, 255:red, 74; green, 144; blue, 226 }  ,draw opacity=1 ]   (344.67,148.25) -- (319,158.5) ;
\draw [color={rgb, 255:red, 74; green, 144; blue, 226 }  ,draw opacity=1 ]   (343.67,194.25) -- (344.67,148.25) ;
\draw [color={rgb, 255:red, 189; green, 16; blue, 224 }  ,draw opacity=1 ]   (344.67,148.25) -- (284.17,113.75) ;
\draw  [color={rgb, 255:red, 189; green, 16; blue, 224 }  ,draw opacity=1 ][fill={rgb, 255:red, 189; green, 16; blue, 224 }  ,fill opacity=0.21 ] (284.17,113.75) -- (343.67,194.25) -- (284.17,173.5) -- cycle ;
\draw  [color={rgb, 255:red, 220; green, 84; blue, 103 }  ,draw opacity=1 ][fill={rgb, 255:red, 227; green, 90; blue, 135 }  ,fill opacity=1 ] (282.33,173.5) .. controls (282.33,172.49) and (283.15,171.67) .. (284.17,171.67) .. controls (285.18,171.67) and (286,172.49) .. (286,173.5) .. controls (286,174.51) and (285.18,175.33) .. (284.17,175.33) .. controls (283.15,175.33) and (282.33,174.51) .. (282.33,173.5) -- cycle ;
\end{tikzpicture}

  \caption{Simplicial complex $\mathbb{S}_1$ for $n=3$.}
  \label{fig:simplex 3}
\end{subfigure}
\begin{subfigure}{.45\textwidth}
  \centering

\tikzset{every picture/.style={line width=0.75pt}} 

\begin{tikzpicture}[x=0.65pt,y=0.65pt,yscale=-1,xscale=1]

\draw  [draw opacity=0] (249.5,178.42) -- (361,152.92) -- (301,127.42) -- cycle ;
\draw  [draw opacity=0] (328,206.42) -- (361,152.92) -- (301,127.42) -- cycle ;
\draw  [draw opacity=0] (249.5,178.42) -- (301,127.42) -- (328,206.42) -- cycle ;
\draw  [draw opacity=0][fill={rgb, 255:red, 74; green, 144; blue, 226 }  ,fill opacity=0.33 ] (301,127.42) -- (361,152.92) -- (328,206.42) -- (249.5,178.42) -- cycle ;
\draw [color={rgb, 255:red, 74; green, 144; blue, 226 }  ,draw opacity=1 ]   (249.5,178.42) -- (328,206.42) ;
\draw [color={rgb, 255:red, 74; green, 144; blue, 226 }  ,draw opacity=1 ]   (328,206.42) -- (361,152.92) ;
\draw [color={rgb, 255:red, 74; green, 144; blue, 226 }  ,draw opacity=1 ]   (301,127.42) -- (328,206.42) ;
\draw [color={rgb, 255:red, 74; green, 144; blue, 226 }  ,draw opacity=1 ]   (301,127.42) -- (249.5,178.42) ;
\draw [color={rgb, 255:red, 74; green, 144; blue, 226 }  ,draw opacity=1 ] [dash pattern={on 4.5pt off 4.5pt}]  (301,127.42) -- (308.11,130.78) ;
\draw [color={rgb, 255:red, 74; green, 144; blue, 226 }  ,draw opacity=1 ]   (317.29,134.52) -- (361,152.92) ;
\draw  [draw opacity=0][fill={rgb, 255:red, 245; green, 166; blue, 35 }  ,fill opacity=0.32 ] (302.5,62.92) -- (301,127.42) -- (230.5,62.92) -- cycle ;
\draw [color={rgb, 255:red, 245; green, 166; blue, 35 }  ,draw opacity=1 ][fill={rgb, 255:red, 245; green, 166; blue, 35 }  ,fill opacity=0.32 ]   (230.5,62.92) -- (302.5,62.92) ;
\draw [color={rgb, 255:red, 245; green, 166; blue, 35 }  ,draw opacity=1 ] [dash pattern={on 4.5pt off 4.5pt}]  (283.29,111.38) -- (301,127.42) ;
\draw [color={rgb, 255:red, 245; green, 166; blue, 35 }  ,draw opacity=1 ][fill={rgb, 255:red, 245; green, 166; blue, 35 }  ,fill opacity=0.32 ]   (230.5,62.92) -- (279.29,107.38) ;
\draw  [draw opacity=0][fill={rgb, 255:red, 189; green, 16; blue, 224 }  ,fill opacity=0.14 ][dash pattern={on 4.5pt off 4.5pt}] (302.5,62.92) -- (361,152.92) -- (301,127.42) -- cycle ;
\draw [color={rgb, 255:red, 189; green, 16; blue, 224 }  ,draw opacity=1 ]   (302.5,62.92) -- (361,152.92) ;
\draw  [color={rgb, 255:red, 189; green, 16; blue, 224 }  ,draw opacity=1 ][fill={rgb, 255:red, 189; green, 16; blue, 224 }  ,fill opacity=0.14 ] (302.5,62.92) -- (328,206.42) -- (301,127.42) -- cycle ;
\draw [color={rgb, 255:red, 74; green, 144; blue, 226 }  ,draw opacity=1 ] [dash pattern={on 4.5pt off 4.5pt}]  (249.5,178.42) -- (311.57,163.95) ;
\draw [color={rgb, 255:red, 74; green, 144; blue, 226 }  ,draw opacity=1 ] [dash pattern={on 4.5pt off 4.5pt}]  (322.11,161.67) -- (361,152.92) ;
\draw [color={rgb, 255:red, 74; green, 144; blue, 226 }  ,draw opacity=1 ]   (315.67,163.22) -- (318.56,162.56) ;
\draw  [color={rgb, 255:red, 189; green, 16; blue, 224 }  ,draw opacity=1 ][fill={rgb, 255:red, 189; green, 16; blue, 224 }  ,fill opacity=0.14 ] (302.5,62.92) -- (249.5,178.42) -- (301,127.42) -- cycle ;
\draw [color={rgb, 255:red, 74; green, 144; blue, 226 }  ,draw opacity=1 ]   (301,127.42) -- (328,206.42) ;
\draw [color={rgb, 255:red, 74; green, 144; blue, 226 }  ,draw opacity=1 ]   (301,127.42) -- (249.5,178.42) ;
\draw [color={rgb, 255:red, 74; green, 144; blue, 226 }  ,draw opacity=1 ] [dash pattern={on 4.5pt off 4.5pt}]  (308.11,130.78) -- (315.22,134.14) ;
\draw  [color={rgb, 255:red, 220; green, 84; blue, 103 }  ,draw opacity=1 ][fill={rgb, 255:red, 227; green, 90; blue, 135 }  ,fill opacity=1 ] (299.17,127.42) .. controls (299.17,126.4) and (299.99,125.58) .. (301,125.58) .. controls (302.01,125.58) and (302.83,126.4) .. (302.83,127.42) .. controls (302.83,128.43) and (302.01,129.25) .. (301,129.25) .. controls (299.99,129.25) and (299.17,128.43) .. (299.17,127.42) -- cycle ;

\end{tikzpicture}

  \caption{Simplicial complex $\mathbb{S}_1$ for $n=4$.}
  \label{fig:simplex 4}
\end{subfigure}
\caption{Simplicial complex $\mathbb{S}_1$ for $n=3,4$.}
\end{figure}

For $k\geq 2$, the situation is a bit more complicated. The primary decomposition of $\mathcal{J}_k$ is given in \cref{lemma:primary k2}, and it has two types of ideals: $\mathcal{Q}_i$ for $i\in[n]$ and $\mathcal{J}_S$ for $S\subsetneq[k]$ with $1\leq |S|\leq \min\{k,n-2\}$.  By \cref{prop: ideal translation 2}, $\mathcal{Q}_i$  corresponds to the smoothable component and $\mathcal{J}_S$ corresponds to the component $\sHilb^{m,|S|+1}(C)$. First, we construct a simplicial complex that describes how the ideals $\mathcal{J}_S$ intersect. 
We can associate to the ring $\mathcal{S}'_k$ the affine space $\mathbb{A}^{nk}$ and the real vector space $\mathbb{R}^{nk}$. We denote the standard vectors of $\R^{nk}$ by $\mathbf{b}_1,\ldots,\mathbf{b}_k,\mathbf{a}_{i,j}$ for $i\in[n]$, $j\in[k]$ with $i\neq j$, where $\mathbf{b}_j$ and $\mathbf{a}_{i,j}$ are the standard vectors associated to the variables $\beta_j$ and $\alpha_{i,j}$ respectively.  Then, we have that 
\[
\mathbb{V}_\mathbb{R}(\mathcal{J}_S)=\mathrm{Span}(\mathbf{b}_j:j\in[k]\setminus S,\mathbf{a}_{i,j}:i\in[n]\setminus S,j\in S).
\]
as a real linear subspace.
In $\mathbb{R}^{nk}$, we consider the simplices $\Delta_{nk-1}:=\mathrm{Conv}(\mathbf{b}_j:j\in[k],\mathbf{a}_{i,j}: i\in[n]\,j\in[k],\,i\neq j)$
and 
$\widehat{\Delta}_{nk}:=\mathrm{Conv}(\mathbf{0},\Delta_{nk-1})$.
Moreover, we associated to $\mathcal{J}_S$ the simplices 
\begin{equation}
    \label{eq: simplex S}
    \begin{array}{l}
    \Delta_S:= \mathrm{Conv}(\mathbf{b}_j:j\in[k]\setminus S,\mathbf{a}_{i,j}:i\in[n]\setminus S,j\in S)=\mathbb{V}(\mathcal{J}_S)\cap \Delta_{nk-1},\\
   \widehat{ \Delta}_S:=\mathrm{Conv}(\mathbf{0},\Delta_S)=\mathbb{V}(\mathcal{J}_S)\cap\widehat{\Delta}_{nk}.
    \end{array}
\end{equation}
Note that $ \widehat{ \Delta}_S$ is the cone of $\Delta_S$. Geometrically, $\Delta_S$ is the polytope associated to the projective space defined by $\mathcal{J}_S$ in $\P^{nk-1}$ and $ \widehat{ \Delta}_S$ is the polytope associated to the cone over this projective space, i.e. the polytope associated to $\mathbb{V}(\mathcal{J}_S)$ in $\A^{nk}$.

\begin{Lem}
    \label{lem:simplicial singularity}
    The simplices $\Delta_S$ for $S\subsetneq[k]$ with $1\leq |S|\leq \min\{k,n-2\}$ form a simplicial complex denoted by $\mathbb{S}_k'$. Similarly, the simplices $\widehat{\Delta}_S$ for $S\subsetneq[k]$ with $1\leq |S|\leq \min\{k,n-2\}$ form a simplicial complex denoted by $\widehat{\mathbb{S}}_k'$.
    Moreover, we have that 
    $$
    \begin{array}{ccc}
    \Delta_S\cap \Delta_{S'} = \mathbb{V}_\mathbb{R}(\mathcal{J}_S+\mathcal{J}_{S'})\cap \Delta_{nk-1} &\text{ and }&
    \widehat{\Delta}_S\cap \widehat{\Delta}_{S'} = \mathbb{V}_\mathbb{R}(\mathcal{J}_S+\mathcal{J}_{S'})\cap \widehat{\Delta}_{nk}
    \end{array}$$
    for $S,S'\subsetneq[k]$ with $1\leq |S|,|S'|\leq \min\{k,n-2\}$.
\end{Lem}
\begin{proof}
   Let $S,S'\subsetneq[k]$ with $1\leq |S|,|S'|\leq \min\{k,n-2\}$. Then, the linear subspace $\mathbb{V}(\mathcal{J}_S+\mathcal{J}_{S'})$ is generated by the vectors $e_{i}$ for $i\in[k]\setminus(S\cup S')$ and $f_{i,j}$ for $i\in[n]\setminus(S\cup S')$ and $j\in S\cap S'$. Using \eqref{eq: simplex S}, we get that
   \[\begin{array}{c}
   \Delta_S\cap\Delta_{S'}=\mathbb{V}(\mathcal{J}_S)\cap\mathbb{V}(\mathcal{J}_{S'})\cap \Delta_{nk-1}= 
   \mathbb{V}(\mathcal{J}_S+\mathcal{J}_{S'})\cap \Delta_{nk-1}=\\
   \mathrm{Conv}(\mathbf{b}_i:i\in[k]\setminus(S\cup S'),\mathbf{a}_{i,j}:i\in[n]\setminus(S\cup S')\text{ and }j\in S\cap S'). 
   \end{array}
   \]
   In particular, $\Delta_S\cap\Delta_{S'}$ is a face of both $\Delta_S$ and $\Delta_{S'}$, and hence they form a simplicial complex. For $\widehat{\mathbb{S}}_k'$, the proof follows from the fact that it is the simplicial complex obtained by taking the cone of $\mathbb{S}_k'$ over the origin.
\end{proof}

From \cref{lem:simplicial singularity}, we conclude that the simplicial complex $\widehat{\mathbb{S}}_k'$ around $\mathbf{0}$ describes how the linear spaces $\mathbb{V}(\mathcal{J}_S)$ intersect. Since $\widehat{\mathbb{S}}_k'$ is the cone of $\mathbb{S}_k'$ over the origin, such combinatorics are also encoded in $\mathbb{S}_k'$.

\begin{Ex}
    Fix $n=3$, $m=3$ and $k=2$, and focus on the singular point $[\langle x_1^2,x_2^2,x_3\rangle]$. This point corresponds to the middle point in the bottom edge of $\mathcal{K}_3^{[3]}$ in \cref{fig: n=3 simplex}. The  corresponding ring is $\mathcal{S}_k=\C[\beta_1,\beta_2,\alpha_{1,2},\alpha_{2,1},\alpha_{3,1},\alpha_{3,2}]$, and we have two possible ideals of the type $\mathcal{J}_S$: 
    \[\begin{array}{ccc}
    \mathcal{J}_{\{1\}}=\langle \beta_1,\alpha_{1,2},\alpha_{3,2}\rangle &\text{and} &\mathcal{J}_{\{2\}}=\langle \beta_2,\alpha_{2,1},\alpha_{3,1}\rangle.
    \end{array}
    \]
    Then, $\mathbb{V}(\mathcal{J}_{\{1\}})\cap \mathbb{V}(\mathcal{J}_{\{2\}})$ is the origin $\mathbf{0}$. Similarly, the simplices associated to them are
    \[\begin{array}{ccc}
    \widehat{\Delta}_{\{1\}}=\mathrm{Conv}(\mathbf{0},\mathbf{b}_2,\mathbf{a}_{2,1},\mathbf{a}_{3,1})& \text{and} &\widehat{\Delta}_{\{2\}}=\mathrm{Conv}(\mathbf{0},\mathbf{b}_1,\mathbf{a}_{2,1},\mathbf{a}_{3,2}).
    \end{array}
    \]
    In this case,
    the corresponding simplicial complex $\widehat{\mathbb{S}}_2'$ is depicted in \cref{fig:simplex 2 3}. The simplicial complex $\mathbb{S}_2'$ is illustrated in dark blue in  \cref{fig:simplex 2 3} as a subcomplex of $\widehat{\mathbb{S}}_2'$.
\end{Ex}

   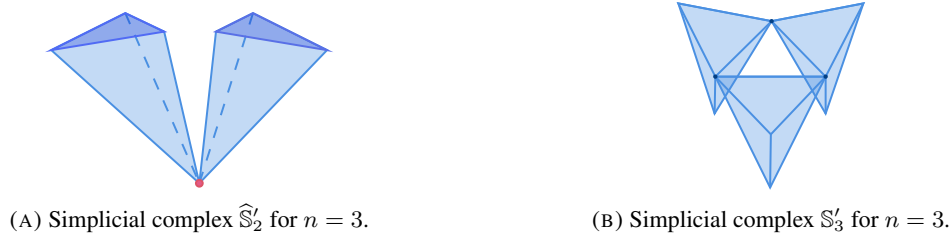
\begin{figure}
\begin{subfigure}{.45\textwidth}
  \centering

\tikzset{every picture/.style={line width=0.75pt}} 

\begin{tikzpicture}[x=0.70pt,y=0.70pt,yscale=-1,xscale=1]

\draw  [color={rgb, 255:red, 74; green, 144; blue, 226 }  ,draw opacity=1 ][fill={rgb, 255:red, 74; green, 144; blue, 226 }  ,fill opacity=0.33 ] (310,182) -- (291,100) -- (270,90) -- (230,110) -- cycle ;
\draw [color={rgb, 255:red, 74; green, 144; blue, 226 }  ,draw opacity=1 ] [dash pattern={on 4.5pt off 4.5pt}]  (270,90) -- (310,182) ;
\draw  [color={rgb, 255:red, 74; green, 144; blue, 226 }  ,draw opacity=1 ][fill={rgb, 255:red, 74; green, 144; blue, 226 }  ,fill opacity=0.33 ] (339,90) -- (319,100) -- (310,182) -- (379,110) -- cycle ;
\draw  [color={rgb, 255:red, 81; green, 107; blue, 244 }  ,draw opacity=1 ][fill={rgb, 255:red, 5; green, 39; blue, 214 }  ,fill opacity=0.27 ] (291,100) -- (270,90) -- (230,110) -- cycle ;
\draw [color={rgb, 255:red, 74; green, 144; blue, 226 }  ,draw opacity=1 ] [dash pattern={on 4.5pt off 4.5pt}]  (339,90) -- (310,182) ;
\draw  [color={rgb, 255:red, 89; green, 112; blue, 230 }  ,draw opacity=1 ][fill={rgb, 255:red, 5; green, 39; blue, 214 }  ,fill opacity=0.27 ] (379,110) -- (339,90) -- (319,100) -- cycle ;
\draw  [color={rgb, 255:red, 220; green, 84; blue, 103 }  ,draw opacity=1 ][fill={rgb, 255:red, 227; green, 90; blue, 135 }  ,fill opacity=1 ] (308.17,182) .. controls (308.17,180.99) and (308.99,180.17) .. (310,180.17) .. controls (311.01,180.17) and (311.83,180.99) .. (311.83,182) .. controls (311.83,183.01) and (311.01,183.83) .. (310,183.83) .. controls (308.99,183.83) and (308.17,183.01) .. (308.17,182) -- cycle ;

\end{tikzpicture}

  \caption{Simplicial complex $\widehat{\mathbb{S}}_2'$ for $n=3$.}
  \label{fig:simplex 2 3}
\end{subfigure}
\begin{subfigure}{.45\textwidth}
  \centering

\tikzset{every picture/.style={line width=0.75pt}} 

\begin{tikzpicture}[x=0.70pt,y=0.70pt,yscale=-1,xscale=1]

\draw  [color={rgb, 255:red, 74; green, 144; blue, 226 }  ,draw opacity=1 ][fill={rgb, 255:red, 74; green, 144; blue, 226 }  ,fill opacity=0.33 ] (231,60.2) -- (260.2,110.2) -- (280.6,50.6) -- cycle ;
\draw  [color={rgb, 255:red, 74; green, 144; blue, 226 }  ,draw opacity=1 ] (231,60.2) -- (280.6,50.6) -- (260.2,90.2) -- cycle ;
\draw [color={rgb, 255:red, 74; green, 144; blue, 226 }  ,draw opacity=1 ]   (260.2,90.2) -- (260.2,110.2) ;
\draw  [color={rgb, 255:red, 74; green, 144; blue, 226 }  ,draw opacity=1 ][fill={rgb, 255:red, 74; green, 144; blue, 226 }  ,fill opacity=0.33 ] (180.6,50.6) -- (200.2,110.2) -- (231,60.2) -- cycle ;
\draw  [color={rgb, 255:red, 74; green, 144; blue, 226 }  ,draw opacity=1 ] (180.6,50.6) -- (231,60.2) -- (200.6,90.2) -- cycle ;
\draw [color={rgb, 255:red, 74; green, 144; blue, 226 }  ,draw opacity=1 ]   (200.6,90.2) -- (200.2,110.2) ;
\draw  [color={rgb, 255:red, 74; green, 144; blue, 226 }  ,draw opacity=1 ][fill={rgb, 255:red, 74; green, 144; blue, 226 }  ,fill opacity=0.33 ] (200.6,90.2) -- (230.2,150.2) -- (260.2,90.2) -- cycle ;
\draw  [color={rgb, 255:red, 74; green, 144; blue, 226 }  ,draw opacity=1 ] (200.6,90.2) -- (260.2,90.2) -- (230.6,121.4) -- cycle ;
\draw [color={rgb, 255:red, 74; green, 144; blue, 226 }  ,draw opacity=1 ]   (230.6,121.4) -- (230.2,150.2) ;
\draw  [color={rgb, 255:red, 2; green, 63; blue, 133 }  ,draw opacity=1 ][fill={rgb, 255:red, 1; green, 65; blue, 144 }  ,fill opacity=1 ] (259.58,90.2) .. controls (259.58,89.86) and (259.86,89.58) .. (260.2,89.58) .. controls (260.54,89.58) and (260.82,89.86) .. (260.82,90.2) .. controls (260.82,90.54) and (260.54,90.82) .. (260.2,90.82) .. controls (259.86,90.82) and (259.58,90.54) .. (259.58,90.2) -- cycle ;
\draw  [color={rgb, 255:red, 2; green, 63; blue, 133 }  ,draw opacity=1 ][fill={rgb, 255:red, 1; green, 65; blue, 144 }  ,fill opacity=1 ] (230.38,60.2) .. controls (230.38,59.86) and (230.66,59.58) .. (231,59.58) .. controls (231.34,59.58) and (231.62,59.86) .. (231.62,60.2) .. controls (231.62,60.54) and (231.34,60.82) .. (231,60.82) .. controls (230.66,60.82) and (230.38,60.54) .. (230.38,60.2) -- cycle ;
\draw  [color={rgb, 255:red, 2; green, 63; blue, 133 }  ,draw opacity=1 ][fill={rgb, 255:red, 1; green, 65; blue, 144 }  ,fill opacity=1 ] (199.98,90.2) .. controls (199.98,89.86) and (200.26,89.58) .. (200.6,89.58) .. controls (200.94,89.58) and (201.22,89.86) .. (201.22,90.2) .. controls (201.22,90.54) and (200.94,90.82) .. (200.6,90.82) .. controls (200.26,90.82) and (199.98,90.54) .. (199.98,90.2) -- cycle ;

\end{tikzpicture}

  \caption{Simplicial complex $\mathbb{S}_3'$ for $n=3$.}
  \label{fig:simplex 3 3}
\end{subfigure}
\caption{Simplicial complex $\widehat{\mathbb{S}'}_2$ and $\mathbb{S}'_3$ for $n=3$.}
\end{figure}

\begin{Ex} 
    Fix $n=3$, $m=4$ and $k=3$, and focus on the singular point $[\langle x_1^2,x_2^2,x_3^2\rangle]$. Such a point corresponds to the middle vertex of $\mathcal{K}_3^{[3]}$ in \cref{fig: n=3 simplex}. The  corresponding ring is $\mathcal{S}_k=\C[\beta_1,\beta_2,\beta_3,\alpha_{1,2},\alpha_{1,3},\alpha_{2,1},\alpha_{2,3},\alpha_{3,1},\alpha_{3,2}]$, and we have three possible ideals of the type $\mathcal{J}_S$: 
    \[\begin{array}{cccc}
    \mathcal{J}_{\{1\}}=\langle \beta_1,\alpha_{1,2},\alpha_{1,3},\alpha_{2,3},\alpha_{3,2}\rangle ,&\mathcal{J}_{\{2\}}=\langle \beta_2,\alpha_{2,1},\alpha_{2,3},\alpha_{1,3},\alpha_{3,1}\rangle&\text{and} &\mathcal{J}_{\{3\}}=\langle \beta_3,\alpha_{3,1},\alpha_{3,2},\alpha_{1,2},\alpha_{2,1}\rangle.
    \end{array}
    \]
    Then $\mathbb{V}(\mathcal{J}_{\{a\}})\cap \mathbb{V}(\mathcal{J}_{\{b\}})$ is $\sp(\C[\beta_c])$ for $\{a,b,c\}=[3]$. Similarly, the simplices associate to them are
    \[\begin{array}{cccc}
    \Delta_{\{1\}}=\mathrm{Conv}(\mathbf{b}_2,\mathbf{b}_3,\mathbf{a}_{2,1},\mathbf{a}_{3,1}),&\Delta_{\{2\}}=\mathrm{Conv}(\mathbf{b}_1,\mathbf{b}_3,\mathbf{a}_{1,2},\mathbf{a}_{3,2})& \text{and} &\Delta_{\{3\}}=\mathrm{Conv}(\mathbf{b}_1,\mathbf{b}_2,\mathbf{a}_{1,3},\mathbf{a}_{2,3}).
    \end{array}
    \]
    In this case,
    the simplicial complex $\mathbb{S}_2$ is depicted in \cref{fig:simplex 3 3}. The simplicial complex $\widehat{\mathbb{S}'}_3$ is the cone over $\mathbb{S}_2'$.
\end{Ex}

Next, we add to our simplicial complexes the polytopes associated to the ideals of the form $\mathcal{Q}_i$ in \cref{lemma: primary k1}. For $i\not \in [k]$, the ring $\mathcal{S}_k'/\mathcal{Q}_i$ is the ring
\[
\C[\alpha_{i,1},\ldots,\alpha_{i,k},\beta_1,\ldots,\beta_k]/\langle\alpha_{i,r}\beta_r-\alpha_{i,s}\beta_s:1\leq r<s\leq k\rangle.
\]
The ideal of this quotient corresponds to the ideal \eqref{eq:toric ideal}. Since the ideal is homogeneous and toric, we can associate to $\mathcal{Q}_i$ the polytopes $P_{i,k}$, corresponding to the projective toric variety, and 
$\widehat{P}_{i,k}$, which is the cone over $P_{i,k}$ and corresponds to the corresponding affine toric variety.
By \cref{prop: toric variety}, $\widehat{P}_{i,k}$ is the polytope $P_k$ and $\widehat{P}_{i,k}$ is the cone over $P_k$. We refer to \eqref{eq:conv hull} for the definition of $P_k$.
We label the vertices of $P_{i,k}$ and $\widehat{P}_{i,k}$ slightly different than the label done for $P_k$. The vertex $\mathbf{0}$ of $P_k$ is denoted by $a_{i,1}$ in $P_{k,i}$. Similarly, the vertices $\mathbf{e}_1$, $\mathbf{e}_2$ and $\mathbf{e}_1+\mathbf{e}_2$ are labeled by $b_2$, $a_{i,2}$ and $b_1$ respectively. The vertices $\mathbf{e}_j$ and $\mathbf{e}_1+\mathbf{e}_2-\mathbf{e}_j$ are labeled by $a_{i,j}$ and $b_{j}$ for $3\leq j\leq k$. The extra point we get after taking the cone over $P_k$ is denoted by $\mathbf{0}$.

\begin{Ex}
    Fix $n=3$ and $k=2$. Then, $P_2$ is the convex hull of $\mathbf{0}$, $\mathbf{e}_1$, $\mathbf{e}_2$ and  $\mathbf{e}_1+\mathbf{e}_2$. The polytope $P_{2,3}$ associated to the ideal $\mathcal{Q}_3$ is the cone over $P_2$. In  \cref{fig:polytope 2 3}, the polytope $P_{2,3}$ is illustrated together with the labeling of its vertices. Moreover, the polytope $P_2$ is also illustrated in dark orange in \cref{fig:polytope 2 3}.

    For $n=4$ and $k=3$, the polytope $P_3$ the convex hull of $\mathbf{0}$, $\mathbf{e}_1$, $\mathbf{e}_2$  $\mathbf{e}_1+\mathbf{e}_2$, $\mathbf{e}_3$  and $\mathbf{e}_1+\mathbf{e}_2-\mathbf{e}_3$. The polytope $P_{3,4}$ associated to the ideal $\mathcal{Q}_4$ is the cone associated to $P_3$. In \cref{fig:polytope 3 4}, the polytope $P_3$ is depicted with the labelling induced by $P_{3,4}$.
\end{Ex}

   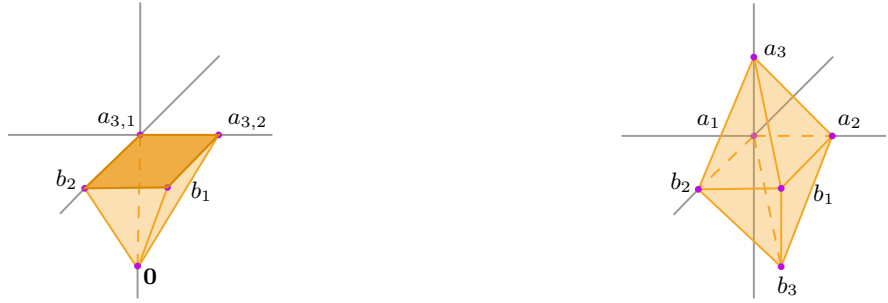
\begin{figure}

\begin{subfigure}{.45\textwidth}
  \centering

\tikzset{every picture/.style={line width=0.75pt}} 

\begin{tikzpicture}[x=1pt,y=1pt,yscale=-1,xscale=1]

\draw [color={rgb, 255:red, 155; green, 155; blue, 155 }  ,draw opacity=1 ]   (262,103.33) -- (312,103.33) ;
\draw [color={rgb, 255:red, 155; green, 155; blue, 155 }  ,draw opacity=1 ]   (312,103.33) -- (312,53.17) ;
\draw [color={rgb, 255:red, 155; green, 155; blue, 155 }  ,draw opacity=1 ]   (312,103.33) -- (342,73.33) ;
\draw  [color={rgb, 255:red, 245; green, 166; blue, 35 }  ,draw opacity=1 ][fill={rgb, 255:red, 245; green, 166; blue, 35 }  ,fill opacity=0.34 ] (312,103.33) -- (341.67,103.33) -- (311,153.88) -- (291,123.55) -- cycle ;
\draw [color={rgb, 255:red, 245; green, 166; blue, 35 }  ,draw opacity=1 ]   (322.33,123.17) -- (311,153.88) ;
\draw [color={rgb, 255:red, 245; green, 166; blue, 35 }  ,draw opacity=1 ] [dash pattern={on 4.5pt off 4.5pt}]  (312,103.33) -- (311,153.88) ;
\draw [color={rgb, 255:red, 245; green, 166; blue, 35 }  ,draw opacity=1 ]   (341.67,103.33) -- (322.33,123.17) ;
\draw [color={rgb, 255:red, 245; green, 166; blue, 35 }  ,draw opacity=1 ]   (322.33,123.17) -- (291,123.55) ;
\draw [color={rgb, 255:red, 155; green, 155; blue, 155 }  ,draw opacity=1 ]   (311,165.21) -- (311,153.88) ;
\draw [color={rgb, 255:red, 155; green, 155; blue, 155 }  ,draw opacity=1 ]   (281.67,133.17) -- (291,123.55) ;
\draw [color={rgb, 255:red, 155; green, 155; blue, 155 }  ,draw opacity=1 ]   (341.67,103.33) -- (362,103.33) ;
\draw  [color={rgb, 255:red, 189; green, 16; blue, 224 }  ,draw opacity=1 ][fill={rgb, 255:red, 189; green, 16; blue, 224 }  ,fill opacity=1 ] (311.08,103.33) .. controls (311.08,102.83) and (311.49,102.42) .. (312,102.42) .. controls (312.51,102.42) and (312.92,102.83) .. (312.92,103.33) .. controls (312.92,103.84) and (312.51,104.25) .. (312,104.25) .. controls (311.49,104.25) and (311.08,103.84) .. (311.08,103.33) -- cycle ;
\draw  [color={rgb, 255:red, 189; green, 16; blue, 224 }  ,draw opacity=1 ][fill={rgb, 255:red, 189; green, 16; blue, 224 }  ,fill opacity=1 ] (290.08,123.55) .. controls (290.08,123.04) and (290.49,122.63) .. (291,122.63) .. controls (291.51,122.63) and (291.92,123.04) .. (291.92,123.55) .. controls (291.92,124.05) and (291.51,124.46) .. (291,124.46) .. controls (290.49,124.46) and (290.08,124.05) .. (290.08,123.55) -- cycle ;
\draw  [color={rgb, 255:red, 189; green, 16; blue, 224 }  ,draw opacity=1 ][fill={rgb, 255:red, 189; green, 16; blue, 224 }  ,fill opacity=1 ] (321.42,123.17) .. controls (321.42,122.66) and (321.83,122.25) .. (322.33,122.25) .. controls (322.84,122.25) and (323.25,122.66) .. (323.25,123.17) .. controls (323.25,123.67) and (322.84,124.08) .. (322.33,124.08) .. controls (321.83,124.08) and (321.42,123.67) .. (321.42,123.17) -- cycle ;
\draw  [color={rgb, 255:red, 189; green, 16; blue, 224 }  ,draw opacity=1 ][fill={rgb, 255:red, 189; green, 16; blue, 224 }  ,fill opacity=1 ] (340.75,103.33) .. controls (340.75,102.83) and (341.16,102.42) .. (341.67,102.42) .. controls (342.17,102.42) and (342.58,102.83) .. (342.58,103.33) .. controls (342.58,103.84) and (342.17,104.25) .. (341.67,104.25) .. controls (341.16,104.25) and (340.75,103.84) .. (340.75,103.33) -- cycle ;
\draw  [color={rgb, 255:red, 189; green, 16; blue, 224 }  ,draw opacity=1 ][fill={rgb, 255:red, 189; green, 16; blue, 224 }  ,fill opacity=1 ] (310.08,152.96) .. controls (310.08,152.46) and (310.49,152.05) .. (311,152.05) .. controls (311.51,152.05) and (311.92,152.46) .. (311.92,152.96) .. controls (311.92,153.47) and (311.51,153.88) .. (311,153.88) .. controls (310.49,153.88) and (310.08,153.47) .. (310.08,152.96) -- cycle ;

\draw  [color={rgb, 255:red, 215; green, 137; blue, 0 }  ,draw opacity=1 ][fill={rgb, 255:red, 232; green, 144; blue, 0 }  ,fill opacity=0.63 ] (312,103.33) -- (341.67,103.33) -- (322.33,123.17) -- (291,123.55) -- cycle ;

\draw (344,94.33) node [anchor=north west][inner sep=0.75pt]   [align=left] {{\small $a_{3,2}$}};
\draw (294.67,93.92) node [anchor=north west][inner sep=0.75pt]   [align=left] {{\small $a_{3,1}$}};
\draw (279,115.33) node [anchor=north west][inner sep=0.75pt]   [align=left] {{\small $b_{2}$}};
\draw (330,120) node [anchor=north west][inner sep=0.75pt]   [align=left] {{\small $b_1$}};
\draw (312,152.67) node [anchor=north west][inner sep=0.75pt]   [align=left] {{\small $\mathbf{0}$}};
\end{tikzpicture}
  \caption{Polytope $P_{2,3}$ with the labeling of the vertices.}
  \label{fig:polytope 2 3}
\end{subfigure}
\begin{subfigure}{.5\textwidth}
  \centering

\tikzset{every picture/.style={line width=0.75pt}} 

\begin{tikzpicture}[x=1pt,y=1pt,yscale=-1,xscale=1]

\draw [color={rgb, 255:red, 155; green, 155; blue, 155 }  ,draw opacity=1 ]   (255.67,119.33) -- (305.67,119.33) ;
\draw [color={rgb, 255:red, 155; green, 155; blue, 155 }  ,draw opacity=1 ]   (305.67,180.83) -- (305.67,69.17) ;
\draw [color={rgb, 255:red, 155; green, 155; blue, 155 }  ,draw opacity=1 ]   (305.67,119.33) -- (335.67,89.33) ;
\draw  [color={rgb, 255:red, 245; green, 166; blue, 35 }  ,draw opacity=1 ][fill={rgb, 255:red, 245; green, 166; blue, 35 }  ,fill opacity=0.34 ] (305.67,89.5) -- (335.33,119.33) -- (316,168.83) -- (284.67,139.55) -- cycle ;
\draw [color={rgb, 255:red, 245; green, 166; blue, 35 }  ,draw opacity=1 ]   (316,139.17) -- (316,168.83) ;
\draw [color={rgb, 255:red, 245; green, 166; blue, 35 }  ,draw opacity=1 ] [dash pattern={on 4.5pt off 4.5pt}]  (305.67,119.33) -- (316,168.83) ;
\draw [color={rgb, 255:red, 245; green, 166; blue, 35 }  ,draw opacity=1 ]   (335.33,119.33) -- (316,139.17) ;
\draw [color={rgb, 255:red, 245; green, 166; blue, 35 }  ,draw opacity=1 ]   (316,139.17) -- (284.67,139.55) ;
\draw [color={rgb, 255:red, 155; green, 155; blue, 155 }  ,draw opacity=1 ]   (275.33,149.17) -- (284.67,139.55) ;
\draw [color={rgb, 255:red, 155; green, 155; blue, 155 }  ,draw opacity=1 ]   (335.33,119.33) -- (355.67,119.33) ;
\draw [color={rgb, 255:red, 245; green, 166; blue, 35 }  ,draw opacity=1 ]   (305.67,89.17) -- (316,139.17) ;
\draw [color={rgb, 255:red, 245; green, 166; blue, 35 }  ,draw opacity=1 ] [dash pattern={on 4.5pt off 4.5pt}]  (305.67,119.33) -- (284.67,139.55) ;
\draw [color={rgb, 255:red, 245; green, 166; blue, 35 }  ,draw opacity=1 ] [dash pattern={on 4.5pt off 4.5pt}]  (305.67,119.33) -- (335.33,119.33) ;
\draw  [color={rgb, 255:red, 189; green, 16; blue, 224 }  ,draw opacity=1 ][fill={rgb, 255:red, 189; green, 16; blue, 224 }  ,fill opacity=1 ] (304.75,89.5) .. controls (304.75,88.99) and (305.16,88.58) .. (305.67,88.58) .. controls (306.17,88.58) and (306.58,88.99) .. (306.58,89.5) .. controls (306.58,90.01) and (306.17,90.42) .. (305.67,90.42) .. controls (305.16,90.42) and (304.75,90.01) .. (304.75,89.5) -- cycle ;
\draw  [color={rgb, 255:red, 189; green, 16; blue, 224 }  ,draw opacity=0.54 ][fill={rgb, 255:red, 189; green, 16; blue, 224 }  ,fill opacity=0.49 ] (304.75,119.33) .. controls (304.75,118.83) and (305.16,118.42) .. (305.67,118.42) .. controls (306.17,118.42) and (306.58,118.83) .. (306.58,119.33) .. controls (306.58,119.84) and (306.17,120.25) .. (305.67,120.25) .. controls (305.16,120.25) and (304.75,119.84) .. (304.75,119.33) -- cycle ;
\draw  [color={rgb, 255:red, 189; green, 16; blue, 224 }  ,draw opacity=1 ][fill={rgb, 255:red, 189; green, 16; blue, 224 }  ,fill opacity=1 ] (315.08,168.83) .. controls (315.08,168.33) and (315.49,167.92) .. (316,167.92) .. controls (316.51,167.92) and (316.92,168.33) .. (316.92,168.83) .. controls (316.92,169.34) and (316.51,169.75) .. (316,169.75) .. controls (315.49,169.75) and (315.08,169.34) .. (315.08,168.83) -- cycle ;
\draw  [color={rgb, 255:red, 189; green, 16; blue, 224 }  ,draw opacity=1 ][fill={rgb, 255:red, 189; green, 16; blue, 224 }  ,fill opacity=1 ] (334.42,119.33) .. controls (334.42,118.83) and (334.83,118.42) .. (335.33,118.42) .. controls (335.84,118.42) and (336.25,118.83) .. (336.25,119.33) .. controls (336.25,119.84) and (335.84,120.25) .. (335.33,120.25) .. controls (334.83,120.25) and (334.42,119.84) .. (334.42,119.33) -- cycle ;
\draw  [color={rgb, 255:red, 189; green, 16; blue, 224 }  ,draw opacity=1 ][fill={rgb, 255:red, 189; green, 16; blue, 224 }  ,fill opacity=1 ] (283.75,139.55) .. controls (283.75,139.04) and (284.16,138.63) .. (284.67,138.63) .. controls (285.17,138.63) and (285.58,139.04) .. (285.58,139.55) .. controls (285.58,140.05) and (285.17,140.46) .. (284.67,140.46) .. controls (284.16,140.46) and (283.75,140.05) .. (283.75,139.55) -- cycle ;
\draw  [color={rgb, 255:red, 189; green, 16; blue, 224 }  ,draw opacity=1 ][fill={rgb, 255:red, 189; green, 16; blue, 224 }  ,fill opacity=1 ] (315.08,139.17) .. controls (315.08,138.66) and (315.49,138.25) .. (316,138.25) .. controls (316.51,138.25) and (316.92,138.66) .. (316.92,139.17) .. controls (316.92,139.67) and (316.51,140.08) .. (316,140.08) .. controls (315.49,140.08) and (315.08,139.67) .. (315.08,139.17) -- cycle ;

\draw (336.33,110.33) node [anchor=north west][inner sep=0.75pt]   [align=left] {{\small $a_2$}};
\draw (283,110.33) node [anchor=north west][inner sep=0.75pt]   [align=left] {{\small $a_1$}};
\draw (308.33,83.59) node [anchor=north west][inner sep=0.75pt]   [align=left] {{\small $a_3$}};
\draw (327.0,136.67) node [anchor=north west][inner sep=0.75pt]   [align=left] {{\small $b_1$}};
\draw (273,132) node [anchor=north west][inner sep=0.75pt]   [align=left] {{\small $b_2$}};
\draw (313.33,171.33) node [anchor=north west][inner sep=0.75pt]   [align=left] {{\small $b_3$}};

\end{tikzpicture}

  \caption{Polytope $P_3$ with the labeling of the vertices induced by $P_{3,3}$.}
  \label{fig:polytope 3 4}
\end{subfigure}
\caption{Polytopes $P_{2,3}$ and $P_3$ with the labeling of their vertices.}

\end{figure}

Now, let $i\in[k]$. Without loss of generality, we may assume that $i=k$. Then, the coordinate ring $\mathcal{S}/\mathcal{Q}_k$ is the ring
\[
\frac{
\C[\alpha_{k,1},\ldots,\alpha_{k,k-1},\beta_1,\ldots,\beta_k]}{\langle \alpha_{k,1}\beta_1-\alpha_{k,s}\beta_s: r\in[k-1]\rangle}
\simeq 
\frac{
\C[\alpha_{k,1},\ldots,\alpha_{k,k-1},\beta_1,\ldots,\beta_{k-1}]}{\langle \alpha_{k,1}\beta_1-\alpha_{k,s}\beta_s: r\in[k-1]\rangle}
\otimes \C[\beta_k].
\]
Therefore, the variety defined by $\mathcal{Q}_i$ is isomorphic to $\mathbb{V}(I_{k-1})\times\A^{1}_\C$ where $I_{k-1}$ is defined in \eqref{eq:toric ideal}. Note that $I_a$ is defined in \eqref{eq:toric ideal} for $a\geq 2$. We set $I_1=\langle 0\rangle$. In particular, the polytope $P_1$ associated to $I_1$ is the one dimensional simplex.
As before, let $P_{k,k}$ and $\widehat{P}_{k,k}$ be the polytopes associated to the affine and projective varieties defined by $\mathcal{Q}_i$ respectively. By construction, $P_{k,k}$ is the cone of the polytope $P_{k-1}$.
Explicitly, we embed $P_{k-1}\subset\R^{k-1}$ in $\R^{k}$. Then, $P_{k,k}$ is the convex hull of $P_{k-1}$ and $\mathbf{e}_k$.
Then, $\widehat{P}_{k,k}$ is the cone of $P_{k,k}$.
 As before, we slightly change the labeling of the vertices of $P_{k,k}$ and $\widehat{P}_{k,k}$ for $i\in[k]$. We label the vertices corresponding to $\mathbf{0}$, $\mathbf{e}_1$, $\mathbf{e}_2$ and $\mathbf{e}_1+\mathbf{e}_2$ by $a_{k,1}$, $b_2$, $a_{k,2}$, and $b_1$ respectively. Similarly, the vertices $\mathbf{e}_j$ and $\mathbf{e}_1+\mathbf{e}_2-\mathbf{e}_j$ for $3\leq j\leq k-1$ are labeled by $a_{k,j}$ and $b_j$ respectively. The extra vertex $\mathbf{e}_k$ is labeled by $b_k$, and the vertex of the cone is denoted by $\mathbf{0}$. In \cref{fig:poyltope 3 3}, the convex hull of $P_2$ and $\mathbf{e}_3$ is depicted with the labeling induced by $\mathcal{Q}_3$. The polytope $P_{3,3}$ is the cone over the polytope in \cref{fig:poyltope 3 3}.

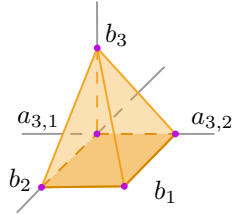
\begin{figure}[h]
    \centering

\tikzset{every picture/.style={line width=0.75pt}} 

\begin{tikzpicture}[x=1pt,y=1pt,yscale=-1,xscale=1]

\draw [color={rgb, 255:red, 155; green, 155; blue, 155 }  ,draw opacity=1 ] [dash pattern={on 4.5pt off 4.5pt}]  (301.92,103.33) -- (312,103.33) ;
\draw [color={rgb, 255:red, 155; green, 155; blue, 155 }  ,draw opacity=1 ]   (312.09,70.81) -- (312,53.17) ;
\draw [color={rgb, 255:red, 155; green, 155; blue, 155 }  ,draw opacity=1 ] [dash pattern={on 4.5pt off 4.5pt}]  (312,103.33) -- (327.5,87.81) ;
\draw [color={rgb, 255:red, 245; green, 166; blue, 35 }  ,draw opacity=1 ] [dash pattern={on 4.5pt off 4.5pt}]  (312.09,70.82) -- (312,103.33) ;
\draw [color={rgb, 255:red, 245; green, 166; blue, 35 }  ,draw opacity=1 ]   (312.09,70.82) -- (322.33,123.17) ;
\draw [color={rgb, 255:red, 155; green, 155; blue, 155 }  ,draw opacity=1 ]   (281.67,133.17) -- (291,123.55) ;
\draw [color={rgb, 255:red, 155; green, 155; blue, 155 }  ,draw opacity=1 ]   (341.67,103.33) -- (356.45,103.33) ;
\draw  [draw opacity=0][fill={rgb, 255:red, 232; green, 144; blue, 0 }  ,fill opacity=0.33 ] (312,103.33) -- (341.67,103.33) -- (322.33,123.17) -- (291,123.55) -- cycle ;
\draw [color={rgb, 255:red, 212; green, 133; blue, 0 }  ,draw opacity=1 ] [dash pattern={on 4.5pt off 4.5pt}]  (312,103.33) -- (341.67,103.33) ;
\draw [color={rgb, 255:red, 212; green, 133; blue, 0 }  ,draw opacity=1 ] [dash pattern={on 4.5pt off 4.5pt}]  (291,123.55) -- (312,103.33) ;
\draw [color={rgb, 255:red, 155; green, 155; blue, 155 }  ,draw opacity=1 ]   (327.5,87.81) -- (337.55,77.91) ;
\draw [color={rgb, 255:red, 155; green, 155; blue, 155 }  ,draw opacity=1 ]   (283.55,103.33) -- (299.25,103.31) ;
\draw  [color={rgb, 255:red, 245; green, 166; blue, 35 }  ,draw opacity=1 ][fill={rgb, 255:red, 245; green, 166; blue, 35 }  ,fill opacity=0.34 ] (322.33,123.17) -- (341.67,103.33) -- (312.09,70.82) -- (291,123.55) -- cycle ;
\draw [color={rgb, 255:red, 212; green, 133; blue, 0 }  ,draw opacity=1 ]   (291,123.55) -- (322.33,123.17) ;
\draw [color={rgb, 255:red, 212; green, 133; blue, 0 }  ,draw opacity=1 ]   (322.33,123.17) -- (341.67,103.33) ;
\draw  [color={rgb, 255:red, 189; green, 16; blue, 224 }  ,draw opacity=1 ][fill={rgb, 255:red, 189; green, 16; blue, 224 }  ,fill opacity=1 ] (311.18,70.82) .. controls (311.18,70.31) and (311.59,69.9) .. (312.09,69.9) .. controls (312.6,69.9) and (313.01,70.31) .. (313.01,70.82) .. controls (313.01,71.32) and (312.6,71.73) .. (312.09,71.73) .. controls (311.59,71.73) and (311.18,71.32) .. (311.18,70.82) -- cycle ;
\draw  [color={rgb, 255:red, 189; green, 16; blue, 224 }  ,draw opacity=1 ][fill={rgb, 255:red, 189; green, 16; blue, 224 }  ,fill opacity=1 ] (290.08,123.55) .. controls (290.08,123.04) and (290.49,122.63) .. (291,122.63) .. controls (291.51,122.63) and (291.92,123.04) .. (291.92,123.55) .. controls (291.92,124.05) and (291.51,124.46) .. (291,124.46) .. controls (290.49,124.46) and (290.08,124.05) .. (290.08,123.55) -- cycle ;
\draw  [color={rgb, 255:red, 189; green, 16; blue, 224 }  ,draw opacity=1 ][fill={rgb, 255:red, 189; green, 16; blue, 224 }  ,fill opacity=1 ] (321.42,123.17) .. controls (321.42,122.66) and (321.83,122.25) .. (322.33,122.25) .. controls (322.84,122.25) and (323.25,122.66) .. (323.25,123.17) .. controls (323.25,123.67) and (322.84,124.08) .. (322.33,124.08) .. controls (321.83,124.08) and (321.42,123.67) .. (321.42,123.17) -- cycle ;
\draw  [color={rgb, 255:red, 189; green, 16; blue, 224 }  ,draw opacity=1 ][fill={rgb, 255:red, 189; green, 16; blue, 224 }  ,fill opacity=1 ] (340.75,103.33) .. controls (340.75,102.83) and (341.16,102.42) .. (341.67,102.42) .. controls (342.17,102.42) and (342.58,102.83) .. (342.58,103.33) .. controls (342.58,103.84) and (342.17,104.25) .. (341.67,104.25) .. controls (341.16,104.25) and (340.75,103.84) .. (340.75,103.33) -- cycle ;
\draw  [color={rgb, 255:red, 189; green, 16; blue, 224 }  ,draw opacity=1 ][fill={rgb, 255:red, 189; green, 16; blue, 224 }  ,fill opacity=1 ] (311.08,103.33) .. controls (311.08,102.83) and (311.49,102.42) .. (312,102.42) .. controls (312.51,102.42) and (312.92,102.83) .. (312.92,103.33) .. controls (312.92,103.84) and (312.51,104.25) .. (312,104.25) .. controls (311.49,104.25) and (311.08,103.84) .. (311.08,103.33) -- cycle ;

\draw (344,94.33) node [anchor=north west][inner sep=0.75pt]   [align=left] {{ $a_{3,2}$}};
\draw (280.67,93.92) node [anchor=north west][inner sep=0.75pt]    [align=left] {{$a_{3,1}$}};
\draw (275,115.33) node [anchor=north west][inner sep=0.75pt]   [align=left] {{ $b_{2}$}};
\draw (314,60.33) node [anchor=north west][inner sep=0.75pt]   [align=left] {{$b_{3}$}};
\draw (330,120) node [anchor=north west][inner sep=0.75pt]   [align=left] {{ $b_1$}};

\end{tikzpicture}

    \caption{The convex hull of $P_2$ and $\mathbf{e}_3$ with the labeling induced by $P_{3,3}$.}
    \label{fig:poyltope 3 3}
\end{figure}

Now, we build the polyhedral complexes $\mathbb{S}_k$ and $\widehat{\mathbb{S}}_k$ by adding to the complex  $\mathbb{S}_k'$ and $\widehat{\mathbb{S}}_k'$ the polytopes $P_{i,k}$ and $\widehat{P}_{i,k}$ respectively.
By \cref{prop: toric variety}, for $i\not\in [k]$, the facets of $P_{i,k}$ are simplices whose vertices are 
\begin{equation}
    \label{eq:face polytope labeled}
    \{a_{i,j}: j\in S \}\cup\{b_j:j\in [k]\setminus S\}
\end{equation}
for $S\subseteq\{1,\ldots,k\}$. We denote such a facet by $F_S$.
We can associate to the facet $F_S$ of $P_{i,k}$ a face of the simplicial complex $\mathbb{S}_k'$ as follows. We consider the face of $\Delta_{S}$ spanned by the vectors
$\mathbf{b}_j$ for $j\in [k]\setminus S$ and $\mathbf{a}_{i,j}$ for $j\in S$. Such face is isomorphic to $F_S$ by identifying the vertices $b_j$ and $a_{i,j}$ of $F_S$ with the vectors $\mathbf{b}_j$ and $\mathbf{a}_{i,j}$ respectively. Similarly, for $i\in[k]$ we can identify each each facet of $P_{i,k}$ with an isomorphic face of the complex $\mathbb{S}_k'$ by identifying the vertices $a_{i,j}$ and $b_j$ with the vectors $\mathbf{a}_{i,j}$ and $\mathbf{b}_j$ respectively. 
The polyhedral complex $\mathbb{S}_k$ is the complex obtained by adding to the simplicial complex $\mathbb{S}_k'$ the polytopes $P_{i,k}$ through the above identification. Finally, the polyhedral complex $\widehat{\mathbb{S}}_k$ is the complex obtained by taking the cone over the complex $\mathbb{S}_k$. The complex $\widehat{\mathbb{S}}_k$ can also be constructed by adding, similarly to the construction of $\mathbb{S}_k$, the polytopes $\widehat{P}_{i,k}$ to the complex $\widehat{\mathbb{S}}_k'$. 

\begin{Prop}\label{prop:complex singularity}
    Let  $I,J$ be two ideals among the ideals in the primary decomposition of $\mathcal{J}_k$, and let $ Q_I$ and $Q_J$ be their corresponding polytopes in the complex $\widehat{\mathbb{S}}_k$. Then, 
    the intersection of $\mathbb{V}(I)$ and $\mathbb{V}(J)$ is the closure of a toric orbit that corresponds to the intersection of $Q_I$ and $Q_J$ in $\widehat{\mathbb{S}}_k$.
\end{Prop}
\begin{proof}
    If $I$ and $J$ are both of the form $\mathcal{J}_S$ for $S\subseteq[k]$, then the proof follows from \cref{lem:simplicial singularity}. Assume now that $I=\mathcal{Q}_i$ and $J=\mathcal{Q}_j$ for $i,j\in[n]$. Then 
    \[\mathcal{Q}_i+\mathcal{Q}_j=\langle \alpha_{i,j}:i\in[n],j\in[k]\setminus \{i\}\rangle.\]
    Therefore, the coordinate ring of $\mathbb{V}(\mathcal{Q}_i)\cap \mathbb{V}(\mathcal{Q}_j)$ is $\C[\beta_1,\ldots,\beta_k]$. The face in $\widehat{P}_{i,k}$ corresponding to this intersection is the convex hull of $b_1,\dots,b_k$, which is exactly the intersection between $\widehat{P}_{i,k}$ and $\widehat{P}_{j,k}$ in $\widehat{\mathbb{S}}_k$.

    Assume now that $I=\mathcal{Q}_i$ and $J=\mathcal{J}_S$. We distinguish two cases. Assume first that $i\in S$, then $\mathcal{Q}_i+\mathcal{J}_S =\langle \beta_s:s\in[k]\setminus S, \,\alpha_{r,s}:r\in[n],s\in[k]\setminus\{r\}\rangle$
    and the coordinate ring of the intersection is $\C[\beta_r:r\in[k]\setminus S]$. Therefore, the face of $\Delta_S$ corresponding to this intersection is the convex hull of the vertices $\mathbf{b}_r$ for $r\in[k]\setminus S$. Such a face in $P_{i,k}$ is the one given by the vertices $b_r$ for $r\in[k]\setminus S$. This face coincides with the intersection of $\Delta_S$ and $P_{i,k}$ in $\mathbb{S}_k$. 

    Assume now that $i\not\in S$. Then, the coordinate ring of the intersection
    $\mathbb{V}(\mathcal{Q}_i+\mathcal{J}_S) $ is $\C[\beta_i:i\in[k]\setminus S,\alpha_{i,s}:s\in S$.
The face of $\Delta_S$ corresponding to this intersection is the convex hull of the vertices $\mathbf{b}_r$ for $r\in[k]\setminus S$ and $\mathbf{a}_{i,s}$ for $s\in S$. Such a face in $P_{i,k}$ is the one given by the vertices $b_r$ for $r\in[k]\setminus S$ and $a_{i,s}$ for $s\in S$. As before, this face coincides with the intersection of $\Delta_S$ and $P_{i,k}$ in $\mathbb{S}_k$. 
\end{proof}

From \cref{prop:complex singularity}, we deduce that singularity type of $\sHilb^m(X_n)$ at $[J]$ is described via the complex $\widehat{\mathbb{S}}_k$ locally around the vertex of the cone.

\begin{Ex}\label{ex: complex S 2}
Fix $n=3$ and $k=2$. The complex $\widehat{\mathbb{S}}_2$ describes the singularity type of $\sHilb^m(X_3)$ at a point of the form $[\langle x_1^{m-i+1},x_2^{i},x_3\rangle]$ for $1\leq i\leq m$.
The ring $\mathcal{S}_k'$ is $\C[\beta_1,\beta_2,\alpha_{1,2},\alpha_{2,1},\alpha_{3,1},\alpha_{3,2}]$. The primary decomposition of $\mathcal{J}_2$ is given by the ideals
\[
\begin{array}{c}
\mathcal{Q}_1=\langle\alpha_{2,1},\alpha_{3,1},\alpha_{3,2}\rangle,\,\,
\mathcal{Q}_2=\langle\alpha_{1,2},\alpha_{3,1},\alpha_{3,2}\rangle,\,\,
\mathcal{Q}_3=\langle\alpha_{1,2},\alpha_{2,1},\alpha_{3,1}\beta_1-\alpha_{3,2}\beta_2\rangle,\\
\mathcal{J}_{\{1\}}=\langle\beta_1,\alpha_{1,2},\alpha_{3,2}\rangle,\,\,\text{ and }\,\,
\mathcal{J}_{\{2\}}=\langle\beta_2,\alpha_{2,1},\alpha_{3,1}\rangle,\\
\end{array}
\]
To each of these ideals, we associate two polytopes. To $\mathcal{Q}_i$, we associate the polytopes $P_{i,2}$ and $\widehat{P}_{i,2}$. The polytopes $P_{1,2}$ and $P_{2,2}$ are the $2$--dimensional simplices with set of vertices $\{b_1,b_2,a_{1,2}\}$ and $\{b_1,b_2,a_{2,1}\}$ respectively. These polytopes are illustrated in purple in \cref{fig: complex S 2}. On the other hand, $P_{3,2}$ is an square and its set of vertices is $\{b_1,b_2,a_{3,1},a_{3,2}\}$. In \cref{fig: complex S 2}, $P_{3,2}$ is depicted in orange. The polytope $\widehat{P}_{i,2}$ is the cone over $P_{i,k}$.
Similarly,  $\mathcal{J}_{\{i\}}$ we associate the polytopes $\Delta_{\{i\}}$ and $\widehat{\Delta}_{\{i\}}$.  The polytope $\Delta_{\{1\}}$ is the $2$ dimensional simplex spanned by the vertices $\mathbf{b}_2,\mathbf{a}_{3,1}$ and $\mathbf{a}_{2,1}$. Analogously, $\Delta_{\{2\}}$ is the $2$ dimensional simplex spanned by the vertices $\mathbf{b}_1,\mathbf{a}_{3,2}$ and $\mathbf{a}_{1,2}$. Both $\Delta_{\{1\}}$ and $\Delta_{\{2\}}$ are illustrated in blue in \cref{fig: complex S 2}.
The complex $\mathbb{S}_2$, which is illustrated in   \cref{fig: complex S 2}, is obtained by gluing the polytopes $P_{1,2},P_{2,2},P_{3,2},\Delta_{\{1\}}$ and $\Delta_{\{2\}}$ through the faces spanned by vertices with the same labeling. The complex $\widehat{\mathbb{S}}_2$ is the cone over the complex $\mathbb{S}_2$.

    \begin{figure}
        \centering

\tikzset{every picture/.style={line width=0.75pt}} 

\begin{tikzpicture}[x=0.75pt,y=0.75pt,yscale=-1,xscale=1]

\draw  [color={rgb, 255:red, 144; green, 19; blue, 254 }  ,draw opacity=1 ][fill={rgb, 255:red, 144; green, 19; blue, 254 }  ,fill opacity=0.2 ] (279.58,71.92) -- (395.67,105.09) -- (325.67,105.09) -- cycle ;
\draw  [color={rgb, 255:red, 74; green, 144; blue, 226 }  ,draw opacity=1 ][fill={rgb, 255:red, 74; green, 144; blue, 226 }  ,fill opacity=0.28 ] (209.18,71.92) -- (279.58,71.92) -- (279.58,149.33) -- cycle ;
\draw  [color={rgb, 255:red, 245; green, 166; blue, 35 }  ,draw opacity=1 ][fill={rgb, 255:red, 245; green, 166; blue, 35 }  ,fill opacity=0.39 ] (325.67,105.09) -- (325.67,182.5) -- (279.58,149.33) -- (279.58,71.92) -- cycle ;
\draw  [color={rgb, 255:red, 74; green, 144; blue, 226 }  ,draw opacity=1 ][fill={rgb, 255:red, 74; green, 144; blue, 226 }  ,fill opacity=0.28 ] (325.67,105.09) -- (395.67,105.09) -- (325.67,182.5) -- cycle ;
\draw  [color={rgb, 255:red, 189; green, 16; blue, 224 }  ,draw opacity=1 ][fill={rgb, 255:red, 189; green, 16; blue, 224 }  ,fill opacity=0.25 ] (209.18,71.92) -- (325.67,105.09) -- (279.58,71.92) -- cycle ;
\draw [color={rgb, 255:red, 74; green, 144; blue, 226 }  ,draw opacity=1 ][line width=1.9]    (209.18,71.92) -- (279.58,149.33) ;
\draw [color={rgb, 255:red, 74; green, 144; blue, 226 }  ,draw opacity=1 ][line width=1.9]    (395.67,105.09) -- (325.67,182.5) ;
\draw [color={rgb, 255:red, 245; green, 166; blue, 35 }  ,draw opacity=1 ][line width=1.9]    (279.58,149.33) -- (325.67,182.5) ;
\draw  [color={rgb, 255:red, 220; green, 84; blue, 103 }  ,draw opacity=1 ][fill={rgb, 255:red, 227; green, 90; blue, 135 }  ,fill opacity=1 ] (277.75,149.33) .. controls (277.75,148.31) and (278.57,147.49) .. (279.58,147.49) .. controls (280.6,147.49) and (281.42,148.31) .. (281.42,149.33) .. controls (281.42,150.34) and (280.6,151.16) .. (279.58,151.16) .. controls (278.57,151.16) and (277.75,150.34) .. (277.75,149.33) -- cycle ;
\draw  [color={rgb, 255:red, 220; green, 84; blue, 103 }  ,draw opacity=1 ][fill={rgb, 255:red, 227; green, 90; blue, 135 }  ,fill opacity=1 ] (323.83,182.5) .. controls (323.83,181.49) and (324.65,180.67) .. (325.67,180.67) .. controls (326.68,180.67) and (327.5,181.49) .. (327.5,182.5) .. controls (327.5,183.51) and (326.68,184.33) .. (325.67,184.33) .. controls (324.65,184.33) and (323.83,183.51) .. (323.83,182.5) -- cycle ;
\draw  [color={rgb, 255:red, 220; green, 84; blue, 103 }  ,draw opacity=1 ][fill={rgb, 255:red, 227; green, 90; blue, 135 }  ,fill opacity=1 ] (323.83,105.09) .. controls (323.83,104.08) and (324.65,103.26) .. (325.67,103.26) .. controls (326.68,103.26) and (327.5,104.08) .. (327.5,105.09) .. controls (327.5,106.1) and (326.68,106.93) .. (325.67,106.93) .. controls (324.65,106.93) and (323.83,106.1) .. (323.83,105.09) -- cycle ;
\draw  [color={rgb, 255:red, 220; green, 84; blue, 103 }  ,draw opacity=1 ][fill={rgb, 255:red, 227; green, 90; blue, 135 }  ,fill opacity=1 ] (393.83,105.09) .. controls (393.83,104.08) and (394.65,103.26) .. (395.67,103.26) .. controls (396.68,103.26) and (397.5,104.08) .. (397.5,105.09) .. controls (397.5,106.1) and (396.68,106.93) .. (395.67,106.93) .. controls (394.65,106.93) and (393.83,106.1) .. (393.83,105.09) -- cycle ;
\draw  [color={rgb, 255:red, 220; green, 84; blue, 103 }  ,draw opacity=1 ][fill={rgb, 255:red, 227; green, 90; blue, 135 }  ,fill opacity=1 ] (277.75,71.92) .. controls (277.75,70.9) and (278.57,70.08) .. (279.58,70.08) .. controls (280.6,70.08) and (281.42,70.9) .. (281.42,71.92) .. controls (281.42,72.93) and (280.6,73.75) .. (279.58,73.75) .. controls (278.57,73.75) and (277.75,72.93) .. (277.75,71.92) -- cycle ;
\draw  [color={rgb, 255:red, 220; green, 84; blue, 103 }  ,draw opacity=1 ][fill={rgb, 255:red, 227; green, 90; blue, 135 }  ,fill opacity=1 ] (207.35,71.92) .. controls (207.35,70.9) and (208.17,70.08) .. (209.18,70.08) .. controls (210.2,70.08) and (211.02,70.9) .. (211.02,71.92) .. controls (211.02,72.93) and (210.2,73.75) .. (209.18,73.75) .. controls (208.17,73.75) and (207.35,72.93) .. (207.35,71.92) -- cycle ;

\draw (315.17,186.29) node [anchor=north west][inner sep=0.75pt]   [align=left] {{\footnotesize $a_{31}$}};
\draw (260.45,152.19) node [anchor=north west][inner sep=0.75pt]   [align=left] {{\footnotesize $a_{32}$}};
\draw (327.74,106.62) node [anchor=north west][inner sep=0.75pt]   [align=left] {{\footnotesize $b_{2}$}};
\draw (273.17,57.32) node [anchor=north west][inner sep=0.75pt]   [align=left] {{\footnotesize $b_{1}$}};
\draw (401.74,100.48) node [anchor=north west][inner sep=0.75pt]   [align=left] {{\footnotesize $a_{21}$}};
\draw (186.86,66.19) node [anchor=north west][inner sep=0.75pt]   [align=left] {{\footnotesize $a_{12}$}};

\end{tikzpicture}

        \caption{Complex $\mathbf{S}_2$ for $n=3$. The simplices in blue correspond to $\Delta_{\{1\}}$ and $\Delta_{\{2\}}$. The simplices $P_{1,2}$ and $P_{2,2}$ are illustrated in purple. The orange square corresponds to $P_{3,2}$. }
        \label{fig: complex S 2}
    \end{figure}
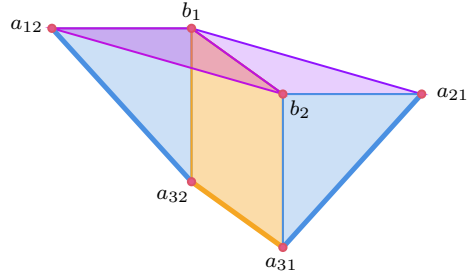

\end{Ex}

\begin{Rem}
    Since the punctual Hilbert scheme $\sHilb_\mathbf{0}^m(X_n)$ is invariant by the torus action, we can also illustrate it locally around $[J]$  through the complexes $\mathbf{S}_k$ and $\widehat{\mathbb{S}}_k$. The irreducible components of $\sHilb_\mathbf{0}^m(X_n)$ that contain $[J]$ correspond to faces of the complex. \cref{prop:local punctual}, allows us to carry out such identification. Given a maximal face $\Delta_S$ of $\mathbb{S}_k$ (see \eqref{eq: simplex S} for the definition of $\Delta_S$), the face $\mathrm{Conv}(\mathbf{a}_{i,j}:i\in[n]\setminus S,j\in S )$ corresponds to the hypersimplex $\Delta_{|S|,n}+\uu-\mathbf{e}_S-\mathbf{1}$ and the Grassmannian $\Sigma(m,n-|S|,\uu-\mathbf{e}_S)$. Similarly, given a maximal face of the form $P_{i,k}$, then the face spanned by the vertices $a_{i,j}$ for $j\in[k]\setminus\{i\}$ corresponds to the hypersimplex $\Delta_{n-1,n}+\uu+\mathbf{e}_i-\mathbf{1}$ and the Grassmannian $\Sigma(m,1,\uu+\mathbf{e}_i-\mathbf{1})$. In particular, the part of $\mathbb{S}_k$ associated to the punctual Hilbert scheme is the subcomplex formed by all the faces spanned by vertices of the form $a_{i,j}$. Analogously, for $\widehat{\mathbb{S}}_k$, these faces are those spanned by the vertex of the cone and the vertices of the form $a_{i,j}$. For instance, following \cref{ex: complex S 2}, for $n=3$ and $k=2$, the faces of $\mathbb{S}_2$ corresponding to the punctual Hilbert scheme are the edges $\overline{a_{1,2}a_{3,2}}$, $\overline{a_{3,2}a_{3,1}}$ and $\overline{a_{2,1}a_{3,1}}$. These edges are represented in \cref{fig: complex S 2} with a thick line. The cone over these three edges is exactly the local picture of $\mathcal{K}^{[m]}_2$ around the vertex corresponding to $[J]$. Such a vertex is an interior point on an edge of $(m-1)\cdot\Delta_2$.
\end{Rem}

Note that the complex $\mathbb{S}_k$ and the ring $\mathcal{S}'_k/\mathcal{J}_k$ depend only on $n$ and $k$. Therefore, the singularity types that appear in $\sHilb^m(X_n)$ depend only on $k$. 
We conclude this section with the following result.

\begin{Cor}
    For $m\geq 2$, any singularity type appearing in $\sHilb^m(X_n)$ also appears in $\sHilb^{n+1}(X_n)$
\end{Cor}
\begin{proof}
Let $[J]\in \sHilb^m(X_n)$ be a singular point. By  \cref{lemma:one parameter family}, there exists $[J_0]\in\sHilb^m_\mathbf{0}(X_n)$ lying in the closure of the $(\C^{*})^n$--orbit of $[J]$ such that $\mu([J_0])$ is a vertex of $\mathcal{K}^{[m]}_n$. 
Let $k$ be the dimension of the face of $(m-1)\cdot \Delta_{n}$ where $\mu([J_0])$ lies. Then, we can see $[J]$ as a point in the complex $\widehat{\mathbb{S}}_k$ and the singularity type of $[J]$ only depends on the relative interior cell of $\widehat{\mathbb{S}}_k$ where it lies. Since for fixed $n$, the complex $\widehat{\mathbb{S}}_k$ depends only on $k$, the possible singularity types of $\sHilb^m(X_n)$ depend only on the possible values of $k$. 
So, it is enough to check that for every $k\in[n]$, there exists $[J_0]\in\sHilb^{n+1}_{\mathbf{0}}(X_n)$ such that 
$\mu([J_0])$ is a vertex of $\mathcal{K}^{[n+1]}_n$ lying in the interior of a $k$--dimensional face of $n\cdot\Delta_{n-1}$. Note that the vertices of $\mathcal{K}^{[n+1]}_n$ coincide with the integer points of $n\cdot\Delta_{n-1}$.
Then, the proof follows from the fact that the interior of every face of $n\cdot \Delta_{n-1}$ has an integer point. Indeed, the point $(n-k+1)\mathbf{e}_1+\mathbf{e}_2+\cdots+\mathbf{e}_k$ lies in the interior of a $k$--dimensional face of  $n\cdot \Delta_{n-1}$. Note that $n\cdot \Delta_{n-1}$ is the first simplex of the form $m\cdot \Delta_{n-1}$ with such a property.
\end{proof}

\section{Smoothable and non-smoothable components}\label{sec:sm and non sm}
With the results of the previous sections at hand, we proceed to analyze the irreducible components of the Hilbert scheme, which display strikingly different behaviors. Interestingly, the smoothable component turns out to be the most singular, whereas the non-smoothable ones exhibit remarkably well-behaved geometry.
\subsection{Non-smoothable components}\label{sec: nonsm}

 Let $C$ be a genus $g$ irreducible curve whose unique singularity $p\in C$ is a rational $n$--fold singularity. In this Subsection, we give a detailed description of the non-smoothable components of $\sHilb^m(C)$ and their normalization.
For $2\leq m'\leq \min\{n-1,m\}$, we consider the irreducible component
$\sHilb^{m,m'}(C)$
of $\sHilb^m(C)$. By \cref{co: irred comp structure}, we may see $\Sym^{m-m'}\left(C\setminus\{p\}\right)\times \Sigma(m',n+1-m' , \mathbf{1})$ as an open subset of this component. We now give a stratification of such a component where this open subset is the biggest strata.

 Let $\nu:\widetilde{C}\rightarrow C$ be the normalization of $C$, and let $p_1,\ldots,p_n$ be the $n$ preimages of the singularity $p$. In particular, $\nu$ gives an isomorphism between $\widetilde{C}\setminus\{p_1,\ldots,p_n\}$ and $C\setminus \{p\}$. We consider the following stratification of the symmetric product $\Sym^{m}(\widetilde{C})$. For $0\leq u\leq m$ and for a partition $\uu\in\Z_{\geq 0}^n$ with $|\uu|=u$, we consider the locus 
 \[
 \Sym^{m}(\widetilde{C})_\uu=\left\{
\sum_{i=1}^nu_ip_i+q: \text{ for } q\in \Sym^{m-u}\left(\widetilde{C}\setminus\{p_1,\ldots,p_n\}\right)
 \right\}\simeq \Sym^{m-u}\left(\widetilde{C}\setminus\{p_1,\ldots,p_n\}\right).
 \]
 Then, the locally closed subvarieties $ \Sym^{m}(\widetilde{C})_\uu$ form a stratification of $\Sym^{m}(\widetilde{C})$.
 For fixed $0\leq u\leq m$, we also consider the variety
 \[
\displaystyle\Sym^{m}(\widetilde{C})_u=\bigcup_{\uu\in\Z_{\geq 0}^n,\, |\uu|=u} \Sym^{m}(\widetilde{C})_\uu.
 \]
Note that this union is actually a disjoint union.
We can use this stratification of $\widetilde{C}$ to describe the irreducible component $\sHilb^{m,m'}(C)$.  For $\uu\in\Z_{\geq 0}^n$ with $0\leq|\uu|\leq m-m'$, we consider the map 
\[
\begin{array}{cccc}
\psi_{m',\uu}:&\Sym^{m-m'}(\widetilde{C})_\uu\times \Sigma(m',n+1-m',\mathbf{1})&\longrightarrow&\sHilb^{m,m'}(C)\\
 & (\sum_{i=1}^nu_ip_i+q,[J])&\longmapsto &\nu(q)\cup \mathbb{V}(\phi_{m',\uu}([J]) ).
\end{array}
\]
Here $\nu$ also denotes the lift of the normalization map $\nu$ to the symmetric product of the curve, and
the map $\phi_{m',\uu}$ is defined in \eqref{eq: new map grass}.
The image of $\psi_{m',\uu}$ is
\[
\sHilb^{m,m',\uu}(C):=\left\{
\left[Z\cup\mathbb{V}(J)\right]: Z\in \Sym^{m-m'-|\uu|}(C\setminus\{\mathbf{0}\}) \text{ and }[J]\in\Sigma(m'+|\uu|,n+1-m',\uu+\mathbf{1})
\right\},
\]
which is isomorphic to $\Sym^{m-m'-|\uu|}(C\setminus\{\mathbf{0}\})\times \Sigma(m'+|\uu|,n+1-m',\uu+\mathbf{1})$. Using that $C\setminus\{0\}\simeq \widetilde{C}\setminus\{p_1,\ldots,p_n\}$, we deduce that $\psi_{m',\uu}$ is an isomorphism onto $\sHilb^{m,m',\uu}(C)$. Note that the varieties $\sHilb^{m,m',\uu}(C)$ do not provide a stratification of $\sHilb^{m,m'}(C)$ since they are not disjoint. Indeed, for $\uu,\mathbf{v}\in\Z^n_{\geq 0}$ with $|\uu|=|\mathbf{v}|$, we have that the intersection $\sHilb^{m,m',\uu}(C)\cap \sHilb^{m,m',\mathbf{v}}(C)$ is the product of $\Sym^{m-m'-|\uu|}(C\setminus\{\mathbf{0}\})$ with 
\[
\Sigma(m'+|\uu|,n+1-m',\uu+\mathbf{1})\cap \Sigma(m'+|\mathbf{v}|,n+1-m',\mathbf{v}+\mathbf{1}).
\]
The above intersection is done in $ \mathcal{G}_{n+1-m',n}^{m'+|\uu|}\subseteq \sHilb_\mathbf{0}^m(C)$ and might not be empty.
To solve this problem, for $0\leq u\leq m-m'$, we consider the map 
\[
\psi_{m',u}:\Sym^{m-m'}(\widetilde{C})_u\times \Sigma(m',n+1-m',\mathbf{1})\longrightarrow\sHilb^{m,m'}(C)
\]
whose restriction to each connected component $\Sym^{m-m'}(\widetilde{C})_\uu$ is $\psi_{m',\uu}$. Note that $\psi_{m',0}$ is the birational morphism between $\Sym^{m-m'}(C)\times\Sigma(m',n+1-m',\mathbf{1})$ and $\sHilb^{m,m'}(C)$. The image of $\psi_{m',u}$ is the locally closed subvariety
\begin{equation}\label{eq:strata component}
\sHilb^{m,m',u}(C)\!:=\!\left\{[Z\cup\mathbb{V}(J)]\!\in\! \sHilb^{m,m'}(X_n): Z\in \Sym^{m-m'-u}(\widetilde{C}\!\setminus\!\{p_1,\ldots,p_n\})\text{ and }[J]\in \mathcal{G}_{n+1-m',n}^{m'+u}\right\}
\end{equation}
Using \cref{theo:reduced punctual}, we may see $\mathcal{G}_{n+1-m',n}^{m'+u}$ as the subvariety of $\sHilb_\mathbf{0}^{m'+u}(C)$ given by the union of the components of the form $\Sigma(m',n+1-m',\uu+\mathbf{1})$ with $|\uu|=u$.
Note that for $u=m-m'$, we get that $\sHilb^{m,m',m-m'}(C)\simeq\mathcal{G}^{m}_{n+1-m',n}$.
From \eqref{eq:strata component}, we get that
\begin{equation}\label{eq: iso strata}
\sHilb^{m,m',u}(C)\simeq \Sym^{m-m'-u}(\widetilde{C}\setminus\{p_1,\ldots,p_n\})\times \mathcal{G}_{n+1-m',n}^{m'+u}.
\end{equation}
Moreover, note that $\sHilb^{m,m',u}(C)$ and $\sHilb^{m,m',v}(C)$ are disjoint for $0\leq u<v\leq m-m'$ and, by \cref{prop:limits}, we deduce that 
\[
\sHilb^{m,m'}(C)=\bigsqcup_{0\leq u\leq m-m'} \sHilb^{m,m',u}(C).
\]
Therefore, the varieties $\sHilb^{m,m',u}(C)$ provide a stratification of $\sHilb^{m,m'}(C)$.

   \begin{figure}

\begin{subfigure}{.45\textwidth}
  \centering

\tikzset{every picture/.style={line width=0.75pt}} 

\begin{tikzpicture}[x=0.5pt,y=0.5pt,yscale=-1,xscale=1]

\draw  [color={rgb, 255:red, 189; green, 16; blue, 224 }  ,draw opacity=1 ][line width=0.75]  (394.8,44) -- (429.42,103.32) -- (360.18,103.32) -- cycle ;
\draw [color={rgb, 255:red, 189; green, 16; blue, 224 }  ,draw opacity=1 ]   (265.63,173.85) -- (288.08,151.28) -- (334.36,104.76) -- (365.48,73.48) -- (394.8,44) ;
\draw [color={rgb, 255:red, 189; green, 16; blue, 224 }  ,draw opacity=1 ]   (300.25,233.17) -- (429.42,103.32) ;
\draw [color={rgb, 255:red, 189; green, 16; blue, 224 }  ,draw opacity=1 ]   (231.01,233.17) -- (251.84,212.14) ;
\draw [color={rgb, 255:red, 189; green, 16; blue, 224 }  ,draw opacity=1 ]   (231.01,233.17) -- (295.38,168.45) -- (343.26,120.32) -- (360.18,103.32) ;
\draw [color={rgb, 255:red, 189; green, 16; blue, 224 }  ,draw opacity=1 ]   (295.38,168.45) -- (299.74,164.07) ;
\draw [color={rgb, 255:red, 189; green, 16; blue, 224 }  ,draw opacity=1 ]   (343.26,120.32) -- (345.52,118.05) ;
\draw  [color={rgb, 255:red, 189; green, 16; blue, 224 }  ,draw opacity=1 ][line width=0.75]  (265.63,173.85) -- (300.25,233.17) -- (231.01,233.17) -- cycle ;
\draw  [color={rgb, 255:red, 189; green, 16; blue, 224 }  ,draw opacity=1 ][fill={rgb, 255:red, 189; green, 16; blue, 224 }  ,fill opacity=0.1 ] (394.8,44) -- (429.42,103.32) -- (300.25,233.17) -- (231.01,233.17) -- (265.63,173.85) -- cycle ;
\draw  [color={rgb, 255:red, 2; green, 23; blue, 232 }  ,draw opacity=1 ][fill={rgb, 255:red, 0; green, 50; blue, 237 }  ,fill opacity=0.4 ][line width=0.75]  (380.14,58.74) -- (414.76,118.05) -- (345.52,118.05) -- cycle ;
\draw  [color={rgb, 255:red, 80; green, 217; blue, 227 }  ,draw opacity=1 ][fill={rgb, 255:red, 80; green, 222; blue, 227 }  ,fill opacity=0.49 ][line width=0.75]  (334.36,104.76) -- (368.98,164.07) -- (299.74,164.07) -- cycle ;
\draw  [color={rgb, 255:red, 74; green, 144; blue, 226 }  ,draw opacity=1 ][fill={rgb, 255:red, 74; green, 144; blue, 226 }  ,fill opacity=0.47 ][line width=0.75]  (286.46,152.82) -- (321.08,212.14) -- (251.84,212.14) -- cycle ;
\draw  [color={rgb, 255:red, 245; green, 166; blue, 35 }  ,draw opacity=1 ][fill={rgb, 255:red, 245; green, 166; blue, 35 }  ,fill opacity=1 ] (284.74,152.82) .. controls (284.74,151.87) and (285.51,151.1) .. (286.46,151.1) .. controls (287.41,151.1) and (288.18,151.87) .. (288.18,152.82) .. controls (288.18,153.77) and (287.41,154.54) .. (286.46,154.54) .. controls (285.51,154.54) and (284.74,153.77) .. (284.74,152.82) -- cycle ;
\draw  [color={rgb, 255:red, 245; green, 166; blue, 35 }  ,draw opacity=1 ][fill={rgb, 255:red, 245; green, 166; blue, 35 }  ,fill opacity=1 ] (344.22,118.05) .. controls (344.22,117.34) and (344.8,116.76) .. (345.52,116.76) .. controls (346.23,116.76) and (346.81,117.34) .. (346.81,118.05) .. controls (346.81,118.77) and (346.23,119.35) .. (345.52,119.35) .. controls (344.8,119.35) and (344.22,118.77) .. (344.22,118.05) -- cycle ;
\draw  [color={rgb, 255:red, 245; green, 35; blue, 64 }  ,draw opacity=1 ][fill={rgb, 255:red, 245; green, 35; blue, 101 }  ,fill opacity=1 ] (319.36,212.14) .. controls (319.36,211.19) and (320.13,210.42) .. (321.08,210.42) .. controls (322.03,210.42) and (322.8,211.19) .. (322.8,212.14) .. controls (322.8,213.09) and (322.03,213.86) .. (321.08,213.86) .. controls (320.13,213.86) and (319.36,213.09) .. (319.36,212.14) -- cycle ;
\draw  [color={rgb, 255:red, 245; green, 35; blue, 64 }  ,draw opacity=1 ][fill={rgb, 255:red, 245; green, 35; blue, 101 }  ,fill opacity=1 ] (298.02,164.07) .. controls (298.02,163.12) and (298.79,162.35) .. (299.74,162.35) .. controls (300.69,162.35) and (301.46,163.12) .. (301.46,164.07) .. controls (301.46,165.02) and (300.69,165.79) .. (299.74,165.79) .. controls (298.79,165.79) and (298.02,165.02) .. (298.02,164.07) -- cycle ;
\draw  [color={rgb, 255:red, 105; green, 195; blue, 5 }  ,draw opacity=1 ][fill={rgb, 255:red, 102; green, 191; blue, 5 }  ,fill opacity=1 ] (332.64,104.76) .. controls (332.64,103.81) and (333.41,103.04) .. (334.36,103.04) .. controls (335.31,103.04) and (336.08,103.81) .. (336.08,104.76) .. controls (336.08,105.71) and (335.31,106.48) .. (334.36,106.48) .. controls (333.41,106.48) and (332.64,105.71) .. (332.64,104.76) -- cycle ;
\draw  [color={rgb, 255:red, 105; green, 195; blue, 5 }  ,draw opacity=1 ][fill={rgb, 255:red, 102; green, 191; blue, 5 }  ,fill opacity=1 ] (413.04,118.05) .. controls (413.04,117.1) and (413.81,116.33) .. (414.76,116.33) .. controls (415.71,116.33) and (416.48,117.1) .. (416.48,118.05) .. controls (416.48,119) and (415.71,119.77) .. (414.76,119.77) .. controls (413.81,119.77) and (413.04,119) .. (413.04,118.05) -- cycle ;

\end{tikzpicture}

  \caption{Strata $\sHilb^{3,2,0}(C)\simeq \P^1\setminus\{p_1,p_2,p_3\}\times \P^2$.}
  \label{fig:strata 3 2 0}
\end{subfigure}
\begin{subfigure}{.45\textwidth}
  \centering

\tikzset{every picture/.style={line width=0.75pt}} 

\begin{tikzpicture}[x=0.65pt,y=0.65pt,yscale=-1,xscale=1]

\draw  [color={rgb, 255:red, 80; green, 217; blue, 227 }  ,draw opacity=1 ][fill={rgb, 255:red, 80; green, 222; blue, 227 }  ,fill opacity=0.49 ][line width=0.75]  (355.56,137.36) -- (390.18,196.67) -- (320.94,196.67) -- cycle ;
\draw  [color={rgb, 255:red, 74; green, 144; blue, 226 }  ,draw opacity=1 ][fill={rgb, 255:red, 74; green, 144; blue, 226 }  ,fill opacity=0.47 ][line width=0.75]  (286.31,137.36) -- (320.94,196.67) -- (251.69,196.67) -- cycle ;
\draw  [color={rgb, 255:red, 2; green, 23; blue, 232 }  ,draw opacity=1 ][fill={rgb, 255:red, 0; green, 50; blue, 237 }  ,fill opacity=0.4 ][line width=0.75]  (320.94,78.04) -- (355.56,137.36) -- (286.31,137.36) -- cycle ;
\draw  [color={rgb, 255:red, 245; green, 166; blue, 35 }  ,draw opacity=1 ][fill={rgb, 255:red, 245; green, 166; blue, 35 }  ,fill opacity=1 ] (284.6,137.36) .. controls (284.6,136.41) and (285.37,135.64) .. (286.31,135.64) .. controls (287.26,135.64) and (288.03,136.41) .. (288.03,137.36) .. controls (288.03,138.31) and (287.26,139.08) .. (286.31,139.08) .. controls (285.37,139.08) and (284.6,138.31) .. (284.6,137.36) -- cycle ;
\draw  [color={rgb, 255:red, 245; green, 35; blue, 64 }  ,draw opacity=1 ][fill={rgb, 255:red, 245; green, 35; blue, 101 }  ,fill opacity=1 ] (319.22,196.67) .. controls (319.22,195.72) and (319.99,194.95) .. (320.94,194.95) .. controls (321.89,194.95) and (322.66,195.72) .. (322.66,196.67) .. controls (322.66,197.62) and (321.89,198.39) .. (320.94,198.39) .. controls (319.99,198.39) and (319.22,197.62) .. (319.22,196.67) -- cycle ;
\draw  [color={rgb, 255:red, 105; green, 195; blue, 5 }  ,draw opacity=1 ][fill={rgb, 255:red, 102; green, 191; blue, 5 }  ,fill opacity=1 ] (353.84,137.36) .. controls (353.84,136.41) and (354.61,135.64) .. (355.56,135.64) .. controls (356.51,135.64) and (357.28,136.41) .. (357.28,137.36) .. controls (357.28,138.31) and (356.51,139.08) .. (355.56,139.08) .. controls (354.61,139.08) and (353.84,138.31) .. (353.84,137.36) -- cycle ;

\end{tikzpicture}

  \caption{Strata $\sHilb^{3,2,1}(C)\simeq \mathcal{G}_{1,3}^{3}$ of $\sHilb^{3,2}(C)$.}
  \label{fig:strata 3 2 1}
\end{subfigure}
\caption{Stratification of $\sHilb^{3,2}(C)$ for a genus $3$ curve with a rational $3$--fold singularity.}
\end{figure}
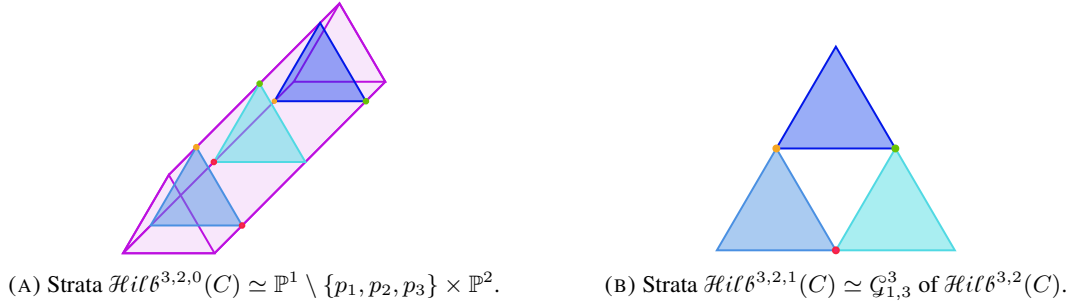

\begin{Ex}\label{ex: strata 3 2}
Let $C$ be an irreducible genus $3$ curve with a rational $3$--fold singularity at $p\in C$. The normalization of $C$ is $\P^1$ and the preimage of $p$ consists of $3$ points $p_1,p_2,p_3\in \P^1$.
   By \cref{co: irred comp structure}, the component $\sHilb^{3,2}(C)$ is birational to $\P^1\times \P^2$. We can stratify $\sHilb^{3,2}(C)$ by $\sHilb^{3,2,0}(C)$ and $\sHilb^{3,2,1}(C)$.  The strata $\sHilb^{3,2,0}(C)$ is the open subset of $\sHilb^{3,2}(C)$ of the form $\left(\P^1\setminus\{p_1,p_2,p_3\}\right)\times \P^2$, which is depicted in \cref{fig:strata 3 2 0}. The strata $\sHilb^{3,2,1}(C)$ is the variety $\mathcal{G}_{2,3}^{3}$, that is obtained by gluing three $\P^2$ by toric invariant points as illustrated in \cref{fig:strata 3 2 1}. 
\end{Ex}

\begin{Ex} 
    Let $C$ be as in \cref{ex: strata 3 2}, then $\sHilb^{4,2}(C)$ is stratified by $\sHilb^{4,2,0}(C)$, $\sHilb^{4,2,1}(C)$ and $\sHilb^{4,2,2}(C)$. The strata $\sHilb^{4,2,0}(C)$ is $\Sym^2(\P^1\setminus\{p_1,p_2,p_3\})\times \P^2 \simeq \P^2\setminus(l_1\cup l_2\cup l_3)\times \P^2$ where $l_i$ corresponds to the line in $\Sym^2(\P^1)\simeq \P^2$ of the form $p_i+q$ for $q\in\P^1$. 
    These lines $l_i$ and $l_j$ intersect in the point $p_i+p_j$. Moreover, the line $l_i$ has an extra marked point $2p_i$. 
    Then, the strata $\sHilb^{4,2,1}(C)$ is 
    \[
    \sHilb^{4,2,1}(C) =\left(\P^1\setminus\{p_1,p_2,p_3\}\right) \times \mathcal{G}_{2,3}^{3}.
    \]
We see that $\sHilb^{4,2,1}(C)$ has three components of the form $\left(\P^1\setminus\{p_1,p_2,p_3\}\right) \times\Sigma(3,2,\mathbf{1}+\mathbf{e}_i)$ for $i\in[3]$ coming from the three components of $\mathcal{G}_{2,3}^{3}$ (see \cref{fig:strata 3 2 1}). Via the map $\psi_{2,1}$, these three components are in correspondence with the three components of 
\[
\Sym^2(\P^1)_1 =  l_1\cup l_2\cup l_3\setminus \{p_i+p_j:1\leq i\leq j\leq 3\}.
\]
     The final strata is $\sHilb^{4,2,1}(C)=\mathcal{G}_{2,3}^4$ which is obtained by gluing $6$ projective planes as illustrated in  \cref{fig:hypercomplex 3 1 4}. Each of these $6$ projective planes corresponds to a hypersimplex in $\mathcal{K}^{[4]}_3$ of the form $\Delta_{1,3}+\mathbf{e}_i+\mathbf{e}_j$. We can associate such hypersimplex to the point $p_i+p_j$ among the $6$ special points in the lines $l_1,l_2,l_3$ via the map $\psi_{2,2}$.
\end{Ex}

The above stratification of $\sHilb^{m,m'}(C)$ allows us to calculate its normalization.

\begin{Thm}\label{thm:normalization}
    The birational map $\psi_{m',0}:\Sym^{m-m'}(\widetilde{C})\times\Sigma(m',n+1-m',\mathbf{1})\dashrightarrow\sHilb^{m,m'}(C)$ extends uniquely to a finite map
    \begin{equation}\label{eq:normalization}
    \psi_{m'}:\Sym^{m-m'}(\widetilde{C})\times \Sigma(m',n+1-m',\mathbf{1})\rightarrow\sHilb^{m,m'}(C)
    \end{equation}
    such that the restriction of $\psi_{m'}$  to $\Sym^{m-m'}(\widetilde{C})_u\times\Sigma(m',n+1-m',\mathbf{1})$ is $\psi_{m',u}$ for $0\leq u\leq m-m'$. In particular, the map \eqref{eq:normalization} is the normalization of $\sHilb^{m,m'}(C)$.
\end{Thm}
\begin{proof}
First, we construct a map $\psi_{m',\mathrm{top}}$
at the level of topological spaces that extends continuously $\psi_{m',0}$. We construct $\psi_{m',\mathrm{top}}$ as the map from topological spaces whose restriction to $\Sym^{m-m'}(\widetilde{C})_\uu\times\Sigma(m',n+1-m',\mathbf{1})$ is $\psi_{m',\uu}$. We claim that $\psi_{m',\mathrm{top}}$ is continuous.
Let $z_0=(q,[J])$ be a point in $\Sym^{m-m'}(\widetilde{C})_\uu\times\Sigma(m',n+1-m',\mathbf{1})$, and let $Z$ be a one parameter family in $\Sym^{m-m'}(\widetilde{C})\times\Sigma(m',n+1-m',\mathbf{1})$ passing through $z_0$. Let $\mathbf{v}\in\Z_{\geq 0}^n$ be  the integer vector with highest $|\mathbf{v}|$ such that $Z\subseteq \Sym^{m-m'}(\widetilde{C})_\mathbf{v}\times\Sigma(m',n+1-m',\mathbf{1})$. Then $|\mathbf{v}|\leq |\uu|$ and $\uu-\mathbf{v}\in \Z_{\geq 0 }^n$. By \cref{prop:limits}, we get that 
\[
\underset{z\rightarrow z_0}{\lim}\psi_{m',\mathrm{top}}(z)=
\underset{z\rightarrow z_0}{\lim}\psi_{m',\mathbf{v}}(z)= q\,\cup\,\mathbb{V}\left(\phi_{m',\uu}([J])\right)=\psi_{m',\uu}(z_0).
\]
Therefore, we conclude that $\psi_{m',\mathrm{top}}$ is continuous. Moreover, by \cref{prop:limits}, $\psi_{m'\mathrm{top}}$ is the only possible extension of $\psi_{m',0}$ at the level of topological spaces.

Next, we show that $\psi_{m',0}$ extends not only at the level of topological spaces, but also at the scheme-theoretical level. Let $\Gamma$ be the graph of $\psi_{m',0}$ and let $\overline{\Gamma}$ be its closure. Consider the projection $\pi:\overline{\Gamma}\rightarrow \Sym^{m-m'}(\widetilde{C})\times \Sigma(m',n+1-m',\mathbf{1})$.
Now, let $y\in \Sym^{m-m'}(\widetilde{C})\times \Sigma(m',n+1-m',\mathbf{1})$ be a point outside the domain of definition of $\psi_{m',0}$. Since $\psi_{m',\mathrm{top}}$ is the unique topological extension of $\psi_{m',0}$, we deduce that $\pi^{-1}(y)$ is the point $(y, \psi_{m',\mathrm{top}}(y))$. Then, by \cite[Theorem 2]{Kanev2021}, $\psi_{m',0}$ extends to the map \eqref{eq:normalization}

Now, we show that \eqref{eq:normalization} is finite. By \cite[Tag 01W6]{stacks-project}, $\psi_{m'}$ is proper. Therefore, it is enough to check that $\psi_{m'}$ is quasi-finite. 
First of all, note that $\psi^{-1}_{m'}(\sHilb^{m,m',u}(C))=\Sym^{m-m'}(\widetilde{C})_{u}\times\Sigma(m',n+1-m',\mathbf{1})$. Since $\sHilb^{m,m',u}(C)$ is a stratification of $\sHilb^{m,m'}(C)$, it is enough to check that the map $\psi_{m',u}$ has finite fibers for all $0\leq u\leq m-m'$. 
Let $(q,[J])\in \sHilb^{m,m',u}(C)$. Since for each $\uu\in\Z_{\geq 0}^{n}$ with $|\uu|=u$, $\psi_{m',\uu}$ is an isomorphism onto $\sHilb^{m,m',\uu}(C)$ the number of fibers of $(q,[J])$ is exactly the number of distinct $\sHilb^{m,m',\uu}(C)$ containing this point. This is equivalent to the number of Grassmannians in $\mathcal{G}_{n+1-m',n}^{m'+u}$ containing $[J]$. In terms of the combinatorics, this is equivalent to counting how many hypersimplices of the form $\Delta_{m'-1,n}+\uu$ contain $\mu([J])$. We conclude that $\psi_{m'}$ is finite. Moreover, since $\psi_{m'}$ is finite, birational and its domain is normal, we conclude that \eqref{eq:normalization} is the normalization of $\sHilb^{m,m'}(C)$.

\end{proof}

\begin{Rem}
     \cref{thm:normalization} states that $\Sym^{m-m'}(\widetilde{C})\times \Sigma(m',n+1-m',\mathbf{1})$ is the normalization of $\sHilb^{m,m'}(C)$. Since this normalization is smooth, the map \eqref{eq:normalization} is also a resolution of singularities of $\sHilb^{m,m'}(C)$. Moreover, we can describe the locus where the normalization map $\psi_{m'}$ is not injective. Consider the strata $\sHilb^{m,m',u}(C)$ of $\sHilb^{m,m'}(C)$. By \eqref{eq: iso strata}, a point in $ \sHilb^{m,m',u}(C)$ is a tuple $(q,[J])$ where $q\in\Sym^{m-m'}(\widetilde{C}\setminus\{p_1,\ldots,p_n\})$ and $[J]\in\mathcal{G}^{m'+u}_{n+1-m',n}$.
     Then $(q,[J])$ lies in the birational locus of $\psi_{m'}$ if and only if $[J]$ is contained in only one irreducible component of $\mathcal{G}^{m'+u}_{n+1-m',n}$. In other words, $\mu([J])$ is contained in a unique hypersimplex of $\mathcal{K}^{[m'+u]}_{n+1-m',n}$. In terms of the fiber, the degree of the fiber $\psi_{m'}^{-1}((q,[J]))$ is the number of hypersimplices containing $\mu([J])$ in $\mathcal{K}^{[m'+u]}_{n+1-m',n}$.
\end{Rem}

Now, we replace $C$ by $X_n$. In this case, the geometry of the non-smoothable components of $\sHilb^m(X_n)$ is simpler than for irreducible curves. 

\begin{Prop}\label{prop: nonsmooth comp Xn}
    The non-smoothable components of $\sHilb^m(X_n)$ are isomorphic to 
    \begin{equation}\label{eq:irred comp X_n}
    \Sym^{u_1}( L_1)\times \cdots \times  \Sym^{u_1}( L_n)\times \Sigma(m',n+1-m',\mathbf{1})
    \end{equation}
    for $2\leq m'\leq \min\{m,n-1\}$ and $\mathbf{u}=(u_1,\dots, u_n)$ partition of $m-m'$. In particular, the non-smoothable components of $\sHilb^m(X_n)$ are smooth.
\end{Prop}
\begin{proof}
Let $Z$ be a non-smoothable component of $\sHilb^m(X_n)$.
    By \cref{theo: irred comp reduced structure}, the non-smoothable components of $\sHilb^m(X_n)$ are birational to \eqref{eq:irred comp X_n}. By \cref{lemma:one parameter family}, it is enough to check the smoothness at a point $[J]$ supported at $\mathbf{0}$ such that $\mu([J])$ is a vertex of $\mathcal{K}_n^{[m]}$. By  \cref{prop:limits}, $[J]$ lies in  $\Sigma(m,n+1-m',\mathbf{1}+\uu)$. By \cref{prop: ideal translation} and \cref{prop: ideal translation 2}, the only component in the completion of the stalk of $\sHilb^m(X_n)$ at $[J]$ corresponding to $Z$ is a component associated to an ideal of the form $\mathcal{J}_S$. Since the ideal $\mathcal{J}_S$ defines an affine space (see  \cref{lemma:primary k2}), we conclude that $Z$ is smooth. 

    Now, in this case, the normalization map \eqref{eq:normalization} in \cref{thm:normalization} maps \eqref{eq:irred comp X_n}
    to $Z$. By the uniqueness of the normalization we deduce that $Z$ is isomorphic to \eqref{eq:irred comp X_n}.
\end{proof}

Note that \cref{prop: nonsmooth comp Xn} is no longer true if we replace $X_n$ by an irreducible curve $C$ with a rational $n$--fold singularity at $p$. From the same arguments used in \cref{prop: nonsmooth comp Xn}, we obtain the following.

\begin{Cor}\label{cor: sing nonsmooth}
Let $C$ be an irreducible curve whose only singularity $p$ is a rational $n$--fold singularity. Then, the singularities of the non-smoothable components $\sHilb^{m,m'}(C)$ are locally union of affine spaces.
\end{Cor}

\subsection{Smoothable components}

For this subsection, we will use the notion of smoothable face, cf. \cref{def: smoothable face}. We start with the following.

\begin{Prop}\label{prop: smooth punctual}
    Let $[J]\in\sHilb_{\mathbf{0}}^{m}(X_n)$. Then $[J]$ is smoothable if and only if $\mu([J])$ lies in a smoothable face of $\mathcal{K}_n^{[m]}$. 
\end{Prop}
\begin{proof}
Let $[J]\in\sHilb_\mathbf{0}^m(X_n)$ such that $\mu([J])$ lies in a smoothable face of $\mathcal{K}_n^{[m]}$. By \cref{prop: ideals to smoothable faces}, there exists $S\subset[n]$ and $\uu\in\Z_{\geq 1}^n$ with $|\uu|=m+|S|$ such that $J =\langle f,x_i^{u_i}:i\in S\rangle$ where $f=\sum_{i\not\in S}a_ix_i^{u_i}$. 
By \cref{prop:limits}, $[J]$ can be obtained as a limit of length $m$ schemes of the form $q\cup\mathbb{V}(J')$ where $q\in \left(X_n\setminus\{\mathbf{0}\}\right)^{m-n+|S|}$ and $J'=\langle f',x_i:i\in S\rangle$ and $f'=\sum_{i\not\in S}a_ix_i$. Therefore, to check that $[J]$ is smoothable, it is enough to check that $[J']$ is smoothable. Consider the ideal $\tilde{J}=\langle f'\rangle$ of the ring $R_{n-|S|}=\C[x_i:i\not\in S]/\langle x_ix_j:i\neq j\rangle$. Then, $[J']$ is smoothable if and only if $[\tilde{J}]$ is smoothable in $\sHilb^{n}(X_{n-|S|})$. Now, by  \cref{theo: irred comp reduced structure}, $[J']$ is  smoothable since it lies in $\Sigma(n,1,\mathbf{1})$.

Assume now that $[J]$ is a smoothable ideal. 
We apply induction on $n$. For $n=2$, $\mu([J])$ lies in a smoothable face since all faces of $\mathcal{K}_2^{[m]}$ are smoothable.
Assume now that the statement holds for all $n'<n$. Let $\Delta_{l,n}+\uu-\mathbf{1}$ be a hypersimplex containing $[J]$. By \cref{theo: irred comp reduced structure} and \cref{remark: comb vs geometry}, we may assume that $\uu=\mathbf{1}$. In particular, we have that $m=l+1\leq n$. If $[J]$ is smoothable, then $[J]$ is the limit of $m$ distinct points $q_1,\ldots,q_m$ in $X_n$. If $m<n$, 
then, along the limit, $q_1,\ldots q_m$ are contained in at most $m$ of the lines of $X_n$. In particular, $[J]$ lies in the Hilbert scheme of points of those $m$ lines. By induction, $\mu([J])$ is contained in a smoothable face. Assume now that $m=n$. Then $[J]$ is the limit of $n$ distinct points $q_1,\ldots,q_n$ in $X_n$. As before, if along such limit, there is a line of $X_n$ not containing any of the points $q_1,\ldots,q_n$, then, we can apply induction. Therefore, we may assume that $q_i$ lies in $L_i$ for each $i\in [n]$. Then, $q_1,\ldots,q_n$ are the intersection of a hyperplane $\mathbb{V}(a_0+\sum a_ix_i)$ with $X_n$. As in the proof of \cref{theo: irred comp reduced structure}, we deduce that such limit is $\mathbb{V}(\sum a_ix_i)$. Therefore, $[J]=[\langle \sum a_ix_i\rangle]$ is contained in $\Sigma(n,1,\mathbf{1})$ and $\mu([J])$ lies in $\Delta_{n-1,n}$ which is smoothable.
\end{proof}

From \cref{prop: smooth punctual}, we derive the following result.

\begin{Cor}
Let $C$ be a curve whose unique singularity $p$ is a rational $n$--fold singularity.
    Let $\mathbb{V}(J)$ be a length $m$ subscheme of $C$, and let $J_0,J_1,\ldots,J_k$ be the ideals in the primary decomposition of $J$ such that $J_0$ is supported at the singularity $p$ and has length $m'$. Then $[J]$ is smoothable if and only if $\mu([J_0])$ lies in a smoothable face of $\mathcal{K}_n^{[m']}$.
\end{Cor}

We finish this section by stating more properties of the singularities of the smoothable components.

\begin{Prop}\label{prop:sing smooth comp}
    Each smoothable component of $\sHilb^m(X_n)$ is normal and has toric singularities.
\end{Prop}
\begin{proof}
    Let $Z$ be a smoothable component of $\sHilb^m(X_n)$. Then $Z$ is birational to 
    \[
    \Sym^{u_1}L_1\times \cdots \times \Sym^{u_n}L_n
    \]
    for some $\uu=(u_1,\ldots,u_n)$ partition of $m$. By \cref{lemma:one parameter family}, it is enough to check the statement for $[J]\in Z$ supported at $\mathbf{0}$ such that $\mu([J])$ is a vertex of $\mathcal{K}_n^{[m]}$. The completion of the stalk of $\sHilb^m(X_n)$ at $[J]$ is computed in \cref{theo:def theory}. The irreducible components of this stalk are calculated in \cref{lemma: primary k1} and \cref{lemma:primary k2}. In \cref{prop: ideal translation} and \cref{prop: ideal translation 2}, we identify which irreducible components of the stalk at $[J]$ correspond to each irreducible component of $\sHilb^m(X_n)$. For $X_n$, this correspondence associates $Z$ to a unique component of the stalk, which corresponds to an ideal of the form $\mathcal{Q}_i$. By \cref{Lem: toric ideal} and  \cref{lem:normal polytope}, $\mathcal{Q}_i$ is normal and toric. We conclude that the completion of the stalk of $Z$ at $[J]$ is normal.
    By \cite[Tag 07QU]{stacks-project},
    the stalk of $Z$ ar $[J]$ is an excellent ring, and
    we deduce that the stalk is normal  by \cite[Theorem 79]{matsumura1989commutative}.
    We conclude that  $Z$ is normal and its singularities are toric.
\end{proof}

As in \cref{cor: sing nonsmooth}, \cref{prop:sing smooth comp} is no longer true if we replace $X_n$ with an irreducible curve $C$ with a rational $n$--fold singularity. By the same techniques used in Proposition \ref{prop:sing smooth comp}, we obtain the following result.

\begin{Cor}\label{cor: sing smooth irred}
    Let $C$ be an irreducible curve whose only singularity $p$ is a rational $n$--fold singularity. Then, the singularities of the smoothable components are locally union of normal toric varieties.
\end{Cor}

\section{Ongoing work and future questions}\label{sec:ongoing work and future questions}

This paper sits in the theory of a combinatorial study of certain properties of Hilbert schemes of curves, not necessarily planar. It naturally leads to questions such as how to go beyond fold-like curves and investigate configurations of lines in projective space. This, in turn, gives rise to the following natural questions.

\begin{Ques}
    When the number of lines in $\C^n$ is greater than $n$, do we still have a combinatorial description of the Hilbert scheme of points?
\end{Ques}

The first interesting case would be four lines in $\C^3$. One possible idea to approach this question is to study degenerations of configurations of lines to a lower dimensional affine space. The authors plan to come back to this problem in the future. Following this line of thought, this leads to.

\begin{Ques}
    Are there combinatorial descriptions of Hilbert schemes of points for configurations of planes or higher dimensional linear subspaces in $\C^n$?
\end{Ques}

Again, by analogy with the case treated in this paper, we believe that for \emph{transversal} unions of planes, similarly to the situation of fold-like curves, rich geometric and combinatorial structures will emerge. Another direction of ongoing work concerns the study of the Quot scheme, which naturally arises in connection with the compactified Jacobian. Indeed, for curves with non--locally planar singularities, the appropriate source of the Abel map is the Quot scheme rather than the Hilbert scheme of points, leading us to investigate the following.

\begin{Ques}
    Is there a nice combinatorial description of the Quot scheme for fold-like curves?
\end{Ques}

The authors expect that such a description can be useful for a concrete study of the alternate compactification of the moduli space of curves described in \cite{HanKassSatrianoExtendingTorelli}.

\appendix

\section{The Hypersimplicial complex  }\label{app: The hypersimplicial complex}

In this Appendix we present the proofs of the combinatorial properties of hypersimplicial complex $\mathcal{K}_n^{[m]}$ used in \cref{sec:moment map}. We refer to \cite{GelfandMacpherson} for further details on the relation between Grassmannians, hypersimplices and their combinatorics.
The hypersimplex $\Delta_{l,n}$ is defined as
\[
\Delta_{l,n}:=\mathrm{Conv}\{\ee_{i_1}+\ldots+\ee_{i_l}:1\leq i_1<\cdots <i_l\leq n\}=\{\sum_{i=1}^n\lambda_i\ee_i:0\leq \lambda_i\leq 1, \text{ and } \sum_{i=1}^n\lambda_i=l\}.
\]
By definition, the hypersimplex $\Delta_{l,n}$ is contained in the dilated simplex $l\cdot \Delta_{n-1}$.
The vertices of $\Delta_{l,n}$ are exactly the vectors $e_{i_1}+\cdots+e_{i_l}$ for $1\leq i_1<\cdots <i_l\leq n$. The number of vertices of $\Delta_{l,n}$ is $\binom{n}{l}$.
For $0\leq r\leq n-1$, the $(n-r)$--faces of the hypersimplex $\Delta_{l,n}$ are of the form
\begin{equation}\label{eq:faces hypersimplex}
\begin{array}{c}
\Delta_{l,n}(S_1,S_2):=
\displaystyle
\left\{\sum_{i\not\in S_1\sqcup S_2}\lambda_ie_i +  \sum_{i\in S_2}e_i
: 0\leq \lambda_i\leq 1\text{ and } \sum_{i\not\in S_1\sqcup S_2}\lambda_i = l-|S_2|\right\} = \\ \\
\displaystyle 
\mathrm{Conv}\left(
e_{i_1}+\cdots+e_{i_{l-|S_2|}}: i_1,\ldots,i_{l-|S_2|}\not\in S_1\sqcup S_2\text{ distinct}
\right)+\sum_{i\in S_2}e_i
\end{array}
\end{equation}
for $S_1\sqcup S_2\subseteq [n]$ and $|S_1|+|S_2| = r-1$. The face $\Delta_{l,n}(S_1,S_2)$ is obtained by setting $\lambda_i=0$ for all $i\in S_1$ and $\lambda_i=1$ for all $i\in S_2$. Moreover, such a face is isomorphic to the hypersimplex $\Delta_{n-r+1,l-|S_2|}$. In particular, every hypersimplex can be seen as a hypersimplicial complex. Recall that a hypersimplicial complex is a polyhedral complex whose faces are hypersimplices. See \cite[Section 2.1.3]{GelfandMacpherson} for these details on hypersimplices.

\begin{Lem}\label{lem: intersection faces}
    Let $m\geq 2$ and $1\leq l,l'\leq \min\{n-1,m-1\}$. Consider two integer vectors $\uu,\mathbf{v}\in\Z_{\geq 0}^{n}$ with $|\uu| = m-1-l$ and $|\mathbf{v}|=m-1-l'$. Then, the intersection of $\left(\Delta_{l,n}+\uu\right)$ and $\left(\Delta_{l',n}+\mathbf{v}\right)$  is nonempty if and only if $\uu-\mathbf{v}\in   
    \{0,1,-1\}^n$. Moreover, in this case we have that 
    \[\left(\Delta_{l,n}+\uu\right)\cap\left(\Delta_{l',n}+\mathbf{v}\right):= \Delta_{l,n}\left(\kappa(\uu-\mathbf{v},1),\kappa(\uu-\mathbf{v},-1)\right) +\uu =\Delta_{l',n}\left(\kappa(\uu-\mathbf{v},-1),\kappa(\uu-\mathbf{v},1)\right) +\mathbf{v}.
    \]
    where the function $\kappa$ is defined as in \eqref{def:kappa indexes}.
\end{Lem}
\begin{proof}
    We can write the translated hypersimplices $\Delta_{l,n}+\uu$ and $\Delta_{l',n}+\mathbf{v}$ as
    \begin{equation}\label{eq:trans hypersimplex}
    \begin{array}{c}
    \displaystyle \Delta_{l,n}+\uu=\left\{\sum_{i=1}^n\lambda_i\mathbf{e}_i: \sum \lambda_i = |\uu|-l\, \text{ and }\, u_i\leq \lambda_i\leq u_i+1 \text{ for all }i\in [n]
     \right\},\\
     \displaystyle \Delta_{l',n}+\mathbf{v}=\left\{\sum_{i=1}^n\lambda_i\mathbf{e}_i: \sum \lambda_i = |\mathbf{v}|-l'\, \text{ and }\, v_i\leq \lambda_i\leq v_i+1 \text{ for all }i\in [n]
     \right\}. 
    \end{array}
    \end{equation}
    Let $\lambda=\sum \lambda_i\mathbf{e}_i$ be a point in the intersection of  $\Delta_{l,n}+\uu$ and $\Delta_{l',n}+\mathbf{v}$.
    By \eqref{eq:trans hypersimplex}, we have that $\lambda_i$ is contained in the intersection of the intervals $[u_i,u_i+1]\cap[v_i,v_i+1]$. This intersection is nonempty if and only if $u_i-v_i\in\{0,1,-1\}$. We deduce that the intersection of  $\Delta_{l,n}+\uu$ and $\Delta_{l',n}+\mathbf{v}$ is nonempty if and only if $\uu-\mathbf{v}\in\{0,1,-1\}^n$.

    Assume now that $\uu-\mathbf{v}\in\{0,1,-1\}^n$. Then, we can write the intersection of the hypersimplices as
    \[
    \begin{array}{c}
\left(\Delta_{l,n}+\mathbf{u}\right)\cap\left(\Delta_{l',n}+\mathbf{v}\right)=\left\{
\sum_{i=1}^n\lambda_i\mathbf{e}_i: \sum\lambda_i=m-1 \text{ and }
\begin{array}{l}
\lambda_i=u_i \text{ for }i\in\kappa(\uu-\mathbf{v},1),
\\
\lambda_i=u_i+1 \text{ for }i\in\kappa(\uu-\mathbf{v},-1),
\\
 u_i\leq \lambda_i\leq u_i+1 \text{ for }i\in\kappa(\uu-\mathbf{v},0)
\end{array}
\right\} =\\ \displaystyle
\left\{
\sum_{i\in\kappa(\uu-\mathbf{v},0)}\lambda_i\mathbf{e}_i:\,\sum_{i\in\kappa(\uu-\mathbf{v},0)}\lambda_i=m-1-|\uu|-|\kappa(\uu-\mathbf{v},-1)|\text{ and } 0\leq \lambda_i\leq 1
\right\}+\mathbf{e}_{\kappa(\uu-\mathbf{v},-1)}+\uu
=\\ \displaystyle
\left\{
\sum_{i\in\kappa(\uu-\mathbf{v},0)}\lambda_i\mathbf{e}_i:\,\sum_{i\in\kappa(\uu-\mathbf{v},0)}\lambda_i=l-|\kappa(\uu-\mathbf{v},-1)|\text{ and } 0\leq \lambda_i\leq 1
\right\}+\mathbf{e}_{\kappa(\uu-\mathbf{v},-1)}+\uu.
\end{array}
    \]
    Using \eqref{eq:faces hypersimplex}, we conclude that 
    \[
\left(\Delta_{l,n}+\mathbf{u}\right)\cap\left(\Delta_{l',n}+\mathbf{v}\right)= \Delta_{l,n}(\kappa(\uu-\mathbf{v},1),\kappa(\uu-\mathbf{v},-1))+\uu.
    \]
\end{proof}

Using \cref{lem: intersection faces}, we can construct a polyhedral complex using the translated hypersimplices $\Delta_{l,n}+\uu$. 

\begin{PropDef}\label{prop: hypersimplicial complex}
For $n\geq 2$ and $m\geq 2$, the set of all faces of 
the translated hypersimplices $\Delta_{l,n}+\uu$ for $1\leq l\leq \min\{n-1,m-1\}$ and $\uu\in \Z_{\geq 0}^n$ with $|u|=m-1-l$ form a hypersimplicial complex $\mathcal{K}_n^{[m]}$ called the $(n,m)$--hypersimplicial complex. Moreover, $\mathcal{K}_n^{[m]}$ forms a subdivision of $(m-1)\cdot\Delta_{n-1}$.
\end{PropDef}
\begin{proof}
    To verify that the translated hypersimplices $\Delta_{l,n}+\uu$ form a polyhedral complex, it is enough to show that any two such hypersimplices intersect in a face of both hypersimplices. This follows from \cref{lem: intersection faces}. It remains to show that all these hypersimplices cover $(m-1)\cdot\Delta_{l,n}$. Let $\lambda=\sum\lambda_i\mathbf{e}_i$ be a point in $(m-1)\cdot\Delta_{l,n}$. In other words, $0\leq \lambda_i\leq m-1$ and $\sum \lambda_i=m-1$. Now, consider the integer vector $\uu=(u_1,\ldots,u_n)$ where $u_i=\lfloor \lambda_i\rfloor$ for every $i\in[n]$. Then  $\uu$ is in $\Z_{\geq 0}^n$ and $|\uu|\leq m-1$. We write $\lambda$ as 
    \[
    \lambda= \sum (\lambda_i-u_i)\mathbf{e}_i+\uu.
    \]
    Note that $0\leq \lambda_i-u_i\leq 1$ and $\sum (\lambda_i-u_i)=m-1-|\uu|$. By construction $m-1-|\uu|\leq n$. Assume first that $m-1-|\uu|\leq n-1$. In this case $m-1-|\uu|\leq \min\{n-1,m-1\}$ and $\lambda$ is contained in the translated hypersimplex $\Delta_{m-1-|\uu|,n}+\uu$. If $m-1-|\uu| = n$, then $\lambda=\uu$ is an integer point in $(m-1)\cdot\Delta_{n-1}$, and there exists $i\in[n]$ such that $\lambda_i\geq 1$. Then  $\lambda$ is contained in $\Delta_{1,n}+\lambda -\mathbf{e_i}$. Therefore, we conclude that any point in $(m-1)\Delta_{n-1}$ is contained in a face of $\mathcal{K}_n^{[m]}$.
\end{proof}

\begin{Ex}\label{Ex:hypercomplex n=2}
    For $n=2$, the only hypersimplex is $\Delta_{1,2}=\Delta_1$. Therefore, the maximal faces $\mathcal{K}_2^{[m]}$ are the segments between the points $(m-1-i,i)$ and $(m-1-i-1,i+1)$ for $0\leq i\leq m-2$. \cref{fig:moment n 2} illustrates the hypersimplicial complexes $\mathcal{K}_2^{[2]}$, $\mathcal{K}_2^{[3]}$, and $\mathcal{K}_2^{[4]}$.
\end{Ex}

\begin{Ex}
    \label{ex: hypercomplex n=3}
    For $n=3$, we have two types of hypersimplices: $\Delta_{1,3}=\Delta_2$ and $\Delta_{2,3}$. For $m=2$, $\mathcal{K}_3^{[2]}$ coincide with $\Delta_{1,3}$ as a simplicial complex. For $m=3$, $\mathcal{K}_3^{[3]}$ has $4$ maximal faces: $\Delta_{2,3}$ and  $\Delta_{1,3}+\mathbf{e}_i$ for $i\in[3]$. Similarly,  $\mathcal{K}_3^{[4]}$ has $9$ maximal faces: $\Delta_{1,3}+\mathbf{e}_i+\mathbf{e}_j$ for $i,j\in[3]$ and $\Delta_{2,3}+\mathbf{e}_i$ for $i\in[3]$. In \cref{fig: n=3 simplex}, the hypersimplicial complexes  $\mathcal{K}_3^{[2]}$, $\mathcal{K}_3^{[3]}$,  and $\mathcal{K}_3^{[4]}$ are depicted.

In general, $\mathcal{K}_3^{[m]}$ has $\binom{m}{2}$ hypersimplices of the form $\Delta_{1,3}$, and $\binom{m-1}{2}$ hypersimplices of the form $\Delta_{2,3}$. The first type of hypersimplices corresponds to the 
 triangles given by the vertices $(u_1,u_2,u_3)+\mathbf{e}_i$ for $i\in[3]$ and $u_1+u_2+u_3=m-2$. Similarly, the hypersimplices of the form $\Delta_{2,3}$ are triangles whose vertices are $(u_1,u_2,u_3)+\mathbf{e}_i+\mathbf{e}_j$ for $1\leq i<j\leq 3$ and $u_1+u_2+u_3=m-3$. 
\end{Ex}

\begin{Ex}
    \label{ex:hypercomplex n=4}
    For $n=4$, there are $3$ possible hypersimplices: $\Delta_{1,4}$, $\Delta_{2,4}$ and $\Delta_{3,4}$. For instance, the maximal faces of $\mathcal{K}_4^{[3]}$ are $\Delta_{2,4}$ and $\Delta_{1,4}+\mathbf{e}_i$ for $i\in [4]$. \cref{fig:n 4 m 3} illustrates $\mathcal{K}_4^{[3]}$.  Similarly, the maximal faces of $\mathcal{K}_4^{[4]}$ are the ten simplices $\Delta_{1,4}+\mathbf{e}_i+\mathbf{e}_j$ for $i,j\in[4]$, the four hypersimplices $\Delta_{2,4}+\mathbf{e}_i$ for $i\in[4]$ and the simplex $\Delta_{3,4}$.  The hypersimplicial complex $\mathcal{K}_4^{[4]}$ is illustrated in  \cref{fig:n 4 m 4}.
\end{Ex}

For $1\leq l\leq \min\{n-1,m-1\}$ fixed, we also consider the hypersimplicial subcomplex $\mathcal{K}_{l,n}^{[m]}$ of $\mathcal{K}_n^{[m]}$ whose maximal faces are the translated hypersimplices $\Delta_{l,n}+\uu$ for $\uu\in\Z_{\geq 0}^n$ with $|\uu|=m-1-l$. In particular, $\mathcal{K}_n^{[m]}$ is the union of the complexes $\mathcal{K}_{l,n}^{[m]}$ for $1\leq l\leq \min\{n-1,m-1\}$. The number of maximal faces in  $\mathcal{K}_{l,n}^{[m]}$ coincides with the number of integer vectors $\uu\in\Z_{\geq 0}^n$ with $|\uu|=m-1-l$, which is 
$
\binom{m+n-l-2}{n-1}
$. Therefore, the number of maximal faces of $\mathcal{K}_n^{[m]}$ is 
\begin{equation}
    \label{eq: num of hypersimplex}
    \displaystyle\sum_{l=1}^{\min\{n-1,m-1\}}\binom{m+n-l-2}{n-1}.
\end{equation}

\begin{Ex}
For $n=3$ and $m=4$, we have two possible hypersimplicial complexes: $\mathcal{K}^{[4]}_{1,3}$ and $\mathcal{K}^{[4]}_{2,3}$.
    The maximal cells of $\mathcal{K}^{[4]}_{1,3}$ are the $6$ hypersimplices $\Delta_{1,3}+\mathbf{e}_i+\mathbf{e}_j$ for $i,j\in [3]$. \cref{fig:hypercomplex 3 1 4} illustrates in purple the complex $\mathcal{K}^{[4]}_{1,3}$. Similarly, the complex  $\mathcal{K}^{[4]}_{2,3}$ is depicted in orange in \cref{fig:hypercomplex 3 1 4} and its maximal cells are the hypersimplices $\Delta_{2,3}+\mathbf{e}_i$ for $i\in[3]$.
\end{Ex}

\begin{Ex}
    The maximal cells of $\mathcal{K}^{[3]}_{1,4}$ are the $3$ hypersimplices $\Delta_{1,4}+\mathbf{e}_i$ for $i\in [4]$.  \cref{fig:hypercomplex 3 1 4} illustrates on the left the complex $\mathcal{K}^{[4]}_{1,3}$. The complex  $\mathcal{K}^{[3]}_{2,4}$ is exactly the hypersimplex $\Delta_{2,4}$. Similarly, the maximal cells of $\mathcal{K}^{[4]}_{2,4}$ are the four hypersimplices $\Delta_{2,4}+\mathbf{e}_i$ for $i\in [4]$. This complex is depicted in the right side of in  \cref{fig:hypercomplex 3 1 4}.
\end{Ex}

By construction, the $(n-r)$--dimensional faces of the $(n,m)$--hypersimplicial complex are the $(n-r)$--dimensional faces of the translated hypersimplices. Using \eqref{eq:faces hypersimplex}, we may describe these faces as follows. For $1\leq l\leq \min\{n-1,m-1\}$, $\uu\in\Z_{\geq 0}^n$ with $|\uu|=m-1-l$, and for $S_1\sqcup S_2\subseteq[n]$ with $|S_1|+|S_2|=r-1$, we get the face
\[
\mathcal{K}_n^{[m]}(S_1,S_2,l,\uu):=\Delta_{l,n}(S_1,S_2)+\uu.
\]
Using this notation, we can write the intersection in \cref{lem: intersection faces} as 
\[
\begin{array}{c}
\displaystyle
\left(\Delta_{l,n}+\uu\right)\cap\left(\Delta_{l',n}+\mathbf{v}\right) = \mathcal{K}_n^{[m]}(\emptyset,\emptyset,l,\uu)\cap \mathcal{K}_n^{[m]}(\emptyset,\emptyset,l',\mathbf{v})=\mathcal{K}_n^{[m]}(\kappa(\uu-\mathbf{v},1),\kappa(\uu-\mathbf{v},-1),l,\uu)=\\ 
\displaystyle\mathcal{K}_n^{[m]}(\kappa(\uu-\mathbf{v},-1),\kappa(\uu-\mathbf{v},1),l',\mathbf{v}).
\end{array}
\]
Note that the face $\mathcal{K}_n^{[m]}(S_1,S_2,l,\uu)$ is contained in the intersection of $\mathcal{K}_n^{[m]}$ with the linear subspace $\{\lambda_i=u_i: \text{ for } i \in S_1\}$ and $\{\lambda_i=u_i+1: \text{ for } i \in S_2\}$. Here, $\lambda_1,\ldots,\lambda_n$ are the coordinates of $\R^n$. In the following result, we study the intersection of such type of linear subspaces with the hypersimplicial complex.

\begin{Prop}\label{prop: linear and hypersimplices}
    For $S\subseteq[n]$ and $\mathbf{a}=(a_i)_{i\in S}\in\Z_{\geq 0}^{|S|}$  with $|\mathbf{a}|\leq m-1$, consider the linear subspace
    \[H(S,\mathbf{a}):=\{ \lambda_i=a_i: \text{ for } i\in S\}.\]
    Then, the intersection of $\mathcal{K}_n^{[m]}$ is isomorphic to the hypersimplicial complex $\mathcal{K}_{n-|s|}^{\left[m-|a|\right]}$.
\end{Prop}
\begin{proof}
    Without loss of generality assume that $S=\{n-|S|+1,\ldots,n \}$,and consider the linear projection 
    \[
    \pi_S:\R^n\rightarrow \R^n/\langle \mathbf{e}_i: i\in S\rangle \simeq\R^{n-|S|}.
    \]
    Note that the restriction of $\pi_S$ to $H(S,\mathbf{a})$ is an isomorphism. We claim that the projection of the intersection of $\mathcal{K}_n^{[m]}$ and $H(S,\mathbf{a})$ is $\mathcal{K}_{n-|S|}^{[m-|\mathbf{a}|]}$. Note that the intersection of $\mathcal{K}_n^{[m]}$ and $H(S,\mathbf{a})$ consists of all the faces of the form $\mathcal{K}_n^{[m]}(S_1,S_2,l,\uu)$ such that $S\subseteq S_1\sqcup S_2$, and $u_i=a_i$ for every $i\in S\cap S_1$ and $u_i+1=a_i$ for every $i\in S\cap S_2$. The projection of such a face is 
    \[\begin{array}{c}
    \displaystyle
    \left\{\sum_{i=1}^{n-|S|}\lambda_i\mathbf{e}_i: \sum \lambda_i = m-1-|\mathbf{a}| \text{ and }
    \begin{array}{l} \displaystyle
\lambda_i=u_i \text{ for } i\in S_1\setminus S
    ,\\ \displaystyle
\lambda_i=u_i+1 \text{ for } i\in S_1\setminus S
    ,\\ \displaystyle
u_i\leq \lambda_i\leq  u_i+1 \text{ for } i\not \in S_1\sqcup S_2
    \end{array}
    \right\} = \\ \displaystyle
     \left\{\sum_{i\not \in S_1\sqcup S_2}\lambda_i\mathbf{e}_i: \sum \lambda_i = m-1-|\mathbf{a}|-|S_2\setminus S| \text{ and }
0\leq \lambda_i\leq 1 \text{ for } i\not \in S_1\sqcup S_2
    \right\} +\mathbf{e}_{S_2\setminus S} + \uu= \\
    \displaystyle \Delta_{ m-1-|a|-|S_2\setminus S|,n-|S|}(S_1\setminus S, S_2\setminus S)+\uu,
    \end{array}
    \]
    which is a maximal face of $\mathcal{K}_{n-|S|}^{[m-|\mathbf{a}|]}$. It remains to show that any face of $\mathcal{K}_{n-|S|}^{[m-|\mathbf{a}|]}$ is achieved from the projection $\pi_S$. This follows from the fact that the image of $H(S,\mathbf{a})\cap \mathcal{K}_n^{[m]]}$ through the projection $\pi_S$ is the simplex $(m-|\mathbf{a}|-1)\cdot \Delta_{n-| S|-1}$.
\end{proof}

Using \cref{prop: linear and hypersimplices}, we introduce the notion of smoothable face.

\begin{Def}\label{def: smoothable face}
    A face $\Gamma$ of $\mathcal{K}_n^{[m]}$ is smoothable if one of the following conditions holds:
    \begin{itemize}
        \item $n=1$ or $n=2$.
        \item $\Gamma$ is contained in $\Delta_{n-1,n}+\uu$ for $\uu\in\Z_{\geq 0}^n$ with $|\uu|=m-n$.
        \item $\Gamma$ is contained in a linear subspace $H(S,\uu)$ and in the intersection of $H(S,\uu)$ and $\mathcal{K}_n^{[m]}$, $\Gamma$ is a smoothable face.
    \end{itemize}
\end{Def}

A first remark from \cref{def: smoothable face} is that all vertices and edges of $\mathcal{K}_n^{[m]}$ are smoothable. 
The third condition in \cref{def: smoothable face} is recursive, since by \cref{prop: linear and hypersimplices} the intersection of $H(S,\uu)$ and $\mathcal{K}_n^{[m]}$ is the hypersimplicial complex $\mathcal{K}_{n-|S|}^{[m-|\uu|]}$. Note that the faces of $\Delta_{n-1,n}$ are all hypersimplices of the form $\Delta_{n'-1,n'}$ for $n'\leq n$. Therefore, a face of $\mathcal{K}_n^{[m]}$ is smoothable if and only if it is a hypersimplex of the form $\Delta_{n'-1,n'}$ for $n'\leq n$. In other words, a face of the form $\mathcal{K}_n^{[m]}(S_1,S_2,l,\uu)$ is smoothable if and only if $n-l-|S_2|=n-|S_1|-|S_2|-1$, which is equivalent to $l=|S_1|+1$.

\begin{Ex}
    \begin{itemize}
        \item For $n=3$ and $m=3$, $\Delta_{2,3}$ is a smoothable face of $\mathcal{K}_3^{[3]}$ by definition. Therefore, for $m\geq 3$, the smoothable faces are the vertices, the edges and the hypersimplices $\Delta_{2,3}+\mathbf{u}-\mathbf{1}$ for $|\uu|=m$, which are illustrated in  \cref{fig:Smoothable faces n 3}.
        
        \item For $n=4$ and $m = 2$, the smoothable faces of  $\mathcal{K}_4^{[2]}$ are the vertices and the edges. For $n=4$ and $m=3$, aside from the vertices and the edges, $\mathcal{K}_4^{[2]}$ has four $2$--dimensional faces that are smoothable. They are faces of $\Delta_{2,4}$ and they arise from the $2$--dimensional smoothable face of $\mathcal{K}_3^{[3]}$. In \cref{fig:smoothable faces}, these four faces are depicted in orange. For $n=4$ and $m=4$, $\mathcal{K}_4^{[4]}$ has $16$ smoothable faces of dimension $2$. They correspond to the translation y $\mathbf{e}_i$ for $i\in[4]$ of the four smoothable faces of $\Delta_{2,4}$ in $\mathcal{K}_4^{[3]}$. Moreover, $\Delta_{3,4}$ is a $3$--dimensional smoothable face of $\mathcal{K}_4^{[4]}$. In  \cref{fig:smoothable faces} all the smoothable faces of $\mathcal{K}_4^{[3]}$  and $\mathcal{K}_4^{[4]}$ are illustrated.
    \end{itemize}
\end{Ex}

Analogously to the notion of smoothable face, we introduce the notion of singular face.

\begin{Def}
    \label{def:singular face}
 We say that a face $\Gamma$ of $\mathcal{K}_n^{[m]}$ is \emph{singular} if one of the following conditions is satisfied:
    \begin{itemize}
        \item $\Gamma$ is in the intersection of two distinct maximal faces.
        \item $\Gamma$ is smoothable of dimension at most $n-2$.
    \end{itemize}
    
\end{Def}

Note that vertices are always singular faces and, for $n\geq 3$, edges are always singular. For example, for $n=3$ the vertices and edges are exactly the only singular faces.

\begin{Ex}
    For $n=4$ and $m=2$, the singular faces of $\mathcal{K}_4^{[2]}$ are the vertices and edges. For $n=4$ and $m=2$, the singular faces of $\mathcal{K}_4^{[2]}$ are the vertices, the edges and the two dimensional faces of $\Delta_{2,4}$.
\end{Ex}

\section{Some Commutative Algebra}\label{app: commutative algebra}

In this technical Appendix we will give first describe the first syzyigies of the ideal representing a point in $\sHilb^m_{\mathbf{0}}(X_n)$, secondly a presentation of the ideal in \eqref{eq:gens ideal def 1} that will be useful for thirdly compute the primary decompositions of the ideals we are interested in. We follow the notations of the previous sections.

\begin{Lem}\label{lemma:syz}
    Let $[J] \in \sHilb^m_{\mathbf{0}}(X_n)$ and let $f_1,\ldots,f_l$ minimal generators of $J$ as in \cref{ideals punctual}. The beginning of the minimal free resolution of $J$ is 
    \begin{equation}\label{eq:syz}
    \begin{array}{cclcc} 
    R(-1)^{n\cdot\binom{l}{2}}&\longrightarrow&R^{l}&\longrightarrow&0\\
     e_{i,j,k}&\longmapsto & A_{k,i}x_ie_j-A_{j,i}x_ie_k& &
   \end{array}   
    \end{equation}
    
    where $1\leq i\leq n , 1\leq j<k\leq l$ and $A_{k,i}$ denotes the entry $(k,i)$ of the matrix $A$ in \cref{ideals punctual}.

\end{Lem}

\begin{proof}
    By construction of the matrix $A$ we have $A_{k,i}x_if_j= A_{k,i}A_{j,i}x_i^{u_i+1}$ and $A_{j,i}x_if_k= A_{j,i}A_{k,i}x_i^{u_i+1}$. Hence  relations in \eqref{eq:syz} are syzygies of $J$. Consider a syzygy of the form 
    \[
    g_1f_1+\cdots+g_lf_l=0,
    \]
    where 
    \[
    g_a=\displaystyle\sum_{i=1}^n \sum_{b=1}^d B_{a,i,b}x_i^b.
    \]
    Note that $g_a$ has no independent coefficient since $f_1,\ldots,f_l$ are linearly independent. We write the syzygy as follows
    \[
    \displaystyle\sum_a g_af_a = \sum_{i=1}^n \sum_{b=1}^d\sum_{a=1}^l B_{a,i,b}A_{a,i} x_i^{u_i+b} = 0.
    \]
    We get that for every $1\leq i\leq n$ and $1\leq b\leq d$, 
    \begin{equation}
        \label{eq:syz constrains}
            \sum_{a=1}^l B_{a,i,b}A_{a,i} x_i^{u_i+b} = x_i^{u_i+b}\sum_{a=1}^l B_{a,i,b}A_{a,i}=0.
    \end{equation}
    Using that $A_{a,i} x_i^{u_i+1}= x_i^1f_a$, we get that 
    \begin{equation}
        \label{eq:syzygy proof}
        \sum_{a=1}^l B_{a,i,b}x_if_a =0
    \end{equation}
    
    for every $1\leq i\leq n$ and $1\leq b\leq d$. 
    Thus, to check that the syzygy $g_1f_1+\cdots+g_lf_l=0$ can be obtained from the syzygies of the form \eqref{eq:syz}, it is enough to show that \eqref{eq:syzygy proof} is a linear combination of the syzygies  \eqref{eq:syz} for every $1\leq i\leq n$ and $1\leq b\leq d$. This is equivalent to showing that the vector $(B_{1,i,b},\ldots,B_{l,i,b})$ is contained in the linear subspace spanned by the vectors $A_{k,i}e_k-A_{j,i}e_j$ for $1\leq j<k\leq l$. Note that this linear subspace is exactly the kernel of the matrix $(A_{1,i},\ldots A_{l,i})$. On the other hand, using \eqref{eq:syz constrains}, we get that 
    \[
    (A_{1,i},\ldots A_{l,i})\begin{pmatrix}
       B_{1,i,b}\\ \vdots \\B_{l,i,b} 
    \end{pmatrix}= 0.
    \]
    We conclude that $(B_{1,i,b},\ldots,B_{l,i,b})$ is a linear combination of the vectors $A_{k,i}e_k-A_{j,i}e_j$ for $1\leq j<k\leq l$, and hence, the syzygy $g_1f_1+\cdots+g_lf_l=0$ can be obtained from the syzygies \eqref{eq:syz}.
\end{proof}

Following the notation introduced in \cref{sec:local hilbert scheme} for $k=1$ we consider the ring 
\[
\mathcal{S}_1 = \C[\alpha_2,\ldots,\alpha_n,A_1,\alpha_{1,1}.\ldots,\alpha_{1,u_1-1}]
\]
and the ideal of $\mathcal{J}_1$ given by
\begin{equation}\label{eq:ideal 1}
\mathcal{J}_1=\langle A_1\alpha_{i}:2\leq i\leq n\rangle+\langle  \alpha_{i}\alpha_{j}\alpha_{1,1} \,:\text{ for } 2\leq i<j\leq n \rangle.
\end{equation}

Similarly, for $k\geq 2$, we consider the ring
\begin{equation}\label{eq:ring Sk}
\mathcal{S}_k = \C[ \alpha_{i,j}: 1\leq i\leq n, 1\leq j\leq k  \text{ and } i\neq j]\otimes \C[\beta_{i}:1\leq i\leq k]\otimes\C[\beta_{i,s}: 1\leq i\leq k \text{ and }2\leq s\leq u_i-1],    
\end{equation}
and the ideal of $S_k$ given by
\begin{equation}\label{eq:ideal Ik}
\begin{array}{ll}
\mathcal{J}_k:=&\left\langle \alpha_{i,j}\alpha_{j,r}: 1\leq i\leq n,\,\, 1\leq j, r\leq k\text{ and } i\neq j,\,\,j\neq r\right\rangle+\\ &
\left\langle \alpha_{i,j}\beta_j-\alpha_{i,r}\beta_r: 1\leq i\leq n,\,\, 1\leq j< r\leq k\text{ and } i\neq j,\,\, i\neq r\right\rangle+\\ &
\left\langle \alpha_{i,r}\alpha_{j,r}\beta_{r}: 1\leq i\leq n,\,\, k+1\leq j\leq n\text{ and } 1\leq r\leq k,\,\, i\neq k \,\, i\neq j \right\rangle.
\end{array}  
\end{equation}
\begin{Lem}\label{lemma: new rings}
    For $1\leq k\leq n$, $S_k/J_k$ is isomorphic to $\mathcal{S}_k/\mathcal{J}_k$.
\end{Lem}
\begin{proof}
    First, we give a simplified list of generators of $J_k$. First of all, using the generators of the form 
    \begin{equation}\label{eq:new gens 1}
    \alpha_{i,j,u_j-1}\alpha_{j,j,l+1}-\alpha_{i,j,l},
     \end{equation}
    we may reduce the last ideal in \eqref{eq:gens ideal def 1} to 
     \begin{equation}\label{eq:new gens 2}
    \langle \alpha_{i,j,u_j-1}\alpha_{j,r,u_r-1}:\text{ for }j\in[k],i\in[n]\setminus\{j\} \text{ and } r\in[k]\setminus\{j\}\rangle.
    \end{equation}
    Similarly, the generators of the second ideal in \eqref{eq:gens ideal def 1} are reduced to 
    \begin{equation}\label{eq:new gens 3}
    \langle A_i\alpha_{j,r,u_r-1}: \text{ for } k+1\leq j\leq n,i\in [n]\setminus\{j\}  \text{ and }  r\in [k]\rangle.
    \end{equation}
    Using the generators \eqref{eq:new gens 3} and the third ideal in \eqref{eq:gens ideal def 1},
    we get that the generators of the first ideal in \eqref{eq:gens ideal def 1} are redundant.
    Now, 
    we conclude that $\mathcal{I}_k$ can be written as 
    \begin{equation}\label{eq:gens ideal def 1 new}
\begin{array}{ll}
\mathcal{I}_k=&
\langle A_i\alpha_{j,r,u_r-1}: \text{ for } k+1\leq j\leq n,i\in [n]\setminus\{j\} \text{ and }  r\in [k]\rangle\\
&+\langle \alpha_{i,j,u_j-1}\alpha_{j,j,1}-A_i:\text{ for }j\in[k]\text{ and }i\in[n]\setminus\{j\}\rangle\\
&+\langle \alpha_{i,j,u_j-1}A_j:\text{ for }j\in[k]\text{ and }i\in[n]\setminus\{j\}\rangle\\
&+\langle \alpha_{i,j,u_j-1}\alpha_{j,j,l+1}-\alpha_{i,j,l}:\text{ for }j\in[k],i\in[n]\setminus\{j\}\text{ and }l\in [u_j-2]\rangle\\
&+\langle \alpha_{i,j,u_j-1}\alpha_{j,r,u_r-1}:\text{ for }j\in[k],i\in[n]\setminus\{j\} \text{ and } r\in[k]\setminus\{j\}\rangle.\\
\end{array}
\end{equation}
Now, assume first that $k=1$. In this case, the ideal $\mathcal{I}_1$ is 
    \begin{equation}\label{eq:gens ideal def 1 k 1}
\begin{array}{ll}
\mathcal{I}_1=&
\langle A_i\alpha_{j,1,u_1-1}: \text{ for } 2\leq j\leq n\text{ and }i\in [n]\setminus\{j\}\\
&+\langle \alpha_{i,1,u_1-1}\alpha_{1,1,1}-A_i:\text{ for }2\leq i\leq n\rangle\\
&+\langle \alpha_{i,1,u_1-1}A_1:\text{ for }2\leq i\leq n\rangle\\
&+\langle \alpha_{i,1,u_1-1}\alpha_{1,1,l+1}-\alpha_{i,1,l}:\text{ for }2\leq i\leq n\text{ and }l\in [u_1-2]\rangle.
\end{array}
\end{equation}
From the generators of $\mathcal{I}_1$, we see that the variables $A_2,\ldots, A_n$ and $\alpha_{i,1,l}$ for $2\leq i\leq n$ and $l\in[u_1-2]$ are redundant. By eliminating these variables we get the ideal
\[
\begin{array}{ll}
\langle\alpha_{i,1,u_1-1}\alpha_{j,1,u_1-1}\alpha_{1,1,1}: \text{ for } 2\leq i< j\leq n\rangle+\langle \alpha_{i,1,u_1-1}A_1:\text{ for }2\leq i\leq n\rangle.
\end{array}
\]
 The proof follows from the fact that this ideal is the kernel of the map $S_1/J_1$ to $\mathcal{S}_1/\mathcal{J}_1$ given by
\[
\begin{array}{cccc}
A_1&\longmapsto& A_1, & \\
\alpha_{i,1,u_1}&\longmapsto&\alpha_i  & \text{for }2\leq i\leq n,\\
\alpha_{1,1,l}&\longmapsto & \alpha_{1,l}& \text{for } l\in[u_1-1].
\end{array}
\]

Now assume that $k\geq2$. Using the second type of generators in \eqref{eq:gens ideal def 1 new}, we get that 
\[
\alpha_{i,j,u_j-1}A_j = \alpha_{i,j,u_j-1}\alpha_{j,r,u_r-1}\alpha_{r,r,1}
\]
for $i\in [n]$, $j\in[k]\setminus\{i\}$, $r\in[k]\setminus\{j\}$. Note that such $r$ exists since $k\geq 2$. Therefore, we deduce that the generators of the form $\alpha_{i,j,u_j-1}A_j$ are contained in the last ideal in \eqref{eq:gens ideal def 1 new}. Similarly,  using the second type of generators in \eqref{eq:gens ideal def 1 new},
we get that 
\[
A_i\alpha_{j,r,u_r-1} = \alpha_{j,r,u_r-1} \alpha_{i,j,u_j-1} \alpha_{j,j,1}
\]
for $k+1\leq j\leq n$, $i\in [n]\setminus\{j\}$ and $r\in[k]$. We conclude that for $k\geq 2 $, we have 
 \begin{equation}\label{eq:gens ideal def 2 new}
\begin{array}{ll}
\mathcal{I}_k=&
\langle  \alpha_{i,j,u_j-1}  \alpha_{j,r,u_r-1}\alpha_{j,j,1}: \text{ for } k+1\leq j\leq n,i\in [n]\setminus\{j\} \text{ and }  r\in [k]\rangle\\
&+\langle \alpha_{i,j,u_j-1}\alpha_{j,j,1}-A_i:\text{ for }j\in[k]\text{ and }i\in[n]\setminus\{j\}\rangle\\

&+\langle \alpha_{i,j,u_j-1}\alpha_{j,j,l+1}-\alpha_{i,j,l}:\text{ for }j\in[k],i\in[n]\setminus\{j\}\text{ and }l\in [u_j-2]\rangle\\

&+\langle \alpha_{i,j,u_j-1}\alpha_{j,r,u_r-1}:\text{ for }j\in[k],i\in[n]\setminus\{j\} \text{ and } r\in[k]\setminus\{j\}\rangle.\\
\end{array}
\end{equation}
Now, consider the map from $\mathcal{S}_k$ to $S_k/J_k$ given by 
\[
\begin{array}{ccc}
    \alpha_{i,j} &\longmapsto& \alpha_{i,j,u_j-1},\\
    \beta_{i} &\longmapsto& \alpha_{i,1,1},\\
    \beta_{i,l} & \longmapsto & \alpha_{i,i,l}.
\end{array}
\]
This map it is surjective and its kernel is exactly the ring $\mathcal{J}_k$

\end{proof}

By \cref{theo:def theory} and \cref{lemma: new rings}, the schematic structure $\sHilb^m(X_n)$ around $[J]$ can be studied through the analysis of the scheme $\sp(\mathcal{S}_k/\mathcal{J}_k)$ around the origin. Since the variables $\beta_{i,s}$ do not appear in the generators of the ideal $\mathcal{J}_k$. It is enough to study the quotient of 
\[
\mathcal{S}'_k:= \C[ \alpha_{i,j},\beta_{i}: 1\leq i\leq n, 1\leq j\leq k  \text{ and } i\neq j]
\]
by $\mathcal{J}_k$. We now compute the primary decomposition of the ideal $\mathcal{J}_k$ in the following technical lemmas. We start with the case $k=1$.

\begin{Lem}\label{lemma: primary k1} 
     The primary decomposition of the ideal $\mathcal{J}_1$ is given by the ideals
     \begin{itemize}
         \item $\langle A_1,\alpha_{1,1}\rangle.$
         \item $\langle \alpha_2,\ldots,\alpha_n\rangle$.
         \item $\langle A_1,\alpha_j:2\leq j\leq n \text{ and } j\neq i\rangle$ for $2\leq i\leq n$.
     \end{itemize}
\end{Lem}
\begin{proof}
Assume first that $A_1\neq 0$. Then, from \eqref{eq:ideal 1}, we deduce that $\alpha_2=\cdots=\alpha_n=0$ and we get the component $\langle \alpha_2,\ldots,\alpha_n\rangle$. On the contrary, suppose that $A_1=0$, and assume that $\alpha_{1,1}\neq 0$. Then, from \eqref{eq:ideal 1} we derive that $\alpha_i\alpha_j=0$ for every $2\leq i<j\leq n$. This leads to the $n-1$ components of the form $\langle A_1,\alpha_j:2\leq j\leq n \text{ and } j\neq i\rangle$ for $2\leq i\leq n$. Finally, assume that $A_1=\alpha_{1,1}=0$. Then, every generator of  $\mathcal{J}_1$  vanishes. Hence, we conclude that  $\langle A_1,\alpha_{1,1}\rangle$ is the last irreducible component of $\mathcal{J}_1$.
\end{proof}

Now we continue with the case $k\geq 2$. To compute the primary decomposition in this case, we need the following lemmas:

\begin{Lem}\label{lemma: primary technical}
   For $S\subseteq [k]$ and $S\neq [n]$, the primary decomposition of the ideal 
   \begin{equation}\label{eq:extra ideal}
   \langle \alpha_{i,j}\alpha_{j,r}: j,r\in S, i\in [n],i\neq j, j\neq r\rangle
   \end{equation}
   is given by the ideals
   \[
  \mathcal{J}_{S,T} = \langle \alpha_{r,s}:r,s\in S\setminus T,r\neq s\rangle + \langle \alpha_{i,j}: \text{ for }j\in T \text{ and }i\in[n]\setminus\{j\}\rangle
   \]
   for every $T\subsetneq S$. For $S=[n]$, the primary decomposition of \eqref{eq:extra ideal} is given by the ideals $\mathcal{J}_{S,T}$ for $T\subsetneq S$ and $T\neq \emptyset$.
\end{Lem}
\begin{proof}
 Assume first that for every $j\in S$ there exists $i_j\in[n]\setminus\{j\}$ such that $\alpha_{i_j,j}\neq 0$. Note that this condition defines an open subset that we denote by $U_{S,\emptyset}$. From  $\alpha_{i_j,j}\alpha_{j,r}=0$, we deduce that $\alpha_{j,r}=0$ for every $j,r\in S$ and $r\neq j$. In particular, we get that the restriction of \eqref{eq:extra ideal} to $U_{S,\emptyset}$ is $\mathcal{J}_{S,\emptyset}$. In particular, since $U_{S,\emptyset}$ is open, we deduce that $\mathcal{J}_{S,\emptyset}$ is part of the primary decomposition

On the contrary, assume now that there exists $j_1\in S$ such that $\alpha_{i,j_0}=0$ for all $i\in[n]\setminus\{j_0\}$. 
We distinguish two cases. First, we assume that for every $j\in S\setminus\{j_0\}$ there exists $i_j\in [n]\setminus\{j\}$ such that $\alpha_{i_j,j} \neq 0$. We denote the set defined by these constraints by $U_{S,\{j_1\}}$. 
As before, we get that $\alpha_{j,r}=0$ for every $j\in S\setminus\{j_0\}$ and $r\in S\setminus\{j\}$. Therefore, the restriction of \eqref{eq:extra ideal} to these constraints is $\mathcal{J}_{S,\{j_1\}}$. Since the ideal $\mathcal{J}_{S,\emptyset}$ is not contained in $\mathcal{J}_{S,\{j_1\}}$ and $U_{S,\{j_1\}}$ is an open subset in the complement of $U_{S,\emptyset}$. we deduce that $\mathcal{J}_{S,\{j_1\}}$  is in the primary decomposition.

Secondly, we assume the contrary. In other words, we assume that there exists $j_2\in S\setminus\{j_1\}$ such that $\alpha_{i,j_2}=0$ for all $i\in[n]\setminus \{j_2\}$. We distinguish again two cases. First, we assume that for every $j\in S\setminus\{j_0\}$ there exists $i_j\in [n]\setminus\{j\}$ such that $\alpha_{i_j,j} \neq 0$. Arguing as before, we get that \eqref{eq:extra ideal} restricted to these constraints is $\mathcal{J}_{S,\{j_1,j_2\}}$ and it forms part of the primary decomposition. Secondly, we assume that there exists $j_3\in S\setminus \{j_1,j_2\}$ such that $\alpha_{i,j_3}=0$ for all $i\in[n]\setminus\{j_3\}$. Recursively, applying these restrictions we get that $\mathcal{J}_{S,T}$ for $T\subsetneq S$ appears in the primary decomposition of \eqref{eq:extra ideal}. Note that if $T=S$, the ideal $\mathcal{J}_{S,S}$ is generated by all $\alpha_{i,j}$ and therefore it is not in the primary decomposition.

Finally, for $S=[n]$, we get the same primary decomposition with one distinction. In this case, $\mathcal{J}_{[n],\emptyset}$ is also generated by all $\alpha_{i,j}$. Therefore, it does not appear in the primary decomposition.
\end{proof}

\begin{Lem}\label{Lem: toric ideal}
    For $k\geq 2$, the ideal 
    \begin{equation}\label{eq:toric ideal}
   I_k:= \langle a_1b_1-a_ib_i:2\leq i\leq k\rangle\end{equation}
    in $\C[a_1,\ldots,a_k,b_1,\ldots,b_k]$ is toric of Krull dimension $k+1$. 
\end{Lem}
\begin{proof}
    It suffices to prove that the binomial ideal $I_k=\langle a_1b_1-a_ib_i:2\leq i\leq k\rangle$ is  prime of Krull dimension $k+1$, we will do it by induction on $k$. 
    For $k=2$, $I_k$  is the ideal of the cone over Segre variety in $\P^3$, and therefore, it is prime of Krull dimension $3$. Now, assume that $I_{k'}$ is prime for all $k'<k$. Assume first that $a_1\neq 0$, then $b_1=\frac{a_2b_2}{a_1}$. After eliminating the variable $b_1$, we get the ideal $\langle a_2b_2-a_ib_i:3\leq i\leq n\rangle$ which by induction is prime of dimension $k$. Therefore, in the open subset $a_1\neq 0$ we get a unique reduced irreducible component of dimension $k+1$.
    Assume now that $a_1 = 0$. Then, $a_ib_i=0$ for every $2\leq i\leq k$. The irreducible components of the corresponding variety are given by the equations $a_1=a_i=b_j=0$ for $i\in S_1$ and $j\in S_2$ with $S_1\sqcup S_2=\{2,\ldots ,k\}$. All these components lie in the boundary of the component obtained by assuming $a_1\neq 0$.
\end{proof}

We now give a description of the projective toric variety defined by the ideal in \cref{Lem: toric ideal}.

\begin{Prop}
    \label{prop: toric variety}
    The polytope $P_k$ associated to the toric ideal \eqref{eq:toric ideal} is 
    \begin{equation}
        \label{eq:conv hull}
       P_k = \mathrm{Conv}\left(
        \mathbf{0},\mathbf{e}_1,\mathbf{e}_2,\mathbf{e}_1+\mathbf{e}_2,\mathbf{e}_i,\mathbf{e}_1+\mathbf{e}_2-\mathbf{e}_i:\text{ for }i\in\{3,\ldots,k\}
        \right).
    \end{equation}
    The facets of $P_k$ are the convex hull of the vertices
    \[
    V_0\cup \{\mathbf{e}_i:i\in S_1\}\cup \{\mathbf{e}_1+\mathbf{e}_2-\mathbf{e}_i:i\in S_2\},
    \]
    for $S_1\sqcup S_2=\{3,\ldots,k\}$ and $V_0 = \{\mathbf{0},\mathbf{e}_1\}$, $\{\mathbf{0},\mathbf{e}_2\}$, $\{\mathbf{e}_1,\mathbf{e}_1+\mathbf{e}_2\}$ or  $\{\mathbf{e}_2,\mathbf{e}_1+\mathbf{e}_2\}$. In particular, the facets of $P_k$ are simplices.
\end{Prop}
\begin{proof}
    We find a monomial parametrization of the projective toric variety defined by \eqref{eq:toric ideal}. Assume that $a_1=1$ and $a_i\neq 0$ for $i\neq 1$. Then, $b_1=a_2b_2$ and $b_i = b_1a_i^{-1}=a_2b_2a_i^{-1}$ for $3\leq i\leq k$. Fixing the coordinates of $\P^{2k-1}$ as $[a_1,\ldots,a_k,b_1,\ldots,b_k]$, the monomial parametrization is given by 
    \[
    \begin{array}{ccc}
    \left(\C^{*}\right)^{k}&\longrightarrow &\P^{2k-1}\\
    t=(t_1,\ldots,t_k)&\longmapsto&[1,t_2,\ldots,t_k,t_1t_2,t_1,t_1t_2t_3^{-1},\ldots,t_1t_2t_k^{-1}].
    \end{array}
    \]
    where $a_1=1$, $a_i=t_i$ for $2\leq i\leq k$ and $b_2=t_1$.
    The description of $P_k$ by \eqref{eq:conv hull} follows from the exponents of this monomial map. 

    To find the faces of the polytope $P_k$, we minimize the scalar product $<\uu,->$ by a vector $\uu=(u_1,\ldots,u_k)\neq \mathbf{0}$ over $P_k$. Let $V$ be the set of vertices among the ones in \eqref{eq:conv hull} where the minimum of $<\uu,->$ is achieved. Assume first that such minimum is $0$. In other words, $\mathbf{0}$ is contained in $V$. This implies that $u_i\geq 0$ for every $i\in [k]$. First, we claim that $\mathbf{e}_1$ and $\mathbf{e}_2$ are not contained simultaneously in $V$. Indeed, assume  that $\mathbf{e}_1,\mathbf{e}_2\in V$. Then, we have that $u_1=u_2=0$. Therefore, $\langle\uu, \mathbf{e}_1+\mathbf{e}_2-\mathbf{e}_i\rangle=-u_i\leq 0$ for $3\leq i\leq k$. Since the minimum is $0$, we deduce $u_i=0$ for $3\leq i\leq k$ and $\uu=\mathbf{0}$. Therefore, $\mathbf{e}_1$ and $\mathbf{e}_2$ are not contained simultaneously in $V$.
    Now, $\mathbf{e}_1+\mathbf{e}_2$ is not contained in $V$. 
    Indeed, if $\mathbf{e}_1+\mathbf{e}_2\in V$, then $\langle\uu, \mathbf{e}_1+\mathbf{e}_2\rangle = u_1+u_1=0$. Since $u_1,u_2\geq 0$, we get that  $u_1=u_2=0$, and hence, $\mathbf{e}_1,\mathbf{e}_2\in V$. We  conclude that $\mathbf{e}_1+\mathbf{e}_2$ is not contained in $V$. 
    Similarly, we claim that $\mathbf{e}_i$ and $\mathbf{e}_1+\mathbf{e}_2+\mathbf{e}_i$ for $3\leq i\leq k$ are not contained simultaneously in $V$. Assume that $\mathbf{e}_i,\mathbf{e}_1+\mathbf{e}_2+\mathbf{e}_i\in V$ for $3\leq i\leq k$. Then, $u_i=0$ and $u_1+u_2-u_i=u_1+u_2=0$. This implies that the minimum is also achieved at $\mathbf{e}_1+\mathbf{e}_2$, and thence, $\mathbf{e}_1+\mathbf{e}_2\in V$. 

    Now assume that the minimum of $\langle\uu,-\rangle$ is obtained at a facet of $P_k$. This implies that $V$ must contain at least $k$ vertices. Since $\mathbf{e}_i$ and $\mathbf{e}_1+\mathbf{e}_2-\mathbf{e}_i$ are not simultaneously contained in $V$, we deduce that among the vertices $\{\mathbf{e}_i,\mathbf{e}_1+\mathbf{e}_2-\mathbf{e}_i:3\leq i\leq k\}$ only $k-2$ can be simultaneously in $V$. Among the other $4$ vertices, only $\mathbf{0}$ and $\mathbf{e}_1$ or  $\mathbf{0}$ and $\mathbf{e}_2$ may be simultaneously contained in $V$. We conclude that $V$ must be of the form 
    \[
    V=V_0\cup \{\mathbf{e}_i:i\in S_1\}\cup \{\mathbf{e}_i:i\in S_2\},
    \]
    for $S_1\sqcup S_2=\{3,\ldots,n\}$ and $V_0 = \{\mathbf{0},\mathbf{e}_1\}$ or $\{\mathbf{0},\mathbf{e}_2\}$.

    Assume now that the minimum of $\langle\uu,-\rangle$ is strictly negative. Then, $\mathbf{0}\not \in V$. As before, if $\mathbf{e}_1$ and $\mathbf{e}_2$ are simultaneously contained in $V$, then the minimum is $u_1=u_2<0$. This is a contradiction since $\langle \uu,\mathbf{e}_1+\mathbf{e}_2\rangle =u_1+u_2<u_1$. Therefore, $\mathbf{e}_1$ and $\mathbf{e}_2$ are  not simultaneously contained in $V$. Similarly, assume that $\mathbf{e}_i$ and $ \mathbf{e}_1+\mathbf{e}_2-\mathbf{e}_i$ are simultaneously contained in $V$. Then, the minimum would be $u_i=u_1+u_2-u_i$. In particular, $\langle \uu, \mathbf{e}_1+\mathbf{e}_2\rangle u_1+u_2=2u_i<u_i $, which is a contradiction since the minimum is $u_i$. In this case, we conclude that $V$ must be of the form 
      \[
    V=V_0\cup \{\mathbf{e}_i:i\in S_1\}\cup \{\mathbf{e}_i:i\in S_2\},
    \]
    for $S_1\sqcup S_2=\{3,\ldots,n\}$ and $V_0 = \{\mathbf{e}_1,\mathbf{e}_1+\mathbf{e}_2\}$ or $\{\mathbf{e}_2,\mathbf{e}_1+\mathbf{e}_2\}$.
    Finally, the facets of $P_n$ are simplices since they have dimension $k-1$ and they are the convex hull of $k$ vertices.
\end{proof}

\begin{Lem}
    \label{lem:normal polytope}
    For $n\geq 2$, the polytope $P_k$ is normal. In particular, the affine cone of $\mathbb{V}(I_k)$ is a normal affine variety.
\end{Lem}
\begin{proof}
We use the fact that a polytope admitting a unimodular simplicial subdivision is normal, and we show that $P_k$ admits such a decomposition by induction on $k$. 
For $k=2$, $P_2=\mathrm{Conv}(\mathbf{0},\mathbf{e}_1,\mathbf{e}_2,\mathbf{e}_1+\mathbf{e}_2)$,  which is normal. In this case, the subdivision is given by the simplices $\mathrm{Conv}(\mathbf{0},\mathbf{e}_1,\mathbf{e}_2)$ and $\mathrm{Conv}(\mathbf{e}_1,\mathbf{e}_2,\mathbf{e}_1+\mathbf{e}_2)$. Assume now that $P_{k-1}$ admits a unimodular simplicial subdivision $\mathcal{P}_{k-1}$ and let $\Delta$ be a simplex in the subdivision.
Since $P_{k-1}$ is contained in $P_{k}$, we consider the simplex $\Delta^+$ obtained by taking the convex hull of $\Delta$ and $\mathbf{e}_n$. Similarly, we consider the simplex $\Delta^{-}$ obtained by taking the convex hull of $\Delta$ and $\mathbf{e}_1+\mathbf{e}_2-\mathbf{e}_n$. We claim that the 
$
\mathcal{P}_k:=\{\Delta^+,\Delta^-:\Delta\in\mathcal{P}_{k-1}\}
$
is a unimodular simplicial subdivision of $P_k$. Indeed, 
by construction the normalized volume of $\Delta^+$ and $\Delta^-$ is one. Therefore, $\mathcal{P}_k$ is unimodular and simplicial. Now, $P_k$ is the union of the polytopes
\[
P_k^+:=\mathrm{Conv}(P_{k-1},\mathbf{e}_n) \text{ and } P_k^-:=\mathrm{Conv}(P_{k-1},\mathbf{e}_1+\mathbf{e}_2-\mathbf{e}_n).
\]
The polytopes $P_k^+$ and $P_k^-$ are subdivided by $\Delta^+$ and $\Delta^-$ respectively for $\Delta\in\mathcal{P}_{k-1}$. Therefore, $\mathcal{P}_k$ is a unimodular simplicial subdivision of $P_k$, and hence, $P_k$ is normal. 
\end{proof}

\begin{Lem}\label{lemma:primary k2}
 The primary decomposition of the ideal $\mathcal{J}_k$ for $k\geq 2$ is given by the ideals
 \begin{itemize}
     \item For every $i\in[n]$,
     \[
     \mathcal{Q}_i:=\langle \alpha_{i,r}\beta_{r}-\alpha_{i,s}\beta_{s}\,:\,r,s\in[k], \, r,s\neq i\rangle + \langle \alpha_{r,s}\,:\,r\in[n]\setminus \{i\},s\in[k],r\neq s\rangle.
     \]
     \item For every $S\subseteq [k]$ with $1\leq |S|\leq \min\{k,n-2\}$ 
     \[
     \mathcal{J}_S := \langle \beta_j:j\in S\rangle + \langle \alpha_{j,r}: j\in S, r\in[k], j\neq r\rangle +\langle \alpha_{r,s}:r\in [n]\setminus S, s\in[k]\setminus S, r\neq s\rangle.
     \]
 \end{itemize}
\end{Lem}
\begin{proof}
We stratify the affine space given by $\mathcal{S}_k$ through the subsets $U_S$ for $S\subseteq [k]$
\[
U_S = \{\beta_j= 0:j\in S\}\cap \{\beta_j\neq 0: j\not\in S\}
\]
and we check the irreducible components of $\mathbb{V}(\mathcal{J}_k)$ restricted to each $U_S$. We first focus on the case $U_\emptyset$. In other words, $S=\emptyset$ and we assume that $\beta_j\neq 0$ for every $j\in [k]$. Under this assumption, the ideal $\mathcal{J}_k$ becomes
\[
\begin{array}{ll}
\left\langle \alpha_{i,j}\alpha_{j,r}: 1\leq i\leq n,\,\, 1\leq j, r\leq k\text{ and } i\neq j,\,\,j\neq r\right\rangle+\\ 
\left\langle \alpha_{i,j}\beta_j-\alpha_{i,r}\beta_r: 1\leq i\leq n,\,\, 1\leq j< r\leq k\text{ and } i\neq j,\,\, i\neq r\right\rangle+\\ 
\left\langle \alpha_{i,r}\alpha_{j,r}: 1\leq i\leq n,\,\, k+1\leq j\leq n\text{ and } 1\leq r\leq k,\,\, i\neq k \,\, i\neq j \right\rangle.
\end{array}
\]
Without loss of generality, we may further assume that $\alpha_{i_0,j_0}\neq 0$ for some $i_0\in[n],j_0\in[k]$ with $i_0\neq j_0$. 
From the generators of $\mathcal{J}_k$, we deduce that $\alpha_{j_0,r}=0$ for $r\in[k]\setminus\{j_0\}$. For every $r\in [k]\setminus\{j_0\}$ we also have that 
\[
\alpha_{i_0,j_0}\beta_{j_0} = \alpha_{i_0,r}\beta_{r}.
\]
Since $\alpha_{i_0,j_0}\beta_{j_0},\beta_{r}\neq 0$, we deduce that $\alpha_{i_0,r}\neq 0$ for every $r\in[k]$. In particular, we obtain that $\alpha_{r,s}=0$ for $r,s\in [k]$ with $r\neq i_0,s$. Similarly, from the generators of $\mathcal{J}_k$, we also get that $\alpha_{i,j_0}=0$ for $i\geq k+1$ and $i\neq i_0$. Thus, we get that 
\[
0 = \alpha_{i,j_0}\beta_{j_0} = \alpha_{i,r}\beta_{r}.
\]
for $r\in[k]$, $i\geq k+1$ and $i\neq 0$. Since $\beta_r\neq 0$, we deduce that $\alpha_{i,r} = 0$ for $r\in [k]$, $i\geq k+1$ and $i\neq i_0$. From all these conditions, we deduce that the restriction of $\mathcal{J}_k$ to the open subsets $U_\emptyset\cap \{\alpha_{i_0,j_0}\}$ is $\mathcal{K}_{i_0}$ which is prime by  \cref{Lem: toric ideal}.

Next, we compute the primary components of $\mathcal{J}_k$ in $U_S$ with $|S|=1$. In other words, assume that there exists $j_0\in [k]$ such that $\beta_{j_0}=0$ and $\beta_j\neq 0$ for $j\in[k]\setminus\{j_0\}$. Since $k\geq 2$, we get that for every $j\in[k]\setminus\{j_0\}$ and $i\in[n]\setminus\{j,j_0\}$ 
\[
 0 = \alpha_{i,j}\beta_{j}-\alpha_{i,j_0}\beta_{j_0} = \alpha_{i,j}\beta_{j}.
\]
Since $j\neq j_0$, we deduce that $\alpha_{i,j}=0$ for $j\in[k]\setminus\{j_0\}$ and $i\in[n]\setminus\{j,j_0\}$. In particular, the restriction of $\mathcal{J}_k$ to $U_{\{j_0\}}$ is given by 
\begin{equation}\label{eq:sum ideals}
\begin{array}{c}
\langle\beta_{j_0}\rangle+ \left\langle \alpha_{i,j_0}\alpha_{j_0,r}: r\in[k]\setminus\{j_0\} \text{ and } i\in[n]\setminus\{j_0\}\right\rangle+
\left\langle \alpha_{j_0,j}\beta_j-\alpha_{j_0,r}\beta_r: r,j\in[k]\setminus\{j_0\}\right\rangle+\\\langle \alpha_{i,j}:i\in[n]\setminus\{j_0\},j\in[k]\setminus\{j_0\},i\neq j\rangle .
\end{array}
\end{equation}
We distinguish two cases:
\begin{itemize}
    \item Assume first that $\alpha_{j_0,r}\neq 0$ for every $r\in[k]\setminus\{j_0\}$. From \eqref{eq:sum ideals} we deduce that $\alpha_{i,j_0}=0$ for every $i\neq j_0$. One may check that under this assumption, $\mathcal{J}_k$ becomes the ideal $\mathcal{Q}_{j_0}+\langle\beta_{j_0}\rangle$. Therefore, it does not lead to a new component of the primary decomposition.
    \item On the contrary, assume that there exists $r\in[k]\setminus\{j_0\}$ with $\alpha_{j_0,r_0}=0$. Then, we get that 
    \[
    0=\alpha_{j_0,r}\beta_r- \alpha_{j_0,r_0}\beta_{r_0} = \alpha_{j_0,r}\beta_r
    \]
    for every $r\in[k]\setminus{j_0}$. Since $\beta_r\neq 0$, we deduce that $\alpha_{j_0,r}=0$ for every $r\in [k]\setminus\{j_0\}$. In particular, in this case, $\mathcal{J}_k$ becomes
    \[
    \langle \beta_{j_0}\rangle + \langle \alpha_{j_0,r}:r\in[k]\setminus\{j_0\}\rangle +\langle \alpha_{i,j}:i\in[n]\setminus\{j_0\},j\in[k]\setminus\{j_0\},i\neq j\rangle =\mathcal{J}_{\{j_0\}}.
    \]
    Since $\mathcal{J}_{\{j_0\}}$ does not contain any of the $\mathcal{Q}_i$, it forms part of the irreducible decomposition of $\mathcal{J}_k$.
\end{itemize}

Next, we apply induction on the cardinality of $S\subseteq [k]$ with $|S|\leq n-2$. Let $ 2\leq a\leq \min\{k,n-2\}$ and assume that $\mathcal{J}_S$ is part of the primary decomposition of $\mathcal{J}_k$ for every $S\subseteq[k]$ with $|S|<a$. Let $S\subseteq[k]$ with $|S|=a$ and restrict $\mathcal{J}_k$ to $U_S$. In other words, assume that $\beta_i=0$ for $i\in S$ and $\beta_i\neq 0$ for $i\not \in S$. We show that the only ideal in the primary decomposition of $\mathcal{J}_k$ appearing in $U_S$ is $\mathcal{J}_S$.
For $r\in[k]\setminus S$, $j\in S$ and $i\in[n]\setminus\{j,r\}$ we get 
\[
0=\alpha_{i,r}\beta_r-\alpha_{i,j}\beta_j = \alpha_{i,r}\beta_r.
\]
Since $r\not\in S$, $\beta_r\neq 0$ and we get $\alpha_{i,r}=0$ for every $r\in[k]\setminus S$ and $i\in[n]\setminus\{r\}$. Therefore, the restriction of $\mathcal{J}_k$ to $U_S$ is 
\begin{equation}\label{eq:restriction Us}
\langle\beta_j:j\in S\rangle +\langle \alpha_{i,r}:r\in[k]\setminus S,i\in [n]\setminus\{j\}\rangle +
\langle \alpha_{i,j}\alpha_{j,r}:j,r\in S,i\in[n], i\neq j\text{ and }j\neq r\rangle .
\end{equation}
By \cref{lemma: primary technical}, the primary decomposition of \eqref{eq:restriction Us} is given by the ideals
\begin{equation}\label{eq:primary restriction}
\langle\beta_j:j\in S\rangle +\langle \alpha_{i,r}:r\in[k]\setminus S,i\in [n]\setminus\{j\}\rangle + \mathcal{J}_{S,T}
\end{equation}
for $T\subsetneq S$. One may check that for $T\neq \emptyset$, we have that $\mathcal{J}_{S,T}$ contains the ideal $\mathcal{J}_{S\setminus T}$. In particular, the ideal \eqref{eq:primary restriction} does not appear in the primary decomposition for $T\neq \emptyset$
For $T=\emptyset$, we obtain that $\mathcal{J}_{S}=\mathcal{J}_{S,\emptyset}$. Moreover, $\mathcal{J}_S$ does not contain and it is not contained in any of the ideals $\mathcal{Q}_i$ and $\mathcal{J}_{S'}$ for $S'\subseteq[k]$ with $|S'|<a$. We conclude that $\mathcal{J}_{S}$ is in the primary decomposition of $\mathcal{J}_k$.

It remains to show that $\mathcal{J}_{S}$ does not appear in the primary decomposition for $|S|= n-1,n$. For $|S|=n-1$, we denote  by $i_S$ the only integer in $[n]\setminus S$. Then, we get that $\mathcal{J}_S$ contains $\mathcal{Q}_{i_S}$. Therefore, it does not appear in the primary decomposition. 
\end{proof}

\bibliography{bib.bib}{}
\bibliographystyle{alpha}
\end{document}